\documentclass[11pt]{article}
\usepackage[a4paper,margin=1in]{geometry}
\usepackage[T1]{fontenc}
\usepackage[utf8]{inputenc}
\usepackage{lmodern,microtype,amsmath,amssymb,amsthm,amscd,mathtools,mathrsfs}
\usepackage{booktabs,tabularx,array,enumitem,etoolbox,xcolor,needspace}
\usepackage{fancyhdr}
\usepackage[colorlinks=true,linkcolor=blue!45!black,citecolor=blue!45!black,urlcolor=blue!45!black]{hyperref}
\usepackage[capitalise,noabbrev]{cleveref}
\hypersetup{pdftitle={Higher Schwarzians},pdfauthor={Amir Jafari}}
\renewcommand{\sectionmark}[1]{}
\renewcommand{\subsectionmark}[1]{}
\newtheorem{maintheorem}{Theorem}

\newtheorem{theorem}{Theorem}[section]
\newtheorem{proposition}[theorem]{Proposition}
\newtheorem{lemma}[theorem]{Lemma}
\newtheorem{corollary}[theorem]{Corollary}
\theoremstyle{definition}
\newtheorem{definition}[theorem]{Definition}
\newtheorem{example}[theorem]{Example}
\newtheorem{convention}[theorem]{Convention}
\theoremstyle{remark}
\newtheorem{remark}[theorem]{Remark}
\numberwithin{equation}{section}
\setlist[enumerate]{label=\textup{(\roman*)},leftmargin=2em}
\makeatletter
\patchcmd{\l@section}{1.0em}{0.65em}{}{}
\patchcmd{\l@part}{2.25em}{1.75em}{}{}
\makeatother
\apptocmd{\thebibliography}{\setlength{\itemsep}{3pt}\setlength{\parsep}{0pt}}{}{}
\newcommand{\Q}{\mathbb Q}
\newcommand{\Z}{\mathbb Z}
\newcommand{\C}{\mathbb C}
\newcommand{\Sch}{\operatorname{Sch}}
\newcommand{\OO}{\mathcal O}
\newcommand{\Id}{\operatorname{Id}}
\newcommand{\End}{\operatorname{End}}
\newcommand{\Hom}{\operatorname{Hom}}
\newcommand{\tr}{\operatorname{tr}}
\newcommand{\res}{\operatorname{res}}
\newcommand{\cub}{\mathcal S}
\newcommand{\SZ}{Szeg\H{o}}
\newcommand{\Hg}{\mathfrak H_g}
\newcommand{\EE}{\mathcal E}
\newcommand{\HH}{\mathcal H}
\newcommand{\nD}{\nabla_t}
\newcommand{\cov}{\mathcal D}
\newcommand{\VV}{\mathcal V}
\newcommand{\MM}{\mathscr M}
\newcommand{\Sp}{\operatorname{Sp}}
\renewcommand{\AA}{\mathscr A}
\newcommand{\RR}{\mathscr T}
\newcommand{\SSS}{\mathscr S}
\newcommand{\Sym}{\operatorname{Sym}}
\newcommand{\GL}{\operatorname{GL}}
\newcommand{\dZ}{\,dZ}
\newcommand{\Dt}{\mathsf D}
\newcommand{\mer}{\mathrm{mer}}
\title{Higher Schwarzians}
\author{Amir Jafari}
\date{}
\begin{document}
\markboth{Higher Schwarzians}{}
\maketitle
\begin{abstract}
The classical Schwarzian can be computed directly from the coefficients
of a second-order linear equation. We extend this coefficient viewpoint
to monic operators of arbitrary order over unital associative
differential algebras over the rationals. The normalized coefficients
freely generate the universal gauge covariants under covariant
differentiation. Projective normalization in the scalar case, followed
by symmetrization, gives higher Schwarzians with compatible finite laws
under changes of frame and parameter.

On a complex curve, these invariants, together with the coefficient
connection and quadratic datum, reconstruct the operator. On a compact
curve of genus $g\ge2$, canonical Szeg\H{o} kernels turn the construction
into a tool for recovering bundle geometry. A bracket of the cubic
invariant recovers marked extensions of acyclic line bundles with
distinct cubics. Higher invariants through degree $4g-1$ separate acyclic
line bundles and detect their self-extensions. For acyclic bundles of arbitrary
rank, the full invariant jets at a point generate the associative span
of the canonical monodromy.

The same universal coefficient calculus gives automorphic tensors on
Siegel modular varieties.
For genus-two filtered symplectic connections with invertible
Kodaira--Spencer map, matrix Schwarzian
invariants determine where scalar equations selected by the Hodge
filtration acquire apparent singularities. Exact Picard--Fuchs
calculations exhibit these divisors and express their equations in
Siegel modular forms. The exact calculations are supplied in a
reproducible ancillary package.
\end{abstract}

\begingroup\small
\par\endgroup

\section{Introduction}\label{intro:main}

Consider a second-order linear equation
$y''+2a_1y'+a_2y=0$. A quotient $y_1/y_2$ of two independent solutions defines
a local projective coordinate: changing the solution basis changes
the quotient by a M\"obius transformation. For a function $f$ with
nonvanishing derivative, the Schwarzian
\begin{equation}\label{I:eq:schwarzian}
 \Sch(f)=\frac{f'''}{f'}-\frac32\left(\frac{f''}{f'}\right)^2
\end{equation}
is unchanged by that transformation. It can also be read directly
from the equation. Where $y_2\ne0$, one has
\begin{equation}\label{I:eq:klein}
 \Sch(y_1/y_2)=2I_2,\qquad I_2=a_2-a_1'-a_1^2.
\end{equation}
Thus the projective differential invariant is available before the
solutions are known. Klein's use of this identity in the icosahedron
problem is a classical example
\cite[Part I, Chap.~III, \S\S6--8]{Klein1888}; see also
\cite{OT2009,OT}.

The coefficient formula suggests extending the construction to
higher-order equations and to coefficients which need not commute.
The latter include the matrix equations arising from vector bundles
and periods. We ask which differential expressions in the coefficients
have intrinsic transformation laws, and what geometric information
those expressions determine. Part~\ref{part:I} constructs the universal
expressions. Part~\ref{part:II} uses them to recover operators and bundle
extensions on curves. Part~\ref{part:III} applies them to modular tensors
and to scalar equations selected by the Hodge filtration.

\subsection{From the coefficient formula to universal invariants}

A change of frame acts on an operator by conjugation. An expression
in its coefficients is a \emph{gauge covariant} if it changes by the
same conjugation. We seek universal covariants: differential
polynomials, with the order of multiplication retained, whose law
holds in every unital associative differential $\Q$-algebra.

The first normalization has a simple form. Write a monic operator of
order $N\ge2$ in the binomial convention and expand it in powers of
the adapted first-order operator:
\begin{equation}\label{intro:normalization}
 L=\sum_{k=0}^N\binom Nk a_kD^{N-k}
  =\sum_{k=0}^N\binom Nk I_k(D+a_1)^{N-k},
 \qquad a_0=I_0=1,\quad I_1=0.
\end{equation}
Here $Da=aD+\delta(a)$, where $\delta$ is coefficient differentiation.
The second expansion defines the normalized coefficients $I_k$;
its quadratic member is \eqref{I:eq:klein}. Their covariant derivative
is $\Delta C=\delta(C)+[a_1,C]$, with $[a,b]=ab-ba$.

\begin{maintheorem}[Universal coefficient invariants]
\label{intro:universal}
For fixed order $N$, the algebra of universal gauge covariants is the
free differential associative algebra on $I_2,\ldots,I_N$, with
derivation $\Delta$.

For each $3\le m\le N$, adding to $I_m$ a universal correction in the
lower normalized coefficients and their covariant derivatives gives
a higher Schwarzian $W_m$. It transforms by conjugation under a
change of frame and as an $m$-differential under a change of parameter.
The representative is specified by scalar projective normalization
and full symmetrization.
\end{maintheorem}

The first assertion says that every universal gauge covariant is a
unique polynomial in ordered products of the covariant derivatives
of the normalized coefficients. Its proof removes $a_1$ by a formal
change of frame, where the remaining coefficient derivatives are
free, and then restores covariance
(Theorem~\ref{I:thm:local-generation}).

Parameter changes present a further obstruction. The quadratic
coefficient acquires an additive Schwarzian term; in the scalar case
it defines a projective connection. A projective parameter sets the
scalar quadratic datum to zero and determines the corrections in
higher degrees. Full symmetrization gives an associative representative
with the same finite parameter law
(Theorem~\ref{I:thm:rp-higher}). The first higher member is
$W_3=I_3-\tfrac32\Delta I_2$. The symmetrization convention matters:
the covariance laws alone allow additional commutator terms.

The scalar starting point is Wilczy\'nski's theory of differential
equations and relative invariants
\cite{Wilczynski1901,Wilczynski1906,Bol}. Noncommutative Schwarzians
were studied in \cite{RetakhShander1993,RetakhRubtsovSharygin2019},
and generalized Wilczy\'nski invariants for nonlinear equations and
systems in \cite{Doubrov2008,DoubrovMedvedev2014}. Our construction
concerns ordered coefficient polynomials in every associative
differential algebra, with universal generation and compatible finite
transformation laws in all orders. Higher Schwarzians of locally
univalent functions, as in \cite{Tamanoi,KimSugawa}, concern a
different input. For scalar and matrix coefficients we also identify
our normalized representatives with the classical primary currents
of \cite{DiFrancescoItzyksonZuber1991,Bilal1995}.

\subsection{From differential invariants to bundle geometry}

The transformation laws allow the local expressions to glue on a
complex curve. With the half-form twists specified in
Section~\ref{II:sec:curve-schwarzians}, the first operator coefficient
defines a connection on the coefficient bundle, the quadratic member
defines an affine quadratic datum, and the higher Schwarzians are
endomorphism-valued differentials. These data determine the operator:
one recovers each coefficient from the preceding ones, and uniqueness
glues the local inverses. This gives the holomorphic and meromorphic
reconstruction correspondence of
Theorem~\ref{II:thm:curve-reconstruction}.

To obtain invariants of a bundle itself, one needs a distinguished
source of coefficients. Let $X$ be a compact complex curve of genus
$g\ge2$. A holomorphic bundle $E$ is \emph{acyclic} if
$H^0(X,E)=H^1(X,E)=0$. Its canonical Szeg\H{o} kernel has identity
residue along the diagonal; expanding near the diagonal gives the
required differential data. The kernel and its connection come from
Ben-Zvi--Biswas \cite{BB}. We identify the resulting cubic differential
with the first higher Schwarzian and denote it by $\cub_E$.

For matrix coefficients, the cubic and its covariant derivative need
not commute. Six times their commutator is the first self-bracket of the
cubic. On an extension of line bundles this bracket has only an
off-diagonal entry. Although its trace and determinant vanish, that
entry can determine the entire extension class. A \emph{marked}
extension keeps fixed its inclusion of the first line bundle and its
identification of the quotient with the second.

\begin{maintheorem}[Recovery of marked extensions]
\label{intro:extensions}
Let $L,M$ be acyclic line bundles on $X$ with
$\cub_L\ne\cub_M$, and let $K_X$ be the canonical bundle.
The cubic self-bracket of a marked extension gives a linear injection
\begin{equation}\label{intro:extension-injection}
 H^1(X,L\otimes M^{-1})\lhook\joinrel\longrightarrow
 H^0(X,L\otimes M^{-1}\otimes K_X^7).
\end{equation}
Its value determines the extension class, and it vanishes precisely
when the extension splits.
\end{maintheorem}

The inverse has a local geometric meaning. Dividing by six times the square of
the difference of the endpoint cubics gives a meromorphic one-form.
Its covariant primitives have principal parts whose differences
recover the extension cocycle
(Theorem~\ref{II:thm:extension-recovery}).

Higher invariants treat pairs that the cubic does not separate.
After fixing a holomorphic projective connection, degrees
$2,\ldots,4g-1$ separate all acyclic line bundles. The same finite
hierarchy detects their self-extensions through its nilpotent
coefficients; the self-brackets vanish for these self-extensions.
For an acyclic bundle of arbitrary rank, the algebra generated by
all invariant jets at one point equals the associative span of the
canonical connection's monodromy
(Section~\ref{II:sec:higher-detection}). The finite degree bound is
specific to line bundles.

On a modular curve the differentials become algebra-valued modular
forms. Their ordered Rankin--Cohen operations extend the
Connes--Moscovici and higher Serre calculus \cite{CM,NSZ}, within the
bundle framework for vector-valued forms of \cite{CF}.
The example on $X_0(23)$ makes the recovery concrete: one Fourier
coefficient of a weight-fourteen matrix-valued cusp form determines
the marked extension coordinate
(Section~\ref{II:sec:x023}).

\subsection{From matrix equations to scalar period equations}

A flat connection can have a regular matrix presentation while a
scalar equation for one of its periods has extra poles. These
\emph{apparent singularities} arise from the scalar presentation;
the underlying local system is nonsingular there. To locate them,
one must keep track of the geometric choice that produces the scalar
equation.

Consider a rank-four flat symplectic bundle on a curve, with a
Lagrangian subbundle. The model is the first de Rham cohomology of
a family of genus-two curves, with its subbundle of holomorphic
differentials. Differentiation modulo this subbundle gives the
directional Kodaira--Spencer map. Where that map is invertible, a
Hodge frame and its first derivatives form a frame of the whole
bundle. The next derivatives determine a second-order matrix
equation, whose normalized coefficient $I$ is the classical
Lagrangian matrix Schwarzian
\cite{Ovsienko,APD,MBS}.

A local section is \emph{cyclic} where its successive covariant
derivatives span the flat bundle. For a line subbundle, the determinant
formed from a nonvanishing local generator has an intrinsic zero
divisor. This cyclicity divisor locates the apparent singularities.
The Hodge filtration supplies two natural choices. The wedge of the two
Hodge frame elements, the Hodge determinant, lies in the rank-five primitive
exterior square, the kernel of contraction by the polarization.
Alternatively, use either null line of the symmetric form induced
on the Hodge bundle by the Kodaira--Spencer map. The two null lines
are separated by an \'etale double cover of the transverse locus.

Two scalar differentials distinguish these choices. Write
$S=I-\tfrac12\tr(I)\Id$ for the trace-free part and $\cov$ for the
adapted covariant derivative. In a parameter $t$, put
\begin{equation}\label{intro:period-invariants}
 q\,dt^4=\tr(S^2)\,dt^4,\qquad
 c\,dt^{10}=\tr([S,\cov S]^2)\,dt^{10}.
\end{equation}
These differentials are independent of the frame and parameter.

\begin{maintheorem}[Cyclicity and Schwarzian divisors]
\label{intro:cyclicity}
On the transverse locus, the Hodge determinant is cyclic in the
primitive exterior square exactly where $c\ne0$. If it is
generically cyclic, twice its cyclicity divisor is the zero divisor
of $c\,dt^{10}$.

If $q\not\equiv0$, the cyclicity divisor of the tautological null
line on the double cover pushes forward to twice the zero divisor
of $q\,dt^4$.
\end{maintheorem}

These are Theorems~\ref{III:thm:primitive-commutator-four} and
\ref{III:thm:null-lines-four}. They determine both the locations and
multiplicities of the failures of the selected scalar presentations.
When the Kodaira--Spencer map has constant rank one, its kernel is a
line. A local generator is cyclic wherever differentiation gives a
second independent Hodge direction. If the line is generically cyclic,
zeros of this derivative determine its apparent points. This identifies the ramification in the
Yang--Zudilin construction \cite{YZ} and relates it to rank-one
Lagrangian geometry \cite{Zelenko,ZelenkoLi}.

The modular setting gives a second use of the universal calculus.
On a Siegel modular variety, covariant differentiation takes values
in symmetric cotangent tensors. Keeping all these tensors in one
differential algebra, as in \cite{YY}, allows the universal
polynomials to be evaluated there. Suitable bracket combinations
are independent of the modular connection; in genus two the scalar
second-order brackets recover those of Eholzer--Ibukiyama \cite{EI}.
Equivariant contractions then produce scalar modular forms
\cite{CFG,CvG}. This ambient construction and differentiation along
a specified period curve have different inputs; their relation is
made explicit by the geodesic and acceleration formulas in
Sections~\ref{III:subsec:geodesic-Schwarzian} and
\ref{III:subsec:ambient-acceleration}.

For the family $y^2=x^5+x^3+x+t$, exact Gauss--Manin reduction
computes the matrix equation, its selected scalar equations, and
their apparent divisors. The Hodge matrix equation is regular at
every smooth finite parameter, while the scalar presentations have
different apparent loci. Igusa--Clebsch invariants express these
loci as pullbacks of Siegel modular forms
(Sections~\ref{III:sec:explicit-gm}--\ref{III:sec:igusa-applications}).
Thus coefficient invariants describe the singularities introduced
by a geometric choice of scalar equation.

\subsection*{Reading guide and conventions}

The algebraic construction occupies Sections~\ref{I:sec:local-ore}--\ref{I:sec:higher-construction};
Section~\ref{I:sec:local-W-algebras} gives its Hamiltonian comparison.
For the curve applications, reconstruction and brackets in
Sections~\ref{II:sec:modular-calculus}--\ref{II:sec:rankin-cohen}
lead to the kernel, extension recovery, and the modular example in
Sections~\ref{II:sec:kernel}--\ref{II:sec:x023}.
Part~\ref{part:III} has two routes: Sections~\ref{III:sec:siegel-ore}--\ref{III:sec:schwarzians-modular-constructions}
develop ambient modular tensors, while
Sections~\ref{III:sec:filtered-connections}--\ref{III:sec:igusa-applications}
pass from filtered connections to explicit scalar equations. The
latter route uses the one-variable calculus of Part~\ref{part:I};
the ambient comparison is given at its points of use. The kernel
applications of Part~\ref{part:II} and the period applications of
Part~\ref{part:III} are independent: their common coefficient calculus
acts on different geometric inputs.

Coefficients in Ore operators are on the left. The correction
defining $W_m$ depends on the operator order $N$, even when this
dependence is suppressed in its notation. Curve derivatives are
ordinary unless a modular normalization is stated.
Appendix~\ref{I:app:formulas} gives formulas through weight six, and
Appendix~\ref{app:computations} describes the combined exact calculations.

\clearpage
\begingroup\fontsize{10}{11}\selectfont
\tableofcontents
\endgroup

\clearpage
\part{Noncommutative invariant theory}\label{part:I}
\markboth{Part I: Universal invariants}{}
The coefficient formula for the classical Schwarzian extends to an
operator with coefficients in any associative differential algebra.
Theorem~\ref{intro:universal} separates two tasks: finding all expressions
that transform by conjugation, and choosing higher expressions with
weighted parameter laws. Sections~\ref{I:sec:local-ore}--\ref{I:sec:local-generation}
solve the first task by normalization and differential generation;
Sections~\ref{I:sec:reparam}--\ref{I:sec:higher-construction} solve the
second by projective normalization and symmetrization.
The geometric applications use these identities. The Hamiltonian
comparison in Section~\ref{I:sec:local-W-algebras} may be read separately:
it identifies their scalar and matrix specializations with primary
currents. The two low-order examples in Section~\ref{I:sec:reparam}
show what changes when the coefficients cease to commute.

\section{Ore operators and universal coefficients}\label{I:sec:local-ore}

The Ore relation separates differentiation from coefficient multiplication.
The free coefficient algebra then permits one proof of each identity to
apply in every associative differential algebra.

Let $A$ be a unital associative $\Q$-algebra with a derivation $\delta$.
The Ore algebra $A[D;\delta]$ is the free left $A$-module with basis
$1,D,D^2,\ldots$, whose multiplication is determined by
\begin{equation}\label{I:eq:local-ore}
 Da=aD+\delta(a),\qquad a\in A.
\end{equation}
Here $D$ is a formal operator and $\delta(a)$ is a coefficient.
Induction gives
\begin{equation}\label{I:eq:local-leibniz}
 D^j a=\sum_{r=0}^j\binom jr\delta^r(a)D^{j-r}.
\end{equation}
To construct this algebra, adjoin a central indeterminate $t$ and
extend $\delta$ to $A[t]$ by $\delta(t)=0$. Let $a\in A$ act by left
multiplication and let $D$ act as $\delta+t$. These additive operators
satisfy \eqref{I:eq:local-ore}. By \eqref{I:eq:local-leibniz}, finite sums
$\sum_j a_jD^j$ are closed under composition. Since
\[
 \left(\sum_j a_jD^j\right)(1)=\sum_j a_jt^j,
\]
their left normal forms are unique. Composition therefore gives the
required associative multiplication. This is the differential case of
Ore's construction \cite{Ore1933}; the same argument works over any
unital differential ring.

Fix $N\ge2$. We use the binomial normalization
\begin{equation}\label{I:eq:local-L}
 L=\sum_{k=0}^N\binom Nk a_kD^{N-k},\qquad a_0=1.
\end{equation}
Over $\Q$, every monic operator of order $N$ has a unique presentation
of this form. We seek polynomial expressions in its coefficients and
their derivatives which transform by conjugation when $L$ does.

To construct these expressions once for all coefficient algebras, set
\begin{equation}\label{I:eq:local-universal}
 \mathcal U_N=\Q\langle a_{k,r}\mid 1\le k\le N,\ r\ge0\rangle,
 \qquad \delta(a_{k,r})=a_{k,r+1}.
\end{equation}
The brackets denote the free associative algebra: its elements are
finite linear combinations of ordered words in the displayed symbols.
The rule for $\delta$ extends uniquely by the Leibniz rule. Write
$a_k=a_{k,0}$, and let $L$ now denote \eqref{I:eq:local-L} over
$\mathcal U_N$. Given coefficients in any differential $\Q$-algebra,
the substitution
\[
 a_{k,r}\longmapsto\delta^r(a_k)
\]
is the unique differential homomorphism with the prescribed values of
$a_{k,0}$. Thus an identity in $\mathcal U_N$ holds after every such
specialization. We first work in this universal algebra; the end of
Section~\ref{I:sec:local-generation} treats general coefficient rings.

\section{The normalized coefficients}\label{I:sec:local-normalization}

The coefficient $a_1$ can be absorbed into the first-order operator
$\nabla=D+a_1$. To expand $L$ in powers of $\nabla$, we first record
the change of basis from $D$ to $D+u$. For a coefficient $u$, define
\begin{equation}\label{I:eq:local-bell}
 \begin{aligned}
 B_0^{\mathrm L}(u)&=B_0^{\mathrm R}(u)=1,\\
 B_{r+1}^{\mathrm L}(u)&=\delta(B_r^{\mathrm L}(u))+uB_r^{\mathrm L}(u),\\
 B_{r+1}^{\mathrm R}(u)&=\delta(B_r^{\mathrm R}(u))-B_r^{\mathrm R}(u)u.
 \end{aligned}
\end{equation}
These are the two Bell polynomial families needed here. The side on
which $u$ is multiplied is part of the definition.

The first terms show the distinction between the two recurrences:
\[
\begin{aligned}
 B_1^{\mathrm L}(u)&=u,&B_1^{\mathrm R}(u)&=-u,\\
 B_2^{\mathrm L}(u)&=\delta(u)+u^2,&
 B_2^{\mathrm R}(u)&=-\delta(u)+u^2.
\end{aligned}
\]
At the next order,
\[
\begin{aligned}
 B_3^{\mathrm L}(u)&=\delta^2(u)+\delta(u)u+2u\delta(u)+u^3,\\
 B_3^{\mathrm R}(u)&=-\delta^2(u)+2\delta(u)u+u\delta(u)-u^3.
\end{aligned}
\]
In particular, replacing \(u\) by \(-u\) in the left recurrence does
not give the right recurrence when the coefficients do not commute.

\begin{lemma}\label{I:lem:local-bell}
For every $j\ge0$,
\begin{equation}\label{I:eq:local-bell-expansions}
 \begin{aligned}
 (D+u)^j&=\sum_{r=0}^j\binom jr B_r^{\mathrm L}(u)D^{j-r},\\
 D^j&=\sum_{r=0}^j\binom jr B_r^{\mathrm R}(u)(D+u)^{j-r}.
 \end{aligned}
\end{equation}
\end{lemma}

\begin{proof}
For the first identity multiply on the left by $D+u$, using
\[
 (D+u)B_r^{\mathrm L}=B_r^{\mathrm L}D+B_{r+1}^{\mathrm L}.
\]
For the second multiply on the left by $D$, using
\[
 DB_r^{\mathrm R}=B_r^{\mathrm R}(D+u)+B_{r+1}^{\mathrm R}.
\]
In both cases Pascal's identity combines the two sums and proves the
induction step. The initial case is $j=0$.
\end{proof}

\begin{proposition}[Normalized expansion]\label{I:prop:local-normalization}
There are unique coefficients $I_0,\ldots,I_N\in\mathcal U_N$ such that
\begin{equation}\label{I:eq:local-normalized}
 L=\sum_{k=0}^N\binom Nk I_k\nabla^{N-k}.
\end{equation}
They satisfy $I_0=1$, $I_1=0$, and
\begin{equation}\label{I:eq:local-I-formulas}
 I_k=\sum_{j=0}^k\binom kj a_jB_{k-j}^{\mathrm R}(a_1),
 \qquad
 a_k=\sum_{j=0}^k\binom kj I_jB_{k-j}^{\mathrm L}(a_1).
\end{equation}
\end{proposition}

\begin{proof}
Substitute the second expansion in Lemma~\ref{I:lem:local-bell} into
\eqref{I:eq:local-L}. The coefficient of $\nabla^{N-k}$ is
\[
 \sum_{j=0}^k\binom Nj\binom{N-j}{k-j}
                   a_jB_{k-j}^{\mathrm R}(a_1)
 =\binom Nk\sum_{j=0}^k\binom kj a_jB_{k-j}^{\mathrm R}(a_1).
\]
This proves the first formula. Each $\nabla^j$ is monic of order $j$,
so expansion in its powers is unique. Since $\binom Nk$ is invertible
over $\Q$, the $I_k$ themselves are unique. Applying the first Bell
expansion to \eqref{I:eq:local-normalized} gives the inverse formula.
Finally, $B_1^{\mathrm R}(a_1)=-a_1$, so $I_1=0$.
\end{proof}

Formula \eqref{I:eq:local-I-formulas} shows that $I_k$ has integer
coefficients and depends only on $a_1,\ldots,a_k$ and their derivatives.
It is independent of $N$ once $N\ge k$. In particular,
\begin{equation}\label{I:eq:local-I23}
 \begin{aligned}
 I_2&=a_2-\delta(a_1)-a_1^2,\\
 I_3&=a_3-3a_2a_1+2a_1^3-\delta^2(a_1)
                         +2[\delta(a_1),a_1].
 \end{aligned}
\end{equation}
The second formula applies when $N\ge3$.
This is an algebraic expansion of the given operator. A gauge removing
\(a_1\) need not belong to the original coefficient algebra; an
injective extension supplying it is constructed in
Section~\ref{I:sec:local-generation}.

Setting every $a_{1,r}$
to zero sends $I_k$ to $a_k$; this restriction will give the inverse
map in the generation theorem.

\subsection{The binary-form comparison}

The algebraic part of the normalization is already visible in a monic
binary form
\[
 F(X,Y)=\sum_{k=0}^N\binom Nk b_kX^{N-k}Y^k,\qquad b_0=1.
\]
The substitution \(X\mapsto X+uY\), with \(u=-b_1\), removes the
next coefficient. The remaining coefficients, for $N\ge3$, begin with
\[
 J_2=b_2-b_1^2,\qquad J_3=b_3-3b_2b_1+2b_1^3.
\]
In the polynomial coordinates \((b_1,J_2,\ldots,J_N)\), this unipotent
group translates \(b_1\) and fixes every \(J_k\). Its invariant ring
on the monic chart is therefore \(\Q[J_2,\ldots,J_N]\).
For scalar differential operators, the derivative terms in
\eqref{I:eq:local-I23} come from the Ore relation. The result is
Wilczy\'nski's normalization \cite[pp.~16--17]{Wilczynski1906}.
In the associative case the same calculation also fixes the order of
each product. The next section proves its covariance without choosing
a solution of a differential equation.

The full-group comparison uses two different derivations. On the full
coefficient space, where \(b_0\) is again a variable, the two unipotent
subgroups of \(\mathrm{SL}_2\) have infinitesimal operators
\cite[p.~523]{Hilbert1890}
\begin{equation}\label{I:eq:Hilbert-operators}
 E=\sum_{k=1}^N kb_{k-1}\frac{\partial}{\partial b_k},\qquad
 \Delta_{\mathrm H}
   =\sum_{k=0}^{N-1}(N-k)b_{k+1}\frac{\partial}{\partial b_k}.
\end{equation}
A homogeneous unipotent invariant of fixed diagonal weight is a
seminvariant. The operator \(E\) kills it, and successive applications
of \(\Delta_{\mathrm H}\) recover the coefficients of a binary
covariant. Full \(\mathrm{SL}_2\)-invariants are killed by both
operators; Hilbert's theorem gives finite generation of their ring
\cite[\S V]{Hilbert1890} (see also \cite{Hilbert1993}).
In the differential problem, the role of \(\Delta\) is different:
it preserves gauge covariance rather than recovering a binary covariant.

\subsection{Changing the adapted connection}\label{I:subsec:connection-translation}
The expansion also separates a change of the adapted connection from
gauge conjugation. Let $u$ be a coefficient, adjoining it and its derivatives freely if
necessary. Define
\[
 L^{\star u}=\sum_{k=0}^N\binom Nk I_k(L)(D+a_1-u)^{N-k}.
\]
Its first coefficient is $a_1-u$, so uniqueness of the normalized
expansion gives
\[
 I_k(L^{\star u})=I_k(L),\qquad
 (L^{\star v})^{\star u}=L^{\star(u+v)}.
\]
Here $u,v$ are held fixed under the translations. Thus the additive
group of coefficients acts by changing the adapted connection. This
operation preserves the normalized coefficients, but it need not preserve
their covariant derivatives; see Section~\ref{I:sec:local-generation}.

\section{Formal gauge transformations}\label{I:sec:local-gauge}

A gauge transformation is conjugation by an invertible coefficient.
To express its law universally, adjoin symbols $g_r$ for $r\ge0$
and an inverse of $g_0$:
\[
 \mathcal G_N=
 \Q\langle a_{k,r},g_r,g_0^{-1}\mid
       g_0g_0^{-1}=g_0^{-1}g_0=1\rangle.
\]
Put $g=g_0$. Extend $\delta$ by
\[
 \delta(g_r)=g_{r+1},\qquad
 \delta(g^{-1})=-g^{-1}g_1g^{-1}.
\]
These rules preserve the two inverse relations and hence define a
derivation. No commutation relation is imposed between $g$ and the
coefficient generators.

By \eqref{I:eq:local-leibniz}, the conjugated operator has coefficients
\begin{equation}\label{I:eq:local-gauge-coefficients}
 L^g=g^{-1}Lg=\sum_{k=0}^N\binom Nk a_k^gD^{N-k},
 \qquad
 a_k^g=\sum_{j=0}^k\binom kj g^{-1}a_j\delta^{k-j}(g).
\end{equation}
The same binomial identity used in
Proposition~\ref{I:prop:local-normalization} gives the factor $\binom Nk$.
Define the differential homomorphism $T_g:\mathcal U_N\to\mathcal G_N$
by $T_g(a_{k,r})=\delta^r(a_k^g)$. For a differential polynomial
$C\in\mathcal U_N$, write $C^g=T_g(C)$. We call $C$ a
\emph{universal gauge covariant} if
\begin{equation}\label{I:eq:local-covariance}
 C^g=g^{-1}Cg.
\end{equation}
Every choice of coefficients and an invertible gauge in a differential
$\Q$-algebra gives a specialization of $\mathcal G_N$. Thus this one
formal identity is equivalent to covariance in every such algebra.
Conjugation is a right action: $(L^g)^h=L^{gh}$.

\begin{proposition}\label{I:prop:local-covariance}
The normalized coefficients are universal gauge covariants:
\[
 I_k^g=g^{-1}I_kg,\qquad 0\le k\le N.
\]
\end{proposition}

\begin{proof}
The first transformed coefficient is
\[
 a_1^g=g^{-1}a_1g+g^{-1}\delta(g).
\]
The Ore relation therefore gives
$D+a_1^g=g^{-1}(D+a_1)g=g^{-1}\nabla g$.
Conjugating \eqref{I:eq:local-normalized} now yields
\[
 L^g=\sum_{k=0}^N\binom Nk
          (g^{-1}I_kg)(D+a_1^g)^{N-k}.
\]
Uniqueness of the normalized expansion identifies these coefficients
with $I_k^g$.
\end{proof}

\section{Covariant differentiation and generation}\label{I:sec:local-generation}

Differentiating $g^{-1}Cg$ introduces derivatives of the gauge.
Commutation with $\nabla$ gives a derivation that respects conjugation:
\begin{equation}\label{I:eq:local-Delta}
 \Delta C=[\nabla,C]=\delta(C)+[a_1,C].
\end{equation}
For the translations in Section~\ref{I:subsec:connection-translation},
replacing \(a_1\) by \(a_1-u\) gives
\(\Delta^{\star u}=\Delta-[u,\cdot]\). This explains why those
translations preserve \(I_k\) but need not preserve \(\Delta I_k\).
Both summands in \eqref{I:eq:local-Delta} are derivations. In particular,
\begin{equation}\label{I:eq:local-covariant-leibniz}
 \nabla C=C\nabla+\Delta C,
 \qquad
 \nabla^j C=\sum_{r=0}^j\binom jr\Delta^r(C)\nabla^{j-r}.
\end{equation}
The second identity in \eqref{I:eq:local-covariant-leibniz} follows by the induction used for
\eqref{I:eq:local-leibniz}. Moreover, for every universal covariant $C$,
\begin{equation}\label{I:eq:local-Delta-law}
 (\Delta C)^g
   =[D+a_1^g,C^g]
   =[g^{-1}\nabla g,g^{-1}Cg]
   =g^{-1}(\Delta C)g.
\end{equation}
Thus $\Delta$ preserves the algebra $\mathcal C_N$ of universal gauge
covariants. All the expressions $\Delta^rI_k$ belong to this algebra.
We show that they give it free generators.

\begin{example}[Ordinary and covariant derivatives]\label{I:ex:matrix-derivative}
Let \(A=M_2(\Q[t])\), with entrywise differentiation, and let \(E_{ij}\)
be its matrix units. For
\[
 L=D^2+2E_{12}D+tE_{11},
\]
one has \(I_2=tE_{11}\) and
\(\Delta I_2=E_{11}-tE_{12}\), since
\([E_{12},E_{11}]=-E_{12}\).
Take \(g=1+tE_{12}\), whose inverse is \(1-tE_{12}\). The gauge
formulas give
\[
 a_1^g=2E_{12},\qquad I_2^g=tE_{11}+t^2E_{12}.
\]
Ordinary differentiation does not respect conjugation:
\[
 \delta(I_2^g)=E_{11}+2tE_{12},\qquad
 g^{-1}\delta(I_2)g=E_{11}+tE_{12}.
\]
The connection term removes the discrepancy. In fact,
\[
 \Delta^g(I_2^g)=E_{11}
       =g^{-1}(\Delta I_2)g.
\]
\end{example}

\subsection{Restriction and reconstruction}

The example concerns a particular coefficient algebra. The generation
theorem concerns polynomials whose covariance holds in every such algebra.
Restriction to a normalized operator will determine these polynomials.

To remove $a_1$, we need a gauge satisfying $\delta(f)=-a_1f$.
Such a gauge need not belong to $\mathcal U_N$, so we first construct
an extension in which it does.

\begin{lemma}\label{I:lem:local-slice}
There is an injective differential extension
$\mathcal U_N\hookrightarrow\mathcal E_N$ with an invertible element
$f$ such that $\delta(f)=-a_1f$. In this extension,
\begin{equation}\label{I:eq:local-slice}
 a_1^f=0,\qquad a_k^f=f^{-1}I_kf\quad(2\le k\le N),
 \qquad
 \delta^r(f^{-1}Cf)=f^{-1}\Delta^r(C)f
\end{equation}
for every $C\in\mathcal U_N$ and $r\ge0$.
\end{lemma}

\begin{proof}
Take the free product
$\mathcal E_N=\mathcal U_N*_{\Q}\Q[f,f^{-1}]$ and prescribe
\[
 \delta(f)=-a_1f,\qquad \delta(f^{-1})=f^{-1}a_1.
\]
Differentiating $ff^{-1}=f^{-1}f=1$ gives zero, so the derivation
is well defined. The algebra homomorphism fixing $\mathcal U_N$ and
sending $f$ to $1$ is a retraction. Hence the inclusion of
$\mathcal U_N$ is injective. This retraction is used only as an algebra
map; it need not preserve the derivation.
The assignment $g_r\mapsto\delta^r(f)$ and $g_0^{-1}\mapsto f^{-1}$
defines a differential homomorphism $\mathcal G_N\to\mathcal E_N$
fixing $\mathcal U_N$, so every universal gauge identity applies to $f$.

The right Bell recurrence in \eqref{I:eq:local-bell} gives
$\delta^r(f)=B_r^{\mathrm R}(a_1)f$ by induction. Substitution in
\eqref{I:eq:local-gauge-coefficients} proves $a_k^f=f^{-1}I_kf$,
including $a_1^f=0$. Finally,
\[
 \delta(f^{-1}Cf)
 =f^{-1}\bigl(a_1C+\delta(C)-Ca_1\bigr)f
 =f^{-1}\Delta(C)f.
\]
Iteration proves the last assertion.
\end{proof}

Let
\[
 \mathcal V_N=\Q\langle x_{k,r}\mid 2\le k\le N,\ r\ge0\rangle,
 \qquad \partial(x_{k,r})=x_{k,r+1}.
\]
Define algebra homomorphisms
\begin{equation}\label{I:eq:local-restriction}
 \begin{aligned}
 \pi_0:\mathcal U_N\longrightarrow\mathcal V_N,&\qquad
    \pi_0(a_{1,r})=0,\quad \pi_0(a_{k,r})=x_{k,r}\quad(k\ge2),\\
 \sigma:\mathcal V_N\longrightarrow\mathcal U_N,&\qquad
    \sigma(x_{k,r})=\Delta^rI_k.
 \end{aligned}
\end{equation}
The map $\pi_0$ sets $a_1$ and all its derivatives to zero; $\sigma$
replaces the remaining coefficient derivatives by covariant derivatives. Since $\pi_0\Delta=\partial\pi_0$ and
$\pi_0(I_k)=x_{k,0}$, we have $\pi_0\sigma=\mathrm{id}$.
This identity proves that the proposed generators satisfy no relations.
The normalizing gauge will prove the complementary assertion: every
universal covariant is reconstructed from its restriction.

\begin{theorem}[Universal generation]\label{I:thm:local-generation}
The map $\sigma$ is an isomorphism of differential algebras
\[
 (\mathcal V_N,\partial)\ \xrightarrow{\ \sim\ }\
 (\mathcal C_N,\Delta).
\]
Equivalently,
\begin{equation}\label{I:eq:local-generation}
 \mathcal C_N=
 \Q\langle\Delta^rI_k\mid 2\le k\le N,\ r\ge0\rangle,
\end{equation}
and every universal covariant has a unique expression as a finite
linear combination of ordered words in these generators.
\end{theorem}

\begin{proof}
The image of $\sigma$ consists of covariants by
Proposition~\ref{I:prop:local-covariance} and \eqref{I:eq:local-Delta-law}.
Since $\pi_0\sigma=\mathrm{id}$, it remains to prove
$\sigma\pi_0(\Phi)=\Phi$ for every $\Phi\in\mathcal C_N$.
Put $R=\pi_0(\Phi)$.
Use the gauge $f$ from Lemma~\ref{I:lem:local-slice}. Since
$a_1^f=0$, all its derivatives vanish, and hence
\[
 \Phi^f
 =R\bigl(\delta^r(a_k^f)\bigr)
 =R\bigl(f^{-1}\Delta^rI_kf\bigr)
 =f^{-1}\sigma(R)f.
\]
The last equality cancels only adjacent factors $ff^{-1}$.
Covariance also gives $\Phi^f=f^{-1}\Phi f$, so cancellation of the
outer gauge factors yields $\Phi=\sigma\pi_0(\Phi)$ in $\mathcal E_N$.
Injectivity of the extension gives the identity in $\mathcal U_N$.
Finally, $\sigma$ intertwines $\partial$ with $\Delta$ by definition.
\end{proof}

The proof also gives a direct procedure: set every $a_{1,r}$ to zero
in a universal covariant, and replace each remaining $a_{k,r}$ by
$\Delta^rI_k$. The resulting polynomial is the original covariant.
Universality is essential here. The theorem describes identities
valid for all coefficient algebras; specialization to a particular
algebra can introduce additional relations among the generators.

\subsection{Integral coefficients}
The normalization and covariance formulas have integer coefficients.
Let $A$ now be any unital associative ring with a derivation. Start with a tuple $(a_1,\ldots,a_N)$ and its operator
\eqref{I:eq:local-L}. Define its normalized coefficients by
\eqref{I:eq:local-I-formulas}, and its gauge-transformed tuple by
\eqref{I:eq:local-gauge-coefficients}. These definitions require no
cancellation hypothesis.

\begin{proposition}\label{I:prop:local-integral}
The Bell identities, the normalized expansion and its inverse,
gauge covariance of the $I_k$, and covariance under $\Delta$ hold
for these tuples over every differential ring.
\end{proposition}

\begin{proof}
The Bell identities and the formula for $I_k$ follow from their
recurrences without division. To obtain the inverse formula for
$a_k$, apply the two Bell expansions to the operator of order $k$
with coefficients $a_1,\ldots,a_k$ and compare constant terms.
The relevant binomial factor is $\binom kk=1$.

The same observation proves gauge covariance without cancelling an
integer. The formula for $I_k$ and the transformed coefficients
$a_j^g$, $j\le k$, are independent of the ambient order. Apply the
two normalized expansions of the conjugated order-$k$ operator and
compare their coefficients of $(D+a_1^g)^0$. Powers of $D+a_1^g$
form a left $A$-basis even in the presence of additive torsion, and
this coefficient has binomial factor $\binom kk=1$. Thus
$I_k^g=g^{-1}I_kg$. The identity for $\Delta$ is the commutator
calculation \eqref{I:eq:local-Delta-law}.

\end{proof}

\begin{theorem}[Integral generation]\label{I:thm:local-integral-generation}
Let $R$ be a commutative ring. The universal covariants of coefficient
tuples over associative differential $R$-algebras with $R$-linear
derivations form the free differential $R$-algebra on
$I_2,\ldots,I_N$ under $\Delta$. The structural map from $R$ is
required to have central image.
\end{theorem}

\begin{proof}
Replace $\Q$ by $R$ in the formal algebras of
Theorem~\ref{I:thm:local-generation}. The extension adjoining $f$ is
injective by the same algebra retraction. Proposition~\ref{I:prop:local-integral}
supplies the covariance identities without division. The maps
$\pi_0,\sigma$ and the restriction-and-reconstruction proof therefore
remain valid over $R$.
\end{proof}

To regard these covariants as expressions in the operator alone,
rather than in its chosen tuple, assume
\begin{equation}\label{I:eq:local-cancellation}
 \binom Nk a=0\ \Longrightarrow\ a=0
 \qquad(a\in A,\ 1\le k\le N).
\end{equation}
Then the operator determines its binomially normalized tuple and
each $I_k$ uniquely. The hypothesis concerns the additive group of
$A$, not merely its characteristic.

Existence of a binomial presentation is a separate requirement:
for $L=D^N+\sum_{k=1}^N c_kD^{N-k}$ it means
$c_k\in\binom Nk A$ for every $k$. Cancellation gives uniqueness
when such a presentation exists; it does not supply division.
Without \eqref{I:eq:local-cancellation}, the tuple remains part of the
data, and all the universal identities and the generation theorem
still apply in that form.

\section{Changes of parameter}\label{I:sec:reparam}

Parameter changes act on the coefficient derivation as well as on the
coefficients. We first define this action and compute its effect on the
normalized coefficients. Their affine transformation law will permit
the symmetrization used in the next section.

\subsection{Algebraic changes of parameter}

Let \((A,\delta)\) and \((B,\widetilde\delta)\) be differential
\(\Q\)-algebras. A change of parameter consists of a unital
\(\Q\)-algebra homomorphism \(\phi:A\to B\) and a central unit
\(\ell\in B\) such that
\begin{equation}\label{I:eq:formal-chain-rule}
 \widetilde\delta(\phi(a))=\ell\,\phi(\delta(a))\qquad(a\in A).
\end{equation}
For functions, these are \(\phi(a)=a\circ\lambda\) and
\(\ell=\lambda'\). If \(\widetilde D\) is the Ore symbol over \(B\),
the assignments \(a\mapsto\phi(a)\) and
\(D\mapsto\ell^{-1}\widetilde D\) respect the Ore relation.
The corresponding monic pullback is
\begin{equation}\label{I:eq:formal-pullback}
 L^\lambda=\ell^N\sum_{k=0}^N\binom Nk
       \phi(a_k)(\ell^{-1}\widetilde D)^{N-k}.
\end{equation}
The superscript denotes these data; an element \(\lambda\) need not
belong to either coefficient algebra. Weight-\(m\) covariance means
\(C(L^\lambda)=\ell^m\phi(C(L))\).

\medskip
\noindent\textbf{Universal parameter jets.}
To prove the transformation identities universally, introduce central symbols
$\lambda_1,\lambda_1^{-1},\lambda_2,\ldots$ and a pulled-back copy of
the coefficient generators. The resulting algebra is
\[
 \mathcal U_N^\lambda
 =\Q[\lambda_1,\lambda_1^{-1},\lambda_2,\ldots]
       \langle\bar a_{k,r}\mid 1\le k\le N,\ r\ge0\rangle,
\]
with derivation
\[
 \partial(\lambda_j)=\lambda_{j+1},\qquad
 \partial(\bar a_{k,r})=\lambda_1\bar a_{k,r+1}.
\]
A bar denotes the homomorphism $a_{k,r}\mapsto\bar a_{k,r}$; thus
$\partial\bar C=\lambda_1\overline{\delta C}$.
For functions, $\lambda_j$ is the $j$-th derivative of the parameter
change and $\bar C=C\circ\lambda$. More generally, the data
$\phi,\ell$ in \eqref{I:eq:formal-chain-rule} specialize this algebra by
$\lambda_j\mapsto\widetilde\delta^{j-1}\ell$ and
$\bar a_{k,r}\mapsto\phi(\delta^r a_k)$.
Derivations preserve the center, so the substituted parameter jets
are central as required.

Let $D_\lambda$ be the Ore symbol for $\partial$. Define the monic
pullback by
\begin{equation}\label{I:eq:rp-operator}
 L^\lambda=\lambda_1^N\sum_{k=0}^N\binom Nk
             \bar a_k(\lambda_1^{-1}D_\lambda)^{N-k}
           =\sum_{k=0}^N\binom Nk a_k^\lambda D_\lambda^{N-k}.
\end{equation}
For any differential polynomial $C$, its transform $C^\lambda$ is
obtained by substituting $a_k^\lambda$ for $a_k$ and $\partial$ for
$\delta$. It is a \emph{parameter covariant of weight $m$} if
\begin{equation}\label{I:eq:rp-weight}
 C^\lambda=\lambda_1^m\bar C.
\end{equation}
The bar pulls back the original coefficients; the superscript also
accounts for the transformed operator.

Put $\rho=\lambda_2/\lambda_1$ and
\[
 S_\lambda=\frac{\lambda_3}{6\lambda_1}
             -\frac{\lambda_2^2}{4\lambda_1^2}.
\]
Thus $S_\lambda$ is one sixth of the classical Schwarzian. The first
transformation formulas are
\begin{equation}\label{I:eq:rp-low-laws}
 \begin{aligned}
 a_1^\lambda&=\lambda_1\bar a_1-\frac{N-1}{2}\rho,\\
 I_2^\lambda&=\lambda_1^2\bar I_2+(N+1)S_\lambda,\\
 I_3^\lambda&=\lambda_1^3\bar I_3+3\lambda_1\lambda_2\bar I_2
                      +\frac{3(N+1)}2\partial S_\lambda.
 \end{aligned}
\end{equation}
The last line applies for $N\ge3$.

\subsection{The quadratic and cubic laws}
Let $\Delta^\lambda=\partial+[a_1^\lambda,\,\cdot\,]$.
Since $\rho$ is central, the first line of \eqref{I:eq:rp-low-laws} gives
\begin{equation}\label{I:eq:rp-Delta-bar}
 \Delta^\lambda\bar C=\lambda_1\overline{\Delta C}.
\end{equation}
Applying this to the second line yields
\[
 (\Delta I_2)^\lambda
 =\lambda_1^3\overline{\Delta I_2}
     +2\lambda_1\lambda_2\bar I_2+(N+1)\partial S_\lambda.
\]
Comparison with the formula for $I_3^\lambda$ shows that
\begin{equation}\label{I:eq:rp-W3}
 W_3=I_3-\frac32\Delta I_2,
 \qquad W_3^\lambda=\lambda_1^3\bar W_3.
\end{equation}
The same subtraction removes both extra terms. Gauge covariance is
preserved because $I_3$ and $\Delta I_2$ are gauge covariants.
The quadratic coefficient behaves differently:
\begin{equation}\label{I:eq:P}
 P=I_2/(N+1),\qquad P^\lambda=\lambda_1^2\bar P+S_\lambda.
\end{equation}
For associative coefficients this is an algebra-valued quadratic datum
with a central Schwarzian term; in the scalar case it is a projective
connection.

\begin{example}[The classical order-two case]\label{I:ex:classical-order-two}
For a scalar second-order equation, remove the first derivative locally
and choose independent solutions $z_1,z_2$ of $z''+I_2z=0$.
Their Wronskian is constant, so, for $f=z_1/z_2$,
\[
 \frac{f''}{f'}=-2\frac{z_2'}{z_2},\qquad
 \Sch(f)=-2\frac{z_2''}{z_2}=2I_2,\qquad P=\frac16\Sch(f).
\]
The law $P^\lambda=\lambda_1^2\bar P+S_\lambda$ is therefore exactly
the Schwarzian chain rule:
\begin{equation}\label{I:eq:schwarzian-chain}
 \Sch(f\circ\lambda)
 = (\lambda')^2(\Sch(f)\circ\lambda)+\Sch(\lambda).
\end{equation}
There are no members $W_m$ with $m\ge3$ when $N=2$.
\end{example}

\begin{example}[A cubic with noncommuting coefficients]\label{I:ex:matrix-W3}
Work in $M_2(\Q[t])$ with entrywise differentiation. Put
\[
 a=tE_{12}+E_{21},\qquad L=D^3+3aD^2.
\]
Thus $a_1=a$, $a_2=a_3=0$, $a^2=t\mathbf1_2$, and
$[a',a]=E_{11}-E_{22}\ne0$. Formula~\eqref{I:eq:local-I23} gives
\[
 I_2=-E_{12}-t\mathbf1_2,\qquad
 I_3=2ta+2(E_{11}-E_{22}),\qquad
 \Delta I_2=-2E_{22}.
\]
Consequently,
\begin{equation}\label{I:eq:matrix-W3}
 W_3=\begin{pmatrix}2&2t^2\\2t&1\end{pmatrix}.
\end{equation}
The commutator in $I_3$ survives in $W_3$. Indeed, for this linear
matrix $a$, the same calculation reads
\[
 W_3=2a^3+\frac32(a'a+aa')+\frac12[a',a].
\]
Its last term is $(E_{11}-E_{22})/2$. Replacing the two constant
matrix units in $a$ by diagonal matrices makes $[a',a]$ vanish;
the remaining terms give the scalar calculation on each diagonal entry.
\end{example}

\subsection{Affine transport of normalized jets}
For the higher coefficients we use a single normal-ordering recurrence,
which also proves \eqref{I:eq:rp-low-laws}. Define central polynomials
$\beta_{j,r}$ by
\begin{equation}\label{I:eq:rp-beta}
 \beta_{0,0}=1,\qquad
 \beta_{j+1,r}=\beta_{j,r}+\partial \beta_{j,r-1}-j\rho \beta_{j,r-1},
\end{equation}
with $\beta_{j,r}=0$ for $r<0$ or $r>j$. The Ore relation gives
\[
 \lambda_1^j(\lambda_1^{-1}D_\lambda)^j
       =\sum_{r=0}^j \beta_{j,r}D_\lambda^{j-r},
 \qquad
 a_k^\lambda=\frac1{\binom Nk}
       \sum_{i=0}^k\binom Ni\lambda_1^i \beta_{N-i,k-i}\bar a_i.
\]
Indeed, if $T_j=\lambda_1^j(\lambda_1^{-1}D_\lambda)^j$, then
$T_{j+1}=(D_\lambda-j\rho)T_j$, which gives \eqref{I:eq:rp-beta}. Substitution into \eqref{I:eq:local-I-formulas} gives
\eqref{I:eq:rp-low-laws} and the higher transformation laws, with the
coefficient factors in their prescribed order.

Symmetrization will preserve the parameter law because that law is
affine in the normalized jets. The next lemma establishes this before
we choose the higher currents. Write $I_{k,r}=\Delta^rI_k$.

\begin{lemma}[Affine transport]\label{I:lem:rp-affine}
For $2\le k\le N$ and $r\ge0$, the normalized jets transform as
\[
 I_{k,r}^\lambda
 =F_{k,r,0}+\sum_{j=2}^k\sum_{s=0}^r
       F_{k,r,j,s}\,\overline{I_{j,s}},
 \qquad
 F_{k,r,j,s}\in\Q[\lambda_1^{\pm1},\lambda_2,\ldots].
\]
The coefficients are central and the sum is finite.
\end{lemma}

Thus parameter change introduces no products of normalized jets. This is
the exact property used to lift scalar identities by symmetrization.

\begin{proof}
First set every $a_1$-jet to zero. Then $a_j=I_j$
for $j\ge2$, pullback is affine in these coefficients, and
$a_1^\lambda=-(N-1)\rho/2$ is central. The Bell formula therefore
makes $I_k^\lambda$ affine in the $\bar I_j$. Applying
$\Delta^\lambda$ preserves this form, since
\[
 \Delta^\lambda(F\overline{I_{j,s}})
  =(\partial F)\overline{I_{j,s}}
       +\lambda_1F\overline{I_{j,s+1}}
\]
for central $F$.

To return from the slice, observe that
$D\mapsto\lambda_1^{-1}D_\lambda$, $a\mapsto\bar a$
defines an algebra homomorphism of Ore algebras. Thus pullback
commutes with gauge conjugation:
\[
 (L^g)^\lambda=(L^\lambda)^{\bar g}.
\]
Adjoin the gauge $f$ from Lemma~\ref{I:lem:local-slice} and its inverse
freely over the central parameter ring, with
$\partial\bar f=-\lambda_1\bar a_1\bar f$.
This extension is injective by the same algebra retraction.
Both sides of the affine formula for $(L^f)^\lambda$ are the
conjugates by $\bar f$ of its two sides for $L^\lambda$.
The central coefficients commute with $\bar f$; cancelling the outer
gauge factors and using injectivity proves the formula before restriction.
\end{proof}

\section{Projective normalization and higher Schwarzians}
\label{I:sec:higher-construction}

The quadratic anomaly is removed by choosing a projective parameter.
We first identify the linear expressions that are unchanged, up to
weight, by the remaining M\"obius freedom. This gives the normalization
condition for the higher Schwarzians. The scalar construction and its
associative lift are then treated separately. The four kinds of
expressions have different parameter laws:
\begin{center}
\begin{tabularx}{\linewidth}{@{}l>{\raggedright\arraybackslash}X>{\raggedright\arraybackslash}X@{}}
\toprule
Expression & Gauge law & Parameter law \\
\midrule
$I_k$ & Conjugation & Affine; generally not pure weight \\
$P=I_2/(N+1)$ & Conjugation & Quadratic datum with a central Schwarzian term \\
$E_m$ & Conjugation & Weight $m$ under M\"obius maps \\
$W_m$ & Conjugation & Weight $m$ under every parameter change \\
\bottomrule
\end{tabularx}
\end{center}
For $m=3$ the last two rows coincide. Weight four is the first case
in which the linear expression needs a nonlinear correction.

\subsection{Linear projective covariants}
Give $I_{k,r}$ weight $k+r$. For $3\le m\le N$ set
\begin{equation}\label{I:eq:rp-E}
 E_m=\sum_{r=0}^{m-2}c_{m,r}I_{m-r,r},
 \qquad
 c_{m,r}=(-1)^r
       \frac{\binom mr\binom{m-1}r}{\binom{2m-2}r}.
\end{equation}
Thus $E_3=W_3$. The following property explains this choice of coefficients.

\begin{proposition}\label{I:prop:rp-projective}
The expression $E_m$ has weight-$m$ covariance under M\"obius
changes of parameter. It is the unique $\Q$-linear combination of the
$I_{m-r,r}$ with this property and coefficient of $I_m$ equal to one.
\end{proposition}

\begin{proof}
\textit{1. The infinitesimal action on the slice.}
Consider the one-parameter group $t\mapsto t/(1-ct/2)$.
On the slice $a_1=0$, write the coefficients as Taylor series.
To first order in $c$, the pullback and Bell formulas give
\[
 I_k^\lambda(t)=I_k(t)+c\left(
     \frac{t^2}{2}I_k'(t)+ktI_k(t)+\binom k2 I_{k-1}(t)\right)
       +O(c^2).
\]
Here $\rho=c+O(c^2)$ and
$a_1^\lambda=-(N-1)c/2+O(c^2)$. The pullback contributes
$-k(N-k)cI_{k-1}/2$, and removal of the first coefficient contributes
$k(N-1)cI_{k-1}/2$. Their sum is
\[
 \frac{k((N-1)-(N-k))}{2}cI_{k-1}
       =\binom k2 cI_{k-1}.
\]

\textit{2. The recurrence for the coefficients.}
Taking $r$ derivatives at zero and transporting from the slice by
Lemma~\ref{I:lem:rp-affine} gives
\begin{equation}\label{I:eq:rp-L1}
 \mathfrak L_1(I_{k,r})
  =\left(rk+\binom r2\right)I_{k,r-1}
                    +\binom k2 I_{k-1,r},
\end{equation}
where $I_{1,r}=0$ and terms with negative indices are zero.
For a linear combination in \eqref{I:eq:rp-E}, cancellation of adjacent
terms is equivalent to
\[
 c_{m,r+1}
   =-\frac{(m-r)(m-r-1)}{(r+1)(2m-2-r)}c_{m,r}.
\]
Starting with $c_{m,0}=1$ determines exactly the coefficients in
\eqref{I:eq:rp-E}. Thus the infinitesimal condition proves uniqueness
as well as determining the candidate.

\textit{3. From the infinitesimal action to the subgroup.}
Let $T_c$ denote the induced action on jets at zero. On jets of
bounded weight, $T_c$ is polynomial in $c$ and
$T_{c+d}=T_cT_d$. Differentiating in $d$ at zero yields
\[
 \frac{d}{dc}T_c(F)=T_c(\mathfrak L_1 F).
\]
For $F=E_m$, the right side is zero. Since the coefficient field is
$\Q$ and $T_0(F)=F$, the polynomial $T_c(F)$ equals $F$ for every $c$.

\textit{4. The full M\"obius law.}
Every M\"obius map fixing zero is a scaling composed with a member
of this subgroup. Homogeneity supplies the factor of weight $m$.
Centering the source and target parameters at arbitrary base points
proves the covariance law there, and hence the asserted identity.
\end{proof}

In scalar Laguerre--Forsyth form, where $a_1=a_2=0$, these are the
linear Wilczy\'nski expressions, with the coefficient of $a_m$
normalized to one \cite[p.~32, equations (47)--(48)]{Wilczynski1906}.
For a general parameter the $E_m$ can still have Schwarzian terms.
To remove them, we choose a formal parameter in which the projective
connection vanishes and evaluate $E_m$ there.

\subsection{The quartic correction}\label{I:subsec:quartic}

The cubic correction is linear. Weight four is the first place where
a product is required. For \(N\ge4\), formula~\eqref{I:eq:rp-E} gives
\[
 E_4=I_4-2\Delta I_3+\frac65\Delta^2I_2.
\]
For this calculation write \(\ell=\lambda_1\), \(S=S_\lambda\),
and use primes for \(\partial\) on central parameter expressions.
Thus \(\rho=\ell'/\ell\) and \(12S=2\rho'-\rho^2\).
Normal ordering by \eqref{I:eq:rp-beta}, followed by
\eqref{I:eq:local-I-formulas}, gives
\begin{equation}\label{I:eq:quartic-I4-law}
\begin{aligned}
 I_4^\lambda={}&\ell^4\bar I_4+6\ell^3\rho\bar I_3
    +\ell^2\bigl(9\rho^2+6(N+5)S\bigr)\bar I_2\\
 &+\frac{9(N+1)}5 S''
    +\frac{3(N+1)(5N+7)}5 S^2.
\end{aligned}
\end{equation}
The remaining two terms of \(E_4\) transform by differentiating
\eqref{I:eq:rp-low-laws}. Using \eqref{I:eq:rp-Delta-bar}, one obtains
\[
\begin{aligned}
 (\Delta I_3)^\lambda={}&\ell^4\overline{\Delta I_3}
 +3\ell^3\rho\bar I_3+3\ell^3\rho\overline{\Delta I_2}\\
 &+3\ell^2(2\rho^2+\rho')\bar I_2+\frac{3(N+1)}2S'',\\
 (\Delta^2I_2)^\lambda={}&\ell^4\overline{\Delta^2I_2}
 +5\ell^3\rho\overline{\Delta I_2}
 +\ell^2(4\rho^2+2\rho')\bar I_2+(N+1)S''.
\end{aligned}
\]
The coefficients of $E_4$ were chosen to cancel every M\"obius
anomaly. This explains the cancellation of the terms involving
$\rho$ without a Schwarzian factor, but does not by itself exclude
derivatives such as $S''$. In the displayed combination, the
coefficient of $S''$ is
\[
 (N+1)\left(\frac95-3+\frac65\right)=0.
\]
The $\bar I_3$ and $\overline{\Delta I_2}$ terms also cancel, while
the remaining coefficient of $\bar I_2$ reduces by
$2\rho'-\rho^2=12S$. Thus precisely the two terms $S\bar I_2$
and $S^2$ survive:
\begin{equation}\label{I:eq:quartic-E4-law}
 E_4^\lambda=\ell^4\bar E_4
 +\frac{6(5N+7)}5\ell^2S\bar I_2
 +\frac{3(N+1)(5N+7)}5S^2.
\end{equation}
On the other hand, centrality of \(S\) gives
\[
 (I_2^2)^\lambda=\ell^4\bar I_2^2
       +2(N+1)\ell^2S\bar I_2+(N+1)^2S^2.
\]
This is why a product correction appears: the square of the
quadratic coefficient has exactly the two remaining anomalies.
The same multiple of this identity cancels both extra terms in
\eqref{I:eq:quartic-E4-law}. Hence
\begin{equation}\label{I:eq:quartic-W4}
 W_4=E_4-\frac{3(5N+7)}{5(N+1)}I_2^2
\end{equation}
has weight-four covariance. Every step is valid with associative
coefficients: only the parameter jets have been commuted. The scalar
restriction at \(I_2=\Delta I_2=\cdots=0\) is \(E_4\), so the
general normalization below gives this same representative.

\subsection{The normalization condition}

For scalar coefficients, we specify a higher Schwarzian by its value
when the projective connection and all its derivatives vanish.
To express this condition, put
\[
 \mathcal J_N=\Q[x_{k,r}\mid 2\le k\le N,\ r\ge0],
 \qquad \operatorname{wt}(x_{k,r})=k+r.
\]
The variable \(x_{k,r}\) represents the \(r\)-th derivative of the
scalar normalized coefficient \(I_k\). The parameter action is the
commutative specialization of Lemma~\ref{I:lem:rp-affine}.
Write \(E_m(x)\) for \eqref{I:eq:rp-E} with \(I_{k,r}\) replaced by
\(x_{k,r}\).

\begin{proposition}[Scalar projective normalization]\label{I:prop:rp-scalar}
For every \(3\le m\le N\), there is a unique scalar polynomial
\(\omega_m\in\mathcal J_N\) with parameter covariance of weight
\(m\) whose restriction to \(x_{2,r}=0\) for all \(r\) equals
\(E_m(x)\) on that slice. It has the form
\[
 \omega_m=x_{m,0}+q_m,
\]
where \(q_m\) is homogeneous of weight \(m\) and uses only variables
\(x_{k,r}\) with \(k<m\) and \(k+r\le m\).
The polynomial is computed by the finite prescription
\eqref{I:eq:rp-normalizer}--\eqref{I:eq:rp-q-definition}.
\end{proposition}

The lift to associative coefficients is a linear symmetrization map.
For free symbols \(X_{k,r}\), define
\begin{equation}\label{I:eq:rp-Sym}
 \operatorname{Sym}(x_{i_1}\cdots x_{i_d})
   =\frac1{d!}\sum_{\pi\in S_d}
        X_{i_{\pi(1)}}\cdots X_{i_{\pi(d)}}.
\end{equation}
Here each \(i_j\) denotes a pair \((k,r)\), repeated factors are
allowed, and \(\operatorname{Sym}(1)=1\). Extend the map linearly.
For two coefficients we write $A\odot B=(AB+BA)/2$; thus
$\operatorname{Sym}(xy)=X\odot Y$. In contrast,
$\operatorname{Sym}(x)\operatorname{Sym}(y)=XY$, so this linear map
is not an algebra homomorphism. Full symmetrization in more than two
factors always means \eqref{I:eq:rp-Sym}, not an iterated binary product.

\begin{theorem}[Associative higher Schwarzians]\label{I:thm:rp-higher}
Let $q_m$ be the scalar correction of Proposition~\ref{I:prop:rp-scalar}.
For $3\le m\le N$, define
\begin{equation}\label{I:eq:rp-Q-definition}
 Q_m=\operatorname{Sym}(q_m),\qquad
 W_m=I_m+Q_m(X_{k,r}=\Delta^rI_k).
\end{equation}
Then $W_m$ is a universal gauge covariant and a parameter covariant
of weight $m$:
\[
 W_m^g=g^{-1}W_mg,\qquad
 W_m^\lambda=\lambda_1^m\bar W_m.
\]
\end{theorem}

\begin{remark}[The canonical choice]\label{I:rem:canonical-choice}
The representative $W_m$ is canonical once the scalar flat-jet
restriction and the full symmetrization convention have been fixed.
Proposition~\ref{I:prop:rp-scalar} makes the scalar polynomial unique;
equation~\eqref{I:eq:rp-Sym} then selects its associative lift.
There is no uniqueness assertion among all associative polynomials
with the two covariance laws, even if their scalar specializations
agree. Example~\ref{I:ex:matrix-ambiguity} below exhibits the ambiguity.
\end{remark}

For a general associative $I_2$, a central parameter change need not
make it vanish: the inhomogeneous term in \eqref{I:eq:rp-low-laws} is
central. We therefore construct the scalar polynomial first and use
affine transport to prove the parameter law after symmetrization.

\subsection{The formal projective parameter}

The normalizer can be constructed coefficient by coefficient. Form
the Taylor series

\[
 J_k(t)=\sum_{r\ge0}x_{k,r}\frac{t^r}{r!},\qquad
 p(t)=\frac{J_2(t)}{N+1},\qquad
 \mathcal L=D_t^N+\sum_{k=2}^N\binom Nk J_k(t)D_t^{N-k}.
\]
Here $D_t$ is the Ore symbol for $d/dt$ on $\mathcal J_N[[t]]$.
The three initial values below fix the residual M\"obius freedom in a
parameter with vanishing projective connection. Choose $\nu(t)$ by
\begin{equation}\label{I:eq:rp-normalizer}
 \begin{gathered}
 \nu(0)=0,\qquad \nu'(0)=1,\qquad \nu''(0)=0,\\
 \nu'''=\frac{3(\nu'')^2}{2\nu'}-6(\nu')^3p(\nu).
 \end{gathered}
\end{equation}
The equation is precisely $p^\nu=0$. In ordinary Taylor coefficients,
the coefficient of $t^r$ on the right depends only on those of $\nu$
through degree $r+2$. On the left it is $(r+3)(r+2)(r+1)$ times
the next coefficient. This recursively gives a unique solution;
for example, $\nu'''(0)=-6p(0)$. Since $\nu'(0)=1$, its reciprocal
exists as a formal series, and all coefficients obtained are rational
polynomials in the $x_{2,r}$. The construction takes place entirely in the formal series ring.

Let $b_k(t)=I_k(\mathcal L^\nu)$. These coefficients are obtained by
\eqref{I:eq:rp-operator} and \eqref{I:eq:local-I-formulas}; in particular,
$b_2=0$. Define commuting polynomials
\begin{equation}\label{I:eq:rp-q-definition}
 \omega_m=\sum_{r=0}^{m-3}c_{m,r}\,b_{m-r}^{(r)}(0),
 \qquad q_m=\omega_m-x_{m,0}.
\end{equation}
For a general scalar operator, evaluate $\omega_m$ by
$x_{k,r}=\left.(d/dt)^r I_k(L)\right|_{t=0}$. These normalized jets
are unchanged by a scalar gauge. Formula \eqref{I:eq:rp-q-definition}
is $E_m(\mathcal L^\nu)$ evaluated at zero: the missing last
summand vanishes because $b_2=0$, and the transformed $a_1$ is central,
so covariant differentiation is ordinary differentiation.

The transformation of $I_k$ uses parameter derivatives only through
order $k+1$; taking $r$ further derivatives raises that bound by $r$.
Thus only the coefficients of $\nu$ through order $m+1$ and the variables
$x_{k,r}$ with $k+r\le m$ enter \eqref{I:eq:rp-q-definition}.
The recurrences are homogeneous if $\nu^{(j)}(0)$ is given weight
$j-1$. Thus $\omega_m$ has weight $m$. The coefficient of $x_{m,0}$ is one, and every other term uses
only $x_{k,r}$ with $k<m$. Equations
\eqref{I:eq:rp-normalizer}--\eqref{I:eq:rp-q-definition} are therefore a
finite prescription for $q_m$. The whole formal series is used to
prove covariance; computing $q_m$ requires only the truncation through
\(\nu^{(m+1)}(0)\). The output is a polynomial in coefficient jets,
not an infinite series in those variables.

\subsection{Covariance and the associative lift}

Two parameters that make the scalar projective connection vanish differ
by a M\"obius transformation. This determines the finite parameter law;
affine transport then carries it through symmetrization.

\begin{proof}[Proof of Proposition~\ref{I:prop:rp-scalar}]
The finite construction above gives the polynomial and its stated shape.

\textit{Scalar covariance.}
Let $\mu$ be any formal change of
parameter fixing zero, and let $\eta$ be the normalized projective
parameter for $\mathcal L^\mu$. The latter operator need not have zero
first coefficient. Its normalized jets are unchanged by a scalar
gauge, so they define $\eta$ by the same prescription.
Pullback composes in the order
$(\mathcal L^\mu)^\eta=\mathcal L^{\mu\circ\eta}$: the two leading
factors multiply to $(\eta'\,\mu'\!\circ\eta)^N$.
Both $\nu$ and $\mu\circ\eta$ make
the projective connection vanish. Their transition
\[
 M=\nu^{-1}\circ\mu\circ\eta
\]
satisfies \(\mu\circ\eta=\nu\circ M\). Applying the projective
connection law to \((\mathcal L^\nu)^M\) gives \(0=S_M\), so
\(M\) has zero Schwarzian. A formal series fixing zero
with zero Schwarzian is M\"obius: for initial derivatives $a$ and
$b$, with $a$ invertible, the unique formal solution is
$at/(1-bt/(2a))$. Here $M'(0)=\mu'(0)$. Proposition~\ref{I:prop:rp-projective} now gives
\[
 \omega_m(\mathcal L^\mu)
 =E_m((\mathcal L^\nu)^M)(0)
 =\mu'(0)^m E_m(\mathcal L^\nu)(0)
 =\mu'(0)^m \omega_m(\mathcal L).
\]
Take the coefficient jets and the jets $\mu^{(j)}(0)$ to be independent
indeterminates, with $\mu'(0)$ inverted. The displayed equality is then
an identity in the polynomial ring with $\mu'(0)$ inverted, so it proves the universal
scalar transformation law and specializes to the abstract chain-rule data.

\textit{Uniqueness.}
When all $x_{2,r}$ vanish, \eqref{I:eq:rp-normalizer} has the solution
$\nu(t)=t$, so $\omega_m$ restricts to $E_m$. Conversely, any scalar
covariant with this restriction has the same value after projective
normalization. Its covariance and $\nu'(0)=1$ recover that value in
the original parameter. This proves uniqueness.

\end{proof}

\begin{proof}[Proof of Theorem~\ref{I:thm:rp-higher}]
\textit{The associative lift.}
By Lemma~\ref{I:lem:rp-affine},
parameter changes act affinely on the normalized jets, with central
coefficients.
Such substitutions commute with \eqref{I:eq:rp-Sym}. Indeed, label the
factors of a monomial before making the substitution. In any term of
the expansion, move its central factors outside the average. The
remaining factors occur in every order with equal multiplicity,
giving their full symmetrization. This argument includes repeated
factors and terms consisting entirely of constants.

We may therefore symmetrize both sides of the scalar transformation
identity, obtaining $W_m^\lambda=\lambda_1^m\bar W_m$.
This uses equivariance of a linear map, not multiplicativity of
$\operatorname{Sym}$. Gauge covariance follows from that of every
$\Delta^rI_k$.
\end{proof}

For $m=3$, the construction gives \eqref{I:eq:rp-W3};
Example~\ref{I:ex:matrix-W3} evaluates it on a noncommutative operator.
At weight five, the failure of uniqueness has a concrete form.
The laws of $I_2$ and $\Delta I_2$ give
\[
 [I_2,\Delta I_2]^\lambda
      =\lambda_1^5\overline{[I_2,\Delta I_2]}.
\]
All extra terms are central or proportional to $\bar I_2$, so their
commutators vanish. The expression is also gauge covariant and has
zero scalar specialization. Thus $W_5+c[I_2,\Delta I_2]$ has the
same two laws for every $c\in\Q$ when $N\ge5$.

\begin{example}[A nonzero weight-five ambiguity]\label{I:ex:matrix-ambiguity}
Over $M_2(\Q[t])$, take
\[
 N=5,\qquad L=D^5+10(tE_{12}+E_{21})D^3.
\]
The operator is already normalized: $I_2=tE_{12}+E_{21}$,
$I_3=I_4=I_5=0$, and $\Delta=\delta$. Hence
\[
 [I_2,\Delta I_2]=E_{22}-E_{11}\ne0.
\]
Here $E_3=-3E_{12}/2$ and $E_5=0$. The formula for $W_5$ in
Appendix~\ref{I:app:formulas} gives
\[
 W_5=\frac{60}{7}\mathbf1_2,\qquad
 W_5+c[I_2,\Delta I_2]
     =\begin{pmatrix}60/7-c&0\\0&60/7+c\end{pmatrix}.
\]
The representatives are distinct for distinct rational $c$.
\end{example}

The next section compares the representative chosen by full
symmetrization with the classical scalar and matrix currents.

\section{Adler brackets and primary fields}
\label{I:sec:local-W-algebras}

Weighted parameter covariance becomes primarity when the Virasoro
Hamiltonian action is reparametrization. We prove this identity of
actions for the scalar and matrix Adler brackets. It identifies our
currents with primary generators of the classical $W_N$-algebra
and with Bilal's matrix primary fields \cite{Bilal1995}.
For scalar coefficients the underlying structures are the classical
Adler--Gelfand--Dickey and Drinfeld--Sokolov brackets
\cite{Adler1979,GelfandDickey1976,DrinfeldSokolov1985};
our primarity convention is that of \cite{BouwknegtSchoutens1993}.
For $3\le m\le N$, the notation is
\[
 \begin{aligned}
 W_m^{\mathrm{sc}}&=\text{scalar specialization of }W_m,
 &w_m&=\binom Nm W_m^{\mathrm{sc}},\\
 V_m&=\binom Nm W_m(L)\quad\text{for matrix }L,
 &\mathcal W_m^{\mathrm B}(-U)&=-V_m.
 \end{aligned}
\]
Here $U$ denotes the tuple of matrix coefficients introduced in
Section~\ref{I:subsec:wa-matrix}. The last equality uses Bilal's sign
convention, established there; the quadratic fields are defined separately.

We begin with commuting coefficients. In this section, write \(\partial\)
for coefficient differentiation and \(D\) for the Ore symbol, so that
\(Da=aD+\partial a\). Primes denote \(\partial\). Put
\begin{equation}\label{I:eq:wa-algebra}
 \mathcal R_N=\Q[u_k^{(r)}\mid 2\le k\le N,\ r\ge0],
 \qquad L=D^N+\sum_{k=2}^N u_kD^{N-k}.
\end{equation}
The derivation is determined by \(\partial u_k^{(r)}=u_k^{(r+1)}\).
The coefficient of $D^{N-1}$ is zero. Our normalization gives
\begin{equation}\label{I:eq:AGD-dictionary}
 u_k=\binom Nk I_k,\qquad
 T=u_2=6\kappa_NP,\qquad
 \kappa_N=\frac{N(N^2-1)}{12}.
\end{equation}

\subsection{The reduced scalar Adler bracket}
A Hamiltonian bracket specifies an infinitesimal change in the coefficients
of $L$ for each functional. The reduced Adler map computes this change
from the variational gradient, while preserving the zero coefficient of
$D^{N-1}$. We give its formula to fix the sign and normalization.
The scalar pseudodifferential algebra $\mathcal R_N((D^{-1}))$
allows series bounded above in powers of $D$. Its multiplication is
given by
\[
 D^j a=\sum_{r\ge0}\binom jr \partial^r(a)D^{j-r},\qquad j\in\Z.
\]
For an operator $A$, let $A_+$ be its part with nonnegative powers of
$D$, and let $\operatorname{res}A$ be its coefficient of $D^{-1}$.
For a negative-order scalar pseudodifferential operator $X$, the
residue of $[L,X]$ is a total derivative. Choose a coefficient $b_X$ with
\[
 \partial b_X=\operatorname{res}[L,X],
\]
and define
\begin{equation}\label{I:eq:wa-reduced-Adler}
 \mathcal H_L(X)=(LX)_+L-L(XL)_++\frac1N[L,b_X].
\end{equation}
Here $b_X$ is a coefficient primitive, not the operator product
$D^{-1}\operatorname{res}[L,X]$. Its existence follows from the
Leibniz identity
\[
 a\partial^r(b)-(-1)^r b\partial^r(a)
 =\partial\!\left(\sum_{s=0}^{r-1}(-1)^s\partial^s(a)\partial^{r-1-s}(b)\right),
 \qquad r\ge1.
\]
Indeed, the residue of each monomial commutator is a scalar multiple
of its left side; terms with $r=0$ cancel. Only finitely many terms
contribute to the residue. The additive constant in $b_X$ has no effect.
The last term preserves the constraint: the coefficient of $D^{N-1}$ in the first two terms
is $-\operatorname{res}[L,X]$, and the correction contributes $\partial b_X$.

Let $\zeta$ be a formal bracket parameter and $h$ an auxiliary
differential indeterminate. Define the coefficient operators by
\begin{equation}\label{I:eq:wa-Hij}
 \mathcal H_L(D^{j-N-1}h)
   =\sum_{i=2}^N \bigl(H_{ij}(\partial)h\bigr)D^{N-i},\qquad
 \{u_j{}_{\zeta}u_i\}=H_{ij}(\zeta).
\end{equation}
The $H_{ij}$ are differential operators with coefficients in
$\mathcal R_N$. Thus the bracket is local, although its definition
uses a primitive. Extend it to differential polynomials by the
Leibniz rules and sesquilinearity:
\[
 \{\partial f{}_{\zeta}g\}=-\zeta\{f{}_{\zeta}g\},\qquad
 \{f{}_{\zeta}\partial g\}=(\partial+\zeta)\{f{}_{\zeta}g\}.
\]
The Adler--Gelfand--Dickey theorem says that this bracket satisfies
skewsymmetry and the Jacobi identity. It is the second Adler bracket
after Dirac reduction by the coefficient of $D^{N-1}$; the resulting
Poisson vertex algebra is the classical principal $W$-algebra of
$\mathfrak{sl}_N$
\cite[Proposition~2.18 and Remark~2.20]{DeSoleKacValeri2015}.

A local functional $\int f$ means the class of
$f$ modulo total derivatives, and its Hamiltonian action is
$\{\int f,g\}=\{f{}_{\zeta}g\}|_{\zeta=0}$.
Equivalently, with the variational derivatives placed on the right,
\begin{equation}\label{I:eq:wa-gradient}
 X_f=\sum_{j=2}^N D^{j-N-1}\frac{\delta f}{\delta u_j},\qquad
 \frac{\delta f}{\delta u_j}
    =\sum_{r\ge0}(-\partial)^r\frac{\partial f}{\partial u_j^{(r)}},\qquad
 \left\{\int f,L\right\}=\mathcal H_L(X_f).
\end{equation}

\subsection{Virasoro and change of parameter}
A field $w$ is primary of weight $m$ with respect to $T$ if
$\{T{}_{\zeta}w\}=(\partial+m\zeta)w$. The next calculation explains why
this definition applies to our higher Schwarzians.

\begin{proposition}\label{I:prop:wa-Virasoro-action}
Let $v$ be an auxiliary differential indeterminate with zero Poisson
brackets. Then
\begin{equation}\label{I:eq:wa-Virasoro-operator}
 \left\{\int vT,L\right\}
 =\left(vD+\frac{N+1}{2}v'\right)L
       -L\left(vD-\frac{N-1}{2}v'\right).
\end{equation}
This is the infinitesimal monic pullback, followed by the scalar gauge
which restores the zero coefficient of $D^{N-1}$. In particular,
\begin{equation}\label{I:eq:wa-TT}
 \{T{}_{\zeta}T\}=(\partial+2\zeta)T+\kappa_N\zeta^3.
\end{equation}
\end{proposition}

\begin{proof}
For $f=vT$, the gradient is $X=D^{1-N}v$. The Ore rule gives
\[
 (LX)_+=vD+v',\qquad (XL)_+=vD-(N-1)v',\qquad
 \operatorname{res}[L,X]=-\frac{N(N-1)}2v''.
\]
Substitution in \eqref{I:eq:wa-reduced-Adler} proves
\eqref{I:eq:wa-Virasoro-operator}.
On the other hand, differentiate the monic pullback
\eqref{I:eq:rp-operator} at $\lambda(t)=t+\varepsilon v(t)$, with
$\varepsilon^2=0$. Its variation is $[vD,L]+Nv'L$, with normalized first coefficient
$a_1=-\varepsilon(N-1)v''/2$.
The gauge $g=1+\varepsilon(N-1)v'/2$ restores that coefficient to
zero and adds $(N-1)[L,v']/2$ to the variation. The result is
\eqref{I:eq:wa-Virasoro-operator}.

Its coefficient of $D^{N-2}$ is
\[
 vT'+2v'T+
 \left(\frac{N-1}{2}\binom N2-\binom N3\right)v'''
 =vT'+2v'T+\kappa_Nv'''.
\]
Comparison of the independent derivatives of $v$ proves
\eqref{I:eq:wa-TT}.
\end{proof}

\begin{theorem}[Primary generators]\label{I:thm:wa-primary}
For $3\le m\le N$, set
\[
 w_m=\binom Nm W_m^{\mathrm{sc}},
\]
where $W_m^{\mathrm{sc}}$ is the scalar specialization from
Section~\ref{I:sec:higher-construction}, defined there by $\omega_m$. This rescaling
gives leading term $u_m$. Then
\begin{equation}\label{I:eq:wa-primary}
 \{T{}_{\zeta}w_m\}=(\partial+m\zeta)w_m.
\end{equation}
Moreover, $T,w_3,\ldots,w_N$ freely generate $\mathcal R_N$ as a
commutative differential algebra. Hence they give primary generators
of the classical $W_N$-algebra, with $T$ as its Virasoro element.
\end{theorem}

\begin{proof}
Theorem~\ref{I:thm:rp-higher} gives weighted covariance under
reparametrization and invariance under scalar gauge. By
Proposition~\ref{I:prop:wa-Virasoro-action}, its infinitesimal form is
\[
 \left\{\int vT,w_m\right\}=vw_m'+mv'w_m.
\]
If $\{T{}_{\zeta}w_m\}=\sum_r c_r\zeta^r$, the left Leibniz rule
identifies the left side with $\sum_r c_r\partial^r(v)$.
Independence of the $v$-jets proves \eqref{I:eq:wa-primary}.
Finally, $w_m=u_m+$ a differential polynomial in
$u_2,\ldots,u_{m-1}$. This triangular substitution has a differential
polynomial inverse, obtained successively for $m=3,\ldots,N$.
It therefore carries the free differential generators
$u_2,\ldots,u_N$ to $T,w_3,\ldots,w_N$.
\end{proof}

\subsection{Comparison of primary generators}
The primary-generator problem is classical
\cite{DiFrancescoItzyksonZuber1991,Dickey1997}. Our representatives are
specified by their values when $u_2$ and all its derivatives vanish:
\begin{equation}\label{I:eq:wa-flat}
 \left.w_m\right|_{u_2^{(r)}=0\ (r\ge0)}
 =\sum_{r=0}^{m-3}(-1)^r
   \frac{\binom{N-m+r}{r}\binom{m-1}{r}}
        {\binom{2m-2}{r}}\,u_{m-r}^{(r)}.
\end{equation}
Indeed, substitute $I_k=u_k/\binom Nk$ in the restriction of $E_m$
from \eqref{I:eq:rp-E} and multiply by $\binom Nm$.
The scalar uniqueness assertion of Proposition~\ref{I:prop:rp-scalar} shows
that weighted covariance and \eqref{I:eq:wa-flat} determine these
representatives at every order.

\begin{proposition}\label{I:prop:wa-classical-comparison}
With the coefficient convention \eqref{I:eq:AGD-dictionary}, our $w_m$
are the normalized primary currents of Di Francesco, Itzykson and
Zuber \cite{DiFrancescoItzyksonZuber1991}.
\end{proposition}

\begin{proof}
Their operator order $n$ and coefficients $a_k$ are our $N$ and $u_k$.
Their currents transform as $m$-differentials. In their formula
\cite[equation (2.11a)]{DiFrancescoItzyksonZuber1991}, every nonlinear
term contains a jet of $u_2$. Setting all $u_2^{(r)}$ to zero therefore
leaves the linear coefficients
\[
 B_{m,m-r}=(-1)^r
 \frac{\binom{m-1}{r}\binom{N-m+r}{r}}{\binom{2m-2}{r}},
\]
given in their equation (2.24). These are exactly
\eqref{I:eq:wa-flat}. Thus $w_m^{\mathrm{DFIZ}}/\binom Nm$ and
$W_m^{\mathrm{sc}}$ have the same covariance and the same full
flat-jet restriction. Proposition~\ref{I:prop:rp-scalar} gives equality.
\end{proof}

The first two fields, in the ordinary coefficient convention, are
\begin{equation}\label{I:eq:wa-low-fields}
 \begin{aligned}
 w_3={}&u_3-\frac{N-2}{2}u_2',\\
 w_4={}&u_4-\frac{N-3}{2}u_3'
       +\frac{(N-2)(N-3)}{10}u_2''
       -\frac{(N-2)(N-3)(5N+7)}{10N(N^2-1)}u_2^2.
 \end{aligned}
\end{equation}
The second formula applies when $N\ge4$; both agree with
\cite[Table~I]{DiFrancescoItzyksonZuber1991}.
At higher weights, a leading term alone is insufficient: for example,
$w_6+c w_3^2$ is also primary when $N\ge6$. Condition
\eqref{I:eq:wa-flat} excludes this freedom.

\subsection{The cubic scalar operator}
For $N=2$, equation \eqref{I:eq:wa-TT} is the Virasoro bracket with
$\kappa_2=1/2$. The first nonlinear example is
$L=D^3+TD+b$. Put $w=b-T'/2$; thus $w=W_3^{\mathrm{sc}}$.
The brackets are
\begin{equation}\label{I:eq:wa-W3-brackets}
 \begin{aligned}
 \{T{}_{\zeta}T\}&=(\partial+2\zeta)T+2\zeta^3,\\
 \{T{}_{\zeta}w\}&=(\partial+3\zeta)w,\\
 \{w{}_{\zeta}w\}&=-\frac16\zeta^5-\frac56T\zeta^3
             -\frac54T'\zeta^2
             -\left(\frac34T''+\frac23T^2\right)\zeta
             -\frac16T'''-\frac23TT'.
 \end{aligned}
\end{equation}
For a direct check of the last line, \eqref{I:eq:wa-Hij} gives
\[
 \begin{aligned}
 \{T{}_{\zeta}b\}&=(\partial+3\zeta)b+T\zeta^2+\zeta^4,\\
 \{b{}_{\zeta}b\}&=(\partial+2\zeta)
       \left(b'-\frac12T''-\frac13T^2\right)
       -\frac16(\partial+2\zeta)^3T-\frac23\zeta^5.
 \end{aligned}
\]
Substitute $w=b-T'/2$ and use sesquilinearity and skewsymmetry;
all terms containing $b$ cancel. These are the classical
$\mathfrak{sl}_3$ brackets in the conventions of
\cite[Example~2.22]{DeSoleKacValeri2015}. The quadratic term $T^2$
exhibits the nonlinear extension of the Virasoro algebra.

\subsection{Matrix coefficients}\label{I:subsec:wa-matrix}
The associative construction also applies to the classical matrix
$W$-algebras of \cite[\S4]{DeSoleKacValeri2015}. Fix $d\ge1$ and put
\[
 \mathcal R_{N,d}=\Q[(U_k)_{ab}^{(r)}\mid
  2\le k\le N,\ 1\le a,b\le d,\ r\ge0],\qquad
 L=\mathbf1_dD^N+\sum_{k=2}^N U_kD^{N-k}.
\]
The entries commute, but the coefficient matrices multiply in their
given order. Equip $\mathcal R_{N,d}$ with the second matrix Adler
bracket after Dirac reduction by the entire coefficient of $D^{N-1}$,
as in \cite[equation (4.11)]{DeSoleKacValeri2015}.

\begin{proposition}\label{I:prop:wa-matrix-primary}
Put $T=\operatorname{tr}U_2$,
$V_2=U_2-(T/d)\mathbf1_d$, and
$V_m=\binom Nm W_m(L)$ for $3\le m\le N$. Then
\[
 \{T{}_{\zeta}T\}=(\partial+2\zeta)T+d\kappa_N\zeta^3,
 \qquad
 \{T{}_{\zeta}(V_m)_{ab}\}=(\partial+m\zeta)(V_m)_{ab}
 \quad(2\le m\le N).
\]
The element $T$, any $d^2-1$ independent linear coordinates of $V_2$,
and all entries of $V_3,\ldots,V_N$ freely generate
$\mathcal R_{N,d}$ as a commutative differential algebra.
\end{proposition}

\begin{proof}
Under the trace-residue pairing, the gradient of $\int vT$ is
$X=D^{1-N}v\mathbf1_d$, where $v$ is a scalar test indeterminate.
Write $\mathcal A_L(X)=(LX)_+L-L(XL)_+$ for the matrix Adler map.
Its coefficient of $D^{N-1}$ is $-\operatorname{res}[L,X]$.
To preserve this constraint, add the normal gradient component
$D^{-N}F$. It changes the coefficient by $-NF'$, since
$\mathcal A_L(D^{-N}F)=FL-LF$.
For this gradient,
\[
 \operatorname{res}[L,X]=-\frac{N(N-1)}2v''\mathbf1_d.
\]
Thus the Dirac correction is given by $F=(N-1)v'\mathbf1_d/2$;
this is the local value of the formal inverse of the constraint block
in \cite[Lemma~4.4(c) and \S4.3]{DeSoleKacValeri2015}. Taking the Adler image gives
\[
 \left\{\int vT,L\right\}
 =\left(vD+\frac{N+1}{2}v'\right)L
       -L\left(vD-\frac{N-1}{2}v'\right).
\]
This is the pullback and restoring gauge of
Proposition~\ref{I:prop:wa-Virasoro-action}. Theorem~\ref{I:thm:rp-higher} therefore
gives the claimed primary law for $m\ge3$. Its quadratic coefficient
is $vU_2'+2v'U_2+\kappa_Nv'''\mathbf1_d$; taking its trace and traceless
part proves the remaining bracket identities. Finally,
$V_m=U_m+$ an ordered differential polynomial in lower coefficients.
This triangular substitution, together with
$U_2=(T/d)\mathbf1_d+V_2$, has a differential polynomial inverse.
\end{proof}

The brackets with $T$ are local although the full reduced matrix
bracket is generally nonlocal. Bilal's theory already supplies a
primary differential basis; we now compare the current polynomials
with his construction.

To compare the expressions exactly, write Bilal's operator as
$L_{\mathrm B}=-\mathbf1_dD^N+\sum_{k=2}^N B_kD^{N-k}$.
His order $m$ and matrix size $n$ are our $N$ and $d$.
For $L_{\mathrm B}=-L$, the coefficient dictionary is $B_k=-U_k$.
For $3\le m\le N$, let $\mathcal W_m^{\mathrm B}$ denote his current polynomial. Then
\begin{equation}\label{I:eq:wa-Bilal-dictionary}
 \mathcal W_m^{\mathrm B}(-U_2,\ldots,-U_N)=-V_m.
\end{equation}
Indeed, Bilal's definition \cite[Theorem~4.6 and equations (4.19)--(4.23)]{Bilal1995}
is the full symmetrization of the DFIZ polynomial, with these sign
changes. Proposition~\ref{I:prop:wa-classical-comparison} identifies the
scalar polynomials, and symmetrization commutes with the diagonal
binomial rescaling. This proves \eqref{I:eq:wa-Bilal-dictionary} in
every order. The identity compares current polynomials; our bracket
normalization is fixed by \cite{DeSoleKacValeri2015}.

For general matrix gradients, the constraint correction requires
primitives of commutators, which need not be differential polynomials. For example, if $L=D^2+U$ and $F$ is a matrix
test coefficient, the Ore rule gives
\[
 \operatorname{res}[L,D^{-1}F]=[U,F]-F''.
\]
For $d>1$, the term $[U,F]$ has no universal entrywise
differential-polynomial primitive: setting all positive-order jets to
zero kills every total derivative but need not kill $[U,F]$.
Dirac reduction therefore uses a formal inverse of the derivation. For $F=v\mathbf1_d$ this commutator vanishes, explaining
why the Virasoro action above remains local.

\begin{remark}[Scope of the Hamiltonian comparison]
This section treats scalar differential polynomials and matrices whose
entries lie in a commutative differential algebra. The ordered
coefficient calculus of the preceding sections applies to arbitrary
associative differential algebras. A Hamiltonian interpretation at
that level requires a noncommutative bracket formalism, such as the
double Poisson vertex algebras of De Sole, Kac and Valeri
\cite{DeSoleKacValeriDouble2015}. Their theory treats noncommutative
Hamiltonian equations, including Gelfand--Dickey hierarchies.
Comparing our universal covariants with that formalism is left to
future work.
\end{remark}

\subsection*{Notation for the universal calculus}
\label{I:sec:notation-index}
\begingroup\small
\renewcommand{\arraystretch}{1.08}
\begin{tabularx}{\linewidth}{@{}l>{\raggedright\arraybackslash}Xl@{}}
\toprule
Symbols & Meaning & Defined in \\
\midrule
$D,\delta,\partial$ & Ore symbol; coefficient derivation; derivation on jets or on the target algebra & \S\ref{I:sec:local-ore}, \S\ref{I:sec:reparam} \\
$\nabla,\Delta,\Delta^\lambda$ & Adapted operator $D+a_1$; covariant derivation; its transformed version & \S\ref{I:sec:local-normalization}, \S\ref{I:sec:local-generation}, \S\ref{I:sec:reparam} \\
$B_r^{\mathrm L},B_r^{\mathrm R}$ & Left and right Bell families & \eqref{I:eq:local-bell} \\
$I_k,\ I_{k,r}$ & Normalized coefficients and their jets $\Delta^rI_k$ & \S\ref{I:sec:local-normalization}, \S\ref{I:sec:reparam} \\
$\pi_0,\sigma$ & Restriction to $a_1=0$ and reconstruction by covariant jets & \eqref{I:eq:local-restriction} \\
$\lambda_j,\ell,\rho,S_\lambda$ & Central parameter jets; $\ell=\lambda_1$; $\rho=\lambda_2/\lambda_1$; one sixth of the Schwarzian & \S\ref{I:sec:reparam} \\
$\beta_{j,r}$ & Central normal-ordering polynomials for parameter pullback & \eqref{I:eq:rp-beta} \\
$P,E_m$ & Affine quadratic datum and linear M\"obius covariants & \eqref{I:eq:P}, \eqref{I:eq:rp-E} \\
$\omega_m,q_m$ & Scalar normalized polynomial and its lower-order correction & \eqref{I:eq:rp-q-definition} \\
$\operatorname{Sym},\odot$ & Full linear symmetrization and its binary product & \eqref{I:eq:rp-Sym} \\
$Q_m,W_m$ & Symmetrized correction and associative higher Schwarzian & \eqref{I:eq:rp-Q-definition} \\
$W_m^{\mathrm{sc}},w_m,V_m,\mathcal W_m^{\mathrm B}$ & Scalar specialization, scalar current, matrix current, and Bilal's current & \S\ref{I:sec:local-W-algebras} \\
\bottomrule
\end{tabularx}
\par\endgroup

\clearpage
\part{Curves, modular forms, and extension classes}\label{part:II}
\markboth{Part II: Curves and bundles}{}
The transformation laws of Part~\ref{part:I} allow coefficient invariants
to be defined on a curve. Sections~\ref{II:sec:modular-calculus}--\ref{II:sec:rankin-cohen}
give the gluing and differentiation rules, keeping the coefficient
connection among the reconstruction data. The kernel construction in
Section~\ref{II:sec:kernel} then supplies these data from an acyclic
bundle on a compact curve. Its cubic self-bracket recovers marked
extensions when the endpoint cubics differ; the higher hierarchy deals
with the remaining pairs and with self-extensions
(Sections~\ref{II:sec:extension-recovery}--\ref{II:sec:higher-detection}).
Section~\ref{II:sec:x023} carries the recovery through rational formulas
and Fourier coefficients on $X_0(23)$.

The kernel applications use a fixed curve and a bundle on that curve.
The period applications of Part~\ref{part:III} use a varying Hodge
subbundle over the parameter curve. They can be read independently
after Part~\ref{part:I}; the shared curve-operator convention is stated
in Section~\ref{II:sec:curve-schwarzians}.

\section{Algebra-valued modular forms and covariant differentiation}
\label{II:sec:modular-calculus}

A higher Schwarzian of degree $m$ will give a modular form of weight
$2m$ with values in a coefficient algebra. To differentiate such a
form, we must account both for its weight and for changes of
coefficient frame. We define the forms first, then construct the two
corrections to differentiation.

Let $\mathbb H=\{\tau\in\C:\operatorname{Im}\tau>0\}$, and let
$\Gamma\subset\operatorname{SL}_2(\mathbb R)$ be a discrete torsion-free
group of finite covolume. Thus its action on $\mathbb H$ is free and
$Y=\Gamma\backslash\mathbb H$ is a smooth complex curve. Write $\bar Y$
for its compactification by cusps; when $Y$ is compact, $\bar Y=Y$.
For $\gamma=\left(\begin{smallmatrix}a&b\\c&d\end{smallmatrix}\right)$,
put
\[
 \gamma\tau=\frac{a\tau+b}{c\tau+d},\qquad
 j(\gamma,\tau)=c\tau+d,\qquad
 \kappa=\frac1{2\pi i},\qquad D=\kappa\frac{d}{d\tau}.
\]
The lift to $\operatorname{SL}_2(\mathbb R)$ fixes the meaning of integer
input weights for differential operators. Our coefficient forms have
even weights and therefore depend only on the effective projective
action. These conventions also apply to a torsion-free projective
congruence group after choosing a lift.

Fix a finite-dimensional unital associative $\C$-algebra $A$, with
identity $1$, and a homomorphism
\[
 \alpha:\Gamma\longrightarrow\operatorname{Aut}_{\C\text{-alg}}(A).
\]
We assume that $\alpha$ is the identity on every parabolic stabilizer.
This assumption gives ordinary Laurent expansions at the cusps; it
will hold for the bundles constructed on $\bar Y$ in
Section~\ref{II:sec:kernel}. Triviality at the cusps does not require
finite image: the character in Section~\ref{II:subsec:x023-multiplier}
has infinite image and is trivial on every parabolic stabilizer.
The associated bundle
\[
 \mathcal A=(\mathbb H\times A)/\Gamma\longrightarrow Y,
 \qquad \gamma(\tau,a)=(\gamma\tau,\alpha(\gamma)a),
\]
is a holomorphic bundle of associative algebras. Its flat connection
is the one induced by entrywise differentiation on $\mathbb H\times A$.

\subsection{Forms with algebra coefficients}
\label{II:subsec:algebra-valued-forms}

For an even integer $k$, an $A$-valued meromorphic modular form of
weight $k$ is a meromorphic function $F:\mathbb H\to A$ satisfying
\begin{equation}\label{II:eq:modular-form-law}
 F(\gamma\tau)=j(\gamma,\tau)^k\alpha(\gamma)F(\tau)
 \qquad(\gamma\in\Gamma),
\end{equation}
with finite Laurent principal parts at the cusps. Meromorphicity is
understood in any vector-space basis of $A$. To specify the cusp
condition, choose $\sigma\in\operatorname{SL}_2(\mathbb R)$ taking
$\infty$ to a cusp. Write $h>0$ for the width, so its stabilizer
is generated, up to the sign of the lift, by
$\sigma\left(\begin{smallmatrix}1&h\\0&1\end{smallmatrix}\right)\sigma^{-1}$.
Set
\begin{equation}\label{II:eq:cusp-form}
 F_\sigma(\tau)=j(\sigma,\tau)^{-k}F(\sigma\tau),\qquad
 q=\exp(2\pi i\tau/h).
\end{equation}
Then $F_\sigma$ is $h$-periodic and must have an expansion
$\sum_{n\ge n_0}F_nq^n$, with $F_n\in A$ and $n_0\in\mathbb Z$.
Write $M_k^{\mathrm{mer}}(\Gamma,A,\alpha)$ for this space.
Its subspace $M_k(\Gamma,A,\alpha)$ consists of forms holomorphic on
$\mathbb H$ for which $n_0\ge0$ at every cusp. Cusp forms have
$F_0=0$ at every cusp. Scalar forms mean $A=\C$ and trivial $\alpha$;
we then abbreviate the spaces to $M_k^{\mathrm{mer}}(\Gamma)$ and
$M_k(\Gamma)$.

Multiplication is ordered multiplication in $A$. Since $\alpha(\gamma)$
is an algebra automorphism,
\[
 M_k^{\mathrm{mer}}(\Gamma,A,\alpha)\,
 M_\ell^{\mathrm{mer}}(\Gamma,A,\alpha)
 \subset M_{k+\ell}^{\mathrm{mer}}(\Gamma,A,\alpha).
\]
Thus the direct sum over even weights is a graded associative algebra.
We will not interchange factors in this algebra.

The coefficient algebra for a vector bundle is its endomorphism
algebra. Accordingly, the principal example is
\begin{equation}\label{II:eq:matrix-multiplier}
 A=\End(\C^r),\qquad
 \alpha(\gamma)P=\rho(\gamma)P\rho(\gamma)^{-1},
 \qquad \rho:\Gamma\to\operatorname{GL}_r(\C).
\end{equation}
Here we assume $\rho$ itself is trivial on parabolic stabilizers.
The bundle
$V_\rho=(\mathbb H\times\C^r)/\Gamma\to Y$, with
$\gamma(\tau,v)=(\gamma\tau,\rho(\gamma)v)$, has its flat connection
$\nabla^0$ induced by $d$, and $\mathcal A=\End(V_\rho)$.
If $F$ has weight $2m$, then
\begin{equation}\label{II:eq:form-tensor}
 F(\tau)(d\tau)^m
 \quad\text{descends to a meromorphic section of}\quad
 \End(V_\rho)\otimes K_Y^m.
\end{equation}
Here $K_Y$ is the holomorphic cotangent line bundle of $Y$, and its
powers are tensor powers. Indeed, $d(\gamma\tau)=j^{-2}d\tau$ cancels the weight factor
in \eqref{II:eq:modular-form-law}.

The parabolic assumption extends $(V_\rho,\nabla^0)$ over $\bar Y$:
a single-valued horizontal frame on a punctured cusp disc extends
as a frame over its center. We use the same notation for this
extension. In such a frame, \eqref{II:eq:form-tensor} becomes
\begin{equation}\label{II:eq:cusp-tensor}
 F_\sigma(q)\left(\frac{h}{2\pi i}\right)^m
             q^{-m}(dq)^m.
\end{equation}
Consequently, for $m\ge1$ a holomorphic $m$-differential on $\bar Y$
corresponds to a modular coefficient vanishing to order at least $m$
at each cusp. Holomorphic modular coefficients alone can give poles
of the tensor at the cusps.

With trivial $\alpha$, the space of forms is simply
$M_k^{\mathrm{mer}}(\Gamma)\otimes_\C A$.
The multiplier in \eqref{II:eq:matrix-multiplier} allows the coefficient
algebra to vary with a flat bundle. It is this case that will arise
from the canonical kernel.

\subsection{The scalar correction to differentiation}
\label{II:subsec:scalar-connection}

Differentiation almost preserves \eqref{II:eq:modular-form-law}.
The chain rule gives
\begin{equation}\label{II:eq:derivative-anomaly}
 (DF)(\gamma\tau)
 =j^{k+2}\alpha(\gamma)(DF)(\tau)
       +k\kappa c j^{k+1}\alpha(\gamma)F(\tau).
\end{equation}
The second term is removed by a scalar meromorphic function $\eta$
with transformation law
\begin{equation}\label{II:eq:eta-law}
 \eta(\gamma\tau)=j^2\eta(\tau)+\kappa cj.
\end{equation}
We call such an $\eta$ a \emph{scalar modular connection}. At a cusp,
its coefficient is
\begin{equation}\label{II:eq:eta-cusp}
 \eta^\sigma(\tau)=j(\sigma,\tau)^{-2}\eta(\sigma\tau)
                    -\frac{\kappa c_\sigma}{j(\sigma,\tau)},
\end{equation}
and we require a finite Laurent principal part in $q$. A holomorphic
modular connection is holomorphic on $\mathbb H$ and has no negative
powers in these cusp expansions. For the classical scalar Serre derivative, see
\cite[Section~5.1]{Zagier2008}.

Define the weight-raising derivative by
\begin{equation}\label{II:eq:scalar-covariant}
 d_k^\eta F=DF-k\eta F.
\end{equation}
Equations \eqref{II:eq:derivative-anomaly} and \eqref{II:eq:eta-law} show that
$d_k^\eta F$ has weight $k+2$. For scalar $F$ and $k=2m$, this
is the coefficient expression of a meromorphic connection on $K_Y^m$:
\begin{equation}\label{II:eq:scalar-connection-tensor}
 \nabla^\eta\bigl(F(d\tau)^m\bigr)
       =\kappa^{-1}(d_{2m}^\eta F)(d\tau)^{m+1}.
\end{equation}
The connection is understood on meromorphic sections. Thus $-k\eta F$
corrects differentiation of the tensor factor; the coefficient-bundle
correction will be introduced separately below. Since $\eta$ is scalar,
\begin{equation}\label{II:eq:graded-Leibniz}
 d_{k+\ell}^\eta(FG)=(d_k^\eta F)G+F(d_\ell^\eta G).
\end{equation}
Both statements hold in cusp coordinates with $\eta$ replaced by
$\eta^\sigma$. Thus the derivative preserves the stated meromorphic
conditions; it preserves holomorphicity when $\eta$ is holomorphic.

Such meromorphic connections exist. For a nonzero scalar form $f$
of nonzero weight $k$, equation \eqref{II:eq:derivative-anomaly} gives
\begin{equation}\label{II:eq:logarithmic-connection}
 \eta_f=\frac{Df}{kf}.
\end{equation}
For existence without an initial form, take a nonconstant meromorphic
function $u$ on $\bar Y$. Its pullback has weight zero, so $Du$ is a
nonzero meromorphic form of weight two, and
$\eta=D^2u/(2Du)$ has the required law and cusp expansions.
The difference of any two scalar modular connections belongs to
$M_2^{\mathrm{mer}}(\Gamma)$.

One additional expression will enter the higher derivatives:
\begin{equation}\label{II:eq:modular-curvature}
 \Omega_\eta=\eta^2-D\eta\in M_4^{\mathrm{mer}}(\Gamma).
\end{equation}
Its modularity follows by differentiating \eqref{II:eq:eta-law}; the
terms involving $c$ cancel. We use the term \emph{modular curvature}
for \eqref{II:eq:modular-curvature}. It is this weight-four expression,
not the exterior curvature of a connection on a curve. If
$v\in M_2^{\mathrm{mer}}(\Gamma)$, direct expansion gives
\begin{equation}\label{II:eq:change-eta-curvature}
 \Omega_{\eta+v}=\Omega_\eta-d_2^\eta v+v^2.
\end{equation}

For the classical modular group, and hence on its subgroups, the
familiar choice is
\[
 \eta=\frac{E_2}{12},\qquad
 \Omega_\eta=\frac{E_4}{144},
\]
where, with $q=e^{2\pi i\tau}$ and
$\sigma_a(n)=\sum_{d\mid n}d^a$,
\[
 E_2=1-24\sum_{n\ge1}\sigma_1(n)q^n,\qquad
 E_4=1+240\sum_{n\ge1}\sigma_3(n)q^n.
\]
The identity $DE_2=(E_2^2-E_4)/12$ proves the displayed value of
$\Omega_\eta$. In this case $d_k^\eta$ is the Serre derivative
\cite[Section 5]{NSZ}.

\subsection{A connection on the coefficient bundle}
\label{II:subsec:coefficient-connection}

The scalar connection accounts for weight. A second datum accounts
for differentiation in the coefficient bundle. Fix
$B\in M_2^{\mathrm{mer}}(\Gamma,A,\alpha)$, and put
\begin{equation}\label{II:eq:adjoint-derivative}
 \delta_BF=DF+[B,F],\qquad
 \mathcal D_k^{\eta,B}F=\delta_BF-k\eta F,
 \qquad [B,F]=BF-FB.
\end{equation}
The commutator with $B$ is a derivation. Consequently,
\begin{equation}\label{II:eq:coefficient-Leibniz}
 \mathcal D_{k+\ell}^{\eta,B}(FG)
   =(\mathcal D_k^{\eta,B}F)G+F(\mathcal D_\ell^{\eta,B}G),
\end{equation}
and $\mathcal D_k^{\eta,B}$ raises weight by two.
This separates the scalar correction from the action on coefficients.
For comparison, a noncentral expression $DF-kCF$ would have the
product-rule defect $-\ell[C,F]G$; it cannot replace
\eqref{II:eq:adjoint-derivative} in the arguments below.

In the matrix case \eqref{II:eq:matrix-multiplier}, $B$ defines a
meromorphic connection on $V_\rho$. Write $\mathcal M(W)$ for the
sheaf of meromorphic sections of a holomorphic bundle $W$. Then
\begin{equation}\label{II:eq:bundle-connection-B}
 \begin{aligned}
 \nabla^B&:\mathcal M(V_\rho)\longrightarrow
                         \mathcal M(V_\rho\otimes K_Y),\\
 \nabla^B&=\nabla^0+\kappa^{-1}B(\tau)\,d\tau.
 \end{aligned}
\end{equation}
For a local endomorphism $F$, its induced connection is
\[
 \nabla^{B,\mathrm{ad}}F=\kappa^{-1}\delta_BF\,d\tau.
\]
Together with \eqref{II:eq:scalar-connection-tensor}, this explains the
two corrections in $\mathcal D_k^{\eta,B}$: the scalar term acts on
the cotangent power, and the commutator acts on the endomorphism
bundle. On columns of input weight $w$, put
$\nabla_w^{\eta,B}s=Ds-w\eta s+Bs$. For $F$ of weight $k$ the identity
\[
 \nabla_{w+k}^{\eta,B}(Fs)-F\nabla_w^{\eta,B}s
                         =(\mathcal D_k^{\eta,B}F)s
\]
gives the coefficient derivative and its weight correction together.
The connection in
\eqref{II:eq:bundle-connection-B} is holomorphic on $Y$ when $B$ is.
It extends holomorphically across a cusp precisely when
$B_\sigma=O(q)$, since its one-form there is $hB_\sigma(q)dq/q$.

Let $g\in M_0^{\mathrm{mer}}(\Gamma,A,\alpha)$ be invertible, with
meromorphic inverse. A change of coefficient frame acts by
\begin{equation}\label{II:eq:B-gauge}
 F^g=g^{-1}Fg,\qquad B^g=g^{-1}Bg+g^{-1}Dg.
\end{equation}
Differentiating $g^{-1}g=1$ gives $D(g^{-1})=-g^{-1}(Dg)g^{-1}$.
Substitution in \eqref{II:eq:adjoint-derivative} then cancels the
derivatives of $g$ and yields
\begin{equation}\label{II:eq:derivative-gauge}
 \mathcal D_k^{\eta,B^g}(F^g)
       =g^{-1}(\mathcal D_k^{\eta,B}F)g.
\end{equation}
Thus the derivative is compatible with changes of coefficient
frame. In the matrix case the trace gives scalar forms and satisfies
\[
 \tr(\mathcal D_k^{\eta,B}F)=d_k^\eta\tr(F),
\]
because $\tr[B,F]=0$. Matrix multiplication retains information
that such scalar expressions can lose. The first self-bracket in
Section~\ref{II:sec:rankin-cohen} will make this distinction explicit.

\Needspace{30\baselineskip}

\subsection{Notation for the curve constructions}
\label{II:subsec:notation}
The differential degree on a curve is half the modular weight.
The following symbols distinguish the several connections and kernels.
\begin{center}
\small\renewcommand{\arraystretch}{1.16}
\begin{tabularx}{\textwidth}{@{}lXl@{}}
\toprule
Symbol & Meaning & Definition \\
\midrule
$A,\mathcal A$ & Coefficient algebra and its associated flat algebra bundle & \S\ref{II:sec:modular-calculus} \\
$D,\kappa$ & Fourier derivative $D=\kappa\partial_\tau$, $\kappa=(2\pi i)^{-1}$ & \S\ref{II:sec:modular-calculus} \\
$\eta,\eta_{\mathrm D}$ & Scalar modular connection; Dedekind eta function & \S\ref{II:subsec:scalar-connection}, \S\ref{II:subsec:x023-curve} \\
$B,\delta_B$ & Coefficient connection; $DF+[B,F]$ & \eqref{II:eq:adjoint-derivative} \\
$B_L^\eta$ & Coefficient connection obtained from the operator $L$ & \eqref{II:eq:operator-B} \\
$\vartheta,\varepsilon$ & Theta characteristic; a flat two-torsion line used to change it & \S\ref{II:subsec:connection}, \S\ref{II:subsec:cubic} \\
$\mathsf H_E,\mathsf A,B_j$ & Regular kernel, its linear connection coefficient, and horizontal coefficients & \eqref{II:eq:regular-kernel}, \eqref{II:eq:horizontal-kernel} \\
$P,p_0$ & Quadratic reconstruction datum; fixed scalar projective connection for the hierarchy & \eqref{II:eq:quadratic-datum-law}, \eqref{II:eq:background-projective-connection} \\
$\cub_E,\mathcal C_E$ & Canonical cubic and its first self-bracket & \eqref{II:eq:canonical-cubic}, \eqref{II:eq:canonical-bracket} \\
$T,d_{LM},c_\xi$ & $\Hom(M,L)$, the cubic difference $\cub_L-\cub_M$, and the extension bracket coefficient & \S\ref{II:sec:extension-recovery} \\
$P_\nabla,\mathfrak F$ & Polar integration; the fibre product in the linearity proof & \S\ref{II:subsec:extension-polar-recovery}, Lemma~\ref{II:lem:extension-linearity} \\
$\ell,\psi(e),e_L$ & Parameter derivative; $\log\theta(e)$; rational divisor frame & \S\ref{II:sec:curve-schwarzians}, \S\ref{II:subsec:theta-cubic}, \S\ref{II:subsec:x023-line-kernel} \\
$P_\infty,P_0$ & The two cusps of $X_0(23)$ & \S\ref{II:subsec:x023-curve} \\
\bottomrule
\end{tabularx}
\end{center}
The letter $L$ denotes an operator in Sections~\ref{II:sec:curve-schwarzians}--\ref{II:sec:rankin-cohen}
and a line bundle in the kernel applications. The scalar $c$ in a
M\"obius matrix is unrelated to $c_\xi$. The iterates $T_{w,n}^\eta$
are differential operators, whereas $T=\Hom(M,L)$ is a line bundle.

\section{Higher Schwarzians on curves and modular reconstruction}
\label{II:sec:modular-operators}

The triangular normalization of Part~\ref{part:I} can be inverted on a
fixed coefficient bundle. We first give the local inverse, then use its
uniqueness to patch operators on a complex curve and on a modular
quotient.

We use the normalized coefficients and higher polynomials defined in
\eqref{I:eq:local-normalized} and \eqref{I:eq:rp-Q-definition}, writing
$Q_m^{(N)}$ when the ambient order must be explicit. On a coordinate
disc the coefficient algebra is the algebra of meromorphic matrix
functions with $\delta=d/dz$. In modular coordinates it consists of
meromorphic $A$-valued functions with $\delta=D=\kappa\,d/d\tau$,
where $\kappa=(2\pi i)^{-1}$. Products keep the order fixed in
Part~\ref{part:I}.

For $z=\phi(t)$, put $\ell=\phi'(t)$ and
$D_z=\kappa\partial_z$, $D_t=\kappa\partial_t$. In the algebraic parameter
laws of Section~\ref{I:sec:reparam}, the substitutions are
$\lambda_1=\ell$, $\lambda_2=D_t\ell$, $\lambda_3=D_t^2\ell$ and
\begin{equation}\label{II:eq:scaled-schwarzian}
 S_\phi^D=\frac{\kappa^2}{6}\Sch(\phi).
\end{equation}
Thus \eqref{I:eq:rp-low-laws} and Theorem~\ref{I:thm:rp-higher} apply
with this derivative convention. Ordinary curve formulas use
$\kappa=1$. The conversion of the resulting differentials is stated in
\eqref{II:eq:derivative-normalization}.

\subsection{The local inverse}
\label{II:subsec:local-reconstruction}

The higher Schwarzians retain all the normalized coefficients:
$W_m$ differs from $I_m$ only by terms whose indices are smaller
than $m$. We record the inverse before using it to patch operators
on a curve or on a modular quotient.

For order three and $a_1=0$, the calculation has a particularly
simple form. Start with $L=D^3+3pD+r$. Then $I_2=p$ and
$W_3=r-\tfrac32Dp$, so prescribing the cubic coefficient $q$
requires $r=q+\tfrac32Dp$. Hence
\begin{equation}\label{II:eq:cubic-operator-example}
 L_{p,q}=D^3+3pD+\tfrac32Dp+q,
 \qquad I_2=p,\quad W_3=q.
\end{equation}
Here $p$ and $q$ are arbitrary elements of the local differential
algebra, and $Dp$ means the derivative $\delta p$. The general
inverse recovers one coefficient at a time in the same way.

\begin{proposition}[Local reconstruction]\label{II:prop:local-reconstruction}
Fix an order $N$, a differential algebra $(\mathfrak A,\delta)$,
and $a_1\in\mathfrak A$. For any $q_2,\ldots,q_N\in\mathfrak A$,
there is a unique monic operator \eqref{I:eq:local-L} with
first coefficient $a_1$ and
\[
 I_2(L)=q_2,\qquad W_m(L)=q_m\quad(3\le m\le N).
\]
\end{proposition}

\begin{proof}
Put $\nabla=D+a_1$ and use the recursion
\begin{equation}\label{II:eq:reconstruction-recursion}
 \begin{aligned}
 \Delta b&=\delta b+[a_1,b],\qquad I_2=q_2,\\
 I_m&=q_m-Q_m^{(N)}(\Delta^jI_k:2\le k<m),
                       \qquad m=3,\ldots,N.
 \end{aligned}
\end{equation}
Each right-hand side uses only coefficients already recovered.
In degrees three and four, when present, this gives
\[
 I_3=q_3+\tfrac32\Delta q_2,\qquad
 I_4=q_4+2\Delta I_3-\tfrac65\Delta^2q_2
                  +\frac{3(5N+7)}{5(N+1)}q_2^2.
\]
Substitution in \eqref{I:eq:local-normalized} gives the required
operator. Conversely, any operator with the prescribed data must
have these $I_m$, one after another, so it is unique.
\end{proof}

We will use this local uniqueness to compare operators written
in different coordinate charts.

\subsection{Higher Schwarzians on a complex curve}
\label{II:sec:curve-schwarzians}

The local construction now gives geometric data on a curve.
We must specify the operator's source and target so that its first
coefficient becomes a connection on the intended bundle.

Let $X$ be a smooth complex curve, with no group action assumed,
and put $K=K_X$. Let $V\to X$ be a holomorphic vector bundle of
rank $r$, and choose a holomorphic line bundle $\vartheta$ with
$\vartheta^2\simeq K$. We consider holomorphic differential operators
of order $N$ of the form
\begin{equation}\label{II:eq:curve-operator-type}
 \mathcal L:V\otimes\vartheta^{1-N}\longrightarrow
                         V\otimes\vartheta^{1+N}.
\end{equation}
The twists serve two purposes. First, the principal symbol takes
values in
\[
 \Hom(V\otimes\vartheta^{1-N},V\otimes\vartheta^{1+N})
                     \otimes K^{-N}\simeq\End(V),
\]
so it makes sense to require that symbol to be $\Id_V$; this is the
meaning of monic in \eqref{II:eq:curve-operator-type}. Second, the
exponent $1-N$ cancels the scalar coordinate term in the first
coefficient. We will see this cancellation in the overlap formula.
Tensor powers of line bundles are understood, with negative powers
denoting duals.

Locally, choose a coordinate $z$, a frame $e$ of $V$, and a frame
$e_\vartheta$ of $\vartheta$ with $e_\vartheta^2=dz$. The operator is specified by
\begin{equation}\label{II:eq:curve-local-action}
 \mathcal L(e\,e_\vartheta^{1-N}y)=e\,e_\vartheta^{1+N}L_zy,\qquad
 L_z=\sum_{k=0}^N\binom Nk a_k\partial_z^{N-k},\quad a_0=\Id.
\end{equation}
Here $y$ is a column of holomorphic functions and $a_k$ are
holomorphic $r\times r$ matrix functions. We now use the ordinary
derivative $\partial_z=d/dz$. All local formulas above apply with
$\delta=\partial_z$ and with $S_\phi^D$ replaced by $\Sch(\phi)/6$.

If $z=\phi(t)$, $\ell=\phi'$, and $e_t=(e_z\circ\phi)G$, the
local frames of $\vartheta$ determine a square root of $\ell$ by
$e_{\vartheta,z}\circ\phi=\ell^{1/2}e_{\vartheta,t}$. Consequently the input
frames in \eqref{II:eq:curve-local-action} differ by
$h=\ell^{(N-1)/2}G$. The half-integer power here is fixed by those
frames. The local operators therefore satisfy
\begin{equation}\label{II:eq:curve-operator-transition}
 L_t=h^{-1}L_z^\phi h.
\end{equation}
The scalar part of the first parameter law cancels the logarithmic
derivative of $\ell^{(N-1)/2}$, giving
\[
 a_{1,t}=G^{-1}\ell(a_{1,z}\circ\phi)G+G^{-1}\partial_tG.
\]
Consequently the first coefficient defines a holomorphic connection
on the specified bundle $V$:
\begin{equation}\label{II:eq:operator-adjoint-connection}
 \begin{aligned}
 \nabla^{\mathcal L}&:V\longrightarrow V\otimes K,
       &\nabla^{\mathcal L}&=d+a_1\,dz,\\
 \nabla^{\mathcal L,\mathrm{ad}}&:\End(V)\longrightarrow\End(V)\otimes K,
       &b&\longmapsto(b'+[a_1,b])\,dz.
 \end{aligned}
\end{equation}
These connections are constructed from $\mathcal L$; neither was
part of its input data. They are flat because $X$ is a complex curve.

The remaining coordinate laws give
\begin{equation}\label{II:eq:global-higher-differentials}
 W_m(\mathcal L)=W_m(L_z)(dz)^m
        \in H^0(X,\End(V)\otimes K^m),\qquad 3\le m\le N.
\end{equation}
For example $W_3(\mathcal L)$ is a cubic differential, equivalently
a holomorphic map $V\to V\otimes K^3$. The quadratic datum has an
affine law:
\begin{equation}\label{II:eq:quadratic-datum-law}
 P_z=\frac{I_2(L_z)}{N+1},\qquad
 P_t=G^{-1}\left(\ell^2(P_z\circ\phi)
                          +\tfrac16\Sch(\phi)\Id\right)G.
\end{equation}
Thus $P$ denotes a collection of local coefficients with this law.
The difference of two such collections is an
$\End(V)$-valued quadratic differential.

\begin{theorem}[Higher Schwarzians and reconstruction on a curve]\label{II:thm:curve-reconstruction}
For fixed $X,V,\vartheta,N$, monic holomorphic operators
\eqref{II:eq:curve-operator-type} correspond bijectively to data
\[
 (\nabla^V,P,w_3,\ldots,w_N),\qquad
 w_m\in H^0(X,\End(V)\otimes K^m),
\]
where $\nabla^V:V\to V\otimes K$ is a holomorphic connection and
$P$ is a collection of holomorphic local coefficients satisfying
\eqref{II:eq:quadratic-datum-law}.
\end{theorem}

\begin{proof}
Write $\nabla^V=d+a_1dz$ locally and
$w_m=w_{m,z}(dz)^m$. Proposition~\ref{II:prop:local-reconstruction},
with ordinary derivatives, gives a unique local operator whose
quadratic coefficient is $I_2=(N+1)P_z$ and whose higher
coefficients are $W_m=w_{m,z}$. On an overlap, transforming that operator by
\eqref{II:eq:curve-operator-transition} gives the same connection,
quadratic datum, and higher differentials as those prescribed in
the second chart. Local uniqueness identifies the two operators,
so they glue. The construction and recovery are inverse.
\end{proof}

The same reconstruction proof applies to meromorphic data by
allowing poles in the local coefficients.

\begin{remark}[When holomorphic operators exist]\label{II:rem:operator-existence}
Suppose here that $X$ is compact. By the Atiyah--Weil criterion,
$V=\bigoplus_iV_i$, with $V_i$ indecomposable, admits a holomorphic
connection if and only if $\deg V_i=0$ for every $i$ \cite{Atiyah,Weil}.
There is no further existence obstruction in
Theorem~\ref{II:thm:curve-reconstruction}. Indeed, uniformization gives
a holomorphic projective connection $p_0$ in every genus
\cite{Gunning}. Normalize it by the scalar law in
\eqref{II:eq:quadratic-datum-law}; then all affine quadratic data are
\[
 P=p_0\Id_V+q,\qquad q\in H^0(X,\End(V)\otimes K^2).
\]
Taking $w_m=0$ therefore produces an operator whenever $V$ admits a
holomorphic connection. In particular, degree zero of $V$ alone is
not sufficient: the criterion concerns each indecomposable summand.
The bundles $V_E$ constructed in Section~\ref{II:sec:kernel}
come with such a connection.

If poles are unrestricted, every $V$ admits a meromorphic operator
of the same type. A meromorphic frame $e_1,\ldots,e_r$ exists by
Riemann--Roch; declare it horizontal by setting
\[
 \nabla^V\!\left(\sum_i f_i e_i\right)=\sum_i e_i\,df_i.
\]
This gives a meromorphic connection on $V$. Taking $P=p_0\Id_V$
and $w_m=0$ supplies the remaining reconstruction data. Thus the
degree obstruction concerns holomorphic operators; no existence
claim is made here with a prescribed pole divisor or pole order.
\end{remark}

To return to modular forms, we record the derivative convention
explicitly. On $\mathbb H$, the conversion from ordinary derivatives
to $D=\kappa\partial_\tau$ is exact: if $\widehat L=\kappa^NL_\tau$ is written
as a monic polynomial in $D=\kappa\partial_\tau$, then
\begin{equation}\label{II:eq:derivative-normalization}
 I_m(\widehat L)=\kappa^m I_m(L_\tau),\qquad
 W_m(\widehat L)=\kappa^m W_m(L_\tau).
\end{equation}
Each term has the same differential weight, which proves these
identities. Thus its geometric $m$-differential has modular expression
$\kappa^{-m}W_m(\widehat L)(d\tau)^m$.

The intrinsic construction has used arbitrary coordinate changes.
For a modular quotient the coordinate changes are fractional linear,
so their scalar Schwarzian is zero. The next two subsections express
the reconstruction in this setting, including its cusp conditions.

\subsection{Modular operators and their first coefficient}
\label{II:subsec:modular-operator-definition}

Return to $\Gamma,A,\alpha$ from Section~\ref{II:sec:modular-calculus}.
For the relation between modular differential equations and their
representations in the scalar case, see~\cite[Sections~2--3]{FM}.
For an integer $w$ and an arbitrary local $A$-valued function $s$, put
\[
 (s|_{w,\alpha}\gamma)(\tau)
 =j(\gamma,\tau)^{-w}\alpha(\gamma)^{-1}s(\gamma\tau).
\]
\begin{definition}[Modular operator]\label{II:def:modular-operator}

A \emph{modular operator of type $(w,w+2N)$} is a monic operator
\eqref{I:eq:local-L} on $\mathbb H$ such that
\begin{equation}\label{II:eq:operator-intertwining}
 L(s|_{w,\alpha}\gamma)=(Ls)|_{w+2N,\alpha}\gamma
\end{equation}
for every $\gamma\in\Gamma$ and every local meromorphic $s$.
This condition is imposed on local functions, so it determines an
identity of differential operators even if a space of global forms
of weight $w$ happens to vanish. For matrix coefficients it is
equivalent to the same identity on columns with slash action
$s|_{w,\rho}\gamma=j^{-w}\rho(\gamma)^{-1}s\circ\gamma$.

At a cusp represented by $\sigma$, define $L_\sigma$ by
\[
 L_\sigma(s|_w\sigma)=(Ls)|_{w+2N}\sigma,\qquad
 s|_w\sigma=j(\sigma,\tau)^{-w}s(\sigma\tau).
\]
We require its coefficients to have finite Laurent principal parts
in the cusp variable $q$. The operator is holomorphic if its
coefficients on $\mathbb H$ and in every such cusp expression are
holomorphic. No extension across a cusp of the integer-weight input
line is needed for this coefficient convention.
\end{definition}

To describe these operators by modular coefficients, we first
replace powers of $D$ by a basis that already respects weights.
Fix a scalar modular connection $\eta$. Its ordinary iterates are
\begin{equation}\label{II:eq:ordinary-iterates}
 T_{w,n}^\eta=d_{w+2n-2}^\eta\circ\cdots\circ d_w^\eta,
 \qquad T_{w,0}^\eta=\mathrm{id}.
\end{equation}
They are monic of order $n$ and raise weight by $2n$.

\Needspace{11\baselineskip}
\begin{proposition}[Modular coefficient expansion]\label{II:prop:modular-expansion}
Every modular operator of type $(w,w+2N)$ has a unique expansion
\begin{equation}\label{II:eq:modular-expansion}
 L=\sum_{j=0}^N\binom Nj c_j^\eta T_{w,N-j}^\eta,
 \qquad c_0^\eta=1,\qquad
 c_j^\eta\in M_{2j}^{\mathrm{mer}}(\Gamma,A,\alpha).
\end{equation}
Conversely, every such tuple of coefficients defines a modular
operator. If $\eta$ is holomorphic, the operator is holomorphic if
and only if all the $c_j^\eta$ are holomorphic.
\end{proposition}

\begin{proof}
Subtract $T_{w,N}^\eta$ from $L$. The remainder satisfies
\eqref{II:eq:operator-intertwining} and has order at most $N-1$.
If its leading term is $bD^r$, comparing the highest derivatives
of $s$ in that identity gives
\[
 b(\gamma\tau)=j(\gamma,\tau)^{2N-2r}\alpha(\gamma)b(\tau).
\]
Subtract $bT_{w,r}^\eta$ and repeat. This gives existence and the
weights in \eqref{II:eq:modular-expansion}; the different leading orders
give uniqueness. Conversely, each term in that expansion satisfies
the intertwining law, since its two factors raise the weight by
$2j$ and $2(N-j)$.

In a cusp expression the same calculation uses $\eta^\sigma$ and
$c_j^\eta|_{2j}\sigma$. Differentiation and multiplication preserve
finite Laurent principal parts. If $\eta$ is holomorphic, the
successive subtractions preserve holomorphicity in both directions.
\end{proof}

For the Serre connection and scalar coefficients this is the familiar
expansion in modular derivatives; compare \cite[Section 6, Theorem
3]{NSZ}. Ordered multiplication gives the same proof for $A$.

The new coefficients $c_j^\eta$ identify the connection supplied
by a modular operator. Put
\begin{equation}\label{II:eq:operator-B}
 s_0=w+N-1,\qquad B_L^\eta=a_1+s_0\eta\,1.
\end{equation}
Since the coefficient of $D^{N-1}$ in $T_{w,N}^\eta$ is
$-\sum_{r=0}^{N-1}(w+2r)\eta=-Ns_0\eta$, comparison in
\eqref{II:eq:modular-expansion} gives
\begin{equation}\label{II:eq:B-first-coefficient}
 B_L^\eta=c_1^\eta\in M_2^{\mathrm{mer}}(\Gamma,A,\alpha),\qquad
 a_1(\gamma\tau)=j^2\alpha(\gamma)a_1(\tau)-s_0\kappa cj\,1.
\end{equation}
For a weight-zero change of frame, \eqref{I:eq:local-gauge-coefficients}
therefore gives
\begin{equation}\label{II:eq:operator-B-gauge}
 B_{L^g}^\eta=g^{-1}B_L^\eta g+g^{-1}Dg.
\end{equation}
Thus $B_L^\eta$ is exactly a coefficient connection as in
\eqref{II:eq:bundle-connection-B} in the matrix case.
Changing $\eta$ to $\eta+v$ adds $s_0v\,1$ to $B_L^\eta$, so the
induced derivative
\begin{equation}\label{II:eq:operator-delta}
 \Delta F=DF+[a_1,F]=DF+[B_L^\eta,F]=\delta_{B_L^\eta}F
\end{equation}
is independent of $\eta$. At the centered input weight $w=1-N$,
$a_1$ itself is modular of weight two.

\subsection{Reconstruction from modular data}
\label{II:subsec:modular-reconstruction}

The parameter laws show why higher Schwarzians become modular
forms. A fractional linear change has zero scalar Schwarzian,
and its remaining Jacobian factor is exactly the modular weight.

\begin{proposition}[Modularity of the invariants]\label{II:prop:modular-invariants}
For a modular operator of type $(w,w+2N)$,
\[
 I_2\in M_4^{\mathrm{mer}}(\Gamma,A,\alpha),\qquad
 W_m\in M_{2m}^{\mathrm{mer}}(\Gamma,A,\alpha)
             \quad(3\le m\le N).
\]
These forms are independent of $\eta$ and transform by conjugation
under weight-zero changes of coefficient frame.
\end{proposition}

\begin{proof}
For $\phi(\tau)=\gamma\tau$, one has $\ell=j^{-2}$ and
$\Sch(\gamma)=0$. The intertwining identity is equivalent to
\begin{equation}\label{II:eq:modular-pullback-gauge}
 L^\gamma=(\alpha(\gamma)L)^{j^{-w}},
\end{equation}
where $\alpha(\gamma)L$ means applying the fixed algebra
automorphism to every coefficient. The formulas for $I_2$ and $W_m$
commute with such automorphisms, and scalar conjugation leaves them
unchanged. Formula \eqref{I:eq:rp-low-laws} and
Theorem~\ref{I:thm:rp-higher} therefore give
\[
 j^{-4}I_2(\gamma\tau)=\alpha(\gamma)I_2(\tau),\qquad
 j^{-2m}W_m(\gamma\tau)=\alpha(\gamma)W_m(\tau).
\]
The same calculation with a cusp matrix $\sigma$ identifies the
invariants of $L_\sigma$ with their weight-$2m$ slash transforms.
Their differential-polynomial formulas preserve finite Laurent
principal parts. Independence of $\eta$ and gauge covariance follow
from the local construction.
\end{proof}

We can therefore prescribe the connection coefficient and these
modular forms independently. The local inverse makes the resulting
operator, and its uniqueness proves the modular law.

\Needspace{16\baselineskip}
\begin{theorem}[Modular reconstruction]\label{II:thm:modular-reconstruction}
Fix $\Gamma,A,\alpha,N,w$ and a meromorphic scalar modular connection
$\eta$. The assignment
\begin{equation}\label{II:eq:modular-data-map}
 L\longmapsto(B_L^\eta,I_2,W_3,\ldots,W_N)
\end{equation}
is a bijection from monic meromorphic modular operators of type
$(w,w+2N)$ to
\begin{equation}\label{II:eq:modular-data-space}
 M_2^{\mathrm{mer}}(\Gamma,A,\alpha)
      \times\prod_{m=2}^N M_{2m}^{\mathrm{mer}}(\Gamma,A,\alpha).
\end{equation}
It respects weight-zero changes of frame, with the connection law
on the first factor and conjugation on the remaining factors.
If $\eta$ is holomorphic, it restricts to a bijection for holomorphic
operators and holomorphic tuples.
\end{theorem}

\begin{proof}
Given $(B,q_2,q_3,\ldots,q_N)$ in \eqref{II:eq:modular-data-space}, set
\[
 a_1=B-(w+N-1)\eta\,1,\qquad
 \Delta F=DF+[a_1,F].
\]
Proposition~\ref{II:prop:local-reconstruction} gives a unique local
monic operator with this first coefficient and the prescribed
$I_2,W_3,\ldots,W_N$.

To prove its modularity, compare the two operators in
\eqref{II:eq:modular-pullback-gauge}. Their first coefficients agree by
the transformation laws of $B$ and $\eta$. Their $I_2$ and $W_m$
agree by the weights of the $q_m$ and the parameter laws. Local
uniqueness therefore identifies the operators, giving
\eqref{II:eq:operator-intertwining}.

At a cusp, the same local inverse applied to $B|_2\sigma$,
$q_m|_{2m}\sigma$, and $\eta^\sigma$ gives $L_\sigma$, with finite
Laurent principal parts. If the data and $\eta$ are holomorphic,
every step of \eqref{II:eq:reconstruction-recursion} preserves
holomorphicity. Conversely, the invariants of a holomorphic
operator are holomorphic by their local polynomial formulas, and
$B_L^\eta=c_1^\eta$ is holomorphic by
Proposition~\ref{II:prop:modular-expansion}. Finally, conjugation of the
normalized expansion and local uniqueness give the asserted
equivariance under changes of coefficient frame.
\end{proof}

The order-three example \eqref{II:eq:cubic-operator-example} now
has a modular interpretation. If
$p\in M_4^{\mathrm{mer}}(\Gamma,A,\alpha)$ and
$q\in M_6^{\mathrm{mer}}(\Gamma,A,\alpha)$, the centered choice
$w=-2$, $B=0$ makes $a_1=0$. Thus $L_{p,q}$ has type $(-2,4)$.
This can also be checked directly in the modular derivative basis:
\[
 T_{-2,3}^\eta=D^3-4\Omega_\eta D-2D\Omega_\eta,
\]
and hence
\[
 L_{p,q}=T_{-2,3}^\eta
 +3\left(p+\tfrac43\Omega_\eta\,1\right)T_{-2,1}^\eta
 +q+\tfrac32d_4^\eta p+2d_4^\eta\Omega_\eta\,1.
\]
Every coefficient has the weight required by
\eqref{II:eq:modular-expansion}.

\begin{remark}[The quadratic coefficient]
The derivative coefficients $c_j^\eta$ and the higher Schwarzian
data are now two complete descriptions of $L$. Their quadratic
comparison locates the scalar modular curvature within $I_2$:
\begin{equation}\label{II:eq:I2-curvature-comparison}
 I_2=c_2^\eta-d_2^\eta B_L^\eta-(B_L^\eta)^2
                                  -\frac{N+1}{3}\Omega_\eta\,1.
\end{equation}
To check it, the first two lower terms of the ordinary iterate are
\[
 \begin{split}
 T_{w,N}^\eta={}&D^N-Ns_0\eta D^{N-1}\\
 &+\binom N2\left(s_0^2\eta^2-s_0D\eta
                    -\frac{N+1}{3}\Omega_\eta\right)D^{N-2}
       +\text{terms of lower order}.
 \end{split}
\]
These coefficients follow by multiplying the factors in
\eqref{II:eq:ordinary-iterates}; alternatively, their displayed formulas
are preserved when the next factor $D-(w+2N)\eta$ is multiplied on
the left. Compare the $D^{N-2}$ terms in
\eqref{II:eq:modular-expansion}, substitute $a_1=B_L^\eta-s_0\eta$, and
use \eqref{I:eq:local-I23}. This gives \eqref{II:eq:I2-curvature-comparison}.
In particular, the modular curvature already enters in order two.
\end{remark}

\section{Ordered Rankin--Cohen operations}
\label{II:sec:rankin-cohen}

The higher Schwarzians are forms with values in the coefficient
algebra. We now combine them with their derivatives to obtain further
forms. The first operation already exhibits the role of the algebra:
its self-bracket vanishes for scalar coefficients, but can be nonzero
for matrices. We will apply it to the canonical cubic constructed in
Section~\ref{II:sec:kernel}.

The ordered brackets belong to the noncommutative Rankin--Cohen
framework of Connes and Moscovici \cite{CM}. We give them in the
coefficient-bundle conventions of this paper, including their
behavior under a change of frame.

\subsection{The first ordered bracket}
\label{II:subsec:first-ordered-bracket}

Fix $B\in M_2^{\mathrm{mer}}(\Gamma,A,\alpha)$, and write
$\delta_B=D+[B,\,\cdot\,]$. If $F,G$ have respective even weights
$k,\ell$, put
\begin{align}
 [F,G]_0^B&=FG,\label{II:eq:RC-zero}\\
 [F,G]_1^B&=kF\delta_BG-\ell(\delta_BF)G.
       \label{II:eq:RC-first}
\end{align}
The order of the factors is part of the definition. To see why the
second expression is modular, choose a scalar modular connection
$\eta$ and use $\mathcal D_k^{\eta,B}=\delta_B-k\eta$:
\[
 kF\mathcal D_\ell^{\eta,B}G
       -\ell(\mathcal D_k^{\eta,B}F)G
   =kF\delta_BG-\ell(\delta_BF)G.
\]
The two terms involving $\eta$ cancel because $\eta$ is scalar.
On the left, both products have weight $k+\ell+2$ and transform
by the same conjugation under a change of coefficient frame.
This proves both properties for the bracket, while the right side
shows that it is independent of $\eta$.

In particular, a form $F$ of weight $2m$ has first self-bracket
\begin{equation}\label{II:eq:first-self-bracket}
 [F,F]_1^B=2m[F,\delta_BF],
\end{equation}
of weight $4m+2$. For matrix coefficients its trace is zero,
although the commutator itself need not vanish. The cubic higher
Schwarzian therefore gives a form of weight $14$.

\begin{example}[A nonzero first self-bracket]
\label{II:ex:ordered-self-brackets}
Use ordinary differentiation in a local projective coordinate $t$,
and take the coefficient connection to be zero. Give
\[
 F(t)=\begin{pmatrix}t&t^2\\1&-t\end{pmatrix}
\]
weight six, so that $F(t)(dt)^3$ is a local endomorphism-valued
cubic. Direct multiplication gives
\[
 [F,F]_1^0=6(FF'-F'F)
 =12\begin{pmatrix}-t&t^2\\1&t\end{pmatrix}.
\]
Thus the first self-bracket need not vanish. Its cancellation for
scalar inputs uses commutativity. This is a local
calculation of the bracket law, with no global automorphy condition
imposed on the polynomial matrix $F$.
\end{example}

To extend the bracket to higher orders, we need higher derivatives
whose scalar connection terms cancel in the same way. Ordinary
iterates of $\mathcal D^{\eta,B}$ do not have this property. The
correction below is the higher Serre derivative of Nagatomo,
Sakai, and Zagier \cite[Section 5]{NSZ}, with the coefficient
connection included.

\subsection{Higher covariant derivatives}
\label{II:subsec:higher-derivatives}

Keep $B$ fixed, choose a scalar modular connection $\eta$, and
write $\Omega=\Omega_\eta$. For a form $F$ of weight $k$, define
\begin{equation}\label{II:eq:higher-derivative-recurrence}
 \begin{aligned}
 H_{k,0}^{\eta,B}F&=F,\qquad
 H_{k,1}^{\eta,B}F=\mathcal D_k^{\eta,B}F,\\
 H_{k,n+1}^{\eta,B}F
   &=\mathcal D_{k+2n}^{\eta,B}(H_{k,n}^{\eta,B}F)
            -n(k+n-1)\Omega H_{k,n-1}^{\eta,B}F
                      \qquad(n\ge1).
 \end{aligned}
\end{equation}
Each summand has weight $k+2n+2$, so $H_{k,n}^{\eta,B}$ raises weight
by $2n$. The recurrence also preserves gauge covariance
\eqref{II:eq:derivative-gauge}. Its first departure from ordinary
iteration is
\begin{equation}\label{II:eq:second-higher-derivative}
 H_{k,2}^{\eta,B}F
   =\mathcal D_{k+2}^{\eta,B}\mathcal D_k^{\eta,B}F-k\Omega F
   =\delta_B^2F-2(k+1)\eta\delta_BF+k(k+1)\eta^2F.
\end{equation}
In the last expression the derivative of $\eta$ has cancelled.
The same cancellation holds in every order.

\begin{lemma}\label{II:lem:higher-derivative-closed}
With $(a)_s=a(a+1)\cdots(a+s-1)$ and $(a)_0=1$, one has
\begin{equation}\label{II:eq:higher-derivative-closed}
 H_{k,n}^{\eta,B}F
   =\sum_{i=0}^n\binom ni(k+i)_{n-i}(-\eta)^{n-i}\delta_B^iF.
\end{equation}
\end{lemma}

\begin{proof}
The cases $n=0,1$ follow from the definitions. In the induction step,
apply $\delta_B-(k+2n)\eta$ to the displayed sum. Since $\eta$ is
scalar, $\delta_B\eta=D\eta$, and differentiation of
$(-\eta)^{n-i}$ contributes a factor $(n-i)D\eta$.
For $0\le i\le n-1$, the coefficient identity
\[
 (n-i)\binom ni(k+i)_{n-i}
  =n(k+n-1)\binom{n-1}i(k+i)_{n-1-i}
\]
shows that these terms cancel the $D\eta$ part of
$-n(k+n-1)(\eta^2-D\eta)H_{k,n-1}^{\eta,B}F$.
For the remaining terms, Pascal's identity and
$(k+i)_{n+1-i}=(k+n)(k+i)_{n-i}$ combine the contributions
to $\delta_B^iF$ into the coefficient in
\eqref{II:eq:higher-derivative-closed} with $n$ replaced by $n+1$.
This proves the formula. All identities are polynomial in $k$, so
they include zero and negative weights.
\end{proof}

\subsection{The bracket and the cancellation of the scalar connection}
\label{II:subsec:ordered-brackets}

For forms $F,G$ of respective even weights $k,\ell$, define
\begin{equation}\label{II:eq:RC-covariant-definition}
 [F,G]_n^B
  =\sum_{r=0}^n(-1)^r
       \binom{k+n-1}{n-r}\binom{\ell+n-1}{r}
       (H_{k,r}^{\eta,B}F)(H_{\ell,n-r}^{\eta,B}G).
\end{equation}
For any integer $a$ and $b\ge0$, the binomial coefficient here means
$\binom ab=a(a-1)\cdots(a-b+1)/b!$. The $F$ factor always precedes
the $G$ factor. The notation omits $\eta$ because the next proposition
shows that it cancels.

\begin{proposition}\label{II:prop:RC-properties}
The expression \eqref{II:eq:RC-covariant-definition} is independent of
$\eta$ and equals
\begin{equation}\label{II:eq:RC-ordinary-definition}
 [F,G]_n^B
  =\sum_{r=0}^n(-1)^r
       \binom{k+n-1}{n-r}\binom{\ell+n-1}{r}
       (\delta_B^rF)(\delta_B^{n-r}G).
\end{equation}
It is a modular form of weight $k+\ell+2n$ with multiplier $\alpha$.
Under \eqref{II:eq:B-gauge},
\begin{equation}\label{II:eq:RC-gauge}
 [F^g,G^g]_n^{B^g}=g^{-1}[F,G]_n^Bg.
\end{equation}
It is holomorphic if $F,G,B$ are holomorphic, including at the cusps.
\end{proposition}

\begin{proof}
Modularity follows term by term in
\eqref{II:eq:RC-covariant-definition}: the weights add to
$k+\ell+2n$. The gauge identity follows from that of $H$ and
the cancellation of the adjacent factors $gg^{-1}$.

For independence of $\eta$, insert
\eqref{II:eq:higher-derivative-closed}. Fix $i,j\ge0$ with
$s=n-i-j\ge0$. The coefficient of the ordered product
$(-\eta)^s(\delta_B^iF)(\delta_B^jG)$ is
\begin{equation}\label{II:eq:RC-cancellation}
 \frac{(-1)^i(k+i)_{n-i}(\ell+j)_{n-j}}{i!j!}
                 \sum_{u=0}^s\frac{(-1)^u}{u!(s-u)!}.
\end{equation}
To see the coefficient, set $r=i+u$ in
\eqref{II:eq:RC-covariant-definition} and cancel the factorials in its
two binomial coefficients and the two factors from
\eqref{II:eq:higher-derivative-closed}. The sum in
\eqref{II:eq:RC-cancellation} is $(1-1)^s/s!$. It vanishes for $s>0$.
For $s=0$, so that $j=n-i$, the remaining coefficient is exactly
the one in \eqref{II:eq:RC-ordinary-definition}. This proves that formula
without commuting $F,G$ or their derivatives.

Finally, \eqref{II:eq:RC-ordinary-definition} uses only derivatives and
products of $F,G,B$. Its cusp expression has the same formula with
$F_\sigma,G_\sigma,B_\sigma$. These operations preserve holomorphic
$q$-series. This proves the last assertion even if the auxiliary
$\eta$ chosen for the proof was meromorphic.
\end{proof}

These operations specialize the construction of Connes--Moscovici,
including the coefficient connection. On the graded algebra of forms,
the dictionary is
\[
 Y_{\rm CM}F=\tfrac{k}{2}F,\qquad
 X_{\rm CM}F=\mathcal D_k^{\eta,B}F,\qquad
 \delta_j^{\rm CM}=0,\qquad \Omega_{\rm CM}=2\Omega_\eta.
\]
The graded Leibniz rule makes $X_{\rm CM}$ a derivation, even when
$B$ is noncentral, and $[Y_{\rm CM},X_{\rm CM}]=X_{\rm CM}$.
A nonzero $\Omega_{\rm CM}$ is compatible with
$\delta_j^{\rm CM}=0$. Indeed, the condition in
\cite[Section~2, equations~(2.1)--(2.2)]{CM} becomes
\[
 \delta_2^{\rm CM}-\tfrac12(\delta_1^{\rm CM})^2
     =\operatorname{ad}(\Omega_{\rm CM})=0,
\]
because $\Omega_{\rm CM}=2\Omega_\eta\,1$ is central; the conditions
$\delta_j^{\rm CM}(\Omega_{\rm CM})=0$ are automatic.
Their recurrences and bracket formula~\cite[Section~2]{CM} give
\eqref{II:eq:higher-derivative-recurrence} and
\eqref{II:eq:RC-covariant-definition}. Here the higher Schwarzians
supply canonical inputs to these established operations.

In the scalar case $B$ has zero commutator action, and
\eqref{II:eq:RC-ordinary-definition} is the classical Rankin--Cohen
formula. In the matrix case $B$ records the chosen coefficient
connection. The bracket is independent of the scalar connection
but can depend on this coefficient connection.

\begin{example}[A nonzero third self-bracket]
\label{II:ex:third-self-bracket}
Use the local coordinate and weight-six matrix $F$ of
Example~\ref{II:ex:ordered-self-brackets}, with zero coefficient
connection. Since $F^{(3)}=0$, only the two middle terms of
\eqref{II:eq:RC-ordinary-definition} survive in order three:
\[
 [F,F]_3^0=224(F''F'-F'F'')
          =-896\begin{pmatrix}0&1\\0&0\end{pmatrix}.
\]
The same matrix therefore has nonzero self-brackets in both
orders one and three.
\end{example}

For $n=0,1$, the general formula gives
\eqref{II:eq:RC-zero}--\eqref{II:eq:RC-first}. In order two, the correction
to ordinary covariant iterates can be seen explicitly. Let $B_2^{\mathrm{it}}(F,G)$ denote the expression
in \eqref{II:eq:RC-covariant-definition} with every $H$ replaced by the
ordinary iterates of $\mathcal D^{\eta,B}$. From
\eqref{II:eq:second-higher-derivative},
\begin{equation}\label{II:eq:ordinary-bracket-correction}
 B_2^{\mathrm{it}}(F,G)
   =[F,G]_2^B+\frac{k\ell(k+\ell+2)}2\Omega_\eta FG.
\end{equation}
Only the two terms with a second derivative contribute to the
difference: their coefficients are
$\ell\binom{k+1}{2}$ and $k\binom{\ell+1}{2}$.
For $\eta=E_2/12$ the correction is a multiple of $E_4/144$ and
does not generally vanish. This is why replacing ordinary derivatives
by Serre iterates in the classical formula requires the correction
\eqref{II:eq:higher-derivative-recurrence}.

\subsection{Brackets of higher Schwarzians}
\label{II:subsec:Schwarzian-brackets}

For a modular operator $L$, take its coefficient connection
$B=B_L^\eta$. By \eqref{II:eq:operator-delta}, the derivative $\delta_B$
depends only on $L$. Proposition~\ref{II:prop:RC-properties} therefore
gives, for $3\le r,s\le N$,
\[
 [W_r(L),W_s(L)]_n^{B_L^\eta}
   \in M_{2(r+s+n)}^{\mathrm{mer}}(\Gamma,A,\alpha).
\]
These forms are independent of $\eta$ and transform by conjugation
with the operator. Thus Part~\ref{part:I} supplies the inputs and their gauge
laws, while the bracket supplies further modular forms from them.

For the cubic, the self-bracket
\eqref{II:eq:first-self-bracket} is $6[W_3,\Delta W_3]$, of weight $14$.
It also has an intrinsic meaning on a complex curve, which will let
us apply the construction without choosing a modular presentation. Let $X$ be a complex curve and let
$\nabla^V:V\to V\otimes K_X$ be a holomorphic connection on a
holomorphic vector bundle. Its induced connection is denoted by
$\nabla^{\mathrm{ad}}$. For
\[
 P\in H^0(X,\End(V)\otimes K_X^m),\qquad
 P=p(z)(dz)^m,
\]
define locally
\begin{equation}\label{II:eq:intrinsic-self-bracket}
 \mathcal C(P)=2m[p,\nabla_z^{\mathrm{ad}}p](dz)^{2m+1}.
\end{equation}
Here $\nabla_z^{\mathrm{ad}}p=p'+[{\mathsf A},p]$ if
$\nabla^V=d+{\mathsf A}\,dz$; the derivative is the ordinary $d/dz$.

\begin{lemma}\label{II:lem:intrinsic-self-bracket}
Formula \eqref{II:eq:intrinsic-self-bracket} defines a holomorphic
section of $\End(V)\otimes K_X^{2m+1}$, independently of the local
coordinate and frame. The same assertion holds meromorphically.
\end{lemma}

\begin{proof}
Gauge covariance follows from the connection law. For $z=\phi(t)$
with $\ell=\phi'$, keep the pulled-back frame and write
$\widetilde p=\ell^m(p\circ\phi)$. Then
\[
 \nabla_t^{\mathrm{ad}}\widetilde p
    =\ell^{m+1}((\nabla_z^{\mathrm{ad}}p)\circ\phi)
                       +m\ell^{m-1}\ell'(p\circ\phi).
\]
The last summand commutes with $\widetilde p$. Taking the commutator
therefore gives $\ell^{2m+1}[p,\nabla_z^{\mathrm{ad}}p]\circ\phi$,
which is precisely the tensor law in
\eqref{II:eq:intrinsic-self-bracket}.
\end{proof}

If $P=F(\tau)(d\tau)^m$ and the connection is presented by $B$ as
in \eqref{II:eq:bundle-connection-B}, the two conventions are related by
\begin{equation}\label{II:eq:intrinsic-modular-bracket}
 [F,F]_1^B(d\tau)^{2m+1}=\kappa\,\mathcal C(P).
\end{equation}
The single factor $\kappa$ comes from the one derivative in the
bracket. The next section constructs both $P$ and its coefficient
connection canonically from a holomorphic bundle.

\section{Canonical higher Schwarzians from a kernel}
\label{II:sec:kernel}

Let $X$ be an arbitrary smooth connected compact complex curve of
genus $g\geq2$, and write $K=K_X$. Let $E\to X$ be a holomorphic
vector bundle of rank $r\geq1$, with fibre $E_x$ at $x$. We will
construct a canonical kernel from $E$. Its first three regular
coefficients will supply the connection, quadratic datum, and cubic
needed for the order-three reconstruction of
Section~\ref{II:sec:curve-schwarzians}. Thus the higher Schwarzian
construction will acquire inputs determined by a holomorphic bundle.
No differential operator or connection is prescribed.
All bundle maps below are holomorphic and cover the identity of $X$.

We assume that $E$ is \emph{acyclic}, meaning
\begin{equation}\label{II:eq:acyclic}
 H^0(X,E)=H^1(X,E)=0.
\end{equation}
Here $H^q(X,E)$ denotes the cohomology of the sheaf of holomorphic
sections of $E$. Thus $H^0(X,E)=0$ says that $E$ has no nonzero
global holomorphic section; the additional vanishing of $H^1(X,E)$
will give existence in the construction below. Riemann--Roch gives
\begin{equation}\label{II:eq:acyclic-degree}
 0=\dim H^0(X,E)-\dim H^1(X,E)=\deg E+r(1-g),
 \qquad \deg E=r(g-1).
\end{equation}
These hypotheses have examples in every rank. For line bundles,
\[
 U:=\operatorname{Pic}^{g-1}(X)\setminus\Theta,
 \qquad
 \Theta=\{L\in\operatorname{Pic}^{g-1}(X):H^0(X,L)\ne0\},
\]
is precisely the acyclic locus. The theta divisor is proper, and
direct sums of lines in $U$ give acyclic bundles of any rank.
Acyclicity is also an open condition in algebraic families, by
semicontinuity of cohomology.

The \emph{Serre dual bundle} of $E$ is
\[
 E^\vee=E^*\otimes K.
\]
Serre duality identifies
$H^q(X,E^\vee)\simeq H^{1-q}(X,E)^*$ for $q=0,1$, so $E^\vee$
is acyclic as well.

\subsection{A kernel with prescribed residue}
\label{II:subsec:residue}

Let $\Delta=\{(x,x):x\in X\}\subset X\times X$. Denote the two
projections by $\operatorname{pr}_1,\operatorname{pr}_2$, and put
\[
 E\boxtimes F=\operatorname{pr}_1^*E\otimes\operatorname{pr}_2^*F.
\]
\begin{definition}[Kernel and residue]\label{II:def:kernel-residue}

The line bundle $\OO(\Delta)$ allows meromorphic sections with at
most a simple pole along $\Delta$. A \emph{kernel on $E$ with at
most a simple diagonal pole} means a section
\begin{equation}\label{II:eq:kernel-space}
 s\in H^0\bigl(X\times X,(E\boxtimes E^\vee)(\Delta)\bigr).
\end{equation}
In particular, for $x\ne y$ its value is a linear map
\begin{equation}\label{II:eq:kernel-fibres}
 s(x,y):E_y\longrightarrow E_x\otimes K_y.
\end{equation}
It depends holomorphically on $(x,y)$ away from $\Delta$. Thus
``kernel'' refers here to a section in two variables, with a
one-form factor in the second variable.

To describe the pole, take both points in one coordinate disc,
write $z=z(x)$ and $w=z(y)$, and choose a holomorphic frame of $E$
on that disc. The kernel has the matrix expression
\begin{equation}\label{II:eq:local-residue}
 s(z,w)=H(z,w)\frac{dw}{w-z},
\end{equation}
where $H$ is a holomorphic $r\times r$ matrix near $w=z$.
We take the residue in $w$, so
\[
 (\res_\Delta s)(z)=H(z,z)\in\End(E_z).
\]
Our normalization will be $H(z,z)=\Id_{E_z}$. It compares the two
fibres precisely at the point where they coincide.
\end{definition}

Taking residues gives the exact sequence
\begin{equation}\label{II:eq:residue-sequence}
 0\longrightarrow E\boxtimes E^\vee
 \longrightarrow (E\boxtimes E^\vee)(\Delta)
 \xrightarrow{\res_\Delta}\iota_*\End(E)
 \longrightarrow0.
\end{equation}
Here $\iota:X\to X\times X$ is $x\mapsto(x,x)$, and
$\iota_*\End(E)$ denotes the sheaf supported on the diagonal with
fibre $\End(E_x)$ at $(x,x)$. Exactness can be seen directly in
\eqref{II:eq:local-residue}: zero residue means that $H(z,w)$ is
divisible by $w-z$, so the pole disappears. Conversely, any local
endomorphism $P(z)$ is the residue of $P(z)\,dw/(w-z)$.

\begin{proposition}[Canonical kernel]\label{II:prop:kernel}
For an acyclic bundle $E$, the residue map induces an isomorphism
\[
 H^0\bigl(X\times X,(E\boxtimes E^\vee)(\Delta)\bigr)
 \xrightarrow{\ \sim\ }H^0(X,\End(E)).
\]
In particular, there is a unique section $s_E$ with residue
$\Id_E$. We call $s_E$ the \SZ\ kernel of $E$.
\end{proposition}

\begin{proof}
The two groups that could obstruct this isomorphism are the zeroth
and first cohomology groups of $E\boxtimes E^\vee$. By K\"unneth,
\[
 H^q(X\times X,E\boxtimes E^\vee)
 =\bigoplus_{a+b=q}H^a(X,E)\otimes H^b(X,E^\vee),
 \qquad q=0,1.
\]
Every summand is zero. The cohomology sequence of
\eqref{II:eq:residue-sequence} therefore gives the asserted isomorphism.
\end{proof}

This is the kernel construction of Ben-Zvi and Biswas
\cite[Section 3.1]{BB}. We have displayed the residue and cohomology
arguments to make clear which condition gives existence and which
gives uniqueness. The name follows their algebro-geometric usage:
$s_E$ is characterized by its diagonal pole and residue, rather
than by a Hardy-space projection. In concurrent work, Hu and
Yu~\cite[Sections~4--5]{HuYu} use the same normalized line-bundle
kernel for finite residue duality on curves over finite fields;
here its regular diagonal coefficients give higher differentials
on complex curves.

The uniqueness has a useful consequence even for maps that are not
isomorphisms. If $u:E\to F$, let $u^\vee:F^\vee\to E^\vee$ be its
Serre dual map.

\begin{lemma}[Functoriality of the kernel]\label{II:lem:kernel-natural}
For a map $u:E\to F$ between acyclic bundles,
\begin{equation}\label{II:eq:kernel-natural}
 (u\boxtimes1)s_E=(1\boxtimes u^\vee)s_F.
\end{equation}
Equivalently, $(u_x\otimes1)s_E(x,y)=s_F(x,y)u_y$ as maps from
$E_y$ to $F_x\otimes K_y$.
\end{lemma}

\begin{proof}
Both sides are sections of $(F\boxtimes E^\vee)(\Delta)$, and both
have residue $u$. Their difference has no pole, so it belongs to
\[
 H^0(X\times X,F\boxtimes E^\vee)
 =H^0(X,F)\otimes H^0(X,E^\vee)=0.
\]
Thus the sections agree.
\end{proof}

For a direct sum, this identity applied to the inclusions and
projections gives
\begin{equation}\label{II:eq:kernel-sum}
 s_{E\oplus F}=s_E\oplus s_F.
\end{equation}

\subsection{The bundle carrying the connection}
\label{II:subsec:connection}

The kernel compares nearby fibres of $E$, but its one-form factor
is carried by the second variable. To extract a connection, we
distribute this factor equally between the variables. Choose a
theta characteristic $\vartheta$, that is, a holomorphic line
bundle with an isomorphism $\vartheta^2\simeq K$, and set
\[
 V_E=E\otimes\vartheta^{-1}.
\]
Since $\deg\vartheta=g-1$, equation~\eqref{II:eq:acyclic-degree} gives
\[
 \deg V_E=\deg E-r\deg\vartheta=0.
\]
Every direct summand $E_i$ of $E$ is acyclic as well. Riemann--Roch
then gives
\[
 \deg E_i=\operatorname{rk}(E_i)(g-1),\qquad
 \deg(E_i\otimes\vartheta^{-1})=0.
\]
Thus every indecomposable summand of $V_E$ satisfies the
Atiyah--Weil condition, in agreement with the canonical connection constructed below.
The connection constructed in this subsection will be
\begin{equation}\label{II:eq:connection-target}
 \nabla^{V_E}:V_E\longrightarrow V_E\otimes K.
\end{equation}
The twist matters: $E$ has positive degree, whereas a holomorphic
bundle carrying a flat connection has degree zero. The construction
of $\nabla^{V_E}$ uses the holomorphic structure of $E$ and the chosen
$\vartheta$; no connection on either bundle is supplied as input.

The same twist separates the fibre maps from the one-form factor:
\[
 E\boxtimes E^\vee
 \simeq (V_E\boxtimes V_E^*)\otimes
                 (\vartheta\boxtimes\vartheta).
\]

On a coordinate disc, choose a local frame $e_\vartheta$ of $\vartheta$
with $e_\vartheta^{\otimes2}=dz$, and write it as $\sqrt{dz}$.
After factoring out $\sqrt{dz}\sqrt{dw}/(w-z)$, the remaining kernel
is a holomorphic map $V_{E,w}\to V_{E,z}$ near the diagonal.
Choose a holomorphic frame $e=(e_1,\ldots,e_r)$ of $V_E$ on the
disc, and expand its matrix:
\begin{equation}\label{II:eq:regular-kernel}
 \begin{aligned}
 s_E(z,w)&=\frac{\sqrt{dz}\sqrt{dw}}{w-z}\,\mathsf H_E(z,w),
 \qquad h=w-z,\\
 \mathsf H_E(z,w)&=\Id+{\mathsf A}(z)h+H_2(z)h^2+H_3(z)h^3+O(h^4).
 \end{aligned}
\end{equation}
The symbols ${\mathsf A},H_2,H_3$ denote holomorphic matrix functions of $z$.
The residue condition makes the constant coefficient the identity.
From now on, primes denote derivatives in the displayed coordinate.

We use ${\mathsf A}$ to define the connection in \eqref{II:eq:connection-target}
by the local rule
\begin{equation}\label{II:eq:connection-local-action}
 \nabla^{V_E}(ev)=e(v'+{\mathsf A}v)\,dz,
 \qquad v:U\longrightarrow\C^r\ \text{holomorphic}.
\end{equation}
We must check that these rules agree on overlaps. First replace the
frame $e$ by $eG$, where $G:U\to\operatorname{GL}_r(\C)$ is
holomorphic. Then
\[
 \mathsf H_E^G(z,w)=G(z)^{-1}\mathsf H_E(z,w)G(w).
\]
The coefficient of $h$ is consequently
\begin{equation}\label{II:eq:connection-frame-law}
 {\mathsf A}^G=G^{-1}{\mathsf A}G+G^{-1}G'.
\end{equation}
This is the required frame law, since the component column of a
section changes from $v$ to $G^{-1}v$.

For a locally biholomorphic coordinate change $z=\phi(t)$,
$w=\phi(u)$, so that $\phi'\ne0$, the two scalar
regularizations differ by
\begin{equation}\label{II:eq:regularization-factor}
 R_\phi(t,u)=
 \frac{(u-t)\sqrt{\phi'(t)\phi'(u)}}{\phi(u)-\phi(t)}.
\end{equation}
The branches are chosen locally so that $R_\phi(t,t)=1$.
Expanding numerator and denominator to first order gives
$R_\phi(t,u)=1+O((u-t)^2)$. It follows that
\begin{equation}\label{II:eq:connection-coordinate-law}
 \widetilde {\mathsf A}(t)=\phi'(t){\mathsf A}(\phi(t)).
\end{equation}
Equations \eqref{II:eq:connection-frame-law} and
\eqref{II:eq:connection-coordinate-law} are exactly the gluing laws for
\begin{equation}\label{II:eq:canonical-connection}
 \nabla^{V_E}:V_E\longrightarrow V_E\otimes K,
 \qquad \nabla^{V_E}=d+{\mathsf A}(z)\,dz.
\end{equation}
The connection is holomorphic. Its curvature is an
endomorphism-valued holomorphic exterior two-form, and hence
vanishes on a complex curve. Thus $\nabla^{V_E}$ is flat.

On a sufficiently small disc, choose a horizontal frame
$e=(e_1,\ldots,e_r)$ for this connection, so
$\nabla^{V_E}e_i=0$. It is obtained by solving $G'=-{\mathsf A}G$ with
invertible initial value. In this frame ${\mathsf A}=0$,
and \eqref{II:eq:regular-kernel} becomes
\begin{equation}\label{II:eq:horizontal-kernel}
 \mathsf H_E(z,w)=\Id+B_2(z)h^2+B_3(z)h^3+O(h^4).
\end{equation}
Two horizontal frames differ by a constant matrix. Tensoring by a
line bundle does not change the endomorphism bundle, so there is a
canonical identification $\End(V_E)=\End(E)$. Denote the induced
connection by
\[
 \nabla_E^{\mathrm{ad}}:\End(E)\longrightarrow\End(E)\otimes K.
\]
It is the connection of \eqref{II:eq:operator-adjoint-connection}
with $a_1={\mathsf A}$: for a local matrix $P(z)$ its coefficient is
\begin{equation}\label{II:eq:end-connection}
 \nabla_{E,z}^{\mathrm{ad}}P=P'+[{\mathsf A},P],\qquad [{\mathsf A},P]={\mathsf A}P-P{\mathsf A}.
\end{equation}

\begin{lemma}[Maps are horizontal]\label{II:lem:horizontal-maps}
If $u:E\to F$ is a map between acyclic bundles, then
$\widehat u=u\otimes\Id_{\vartheta^{-1}}:V_E\to V_F$ is horizontal:
\[
 \nabla^{V_F}\widehat u=(\widehat u\otimes1)\nabla^{V_E}.
\]
In horizontal frames its matrix is constant and intertwines every
coefficient of the regularized kernels.
\end{lemma}

\begin{proof}
In local frames, write $U(z)$ for the matrix of $\widehat u$.
The kernel identity \eqref{II:eq:kernel-natural} reads
\[
 U(z)\mathsf H_E(z,w)=\mathsf H_F(z,w)U(w).
\]
Its linear coefficient is
$U{\mathsf A}_E=U'+{\mathsf A}_FU$. This is the asserted connection identity.
In horizontal frames it becomes $U'=0$; comparison of the remaining
coefficients then gives, for example,
$UB_{j,E}=B_{j,F}U$ for $j=2,3$.
\end{proof}

\subsection{A cubic independent of the coordinate}
\label{II:subsec:cubic}

The remaining regular coefficients must supply the quadratic
and cubic data of the reconstruction theorem. The cubic is to be
a differential with values in $\End(E)$:
\begin{equation}\label{II:eq:cubic-type}
 \cub_E\in H^0(X,\End(E)\otimes K^3),
 \qquad \cub_E:E\longrightarrow E\otimes K^3.
\end{equation}
Writing $\cub_E=F(z)(dz)^3$, the laws of
Theorem~\ref{I:thm:rp-higher} require $F$ to transform
by conjugation under a frame change and satisfy
\begin{equation}\label{II:eq:cubic-tensor-law}
 \widetilde F(t)=\phi'(t)^3F(\phi(t))
\end{equation}
under $z=\phi(t)$ with the frame fixed.

The coefficients $B_2$ and $B_3$ in
\eqref{II:eq:horizontal-kernel} do not by themselves have these tensor
transformation laws. We will check that the combination
$6B_3-3B_2'$ has exactly the law \eqref{II:eq:cubic-tensor-law}.

\begin{proposition}[The canonical cubic]\label{II:prop:canonical-cubic}
In a horizontal frame, the expression
\begin{equation}\label{II:eq:canonical-cubic}
 \cub_E=(6B_3-3B_2')(dz)^3
\end{equation}
defines a holomorphic section of $\End(E)\otimes K^3$, called the
\emph{canonical cubic} of $E$. It is natural
under maps of acyclic bundles:
\begin{equation}\label{II:eq:cubic-natural}
 (u\otimes1)\cub_E=\cub_Fu\qquad(u:E\to F).
\end{equation}
Both $\cub_E$ and $\nabla_E^{\mathrm{ad}}$ are independent of the
choice of $\vartheta$.
\end{proposition}

\begin{proof}
\noindent\textit{1. Change of coordinate.} Keep a horizontal frame for $\nabla^{V_E}$ and make the coordinate
change used in \eqref{II:eq:regularization-factor}. Write $s=u-t$.
Primes on $\phi$ denote $t$-derivatives, while $B_j'$ denotes the
derivative of $B_j$ in $z$. With $\Sch(\phi)$ as in
\eqref{I:eq:schwarzian}, the scalar factor through third
order is
\begin{equation}\label{II:eq:R-cubic}
 R_\phi(t,t+s)=1+\frac{\Sch(\phi)}{12}s^2
                   +\frac{\Sch(\phi)'}{24}s^3+O(s^4).
\end{equation}
For clarity, this follows by writing
$a=\phi''/\phi'$, $b=\phi'''/\phi'$, and $c=\phi''''/\phi'$ at $t$.
Then the factor is
\[
 \frac{\sqrt{1+as+\tfrac12bs^2+\tfrac16cs^3+O(s^4)}}
      {1+\tfrac12as+\tfrac16bs^2+\tfrac1{24}cs^3+O(s^4)}.
\]
Its quadratic and cubic coefficients are
$b/12-a^2/8$ and $c/24-ab/6+a^3/8$, respectively, which are
the two coefficients in \eqref{II:eq:R-cubic}.

\noindent\textit{2. Cancellation and gluing.} Multiply \eqref{II:eq:horizontal-kernel} by this factor and use
$\phi(t+s)-\phi(t)=\phi's+\tfrac12\phi''s^2+O(s^3)$.
The new coefficients are
\begin{align}
 \widetilde B_2
  &=\phi'^2(B_2\circ\phi)+\tfrac1{12}\Sch(\phi)\Id,
       \label{II:eq:B2-coordinate}\\
 \widetilde B_3
  &=\phi'^3(B_3\circ\phi)+\phi'\phi''(B_2\circ\phi)
                    +\tfrac1{24}\Sch(\phi)'\Id.
       \label{II:eq:B3-coordinate}
\end{align}
Differentiating \eqref{II:eq:B2-coordinate} gives
\[
 \widetilde B_2'
 =\phi'^3(B_2'\circ\phi)+2\phi'\phi''(B_2\circ\phi)
                       +\tfrac1{12}\Sch(\phi)'\Id.
\]
Thus both extra terms cancel in
\begin{equation}\label{II:eq:cubic-coordinate}
 6\widetilde B_3-3\widetilde B_2'
 =\phi'^3\bigl((6B_3-3B_2')\circ\phi\bigr).
\end{equation}
This is precisely \eqref{II:eq:cubic-tensor-law}. A horizontal
frame change is constant and conjugates the coefficient, so the
local expressions glue to a section of $\End(E)K^3$.
Lemma~\ref{II:lem:horizontal-maps} gives
\eqref{II:eq:cubic-natural}, including the derivative term because
the matrix of $u$ is constant in those frames.

\noindent\textit{3. Independence of the theta characteristic.} Finally, write a second theta characteristic as
$\vartheta' =\vartheta\otimes\varepsilon$, where $\varepsilon^2\simeq\OO_X$
and $\OO_X$ is the trivial holomorphic line bundle. Local frames of
$\varepsilon$ can be chosen with
transition functions in $\{1,-1\}$, which define a flat connection.
Replacing $\vartheta$ by $\vartheta\otimes\varepsilon$ replaces $V_E$ by
$V_E\otimes\varepsilon^{-1}$. In these locally constant frames, the matrices
in \eqref{II:eq:regular-kernel} are unchanged. The induced connection
on endomorphisms and the cubic are therefore unchanged as well.
\end{proof}

The scalar correction in \eqref{II:eq:B2-coordinate} explains why
$B_2(dz)^2$ alone is not a quadratic differential. The cubic removes
its derivative together with the corresponding correction in $B_3$.
Thus $\cub_E$ is defined from the holomorphic bundle $E$ without
an additional choice of a connection on $E$ or a coordinate on $X$.

The quadratic datum is read from the same transformation:
in a horizontal frame put $P_{E,z}=2B_2(z)$. Then
\[
 \widetilde P_E=\phi'^2(P_E\circ\phi)
                             +\tfrac16\Sch(\phi)\Id.
\]
Together with conjugation under a change of horizontal frame,
this is exactly \eqref{II:eq:quadratic-datum-law}. We have therefore
obtained all three data required for reconstruction in order three:
\[
 \nabla^{V_E},\qquad P_E,\qquad
 \cub_E\in H^0(X,\End(E)\otimes K^3).
\]

\subsection{Reconstruction and the Taylor operator}
\label{II:subsec:part-I-cubic}

Theorem~\ref{II:thm:curve-reconstruction}, with $N=3$, therefore
supplies a unique global monic operator
\begin{equation}\label{II:eq:global-kernel-operator}
 \mathcal L_{E,3}:V_E\otimes\vartheta^{-2}\longrightarrow
                            V_E\otimes\vartheta^4
\end{equation}
whose connection is $\nabla^{V_E}$, quadratic datum is $P_E$, and
cubic is $\cub_E$. Its horizontal expression is
\begin{equation}\label{II:eq:global-kernel-operator-expression}
 \mathcal L_{E,3}=\partial_z^3+24B_2\partial_z+6B_3+9B_2'.
\end{equation}
Indeed, the normalization requires $I_2=4P_E=8B_2$ and
$I_3=(6B_3-3B_2')+\tfrac32(8B_2)'=6B_3+9B_2'$.
Thus \eqref{I:eq:local-normalized} gives the displayed formula and
\[
 W_3(\mathcal L_{E,3})=\cub_E.
\]

The operator is canonical for the chosen $\vartheta$. If
$\vartheta$ is replaced by $\vartheta\otimes\varepsilon$, with
$\varepsilon^2\simeq\OO_X$, its source and target are both tensored by
$\varepsilon$, and
\[
 \mathcal L^{\vartheta\otimes\varepsilon}_{E,3}
      =\mathcal L^\vartheta_{E,3}\otimes\Id_\varepsilon.
\]
Here tensoring uses the flat connection whose local transition
functions are $\pm1$. The adjoint connection and $\cub_E$ remain
unchanged under the canonical identification
$\End(V_E\otimes\varepsilon)\simeq\End(V_E)$.

\begin{remark}[Taylor normalization]
A second operator records the Taylor coefficients directly.
It will be useful when we form the higher residue operators in
Section~\ref{II:sec:higher-detection}. In the chosen coordinate, set
\begin{equation}\label{II:eq:kernel-operator-action}
 (\mathcal T_{E,3}y)(z)=
 \left.\frac{\partial^3}{\partial w^3}
       \bigl(\mathsf H_E(z,w)y(w)\bigr)\right|_{w=z}.
\end{equation}
The product rule gives
\begin{equation}\label{II:eq:local-kernel-operator}
 \mathcal T_{E,3}=\partial^3+3{\mathsf A}\partial^2+6H_2\partial+6H_3.
\end{equation}
Thus the factors $3,6,6$ come from differentiating the Taylor series.
The law $\mathsf H_E^G(z,w)=G(z)^{-1}\mathsf H_E(z,w)G(w)$ also
shows directly in \eqref{II:eq:kernel-operator-action} that
$\mathcal T_{E,3}^G=G^{-1}\circ\mathcal T_{E,3}\circ G$.
In a horizontal frame for $\nabla^{V_E}$ the expression is
\[
 \mathcal T_{E,3}=\partial^3+6B_2\partial+6B_3,
 \qquad I_2=2B_2,\quad I_3=6B_3.
\]
Consequently,
\begin{equation}\label{II:eq:kernel-W3}
 W_3(\mathcal T_{E,3})(dz)^3=\cub_E.
\end{equation}
Here we have used the Ore normalization of
Section~\ref{I:sec:local-normalization}. The coordinate covariance of
the cubic is the calculation \eqref{II:eq:cubic-coordinate}.
The Taylor operator has $I_2=2B_2$, whereas the reconstructed
operator has $I_2=8B_2$. They have the same cubic; it is
$\mathcal L_{E,3}$ that has the global source and target
\eqref{II:eq:global-kernel-operator}.
\end{remark}

\subsection{Hyperelliptic kernels and a genus-two example}
\label{II:subsec:hyperelliptic-kernels}

On an even-degree hyperelliptic curve, the normalized kernel has a
rational formula. We derive its cubic from the Taylor operator just
defined, then calculate the horizontal expansion on a genus-two curve.

\Needspace{24\baselineskip}
\begin{lemma}[A hyperelliptic kernel and its cubic]
\label{II:lem:hyperelliptic-kernel}
Let $X$ be the smooth completion of $y^2=f(x)$, where $f$ is a
monic squarefree polynomial of degree $2g+2$, with $g\ge2$.
Denote its two points at infinity by
\[
 P:\quad y/x^{g+1}\longrightarrow-1,
 \qquad Q:\quad y/x^{g+1}\longrightarrow1.
\]
Let $S$ be the polynomial part of $+\sqrt{f}$ at infinity and put
\[
 R=f-S^2,\qquad \deg R\le g,\qquad
 \omega=\frac{dx}{y},\qquad L=\OO(gP-Q).
\]
Then $L$ is acyclic. Use the rational section represented by $1$
in the divisor model $L=\OO(gP-Q)$; its divisor as a section
of $L$ is $gP-Q$. In this frame, the canonical kernel is
\begin{equation}\label{II:eq:hyperelliptic-kernel}
 s_L(z,w)=
 \frac{y(z)+y(w)-S(x(z))+S(x(w))}
      {2(x(w)-x(z))}\,\frac{dx(w)}{y(w)}.
\end{equation}
The factors from $L_z\otimes L_w^{-1}$ are understood.
Its cubic, with derivatives taken in $x$, is
\begin{equation}\label{II:eq:hyperelliptic-cubic}
 \cub_L=\frac{S''R'-S'R''}{8}\,\omega^3.
\end{equation}
In particular, $\cub_L\ne0$ if and only if $R$ is nonconstant.
\end{lemma}

\begin{proof}
\noindent\textit{1. Acyclicity.} First, the differentials $x^j\omega$, $0\le j\le g-1$, have
orders $g-1-j$ at $P$. They form a basis of $H^0(X,K)$, so
$H^0(X,K(-gP))=0$. Riemann--Roch gives
$h^0(X,\OO(gP))=1$. Its sections are the constants, and a
nonzero constant does not vanish at $Q$. Thus $H^0(X,L)=0$;
since $\deg L=g-1$, also $H^1(X,L)=0$.

\noindent\textit{2. Divisor bounds.} Write the rational factor in \eqref{II:eq:hyperelliptic-kernel} as
$F(z,w)$. At infinity we have
\[
 y=-S+O(x^{-1})\quad\hbox{at }P,
 \qquad y=S+O(x^{-1})\quad\hbox{at }Q.
\]
For fixed $w$, these expansions show that $F$ has a pole of
order at most $g$ at $z=P$ and vanishes at $z=Q$. For fixed
$z$, it vanishes at $w=P$ and has a pole of order at most $g$
at $w=Q$. Together with
$\operatorname{div}(\omega)=(g-1)(P+Q)$, these are exactly
the bounds required by $L_z\otimes L_w^{-1}\otimes K_w$.

\noindent\textit{3. The diagonal residue.} Away from infinity the only possible additional pole lies over
$x(w)=x(z)$. On the conjugate sheet the numerator vanishes,
so that pole cancels. At a branch point, use local parameters
$x(z)=a+t^2$, $x(w)=a+s^2$, with
$y(z)=tU(t^2)$ and $y(w)=sU(s^2)$, where $U(0)\ne0$.
The numerator is divisible by $s+t$, leaving only the
diagonal factor $s-t$. Its residue in the second variable is
$+1$. The residue characterization in
Proposition~\ref{II:prop:kernel} therefore proves the formula.

\noindent\textit{4. The normalized cubic.} In a coordinate disc with coordinate $x$, regularization by
$\sqrt{dx}\sqrt{dw}/(w-x)$ gives
\[
 H(x,w)=\frac{y(x)+y(w)-S(x)+S(w)}{2y(w)}.
\]
This is a rational frame, so its linear coefficient need not
vanish. Write $\epsilon=w-x$ and set
\begin{equation}\label{II:eq:hyperelliptic-factorial-expansion}
 \begin{aligned}
 H(x,x+\epsilon)&=1+u(x)\epsilon+\tfrac12v(x)\epsilon^2
                      +\tfrac16a_3(x)\epsilon^3+O(\epsilon^4),\\
 a_j(x)&=\left.\partial_w^j H(x,w)\right|_{w=x},
       \qquad a_1=u=\frac{S'-y'}{2y},\qquad a_2=v.
 \end{aligned}
\end{equation}
The identity
$2y(w)H(x,w)=y(x)+y(w)-S(x)+S(w)$ gives the next two coefficients
without expanding a quotient:
\[
 v=\frac{S''-y''}{2y}-2u\frac{y'}y,\qquad
 a_3=\frac{S'''-y'''}{2y}-3v\frac{y'}y-3u\frac{y''}y.
\]
The Taylor operator uses these $a_j$ by
\eqref{II:eq:kernel-operator-action}. We first express it in powers
of $\nabla=\partial_x+u$, using the Ore normalization
\eqref{I:eq:local-I23}. For scalar coefficients this reads
\[
 I_2=v-u'-u^2,\qquad
 I_3=a_3-u''-3uu'-u^3-3I_2u,\qquad
 W_3=I_3-\tfrac32 I_2'.
\]
Substituting these expressions and using $y^2=f$ gives
\[
 I_3-\tfrac32I_2'
   =\frac{2(S')^3-S'f''+S''f'}{8y^3}
   =\frac{S''R'-S'R''}{8y^3},
\]
since $f=S^2+R$. This proves \eqref{II:eq:hyperelliptic-cubic}.
If $R$ has degree
$r\ge1$ and leading coefficient $a$, this numerator has leading
coefficient $(g+1)r(g-r+1)a$, which is nonzero because $r\le g$.
If $R$ is constant, the numerator is zero.
\end{proof}

\begin{example}[A genus-two kernel through cubic order]
\label{II:ex:small-kernel}
Let $X$ be the smooth completion of $y^2=x^6+x+1$. The polynomial
is squarefree. Let $P_-,P_+$ be the points at infinity where
$y/x^3$ tends to $-1,+1$, respectively, and take
$L=\OO(2P_--P_+)$. Lemma~\ref{II:lem:hyperelliptic-kernel} proves
acyclicity and gives its kernel with $S=x^3$ and $R=x+1$.
In the divisor frame of $L$ and the coordinate half-form
$\sqrt{dx}$, its regular part is
\[
 \mathsf H_L(x,w)
 =\frac{y(x)+y(w)-x^3+w^3}{2y(w)},\qquad
 \mathsf A(x)=\frac{3x^2-y'(x)}{2y(x)}.
\]
Choose the branch $y(0)=1$. The first expansion is
\[
 \mathsf H_L(0,w)=1-\frac w4+\frac{3w^2}{16}
                         +\frac{11w^3}{32}+O(w^4).
\]
It is not yet horizontal: $\mathsf A(0)=-1/4$.
Let $G'=-\mathsf A G$ and $G(0)=1$. Since
$\mathsf A'(0)=1/4$ and $\mathsf A''(0)=5/2$,
\[
 G(w)=1+\frac w4-\frac{3w^2}{32}
                         -\frac{57w^3}{128}+O(w^4),\qquad
 \mathsf H_L^{\mathrm{hor}}(0,w)
      =1+\frac{w^2}{32}-\frac{w^3}{32}+O(w^4).
\]
Thus $B_2(0)=1/32$ and $B_3(0)=-1/32$. Keeping the base point
$x$ variable gives
\[
 B_2(x)=\frac{1-36x^4-24x^5}{32(x^6+x+1)^2},\qquad
 6B_3-3B_2'=\frac{3x}{4y^3}.
\]
In particular $B_2'(0)=-1/16$: the cubic vanishes at $x=0$ even
though $B_3(0)\ne0$. Globally,
$\cub_L=(3x/4)(dx/y)^3$ is nonzero. This distinguishes the cubic
tensor from a single Taylor coefficient and shows why the
horizontal normalization enters its definition.
\end{example}

\subsection{Duality}
\label{II:subsec:kernel-extensions}

Duality supplies pairs of line bundles with opposite cubics.
These will give the fixed-determinant extensions in
Section~\ref{II:subsec:extension-consequences}.

The bundle associated to the Serre dual is
\[
 V_{E^\vee}=E^*\otimes K\otimes\vartheta^{-1}
           =E^*\otimes\vartheta=V_E^*.
\]
For an endomorphism $P$ of $E$, its transpose $P^{\mathsf t}$ acts
on a covector $\lambda$ by $P^{\mathsf t}\lambda=\lambda\circ P$;
it also acts on $E^\vee=E^*\otimes K$ through the first factor.

\begin{lemma}[Duality]\label{II:lem:cubic-duality}
Under the identification $V_{E^\vee}=V_E^*$, the connection
$\nabla^{V_{E^\vee}}$ is dual to $\nabla^{V_E}$: if the latter has
matrix ${\mathsf A}\,dz$, the former has matrix $-{\mathsf A}^{\mathsf t}\,dz$ in the
dual frame. Moreover,
\begin{equation}\label{II:eq:cubic-duality}
 \cub_{E^\vee}=-\cub_E^{\mathsf t}.
\end{equation}
In particular, for an acyclic line bundle $L$,
$\cub_{KL^{-1}}=-\cub_L$.
\end{lemma}

\begin{proof}
Interchange the two variables and transpose the matrix of $s_E$.
The resulting section takes values in
$(E^\vee\boxtimes E)(\Delta)$ and has residue $-\Id$:
the factor $1/(w-z)$ changes sign. Uniqueness of the kernel gives
\begin{equation}\label{II:eq:kernel-duality}
 s_{E^\vee}(z,w)=-s_E(w,z)^{\mathsf t}.
\end{equation}
The first regular coefficient on the right is $-{\mathsf A}^{\mathsf t}$,
so its connection is the dual connection.

Now choose a horizontal frame and its dual. With $h=w-z$,
the regular part on the right of \eqref{II:eq:kernel-duality} is
\[
 \Id+B_2(w)^{\mathsf t}h^2-B_3(w)^{\mathsf t}h^3+O(h^4).
\]
Expanding at $z$ shows that the dual coefficients are
$B_2^{\mathsf t}$ and $(B_2')^{\mathsf t}-B_3^{\mathsf t}$.
Their cubic combination is
$3(B_2')^{\mathsf t}-6B_3^{\mathsf t}=-(6B_3-3B_2')^{\mathsf t}$,
which proves \eqref{II:eq:cubic-duality}.
\end{proof}

\subsection{The cubic of a line bundle}
\label{II:subsec:theta-cubic}

For a line bundle, the cubic has a direct theta formula. Besides
giving a way to calculate it, the formula will show that distinct
cubics occur for a general pair of acyclic lines.
Choose a symplectic marking of $X$, normalized holomorphic
differentials $\omega_1,\ldots,\omega_g$, and the associated Riemann
theta function $\theta$. Choose the symmetric identification of
$\operatorname{Pic}^{g-1}(X)$ with the Jacobian for which the theta
divisor has equation $\theta(e)=0$. Write $L_e$ for
the line represented by $e$, with the difference convention
$L_e(w-z)=L_{e+u(z,w)}$, where
\[
 u_i(z,w)=\int_z^w\omega_i.
\]
Thus $L_e$ is acyclic exactly when $\theta(e)\ne0$. Derivatives
$\partial_i$ below are taken in the Abelian coordinates of $e$.

\begin{proposition}[Theta formula for the cubic]
\label{II:prop:theta-cubic}
For every $e$ with $\theta(e)\ne0$,
\begin{equation}\label{II:eq:theta-cubic}
 \cub_{L_e}
 =\sum_{i,j,k=1}^g
   \partial_i\partial_j\partial_k\log\theta(e)\,
   \omega_i\omega_j\omega_k.
\end{equation}
\end{proposition}

\begin{proof}
We calculate from the classical line-kernel formula
\cite[Section~4.2, equation~(4.2.1)]{BB}, using the residue and
Abel-difference conventions above. In a coordinate $z$, write
$\omega_i=f_i(z)\,dz$, and write the prime form as
$\epsilon(z,w)/(\sqrt{dz}\sqrt{dw})$, normalized by
$\epsilon(z,w)=(w-z)+O((w-z)^3)$. Put
\[
 r(z,w)=\frac{w-z}{\epsilon(z,w)},\qquad
 \psi(e)=\log\theta(e).
\]
The regularized theta kernel is
$r(z,w)\theta(e+u)/\theta(e)$. Its linear coefficient is
${\mathsf A}(z)=\sum_i\psi_i(e)f_i(z)$, so passage to the canonical horizontal
frame multiplies it by $\exp(-\sum_i\psi_i(e)u_i)$. Consequently,
\begin{equation}\label{II:eq:theta-horizontal-kernel}
 \mathsf H^{\mathrm{hor}}_{L_e}(z,w)
 =r(z,w)\exp\!\left(
   \psi(e+u)-\psi(e)-\sum_i\psi_i(e)u_i\right).
\end{equation}
This subtraction is essential: the theta trivialization itself
need not be horizontal for the canonical connection.

The prime form is antisymmetric, hence $r(z,w)=r(w,z)$.
Since $r(z,z)=1$, its expansion for $h=w-z$ has the form
\[
 r(z,z+h)=1+r_2(z)h^2+\tfrac12r_2'(z)h^3+O(h^4).
\]
Substitute $u_i=f_i h+\tfrac12f_i'h^2+O(h^3)$ into
\eqref{II:eq:theta-horizontal-kernel}. The quadratic and cubic
coefficients are
\begin{align*}
 B_2&=r_2+\tfrac12\sum_{i,j}\psi_{ij}f_if_j,\\
 B_3&=\tfrac12r_2'
       +\tfrac12\sum_{i,j}\psi_{ij}f_if_j'
       +\tfrac16\sum_{i,j,k}\psi_{ijk}f_if_jf_k.
\end{align*}
In $6B_3-3B_2'$, the prime-form term and all second-derivative
terms cancel. Equation~\eqref{II:eq:canonical-cubic} gives the claim.
\end{proof}

This is the third-order companion of the second logarithmic
derivative formula in \cite[Section~4.4]{BB}. The kernel formula
is classical; the calculation identifies its cubic coefficient
with the normalization used here. Since $L_{-e}=K\otimes L_e^{-1}$
and $\theta$ is even, \eqref{II:eq:theta-cubic} also recovers
$\cub_{K\otimes L_e^{-1}}=-\cub_{L_e}$.

\begin{example}[The genus-two theta formula]
\label{II:ex:theta-genus-two}
Put $\psi_{ijk}=\partial_i\partial_j\partial_k\log\theta(e)$.
In genus two, Proposition~\ref{II:prop:theta-cubic} reads
\[
 \cub_{L_e}=\psi_{111}\omega_1^3
 +3\psi_{112}\omega_1^2\omega_2
 +3\psi_{122}\omega_1\omega_2^2+\psi_{222}\omega_2^3.
\]
The factors three count the permutations of the mixed indices.
For an acyclic theta characteristic $L$, duality gives
$KL^{-1}=L$ and hence $\cub_L=-\cub_L=0$.
In theta coordinates centered at such a characteristic, the same
vanishing follows because the nonvanishing theta function is even,
so its third logarithmic derivatives at the center vanish.
\end{example}

\begin{corollary}[Distinct cubics occur generically]
\label{II:cor:generic-distinct-cubics}
The assignment
\[
 U\longrightarrow H^0(X,K^3),\qquad L\longmapsto\cub_L,
\]
is a nonconstant algebraic morphism. Hence
$\{(L,M)\in U\times U:\cub_L\ne\cub_M\}$ is a nonempty
Zariski-open subset. For each fixed $L\in U$, a general $M\in U$
has $\cub_M\ne\cub_L$.
\end{corollary}

\begin{proof}
First, algebraicity follows from the kernel construction in
families, rather than from the analytic theta expression alone.
Take a Poincar\'e line bundle on $X\times U$. Acyclicity, together
with cohomology and base change, makes the relative residue map an
isomorphism, just as in Proposition~\ref{II:prop:kernel}. Its inverse
therefore gives an algebraic family of normalized kernels.
The cubic is obtained from the first three regular coefficients
by differential-polynomial expressions, so it is algebraic in
the parameter as well. Horizontal frames are needed for the
short formula \eqref{II:eq:canonical-cubic}, not for this conclusion.

To prove nonconstancy, approach a smooth point $e_0$ of the theta
divisor. Such points exist, and $d\theta(e_0)\ne0$; see
\cite[Section~4.3]{BB}. Formula~\eqref{II:eq:theta-cubic} gives
\[
 \lim_{e\to e_0}\theta(e)^3\cub_{L_e}
 =2\left(\sum_i\partial_i\theta(e_0)\omega_i\right)^3\ne0.
\]
The last inequality follows because the $\omega_i$ are a basis.
Thus the cubic map has a pole at $e_0$ and cannot be constant.
The locus $U$ is irreducible, being a nonempty open subset of
$\operatorname{Pic}^{g-1}(X)$, a translate of the Jacobian.
Each fixed-value fibre is proper and closed, so its complement
is dense and open.
Equality of the two
cubics is the inverse image of the diagonal under the product map,
so it is also a proper closed condition.
\end{proof}

\subsection{The canonical modular forms}
\label{II:subsec:kernel-modular-forms}

The preceding construction supplies both inputs to
\eqref{II:eq:intrinsic-self-bracket}: the connection on $V_E$ and its
endomorphism-valued cubic $\cub_E$. Put
\begin{equation}\label{II:eq:canonical-bracket}
 \mathcal C_E=\mathcal C(\cub_E)
     =6[F,\nabla_{E,z}^{\mathrm{ad}}F](dz)^7
       \in H^0(X,\End(E)\otimes K^7),
 \qquad \cub_E=F(z)(dz)^3.
\end{equation}
Lemma~\ref{II:lem:intrinsic-self-bracket} proves that this expression
is intrinsic. Proposition~\ref{II:prop:canonical-cubic} shows that it
is independent of $\vartheta$ as well. It is natural under maps of
acyclic bundles: in horizontal frames such maps have constant
matrices intertwining $F$, and hence intertwining $F'$ and
$[F,F']$.

To obtain automorphic coefficients, uniformize the compact curve
as $X=\Gamma_X\backslash\mathbb H$ and choose a lift of its group to
$\operatorname{SL}_2(\mathbb R)$. The pullback of $\nabla^{V_E}$ has
a global horizontal frame on $\mathbb H$. Let
$\rho_E:\Gamma_X\to\operatorname{GL}_r(\C)$ be its descent
representation. In this frame write
\begin{equation}\label{II:eq:kernel-modular-coefficients}
 \pi^*\cub_E=F_E(\tau)(d\tau)^3,\qquad
 \Phi_E=[F_E,F_E]_1^0=6[F_E,DF_E],
 \qquad \pi:\mathbb H\to X.
\end{equation}

\begin{proposition}[Automorphic coefficients on a compact quotient]
\label{II:prop:canonical-modular-forms}
Let $E$ be an acyclic bundle on a compact curve
$X=\Gamma_X\backslash\mathbb H$ of genus at least two.
The two coefficients in \eqref{II:eq:kernel-modular-coefficients} are
holomorphic forms of weights $6$ and $14$ with multiplier
$\operatorname{Ad}\rho_E$:
\begin{equation}\label{II:eq:canonical-modular-laws}
 \begin{aligned}
 F_E(\gamma\tau)&=j(\gamma,\tau)^6
                      \rho_E(\gamma)F_E(\tau)\rho_E(\gamma)^{-1},\\
 \Phi_E(\gamma\tau)&=j(\gamma,\tau)^{14}
                      \rho_E(\gamma)\Phi_E(\tau)\rho_E(\gamma)^{-1}.
 \end{aligned}
\end{equation}
They satisfy
\begin{equation}\label{II:eq:canonical-modular-normalization}
 \Phi_E(\tau)(d\tau)^7=\kappa\,\pi^*\mathcal C_E.
\end{equation}
Changing the horizontal frame conjugates the coefficients and the
representation by the same constant matrix.

\end{proposition}

\begin{proof}
The first law in \eqref{II:eq:canonical-modular-laws} is the tensor law
for the cubic in a horizontal frame. The second follows from
Proposition~\ref{II:prop:RC-properties} with $k=\ell=6$ and $B=0$;
equivalently, the extra derivative term in the transformation of
$DF_E$ is proportional to $F_E$ and vanishes in the commutator.
Equation~\eqref{II:eq:canonical-modular-normalization} is
\eqref{II:eq:intrinsic-modular-bracket} for $m=3$.

\end{proof}

\begin{corollary}[Cusp orders on a compactification]
\label{II:cor:kernel-cusp-forms}
Let $Y=\Gamma\backslash\mathbb H$ be a noncompact quotient as in
Section~\ref{II:sec:modular-calculus}, and suppose its smooth
compactification $X=\bar Y$ has genus at least two. Let $E$ be an
acyclic bundle on $X$. Construct its kernel on $X$ and define
$F_E,\Phi_E$ by~\eqref{II:eq:kernel-modular-coefficients} after
pullback to $\mathbb H$. These coefficients satisfy
\eqref{II:eq:canonical-modular-laws}--\eqref{II:eq:canonical-modular-normalization},
and their representation is trivial at parabolic stabilizers. Both coefficients are cusp forms; at each cusp,
\begin{equation}\label{II:eq:kernel-cusp-orders}
 F_{E,\sigma}=O(q^3),\qquad \Phi_{E,\sigma}=O(q^7).
\end{equation}
\end{corollary}
\begin{proof}
The transformation laws follow from the tensor calculation in the
preceding proposition. The connection is holomorphic on a
disc about each cusp, so a horizontal frame exists on the entire
disc. A loop about its center has trivial monodromy. Both $\cub_E$
and $\mathcal C_E$ are holomorphic tensors on $X$; applying
\eqref{II:eq:cusp-tensor} with $m=3$ and $m=7$ gives
\eqref{II:eq:kernel-cusp-orders}.
\end{proof}

Thus the same canonical construction works on every compact complex
curve of genus at least two and, when the curve is the compactification
of a congruence quotient, gives matrix-valued modular cusp forms on
that quotient. The forms and their multiplier come from the bundle
and its kernel connection. Their ability to retain extension data
will be the application developed next.

\section{Extension recovery from the canonical cubic}
\label{II:sec:extension-recovery}

The preceding section constructs a canonical cubic and a modular
self-bracket from an acyclic bundle. We now ask what the bracket
remembers about an extension of two line bundles. The main point is
that its off-diagonal entry recovers the extension class when the
two endpoint cubics differ.

Keep $X$ a smooth compact complex curve of genus $g\geq2$, put
$K=K_X$, and fix two acyclic holomorphic line bundles $L,M$ on $X$.
Corollary~\ref{II:cor:generic-distinct-cubics} shows that
$\cub_L\ne\cub_M$ holds for a general pair.

\subsection{Marked extensions and triangular forms}
\label{II:subsec:marked-extensions}

\begin{definition}[Marked extension]\label{II:def:marked-extension}

An \emph{extension of $M$ by $L$} is a short exact sequence of
holomorphic bundles
\begin{equation}\label{II:eq:line-extension}
 0\longrightarrow L\xrightarrow{i}E\xrightarrow{\pi}M
 \longrightarrow0.
\end{equation}
The middle bundle $E$ has rank two, $i$ identifies $L$ with a
subbundle, and $\pi$ identifies the quotient with $M$. A splitting is a holomorphic map $s:M\to E$ with
$\pi s=\Id_M$; it identifies $E$ with $L\oplus M$ compatibly
with $i$ and $\pi$.

Local splittings exist. Choose them on an open cover $\{U_i\}$,
writing $s_i:M|_{U_i}\to E|_{U_i}$. Put
\[
 T=\Hom(M,L)=L\otimes M^{-1}.
\]
On an overlap $U_i\cap U_j$, the difference lands in $i(L)$,
so there is a unique section $\varphi_{ij}$ of $T$ with
\begin{equation}\label{II:eq:extension-cocycle}
 s_j-s_i=i\varphi_{ij},\qquad
 \varphi_{ij}+\varphi_{jk}=\varphi_{ik}.
\end{equation}
Replacing $s_i$ by $s_i+i\psi_i$, with
$\psi_i\in H^0(U_i,T)$, changes $\varphi_{ij}$ to
$\varphi_{ij}+\psi_j-\psi_i$. Thus the cocycle defines the
\emph{extension class}
\begin{equation}\label{II:eq:extension-class}
 \xi(E)=[\varphi_{ij}]\in H^1(X,T).
\end{equation}
It vanishes exactly when the local splittings can be modified to
agree, giving a global splitting. We write $E_\xi$ for an extension
representing $\xi$, retaining the inclusion $i$ and quotient $\pi$
in \eqref{II:eq:line-extension}. Isomorphisms of these marked extensions
must commute with $i$ and $\pi$ and induce the identity on $L$ and $M$.
\end{definition}

The cohomology sequence of \eqref{II:eq:line-extension} gives
$H^q(X,E)=0$ for $q=0,1$, because the corresponding groups for $L$
and $M$ vanish. Thus $E$ is acyclic, and all the preceding
constructions apply to it. A frame \emph{adapted to the extension}
consists of a local frame of $L$ followed by the image under a local
splitting of a frame of $M$.

\begin{corollary}[Triangular form]\label{II:cor:triangular-kernel}
For the marked extension~\eqref{II:eq:line-extension}, kernel
naturality has the following consequences.
\begin{enumerate}[label=\textup{(\roman*)},leftmargin=*]
\item The connection on $V_E$ preserves $L\vartheta^{-1}$ and induces the
canonical connection on $M\vartheta^{-1}$; these are the connections
$\nabla^{V_E}$, $\nabla^{V_L}$, and $\nabla^{V_M}$ constructed above.
In local frames adapted
to \eqref{II:eq:line-extension}, the kernel and the cubic are upper
triangular, with diagonal entries $s_L,s_M$ and $\cub_L,\cub_M$,
respectively.
\item The determinant is independent of the extension class:
\begin{equation}\label{II:eq:extension-determinant}
 \det s_E=s_Ls_M=\det s_{L\oplus M}
\end{equation}
under the determinant-line identification induced by the extension.
Here the determinant is taken fibrewise:
$\det s_E(x,y):\det E_y\to\det E_x\otimes K_y^2$ for $x\ne y$.
\end{enumerate}
\end{corollary}

\begin{proof}
Apply Lemma~\ref{II:lem:horizontal-maps} to $i$ and $\pi$ to obtain
the assertion about the connection. Apply
\eqref{II:eq:kernel-natural} to the same maps: the first column is
$(s_L,0)^{\mathsf t}$ and the second row is $(0,s_M)$.
Thus the kernel is upper triangular with the stated diagonal.
Equation~\eqref{II:eq:cubic-natural} gives the identical assertion
for the cubic. Taking determinants proves
\eqref{II:eq:extension-determinant}.
\end{proof}

For example, if $a$ and $d$ are the local coefficients of
$\cub_L$ and $\cub_M$, the cubic coefficient of $E$ has the form
\[
 F=\begin{pmatrix}a&b\\0&d\end{pmatrix}.
\]
Its trace and determinant are $a+d$ and $ad$, independently of $b$.
Write $\varphi(z)$ for the coefficient of a local section of $T$
in the chosen line frames. Changing the splitting by this section
gives the frame change
$G=\left(\begin{smallmatrix}1&\varphi\\0&1\end{smallmatrix}\right)$.
Direct multiplication gives
\begin{equation}\label{II:eq:cubic-splitting-change}
 G^{-1}FG=\begin{pmatrix}a&b+(a-d)\varphi\\0&d\end{pmatrix}.
\end{equation}
Thus $b$ itself depends on the local splitting. The bracket
$\mathcal C_E$, however, is already independent of this choice.
In a horizontal frame adapted to the preserved line and quotient,
\begin{equation}\label{II:eq:triangular-bracket}
 \mathcal C_E
 =6\begin{pmatrix}
     0&(a-d)b'-b(a'-d')\\0&0
   \end{pmatrix}(dz)^7.
\end{equation}
This follows by multiplying $F$ and $F'$ in both orders. In
particular, $\mathcal C_E$ kills $L$ and takes its values in
$L\otimes K^7$. It therefore factors through the specified quotient
and inclusion as
\[
 E\xrightarrow{\pi}M\longrightarrow L\otimes K^7
                      \xrightarrow{i\otimes1}E\otimes K^7,
\]
and can be viewed as a section of $T\otimes K^7$, where
$T=\Hom(M,L)$. The upper-right entry is the part that can depend on the extension.
We now show that, when $\cub_L\ne\cub_M$, it determines the marked
extension class.

For $E_\xi$, write this factorization as
\begin{equation}\label{II:eq:extension-bracket-factor}
 c_\xi\in H^0(X,T\otimes K^7),\qquad
 \mathcal C_{E_\xi}=(i\otimes1)c_\xi\pi.
\end{equation}
An isomorphism of marked extensions intertwines the brackets and
induces the identity on $L,M$. Thus $c_\xi$ depends only on $\xi$.

\subsection{The recovery theorem}
\label{II:subsec:extension-recovery-theorem}

The connection needed to read this differential is fixed before
choosing the extension. Write $V_L=L\otimes\vartheta^{-1}$ and
$V_M=M\otimes\vartheta^{-1}$. Their canonical connections induce
\begin{equation}\label{II:eq:fixed-T-connection}
 T=\Hom(V_M,V_L),\qquad
 \nabla^T u=\nabla^{V_L}\circ u-(u\otimes1)\circ\nabla^{V_M}.
\end{equation}
This is a holomorphic flat connection on $T$. It is the restriction
of the induced connection on $\End(V_{L\oplus M})$, and hence is
independent of $\vartheta$ by Proposition~\ref{II:prop:canonical-cubic}.
It depends on the fixed lines $L,M$, not on $\xi$.

Put
\begin{equation}\label{II:eq:cubic-gap}
 d_{LM}=\cub_L-\cub_M\in H^0(X,K^3).
\end{equation}
When $d_{LM}\not\equiv0$, let $Z=(d_{LM})_0$ be its zero divisor.
The quotient sheaf
$\mathcal P_Z(T)=T(Z)/T$ records the polar parts of meromorphic
sections of $T$ whose poles are bounded by $Z$. The exact sequence
\begin{equation}\label{II:eq:extension-principal-sequence}
 0\longrightarrow T\longrightarrow T(Z)
   \longrightarrow\mathcal P_Z(T)\longrightarrow0
\end{equation}
gives a connecting map
$\partial_Z:H^0(X,\mathcal P_Z(T))\to H^1(X,T)$.
We use the sign convention of \eqref{II:eq:extension-cocycle}:
local meromorphic lifts $u_i$ of a principal part give the cocycle
$u_j-u_i$.

\Needspace{24\baselineskip}
\begin{theorem}[Recovery of extension classes]
\label{II:thm:extension-recovery}
Let $X$ be a smooth compact complex curve of genus $g\geq2$, and
let $L,M$ be acyclic line bundles on $X$ with $\cub_L\ne\cub_M$.
For extensions with their endpoint identifications retained, the canonical bracket gives a
$\C$-linear injection
\begin{equation}\label{II:eq:extension-injection}
 H^1(X,L\otimes M^{-1})\lhook\joinrel\longrightarrow
 H^0(X,L\otimes M^{-1}\otimes K^7),\qquad
 \xi\longmapsto c_\xi.
\end{equation}
The map is independent of the choice of theta characteristic.
In particular, $c_\xi=0$ if and only if the marked extension splits.
\end{theorem}

\begin{proposition}[Recovery by polar integration]
\label{II:prop:polar-recovery}
Under the hypotheses of Theorem~\ref{II:thm:extension-recovery}, set
$T=LM^{-1}$, $d_{LM}=\cub_L-\cub_M$, and $Z=(d_{LM})_0$.
Use the fixed connection~\eqref{II:eq:fixed-T-connection} on $T$
and the connecting map~\eqref{II:eq:extension-principal-sequence}.
Form the meromorphic
$T$-valued one-form
\begin{equation}\label{II:eq:recovery-one-form}
 \omega_\xi=\frac{c_\xi}{6d_{LM}^2}.
\end{equation}
At a zero of $d_{LM}$ of order $m$, this form has a pole of order
at most $m+1$ and zero residue in a horizontal frame for $\nabla^T$.
Integrating its polar terms gives a principal part of order at
most $m$. These principal parts define
$P_\nabla(\omega_\xi)\in H^0(X,\mathcal P_Z(T))$, and
\begin{equation}\label{II:eq:extension-recovery-formula}
 \boxed{\ \partial_Z P_\nabla\!
   \left(\frac{c_\xi}{6d_{LM}^2}\right)=\xi.\ }
\end{equation}
\end{proposition}

The proof rests on a local observation. A holomorphic splitting
makes the cubic upper triangular. Divide its upper-right entry $b_i$
by $d_{LM}$. The resulting meromorphic section $t_i=b_i/d_{LM}$
satisfies $t_j-t_i=\varphi_{ij}$, so its principal parts determine
the extension class. We will show that the bracket determines these
principal parts. After constructing the inverse, we verify that the
bracket respects addition and scalar multiplication of classes.

\noindent\textit{Proof of Theorem~\ref{II:thm:extension-recovery} and
Proposition~\ref{II:prop:polar-recovery}.}
The next three subsections establish the local identity, construct
the polar inverse, and prove linearity. The final proof paragraph
assembles these steps.

\subsection{The local identity}
\label{II:subsec:extension-local-identity}

Choose holomorphic splittings $s_i:M|_{U_i}\to E_\xi|_{U_i}$.
Twisting by $\vartheta^{-1}$ gives splittings
$\widehat s_i:V_M|_{U_i}\to V_{E_\xi}|_{U_i}$; write
$\widehat i:V_L\to V_{E_\xi}$ for the twisted inclusion.
The splittings $\widehat s_i$ need not be horizontal for the
canonical connections. We must therefore keep both the off-diagonal
cubic entry and the off-diagonal connection entry; the latter
measures this failure of horizontality. There are unique holomorphic
sections
\[
 b_i\in H^0(U_i,T\otimes K^3),\qquad
 \nu_i\in H^0(U_i,T\otimes K)
\]
defined by
\begin{equation}\label{II:eq:off-diagonal-data}
 \begin{aligned}
 (\widehat i\otimes1)b_i
   &=\cub_{E_\xi}\widehat s_i
                   -(\widehat s_i\otimes1)\cub_M,\\
 (\widehat i\otimes1)\nu_i
   &=\nabla^{V_{E_\xi}}\circ\widehat s_i
                   -(\widehat s_i\otimes1)\circ\nabla^{V_M}.
 \end{aligned}
\end{equation}
In the second line, the derivative terms cancel, so the difference
is $\OO_X$-linear. Both right-hand sides take values in $V_L$
with the indicated tensor factor, by naturality for the quotient.
Thus $b_i$ is the off-diagonal cubic, while $\nu_i$ measures the
failure of the chosen splitting to be horizontal.

To compute the bracket, take a coordinate $z$ on $U_i$ and
holomorphic frames of $V_L,V_M$. Write
\[
 \begin{gathered}
 \cub_L=a(z)(dz)^3,\qquad \cub_M=d(z)(dz)^3,\\
 d_{LM}=d_{LM,z}(dz)^3,\qquad d_{LM,z}=a-d,\\
 b_i=b_{i,z}(dz)^3,\qquad \nu_i=q_i\,dz.
 \end{gathered}
\]
The coefficients in the frame determined by $\widehat s_i$ are
\[
 F_i=\begin{pmatrix}a&b_{i,z}\\0&d\end{pmatrix},\qquad
 \nabla^{V_{E_\xi}}_z=\partial_z+
       \begin{pmatrix}{\mathsf A}_L&q_i\\0&{\mathsf A}_M\end{pmatrix},\qquad
 \nabla^T_z=\partial_z+{\mathsf A}_L-{\mathsf A}_M.
\]
Here ${\mathsf A}_L\,dz$ and ${\mathsf A}_M\,dz$ are the matrices of the fixed line
connections. In this frame
\[
 \nabla^{\mathrm{ad}}_{E_\xi,z}F_i
   =\begin{pmatrix}
       a'&\nabla^T_z b_{i,z}-d_{LM,z}q_i\\
       0&d'
     \end{pmatrix}.
\]
Multiplying this matrix and $F_i$ in the two orders gives
\begin{equation}\label{II:eq:extension-local-bracket}
 c_\xi=6\bigl(
     d_{LM,z}\nabla^T_z b_{i,z}
       -b_{i,z}d_{LM,z}'-d_{LM,z}^2q_i\bigr)(dz)^7.
\end{equation}
Assume $d_{LM}\not\equiv0$ and put
\begin{equation}\label{II:eq:meromorphic-splitting-coefficient}
 t_i=\frac{b_i}{d_{LM}}\in H^0(U_i,T(Z)).
\end{equation}
The quotient is an actual meromorphic section of $T$: its two
factors of $K^3$ cancel. Dividing
\eqref{II:eq:extension-local-bracket} by $6d_{LM}^2$ gives
the identity of $T$-valued one-forms
\begin{equation}\label{II:eq:extension-basic-identity}
 \boxed{\ \omega_\xi=\nabla^T t_i-\nu_i.\ }
\end{equation}

The changes between splittings explain what this identity remembers.
By \eqref{II:eq:extension-cocycle},
$s_j-s_i=i\varphi_{ij}$, where $[\varphi_{ij}]=\xi$.
Using \eqref{II:eq:off-diagonal-data} and the naturality of the
endpoint connections and cubics gives
\begin{equation}\label{II:eq:extension-data-transitions}
 b_j-b_i=d_{LM}\varphi_{ij},\qquad
 t_j-t_i=\varphi_{ij},\qquad
 \nu_j-\nu_i=\nabla^T\varphi_{ij}.
\end{equation}
The first identity is also the splitting change
\eqref{II:eq:cubic-splitting-change}. The last two show directly that
the right-hand side of \eqref{II:eq:extension-basic-identity} agrees
on overlaps.

\subsection{Recovering the cocycle from polar parts}
\label{II:subsec:extension-polar-recovery}

Let $p$ be a zero of $d_{LM}$ of multiplicity $m$. Choose a small
disc about $p$, a coordinate $z$ with $z(p)=0$, and a nonvanishing
horizontal frame $e$ of $T$ on the whole disc:
$\nabla^T e=0$. Such a frame exists because $\nabla^T$ is
holomorphic there. At a simple zero of $d_{LM}$, a possible polar
term of $\omega_\xi$ is $u_{-2}e\,z^{-2}dz$. Its primitive has
principal part $-u_{-2}e/z$; the integration constant is
holomorphic and has no effect. A multiple zero requires the same
calculation for finitely many negative powers.

In general, $t_i$ has a pole of order at most $m$ and $\nu_i$ is
holomorphic. Thus \eqref{II:eq:extension-basic-identity} gives
\begin{equation}\label{II:eq:extension-Laurent-recovery}
 \omega_\xi=e\left(\sum_{n=-m-1}^{\infty}u_nz^n\right)dz,
 \qquad u_{-1}=0.
\end{equation}
The residue is zero because differentiating a meromorphic function
produces no term in $z^{-1}dz$. The polar part of $t_i$ is therefore
\begin{equation}\label{II:eq:extension-polar-primitive}
 P_\nabla(\omega_\xi)|_p
   =\left[e\sum_{n=-m-1}^{-2}\frac{u_n}{n+1}z^{n+1}\right]
       \in\mathcal P_Z(T)_p.
\end{equation}
The brackets here mean the class modulo holomorphic sections.
No integration constant is needed.

Formula~\eqref{II:eq:extension-polar-primitive} is independent of the
coordinate and horizontal frame. Intrinsically it is the unique
principal part with a meromorphic lift $v$ such that
$\nabla^T v-\omega_\xi$ is holomorphic. For uniqueness, the difference
of two such lifts has holomorphic covariant derivative. In a
horizontal frame a pole of order $r\geq1$ would give a derivative
with a nonzero pole of order $r+1$, so the difference has no pole.
This also explains why residues in
\eqref{II:eq:extension-Laurent-recovery} are taken in horizontal frames.

We can now recover the class. The local identity gives the pole
bound and residue vanishing in Theorem~\ref{II:thm:extension-recovery}. The principal parts computed in
\eqref{II:eq:extension-polar-primitive} are precisely those of $t_i$.
By \eqref{II:eq:extension-data-transitions}, $t_j-t_i$ is holomorphic.
The principal parts therefore agree on overlaps and define a section
$P_\nabla(\omega_\xi)$ of $\mathcal P_Z(T)$. Applying
the connecting map in \eqref{II:eq:extension-principal-sequence} gives
\[
 \partial_Z P_\nabla(\omega_\xi)
      =[t_j-t_i]=[\varphi_{ij}]=\xi,
\]
with exactly the sign fixed in \eqref{II:eq:extension-cocycle}.
This proves the recovery formula and injectivity. It uses only
the fixed data $(T,\nabla^T,d_{LM})$ and the given differential
$c_\xi$.

To realize the recovered class as an explicit extension, cover $X$ by $U_0=X\setminus|Z|$ and small discs about
the points of $Z$. Set $v_0=0$ and on each disc choose the
meromorphic representative $v_i$ displayed in that formula.
The differences $\psi_{ij}=v_j-v_i$ are holomorphic on overlaps
and satisfy $\psi_{ij}+\psi_{jk}=\psi_{ik}$. Glue the local copies
of $(L\oplus M)|_{U_i}$ by
\begin{equation}\label{II:eq:recovered-extension-gluing}
 (\ell,m)_j\longmapsto
       \bigl(\ell+\psi_{ij}(m),m\bigr)_i.
\end{equation}
The resulting extension has class $[\psi_{ij}]=\xi$ by
\eqref{II:eq:extension-recovery-formula}. Changing the meromorphic
representatives by holomorphic sections changes this cocycle by
a coboundary and gives an isomorphic extension.

The pole bound in
\eqref{II:eq:extension-Laurent-recovery} implies the useful constraint
\begin{equation}\label{II:eq:extension-vanishing-order}
 \operatorname{ord}_p(c_\xi)\geq m-1
       \qquad\text{if }\operatorname{ord}_p(d_{LM})=m.
\end{equation}
For an input in the image of \eqref{II:eq:extension-injection}, the
pole and residue conditions guarantee that every step of the
recovery formula is defined.

Equivalently, if $\Xi_{L,M}(\xi)=c_\xi$ and
\[
 \Psi(c)=\partial_ZP_\nabla\!\left(\frac{c}{6d_{LM}^2}\right),
 \qquad c\in\operatorname{im}\Xi_{L,M},
\]
then the construction proves $\Psi\circ\Xi_{L,M}=\Id$.
The indicated domain of $\Psi$ retains the pole and residue
conditions required for its definition.

\subsection{Linearity in the extension class}
\label{II:subsec:extension-linearity}

It remains to prove that $\xi\mapsto c_\xi$ is linear. The local
commutator formula by itself does not prove this, because both
the cubic and its connection depend on the extension. Instead we
use their compatibility with bundle maps. By
Section~\ref{II:subsec:kernel-modular-forms}, for a holomorphic map
$u:E\to F$ between acyclic bundles,
\begin{equation}\label{II:eq:bracket-functoriality}
 (u\otimes1)\mathcal C_E=\mathcal C_Fu.
\end{equation}
On a line bundle the self-bracket is zero. On a direct sum of line
bundles it is also zero: the kernel, connection, and cubic are direct
sums, so the cubic and its covariant derivative are diagonal.

\begin{lemma}\label{II:lem:extension-linearity}
The assignment $\xi\mapsto c_\xi$ is linear for any fixed acyclic
$L,M$, including when $\cub_L=\cub_M$.
\end{lemma}

\begin{proof}
\noindent\textit{1. Pull back the two quotients.} Let $E_1,E_2$ represent $\xi_1,\xi_2$. To add the extension classes,
first identify their quotients by taking the fibre product
\[
 {\mathfrak F}=E_1\times_M E_2
   =\{(e_1,e_2):\pi_1(e_1)=\pi_2(e_2)\}.
\]
The projections $p_r:{\mathfrak F}\to E_r$ form the pullback square
\[
\begin{CD}
 {\mathfrak F} @>{p_2}>> E_2\\
 @V{p_1}VV @VV{\pi_2}V\\
 E_1 @>>{\pi_1}> M.
\end{CD}
\]
It is a rank-three holomorphic bundle in an exact sequence
\[
 0\longrightarrow L\oplus L\xrightarrow{j}{\mathfrak F}
   \xrightarrow{\pi_{\mathfrak F}}M\longrightarrow0.
\]
The cohomology sequence shows that ${\mathfrak F}$ is acyclic. Applying
\eqref{II:eq:bracket-functoriality} to $j$ and $\pi_{\mathfrak F}$ gives
$\mathcal C_{\mathfrak F}j=0$ and $(\pi_{\mathfrak F}\otimes1)\mathcal C_{\mathfrak F}=0$.
Consequently the bracket has a unique factorization
\[
 \mathcal C_{\mathfrak F}=(j\otimes1)c_{\mathfrak F}\pi_{\mathfrak F},\qquad
 c_{\mathfrak F}:M\longrightarrow(L\oplus L)\otimes K^7.
\]
Let $\operatorname{pr}_r:L\oplus L\to L$ be projection to the
$r$th factor, for $r=1,2$. The relations
$p_rj=i_r\operatorname{pr}_r$ and $\pi_rp_r=\pi_{\mathfrak F}$ turn
naturality for $p_r$ into
\[
 (i_r\otimes1)(\operatorname{pr}_r\otimes1)c_{\mathfrak F}\pi_{\mathfrak F}
       =(i_r\otimes1)c_{\xi_r}\pi_{\mathfrak F}.
\]
Since $i_r$ is injective and $\pi_{\mathfrak F}$ is surjective, this implies
\begin{equation}\label{II:eq:bracket-fibre-product}
 c_{\mathfrak F}=\begin{pmatrix}c_{\xi_1}\\c_{\xi_2}\end{pmatrix}.
\end{equation}

\noindent\textit{2. Push out by addition.} Next add the two copies of $L$. Let
$a:L\oplus L\to L$ be $a(\ell_1,\ell_2)=\ell_1+\ell_2$ and
form the pushout
\[
 E_+=({\mathfrak F}\oplus L)/
     \{(j(\ell_1,\ell_2),-a(\ell_1,\ell_2))\}.
\]
It is an acyclic extension of $M$ by $L$. Indeed, its subbundle
and quotient are $L$ and $M$. Writing $i_+:L\to E_+$ and $\pi_+:E_+\to M$ for its inclusion
and quotient, the map $u(p)=[(p,0)]$ gives the diagram
\[
\begin{CD}
 0 @>>> L\oplus L @>{j}>> {\mathfrak F} @>{\pi_{\mathfrak F}}>> M @>>> 0\\
 @. @V{a}VV @V{u}VV @V{\Id_M}VV @.\\
 0 @>>> L @>{i_+}>> E_+ @>{\pi_+}>> M @>>> 0.
\end{CD}
\]
Both rows are exact. If the two original extensions have cocycles
$\varphi^{(1)}_{ij},\varphi^{(2)}_{ij}$, their local splittings
give the cocycle $\varphi^{(1)}_{ij}+\varphi^{(2)}_{ij}$ for
$E_+$. Thus $E_+$ represents $\xi_1+\xi_2$.
Applying \eqref{II:eq:bracket-functoriality} to $u$ and using
\eqref{II:eq:bracket-fibre-product} now gives
\[
 c_{\xi_1+\xi_2}=(a\otimes1)c_{\mathfrak F}=c_{\xi_1}+c_{\xi_2}.
\]

\noindent\textit{3. Scalar multiplication.} For $\lambda\in\C$, the pushout
\[
 E_\lambda=(E_\xi\oplus L)/
            \{(i(\ell),-\lambda\ell):\ell\in L\}
\]
has cocycle $\lambda\varphi_{ij}$ and hence class $\lambda\xi$.
Its natural map from $E_\xi$ induces $\lambda\Id_L$ on $L$ and
$\Id_M$ on $M$. Naturality gives $c_{\lambda\xi}=\lambda c_\xi$.
This also applies at $\lambda=0$, when the pushout is split.
Every bundle occurring in these maps is acyclic, so each application
of \eqref{II:eq:bracket-functoriality} is valid.
\end{proof}

The fixed endpoint identifications in
\eqref{II:eq:line-extension} are part of the assertion. Changing
the inclusion or quotient identification can rescale the extension
class and its differential together.

\begin{proof}[Completion of the recovery argument]
The local calculation and polar construction prove the recovery
formula \eqref{II:eq:extension-recovery-formula} and injectivity.
Lemma~\ref{II:lem:extension-linearity} proves linearity.
Independence of $\vartheta$ follows from that of
$\mathcal C_{E_\xi}$ in \eqref{II:eq:canonical-bracket} and the fixed
factorization \eqref{II:eq:extension-bracket-factor}.
This proves Theorem~\ref{II:thm:extension-recovery} and
Proposition~\ref{II:prop:polar-recovery}.
\end{proof}

\subsection{Splitting and scalar characteristic invariants}
\label{II:subsec:extension-consequences}

\begin{corollary}\label{II:cor:extension-splitting}
Under the hypothesis $\cub_L\ne\cub_M$, the following conditions
are equivalent:
\[
 \xi=0,\qquad
 E_\xi\text{ splits as the extension \eqref{II:eq:line-extension}},
 \qquad c_\xi=0,\qquad \mathcal C_{E_\xi}=0.
\]
For fixed endpoint identifications, two extensions have the same
differential $c_\xi$ if and only if they have the same extension
class.
\end{corollary}

\begin{proof}
The first two conditions are equivalent by
\eqref{II:eq:extension-class}. Theorem~\ref{II:thm:extension-recovery}
identifies the first and third, and the factorization
\eqref{II:eq:extension-bracket-factor} identifies the last two.
The last assertion is injectivity.
\end{proof}

The invariant retains information despite the vanishing of its
scalar characteristic invariants. The composition
$\mathcal C_{E_\xi}^2:E_\xi\to E_\xi\otimes K^{14}$ is zero,
because $\pi i=0$ in \eqref{II:eq:extension-bracket-factor}.
Locally, writing $\mathcal C_{E_\xi}=C_z(dz)^7$ in an adapted
frame gives
\begin{equation}\label{II:eq:extension-nilpotence}
 C_z=\begin{pmatrix}0&c_z\\0&0\end{pmatrix},\qquad
 C_z^2=0,\qquad \tr C_z=\det C_z=0,\qquad
 \det(\lambda\Id-C_z)=\lambda^2.
\end{equation}
For every nonzero extension class, the theorem gives $C_z\not\equiv0$,
although all these scalar expressions agree with those of the split
extension.

The cubic and its bracket also exhibit the role of ordered
multiplication. In the notation of
\eqref{II:eq:extension-local-bracket},
\begin{equation}\label{II:eq:extension-noncommutativity}
 [F_i,C_z]=d_{LM,z} C_z.
\end{equation}
At a point where $d_{LM,z}c_z\ne0$ they do not commute. Indeed,
$\Id,F_i,C_z$ are then linearly independent and span the algebra
of upper triangular $2$-by-$2$ matrices in this frame: the first
two supply arbitrary diagonal entries, and the third supplies
the remaining upper-right entry. The scalar characteristic
expressions in \eqref{II:eq:extension-nilpotence} discard precisely
that entry.

The distinct-cubic hypothesis has a concrete role. If
$\cub_L=\cub_M$, then $d_{LM,z}=d_{LM,z}'=0$, and
\eqref{II:eq:extension-local-bracket} gives $c_\xi=0$ for every
extension class. The first self-bracket of the cubic therefore
requires a nonzero difference of endpoint cubics to detect these
extensions.

One useful specialization follows immediately from duality.
If $L$ is acyclic and $\cub_L\not\equiv0$, put $M=K\otimes L^{-1}$.
Then $M$ is acyclic, $d_{LM}=2\cub_L$, and
\eqref{II:eq:extension-injection} becomes
\begin{equation}\label{II:eq:fixed-determinant-injection}
 H^1(X,L^2\otimes K^{-1})
       \lhook\joinrel\longrightarrow H^0(X,L^2\otimes K^6).
\end{equation}
These are extensions with determinant $K$; their nonzero classes
give nonzero nilpotent endomorphism-valued seventh differentials.

\subsection{The modular forms and their fixed multiplier}
\label{II:subsec:extension-modular-forms}

We now return to the modular formulation of
Section~\ref{II:sec:modular-calculus}. Let $Y=\Gamma\backslash\mathbb H$
and $X=\bar Y$ as in Proposition~\ref{II:prop:canonical-modular-forms};
this includes the compact case $Y=X$. Write
$\pi_X:\mathbb H\to Y\subset X$ for the covering map.
Let $\rho_L,\rho_M:\Gamma\to\C^\times$ be the descent characters
of the canonical connections on $V_L,V_M$. Their ratio
\begin{equation}\label{II:eq:fixed-extension-character}
 \chi=\rho_L\rho_M^{-1}
\end{equation}
is the character of $(T,\nabla^T)$. In particular, it is fixed
as $\xi$ varies and is independent of $\vartheta$.

Fix horizontal frames of $\pi_X^*V_L$ and $\pi_X^*V_M$. They
determine a horizontal frame $e_T$ of $\pi_X^*T$. Define
$f_\xi:\mathbb H\to\C$ by
\begin{equation}\label{II:eq:extension-modular-coefficient}
 \pi_X^*c_\xi=\kappa^{-1}f_\xi(\tau)e_T(d\tau)^7,
 \qquad \kappa=(2\pi i)^{-1}.
\end{equation}
Choose a horizontal lift of the frame of $\pi_X^*V_M$ to
$\pi_X^*V_{E_\xi}$. Such a lift exists on the simply connected
upper half-plane. In the resulting adapted horizontal frame,
the matrix modular form from
\eqref{II:eq:kernel-modular-coefficients} is
\begin{equation}\label{II:eq:extension-matrix-modular-form}
 \Phi_{E_\xi}=[F_{E_\xi},F_{E_\xi}]_1^0
      =6[F_{E_\xi},DF_{E_\xi}]
      =f_\xi(\tau)\begin{pmatrix}0&1\\0&0\end{pmatrix}.
\end{equation}
This normalization follows from
\eqref{II:eq:canonical-modular-normalization}. Changing the horizontal
lift adds a constant multiple of the frame of $\pi_X^*V_L$;
it leaves the upper-right coefficient in
\eqref{II:eq:extension-matrix-modular-form} unchanged.

\begin{corollary}\label{II:cor:extension-modular-injection}
If $\cub_L\ne\cub_M$, then $\xi\mapsto f_\xi$ is a linear
injection into the holomorphic weight-$14$ forms with the fixed
character $\chi$:
\begin{equation}\label{II:eq:extension-weight-fourteen}
 f_\xi(\gamma\tau)=j(\gamma,\tau)^{14}
                         \chi(\gamma)f_\xi(\tau).
\end{equation}
When $Y$ is noncompact, $\chi$ is trivial on parabolic stabilizers
and, at every cusp,
\begin{equation}\label{II:eq:extension-cusp-vanishing}
 f_{\xi,\sigma}=O(q^7).
\end{equation}
Thus the construction gives a $(g-1)$-dimensional space of cusp
forms in that case. Their matrix realizations
\eqref{II:eq:extension-matrix-modular-form} have zero trace and
determinant and vanish identically exactly for the split extension.
\end{corollary}

\begin{proof}
The tensor law for $c_\xi$ gives
\eqref{II:eq:extension-weight-fourteen}. It also follows from the
matrix law in \eqref{II:eq:canonical-modular-laws}, since in an
adapted frame
\[
 \rho_{E_\xi}(\gamma)=
 \begin{pmatrix}\rho_L(\gamma)&r_\xi(\gamma)\\
                 0&\rho_M(\gamma)\end{pmatrix},
 \qquad
 \rho_{E_\xi}(\gamma)
 \begin{pmatrix}0&1\\0&0\end{pmatrix}
 \rho_{E_\xi}(\gamma)^{-1}
   =\chi(\gamma)\begin{pmatrix}0&1\\0&0\end{pmatrix}.
\]
Thus the varying off-diagonal part of $\rho_{E_\xi}$ does not
change the character on $\Hom(V_M,V_L)$.
Linearity and injectivity follow from
Theorem~\ref{II:thm:extension-recovery} and the fixed normalization
\eqref{II:eq:extension-modular-coefficient}.
The cusp assertions follow from holomorphicity of $\nabla^T$ and
$c_\xi$ on $X$, as in the proof of
Proposition~\ref{II:prop:canonical-modular-forms}.

Finally, $\deg L=\deg M=g-1$ by acyclicity, so $\deg T=0$.
The inequality $\cub_L\ne\cub_M$ implies $L\not\simeq M$.
A nonzero section of a degree-zero line bundle has no zeros and
trivializes it; hence $H^0(X,T)=0$. Riemann--Roch gives
$h^1(X,T)=g-1$, which is the dimension of the image.
The remaining assertions are
Corollary~\ref{II:cor:extension-splitting} and
\eqref{II:eq:extension-nilpotence}.
\end{proof}

The modular coefficient $f_\xi$, together with the fixed data
$(T,\nabla^T,d_{LM})$, therefore determines the extension class.
Pullback identifies it with $c_\xi$ by
\eqref{II:eq:extension-modular-coefficient}, and the local formula
\eqref{II:eq:extension-polar-primitive} recovers the class. The
ordered bracket of Section~\ref{II:sec:rankin-cohen}, applied to the
canonical cubic of Section~\ref{II:sec:kernel}, has produced a modular
form with this explicit geometric meaning.

\section{Higher invariants and the limits of cubic detection}
\label{II:sec:higher-detection}

Suppose that $\cub_L=\cub_M$. There are two cases to distinguish.
If $L\not\simeq M$, a differential of another degree may have
different values on the two lines, and the recovery argument of
Section~\ref{II:sec:extension-recovery} can still apply. If $L=M$,
every endpoint value agrees. In that case the information must
come from the nilpotent part of the invariant on the middle bundle.

We construct one higher differential $\mathcal P_j(E)$ in each
degree $j\geq2$, with $\mathcal P_3(E)=\cub_E$. We then treat
the two cases in turn. A pole estimate for the kernel will show
that degrees $2\leq j\leq4g-1$ suffice both to separate distinct
endpoints and to detect self-extensions. The final result explains
what the entire hierarchy retains in any rank: its jets generate
the associative monodromy algebra of the canonical connection.

\subsection{A fixed hierarchy of higher Schwarzians}
\label{II:subsec:fixed-hierarchy}

Let $X$ remain a smooth compact complex curve of genus $g\geq2$.
To compare higher coefficients of different kernels, we must use
the same construction for every bundle. We make two choices.
A holomorphic projective connection $p_0$ supplies compatible
coordinates for the kernel expansion. The higher Schwarzian of
degree $j$ also depends on the order $N$ of the operator to which
Part~\ref{part:I} is applied; we will use $N=j$. These choices specify the
individual differentials. The algebra generated by the whole
hierarchy will be independent of them.

Choose $p_0$ in the normalization
\begin{equation}\label{II:eq:background-projective-connection}
 p_{0,t}=\phi'^2(p_{0,z}\circ\phi)+\frac16\Sch(\phi)
             \qquad(z=\phi(t)).
\end{equation}
Equivalently, choose an atlas of projective coordinates: its
transition functions are M\"obius maps and $p_0=0$ in its charts.
The uniformization of the compact curve supplies one such choice.
Throughout this section, all bundles are compared using the same
$p_0$. When $X=\bar Y$ is a compactified modular curve, this
background is chosen holomorphic on $X$, including at its cusps.

We need a global operator in each order whose coefficients retain
the regular coefficients of the kernel. The Cauchy germ in a
projective coordinate provides it. With $e_\vartheta^2=dz$ a local
frame of $\vartheta$, set
\begin{equation}\label{II:eq:projective-Cauchy-kernel}
 \mathsf k_{p_0}(z,w)=\frac{e_\vartheta(z)e_\vartheta(w)}{w-z}.
\end{equation}
This defines a meromorphic germ along the diagonal with values in
$\vartheta\boxtimes\vartheta$. Indeed, for a M\"obius change of
coordinate,
\[
 \frac{\sqrt{\phi'(t)\phi'(s)}}{\phi(s)-\phi(t)}
       =\frac1{s-t},
\]
where the square roots are chosen compatibly near $s=t$.
Thus the local germs in \eqref{II:eq:projective-Cauchy-kernel} agree.

For an acyclic bundle $E$, multiply its kernel by
$\mathsf k_{p_0}^{\,N}$ and take the residue in the second
variable:
\begin{equation}\label{II:eq:higher-residue-operator}
 (L^{p_0}_{E,N}u)(z)
    =N!\res_{w=z}
       \bigl(s_E(z,w)\mathsf k_{p_0}(z,w)^N u(w)\bigr).
\end{equation}
This is the usual kernel interpretation of differential operators;
see \cite[Sections 2.2 and 2.3.8]{BB}.
The formula specifies its normalization here. Its source and target are
\begin{equation}\label{II:eq:higher-residue-type}
 L^{p_0}_{E,N}:V_E\otimes\vartheta^{1-N}
                   \longrightarrow V_E\otimes\vartheta^{1+N},
 \qquad V_E=E\otimes\vartheta^{-1}.
\end{equation}
In fact, the kernel in \eqref{II:eq:higher-residue-operator} has
the half-differential factors
$e_\vartheta(z)^{N+1}e_\vartheta(w)^{N+1}$.
Multiplication by an input with factor $e_\vartheta(w)^{1-N}$
leaves $dw$ in the second variable and the required factor
$e_\vartheta(z)^{N+1}$ in the first.

In a horizontal frame for $\nabla^{V_E}$, write the regularized
kernel as in Section~\ref{II:sec:kernel}:
\[
 \mathsf H_E(z,w)=\Id+\sum_{j\geq2}B_j(z)(w-z)^j.
\]
The coefficient of $(w-z)^{-1}dw$ in
\eqref{II:eq:higher-residue-operator} gives
\begin{equation}\label{II:eq:higher-residue-coefficients}
 L^{p_0}_{E,N}
   =\partial_z^N+
      \sum_{j=2}^N\frac{N!}{(N-j)!}B_j\partial_z^{N-j},
 \qquad I_j(L^{p_0}_{E,N})=j!B_j.
\end{equation}
Thus the operator is monic and its coefficient connection is
the canonical $\nabla^{V_E}$. The residue construction proves
that these local expressions define the global operator
\eqref{II:eq:higher-residue-type}.

For $3\leq j\leq N$, let
$\mathcal W_{j,N}(E)=W_j(L^{p_0}_{E,N})$, viewed as an
$\End(E)$-valued $j$-differential by
Section~\ref{II:sec:modular-operators}. The polynomial defining
$W_j$ depends on $N$, so $\mathcal W_{j,N}$ for $N>j$ may differ
from $\mathcal W_{j,j}$. We select the latter in each degree:
\begin{equation}\label{II:eq:chosen-higher-hierarchy}
 \mathcal P_j(E)=\mathcal W_{j,j}(E)\quad(j\geq3),\qquad
 \mathcal P_2(E)=2B_2(dz)^2
       \quad\text{in projective coordinates for }p_0.
\end{equation}
In an arbitrary coordinate the last definition reads
$\mathcal P_2(E)=(2B_2-p_{0,z}\Id)(dz)^2$, by
\eqref{II:eq:B2-coordinate}. Each $\mathcal P_j(E)$ is therefore
a holomorphic section of $\End(E)\otimes K^j$.
These assignments are natural under holomorphic maps of acyclic
bundles: in horizontal frames such maps are constant and intertwine
every $B_j$, hence every polynomial defining the invariant.
They are independent of $\vartheta$ by the same locally constant
twist argument as in Proposition~\ref{II:prop:canonical-cubic}.

The first formulas, in the projective coordinates just chosen, are
\begin{align}
 \mathcal P_3(E)&=(6B_3-3B_2')(dz)^3=\cub_E,
       \label{II:eq:selected-cubic}\\
 \mathcal P_4(E)&=
  \left(24B_4-12B_3'+\frac{12}{5}B_2''
                      -\frac{324}{25}B_2^2\right)(dz)^4.
       \label{II:eq:selected-quartic}
\end{align}
The relation with the kernel coefficients is triangular:

\begin{equation}\label{II:eq:hierarchy-triangular}
 \mathcal P_j(E)=
 \bigl(j!B_j+Q_j(B_2,\ldots,B_{j-1};
                         \partial_zB_2,\ldots)\bigr)(dz)^j,
\end{equation}
where $Q_j$ is the fixed noncommutative differential polynomial
from Part~\ref{part:I} at order $N=j$.

The residue operators and the reconstructed operator of
\eqref{II:eq:global-kernel-operator} serve different purposes.
The reconstruction realizes the kernel's affine quadratic datum
and cubic intrinsically. The residue construction gives a family
in every order with the simple coefficient identity $I_j=j!B_j$.
At $N=3$ its quadratic datum is
$p_0\Id+\frac14\mathcal P_2(E)$, while that of
$\mathcal L_{E,3}$ is $p_0\Id+\mathcal P_2(E)$; both have cubic
$\cub_E$. Thus the cubic is independent of $p_0$, although the
other selected $\mathcal P_j$ may depend on it.
Subsection~\ref{II:subsec:hierarchy-monodromy} will show that the
algebra generated by their jets is independent of $p_0$ as well.

\subsection{Recovery in any separating degree}
\label{II:subsec:higher-degree-recovery}

The argument of Section~\ref{II:sec:extension-recovery} applies to
any natural differential whose endpoint values differ.

\begin{theorem}\label{II:thm:higher-degree-recovery}
Let $m\geq2$. Suppose that every acyclic bundle $E$ is assigned
a holomorphic
$P_E\in H^0(X,\End(E)\otimes K^m)$, naturally under holomorphic
maps. Use the canonical kernel connections on $V_E$.
For acyclic lines $L,M$, put
\[
 T=L\otimes M^{-1},\qquad
 \delta_P=P_L-P_M\in H^0(X,K^m).
\]
If $\delta_P\not\equiv0$, the first self-bracket of $P_{E_\xi}$
gives a linear injection
\begin{equation}\label{II:eq:higher-extension-injection}
 H^1(X,T)\lhook\joinrel\longrightarrow
 H^0(X,T\otimes K^{2m+1}),\qquad
 \xi\longmapsto c^P_\xi.
\end{equation}
Writing $Z_P=(\delta_P)_0$, the recovery formula is
\begin{equation}\label{II:eq:higher-recovery-formula}
 \partial_{Z_P}P_\nabla
       \left(\frac{c^P_\xi}{2m\,\delta_P^2}\right)=\xi.
\end{equation}
Here $\nabla^T$ is \eqref{II:eq:fixed-T-connection}, and the
principal parts and connecting map have the conventions of
Section~\ref{II:subsec:extension-polar-recovery}.
\end{theorem}

\begin{proof}
Locally write $P_E=p_z(dz)^m$. Lemma~\ref{II:lem:intrinsic-self-bracket}
defines the intrinsic bracket
\[
 2m[p_z,\nabla_z^{\mathrm{ad}}p_z](dz)^{2m+1}.
\]
It vanishes on the two line bundles, so it factors through
$T\otimes K^{2m+1}$ as in \eqref{II:eq:extension-bracket-factor}.
The fibre product and pushouts in
Lemma~\ref{II:lem:extension-linearity} show that this factor is
linear in $\xi$.

To check recovery, use a local splitting and let $b_i$ be the
off-diagonal part of $P_E$; it is now a section of $T\otimes K^m$.
Keep the connection one-form $\nu_i=q_i\,dz$ from
\eqref{II:eq:off-diagonal-data}. Writing
$\delta_P=\delta_z(dz)^m$ and $b_i=b_{i,z}(dz)^m$,
the matrix multiplication in \eqref{II:eq:extension-local-bracket}
gives
\[
 c^P_\xi=2m\bigl(\delta_z\nabla^T_zb_{i,z}
       -b_{i,z}\delta_z'-\delta_z^2q_i\bigr)(dz)^{2m+1}.
\]
Consequently
\begin{equation}\label{II:eq:higher-basic-identity}
 \frac{c^P_\xi}{2m\,\delta_P^2}
       =\nabla^T(b_i/\delta_P)-\nu_i,\qquad
 \frac{b_j}{\delta_P}-\frac{b_i}{\delta_P}=\varphi_{ij}.
\end{equation}
The first identity recovers the polar parts of $b_i/\delta_P$,
and the second identifies their connecting class with $\xi$.
This proves \eqref{II:eq:higher-recovery-formula} and injectivity.
\end{proof}

Apply this theorem with $P_E=\mathcal P_m(E)$.
On a compactified modular curve, the resulting matrix form has
weight $4m+2$ and cusp order at least $2m+1$:
\begin{equation}\label{II:eq:higher-modular-output}
 [F_m,F_m]_1^0=2m[F_m,DF_m],\qquad
 [F_m,F_m]_1^0(d\tau)^{2m+1}
       =\kappa\,\pi_X^*\mathcal C(\mathcal P_m(E)).
\end{equation}
The scalar upper-right coefficient has the fixed character
$\rho_L\rho_M^{-1}$, by the argument of
Corollary~\ref{II:cor:extension-modular-injection}.
These are pullbacks of the intrinsic differentials:
$\tau$ need not be a projective coordinate for the chosen $p_0$.

Degree two is available as well. Replacing $p_0$ changes
$\mathcal P_2(E)$ by a scalar quadratic differential times $\Id$,
which has no effect on its first self-bracket. Thus a nonzero
difference $\mathcal P_2(L)-\mathcal P_2(M)$ gives a canonical
weight-$10$ construction. For the fixed-determinant pair
$M=K\otimes L^{-1}$, the quadratic coefficients agree by the
duality computation of Section~\ref{II:subsec:kernel-extensions}.
The cubic is then the first possible separating degree.

\subsection{An effective degree bound for separating lines}
\label{II:subsec:effective-line-separation}

We now find a degree in which two distinct lines are separated.
The same estimate will also handle self-extensions. In each case,
the kernel gives a meromorphic comparison map with a controlled
pole divisor. Agreement of its first coefficients makes the
covariant derivative vanish to high order at one point. A
zero-and-pole count then forces that derivative to be zero.

Fix $y\in X$ and trivialize the fixed factors at $y$ in the
line kernel $s_M(\,\cdot\,,y)$. It becomes a nonzero meromorphic
section of $M$ with a simple pole at $y$ and no other poles.
Since $\deg M=g-1$, its zero divisor $Z_M(y)$ satisfies
\begin{equation}\label{II:eq:kernel-zero-degree}
 \deg Z_M(y)=g,\qquad y\notin\operatorname{supp}Z_M(y).
\end{equation}
Equivalently, it is a section of $M(y)$ that is nonzero at $y$.

The following estimate isolates the common step in the two
arguments. Here $Z_{\mathrm{red}}$ counts each point of a divisor
once.

\Needspace{14\baselineskip}
\begin{lemma}[The pole estimate]\label{II:lem:kernel-pole-estimate}
Let $T$ be a degree-zero line bundle with a holomorphic connection,
and let $v$ be a meromorphic section with poles bounded by
$Z_M(y)$. Then
\begin{equation}\label{II:eq:pole-count-bound}
 \begin{aligned}
 \nabla^T v&\in
 H^0\!\left(X,T\otimes K
       \bigl(Z_M(y)+Z_M(y)_{\mathrm{red}}\bigr)\right),\\
 \deg\!\left(T\otimes K
       \bigl(Z_M(y)+Z_M(y)_{\mathrm{red}}\bigr)\right)
       &\leq4g-2.
 \end{aligned}
\end{equation}
If $\nabla^T v$ vanishes to order at least $4g-1$ at $y$, it is
identically zero.
\end{lemma}

\begin{proof}
Differentiation increases each pole order by at most one. This
gives the first inclusion in \eqref{II:eq:pole-count-bound}; the
degree bound follows from $\deg K=2g-2$ and
$\deg Z_M(y)_{\mathrm{red}}\leq\deg Z_M(y)=g$.
A nonzero holomorphic section has a zero divisor of degree equal
to that of its line bundle. Since $y$ is outside the pole divisor,
a zero of order $4g-1$ at $y$ exceeds that degree.
\end{proof}

\begin{theorem}[Bounded separation]\label{II:thm:bounded-line-separation}
For any fixed $p_0$, the map
\begin{equation}\label{II:eq:bounded-separating-list}
 L\longmapsto
 \bigl(\mathcal P_2(L),\ldots,\mathcal P_{4g-1}(L)\bigr)
       \in\bigoplus_{j=2}^{4g-1}H^0(X,K^j)
\end{equation}
is injective on isomorphism classes of acyclic holomorphic line
bundles. The entries are global differentials on $X$; equality
means equality as sections. In particular, if $L\not\simeq M$, then
$\mathcal P_m(L)\ne\mathcal P_m(M)$ for some
$2\leq m\leq4g-1$.
\end{theorem}

\begin{proof}
We compare the two kernels by taking their ratio. Equality of
the invariants will make this ratio horizontal to high order;
Lemma~\ref{II:lem:kernel-pole-estimate} will then make it horizontal
everywhere.

Put $n=4g-1$ and suppose that the two lines have the same
invariants through degree $n$. In a common projective coordinate,
\eqref{II:eq:hierarchy-triangular} successively gives
\[
 B_j(L)=B_j(M)\qquad(2\leq j\leq n).
\]
Choose horizontal line frames near $y$, with their values
identified at $y$. The regularized kernels consequently satisfy
\begin{equation}\label{II:eq:equal-kernel-jets}
 \mathsf H_L(z,w)-\mathsf H_M(z,w)=O((w-z)^{n+1}).
\end{equation}

Choose a nonzero element of $K_y$ and suppress that common factor
in $s_L(z,y),s_M(z,y)$, writing the resulting maps as
$k_L(z):L_y\to L_z$ and $k_M(z):M_y\to M_z$.
For the chosen isomorphism $\iota_y:M_y\to L_y$, define
\[
 r_y(z)=k_L(z)\,\iota_y\,k_M(z)^{-1}
                   \in T_z,\qquad T=L\otimes M^{-1}.
\]
This is a meromorphic section of $T$. Its possible poles are
bounded by $Z_M(y)$. It is regular at $y$, because the simple
poles of the two kernels cancel, and $r_y(y)=\iota_y\ne0$.
In the horizontal frame of $T$ determined by $\iota_y$,
\eqref{II:eq:equal-kernel-jets}, with $w=y$, gives
\[
 r_y(z)=1+O((z-y)^{n+1}),\qquad
 \nabla^T r_y=O((z-y)^n)\,dz.
\]
Apply Lemma~\ref{II:lem:kernel-pole-estimate}. Since $n=4g-1$, it
forces $\nabla^T r_y=0$.

A meromorphic horizontal section of a holomorphic flat line
bundle has no poles: in a horizontal frame on a disc its
coefficient is constant. Since $r_y(y)\ne0$, that constant is
nonzero everywhere. Thus $r_y$ trivializes $T$ and gives
$L\simeq M$.
\end{proof}

Combining this theorem with
Theorem~\ref{II:thm:higher-degree-recovery} gives a finite recovery
procedure even when the cubics agree. Compute the endpoint values
in degrees $2,\ldots,4g-1$, choose a degree $m$ with nonzero
difference, and apply \eqref{II:eq:higher-recovery-formula}.
Equivalently, for fixed nonisomorphic $L,M$ there is a linear
injection
\begin{equation}\label{II:eq:bounded-bracket-injection}
 H^1(X,T)\lhook\joinrel\longrightarrow
 \bigoplus_{j=2}^{4g-1}H^0(X,T\otimes K^{2j+1}),\qquad
 \xi\longmapsto(c^{\mathcal P_j}_\xi)_{j=2}^{4g-1}.
\end{equation}
At least one component alone recovers $\xi$. Its modular weight
$4m+2$ is at most $16g-2$.
The bound $4g-1$ is a uniform upper bound; no minimality is asserted.
For $g=2$ it uses degrees $2,\ldots,7$. The pair of line bundles
on $X_0(23)$ in Section~\ref{II:sec:x023} is already separated by the
cubic. A bound for every pair and a separating degree for one pair
answer different questions.

\begin{example}[What the genus-two bound provides]
\label{II:ex:genus-two-bound}
For $g=2$, the bounded hierarchy is
$\mathcal P_2,\mathcal P_3,\ldots,\mathcal P_7$. Riemann--Roch
gives $h^0(K^j)=2j-1$ for $j\geq2$, so the combined target has
dimension $3+5+7+9+11+13=48$. Once the common projective
connection $p_0$ is fixed, every pair of nonisomorphic acyclic
lines differs in at least one of these six components.
The theorem gives a terminating comparison procedure: compute
the six differences and use any nonzero one in
\eqref{II:eq:higher-recovery-formula}. It does not prescribe which
component first differs. For self-extensions, the next theorem
embeds the two-dimensional space $H^1(X,\OO_X)$ into the same
48-dimensional target by its nilpotent coefficients.
\end{example}

\subsection{Self-extensions and nilpotent coefficients}
\label{II:subsec:higher-self-extensions}

Now take a self-extension
\begin{equation}\label{II:eq:higher-self-extension}
 0\longrightarrow L\xrightarrow{i}E_\xi\xrightarrow{\pi}L
      \longrightarrow0,\qquad
 \xi\in H^1(X,\OO_X).
\end{equation}
Every endpoint difference is zero. Naturality instead gives the
unique decomposition
\begin{equation}\label{II:eq:self-extension-coefficients}
 \mathcal P_j(E_\xi)
    =\mathcal P_j(L)\Id_{E_\xi}
          +\beta_{j,L}(\xi)\varepsilon_E,\qquad
 \beta_{j,L}(\xi)\in H^0(X,K^j),\qquad
 \varepsilon_E=i\pi,\quad \varepsilon_E^2=0.
\end{equation}
Indeed, after subtracting $\mathcal P_j(L)\Id$, the invariant
kills the subline and has image in that line. It factors through
$\Hom(L,L)\otimes K^j=K^j$.
The endpoint identifications in \eqref{II:eq:higher-self-extension}
fix $\varepsilon_E$ and the coefficient $\beta_{j,L}(\xi)$.

Each map $\beta_{j,L}$ is linear. To verify this, use the fibre
product and pushouts of Lemma~\ref{II:lem:extension-linearity},
replacing $\mathcal C_E$ by
$\mathcal P_j(E)-\mathcal P_j(L)\Id_E$.
This natural map vanishes on every copy of $L$, and the same
factorizations now end in $L\otimes K^j$. They give
\[
 \beta_{j,L}(\xi+\zeta)=\beta_{j,L}(\xi)+\beta_{j,L}(\zeta),
 \qquad
 \beta_{j,L}(\lambda\xi)=\lambda\beta_{j,L}(\xi).
\]

\begin{theorem}[Bounded detection of self-extensions]
\label{II:thm:bounded-self-extension}
Let $X$ be a smooth compact complex curve of genus $g\geq2$,
let $L$ be an acyclic line bundle on $X$, and fix the common
projective connection $p_0$ used to define the tensors
$\mathcal P_j$ in Section~\ref{II:subsec:fixed-hierarchy}.
Then the map
\begin{equation}\label{II:eq:bounded-self-injection}
 H^1(X,\OO_X)\lhook\joinrel\longrightarrow
       \bigoplus_{j=2}^{4g-1}H^0(X,K^j),\qquad
 \xi\longmapsto
       \bigl(\beta_{2,L}(\xi),\ldots,\beta_{4g-1,L}(\xi)\bigr)
\end{equation}
is a linear injection.
\end{theorem}

\begin{proof}
Linearity was just proved. We use the kernel to extend a splitting
at one point to a meromorphic splitting on $X$. The pole estimate
will show that this splitting is horizontal, hence holomorphic.

Suppose that all the displayed coefficients vanish, and put
$n=4g-1$.
In an adapted horizontal frame the invariants of $E_\xi$ are
scalar through degree $n$. The triangular identities give
\[
 B_j(E_\xi)=B_j(L)\Id\quad(2\leq j\leq n),\qquad
 \mathsf H_{E_\xi}(z,w)
    =\mathsf H_L(z,w)\Id+O((w-z)^{n+1}).
\]

Fix $y\in X$ and choose a linear splitting
$u_y:L_y\to E_{\xi,y}$ of $\pi_y$. Suppress a fixed nonzero
element of $K_y$ in the kernels, obtaining
$k_E(z):E_{\xi,y}\to E_{\xi,z}$ and $k_L(z):L_y\to L_z$.
The formula
\begin{equation}\label{II:eq:kernel-meromorphic-splitting}
 s_y(z)=k_E(z)u_y k_L(z)^{-1}:L_z\longrightarrow E_{\xi,z}
\end{equation}
gives a meromorphic splitting of $\pi$. To check this, kernel
naturality gives $\pi_z k_E(z)=k_L(z)\pi_y$, hence
$\pi_zs_y(z)=\Id$. Its poles are bounded by $Z_L(y)$,
the degree-$g$ zero divisor of $k_L$.
The simple poles of $k_E$ and $k_L$ cancel at $y$, so $s_y$
is regular there.

Twist this splitting by $\vartheta^{-1}$ and let
$\nu_y$ be its off-diagonal connection one-form, as in
\eqref{II:eq:off-diagonal-data}. Now $\Hom(V_L,V_L)=\OO_X$,
and therefore
\begin{equation}\label{II:eq:self-extension-pole-estimate}
 \nu_y\in H^0\!\left(X,K
             \bigl(Z_L(y)+Z_L(y)_{\mathrm{red}}\bigr)\right).
\end{equation}
Choose the adapted horizontal frame near $y$ with initial
splitting $u_y$. The displayed kernel expansion shows that
the coefficient of $s_y$ in this frame is its constant value
at $y$ plus $O((z-y)^{n+1})$. Differentiation gives
$\nu_y=O((z-y)^n)\,dz$.
The degree estimate of Lemma~\ref{II:lem:kernel-pole-estimate}
applies to \eqref{II:eq:self-extension-pole-estimate}. Its line
bundle has degree at most $4g-2$, so $\nu_y=0$.
Thus the twisted splitting is horizontal. In horizontal frames
near each possible pole it is a constant matrix, so all poles
are removable. It is a holomorphic splitting on $X$, and
$\xi=0$.

\end{proof}

\begin{corollary}[Vanishing self-brackets]
\label{II:cor:self-extension-brackets}
For every marked self-extension of an acyclic line bundle, the first
self-bracket of $\mathcal P_j(E_\xi)$ vanishes for every $j\geq2$.
\end{corollary}
\begin{proof}
Naturality makes $\varepsilon_E$ horizontal.
In a coordinate, the coefficient of \eqref{II:eq:self-extension-coefficients}
and its covariant derivative have the forms
$p_j\Id+b_j\varepsilon_E$ and
$p_j'\Id+b_j'\varepsilon_E$. They commute, so their first
self-bracket is zero.
\end{proof}

The coefficients in \eqref{II:eq:bounded-self-injection} also recover
the class by finite-dimensional linear algebra. Choose a basis
$\xi_1,\ldots,\xi_g$ of $H^1(X,\OO_X)$ and bases of the target
spaces. The column vectors
$(\beta_{j,L}(\xi_a))_{j=2}^{4g-1}$ form a matrix of rank $g$.
Any nonzero $g$-by-$g$ minor gives a linear inverse on the image.

There is an equivalent deformation interpretation. Under
$T_L\operatorname{Pic}^{g-1}(X)=H^1(X,\OO_X)$, one has
\begin{equation}\label{II:eq:self-extension-differential}
 \beta_{j,L}(\xi)=(d\mathcal P_j)_L(\xi).
\end{equation}
For verification, represent $\xi$ by a first-order line bundle
over $\C[\epsilon]/(\epsilon^2)$. Its underlying rank-two bundle
is \eqref{II:eq:higher-self-extension}, and multiplication by
$\epsilon$ is $\varepsilon_E$. The residue characterization of
the kernel commutes with this base change: acyclicity and the
exact diagonal sequence still give a unique section with residue
the identity. The finite-jet operations defining $\mathcal P_j$
also commute with it. Expanding
$\mathcal P_j(L_{\epsilon})$ to first order proves
\eqref{II:eq:self-extension-differential}.

On a compactified modular curve, the coefficients
$\beta_{j,L}(\xi)$ give scalar forms of weight $2j$ and cusp order
at least $j$; multiplying by $\varepsilon_E$ realizes them as
nilpotent matrix forms. Their traces and determinants vanish.
The coefficients $\beta_{j,L}(\xi)$ therefore supply the
extension data in this case, even though all the first
self-brackets vanish.

\subsection{The jet algebra and monodromy}
\label{II:subsec:hierarchy-monodromy}

The finite results above concern line bundles and their
extensions. We now return to an acyclic bundle $E$ of arbitrary
rank and ask what its entire hierarchy retains. At a point, the
coefficients and their derivatives are matrices, so they generate
an associative algebra. We will identify this algebra with the
span of the monodromy matrices of the canonical connection.

Fix $x\in X$, a local projective coordinate $z$, and a horizontal
frame for $V_E$. Let $P_{j,z}$ be the coefficient of
$\mathcal P_j(E)$.
Define the unital algebra
\begin{equation}\label{II:eq:hierarchy-jet-algebra}
 \mathfrak A_x(E)=
 \operatorname{alg}_{\C}\{
       \partial_z^aP_{j,z}(x):j\geq2,\ a\geq0\}
       \subset\End(V_{E,x}).
\end{equation}
The identity matrix is included in $\operatorname{alg}_{\C}$.

The triangular formula \eqref{II:eq:hierarchy-triangular} gives
an equivalent description:
\begin{equation}\label{II:eq:kernel-jet-algebra}
 \mathfrak A_x(E)=
 \operatorname{alg}_{\C}\{
       \partial_z^aB_j(x):j\geq2,\ a\geq0\}.
\end{equation}
Indeed, $\mathcal P_2$ gives $B_2$ and all its derivatives.
Once the jets of $B_2,\ldots,B_{j-1}$ are known,
\eqref{II:eq:hierarchy-triangular}, and its successive derivatives,
recover the jets of $B_j$ by division by $j!$.
The reverse inclusion follows from the same polynomial formulas.
Using all $\mathcal W_{j,N}$ instead of the selected $\mathcal P_j$
therefore gives the same algebra.

This algebra does not depend on the coordinate or on $p_0$.
To see this directly in \eqref{II:eq:kernel-jet-algebra}, use the
two-variable Taylor series of $\mathsf H_E(z,w)$ at $(x,x)$.
A coordinate change substitutes an invertible analytic change
of variables and multiplies this series by an invertible scalar
series. Its new coefficients are linear combinations of the old
ones and scalar multiples of $\Id$; the inverse change gives
the reverse inclusion. A change of horizontal frame conjugates
every matrix by the same constant matrix.

Let
\[
 \rho_E:\pi_1(X,x)\longrightarrow\operatorname{GL}(V_{E,x}),
 \qquad
 R_x(E)=\operatorname{Span}_{\C}
            \{\rho_E(\gamma):\gamma\in\pi_1(X,x)\}.
\]
The space $R_x(E)$ is an algebra: products of monodromy matrices
are again monodromy matrices, and the identity is present.

\begin{theorem}[The kernel jet algebra]\label{II:thm:kernel-monodromy-algebra}
Let $X$ be a smooth compact complex curve of genus $g\geq2$,
and let $E$ be an acyclic holomorphic bundle on $X$. Fix $x\in X$,
a local projective coordinate, and a horizontal frame for
$V_E=E\otimes\vartheta^{-1}$. The algebra generated by the invariant
jets in~\eqref{II:eq:hierarchy-jet-algebra} equals the monodromy span:
\begin{equation}\label{II:eq:kernel-monodromy-equality}
 \boxed{\ \mathfrak A_x(E)=R_x(E).\ }
\end{equation}
\end{theorem}

\begin{proof}
We prove the two inclusions separately. For the first, continuation
of the kernel recovers monodromy as residues. For the second, we
project the kernel to a quotient bundle and use acyclicity to
show that the projection is zero.

\noindent\textit{1. Continue the kernel in the fixed jet algebra.} Pull $V_E$ and $\vartheta$ to the universal cover $\widetilde X$.
Extend the chosen local horizontal frame to a horizontal frame $e$
of the former, and choose a nowhere vanishing
holomorphic frame $\beta$ of the latter. Write the lifted kernel as
\[
 \widetilde s_E(z,w)
   =e(z)\,\mathsf S(z,w)\,e(w)^{-1}\beta(z)\beta(w).
\]
Near a lift $(\widetilde x,\widetilde x)$, the Taylor coefficients
of the regularized matrix kernel belong to $\mathfrak A_x(E)$.
This algebra is a finite-dimensional linear subspace of the matrix
space, so convergence implies that $\mathsf S$ takes values in it
on an open set away from the diagonal. If a linear functional on
the matrix space annihilates $\mathfrak A_x(E)$, its value on
$\mathsf S$ is a meromorphic function vanishing on that open set.
The identity theorem on the connected product
$\widetilde X\times\widetilde X$ makes it identically zero. Thus
$\mathsf S$ takes values in this same fixed algebra everywhere it
is regular; no invariance of the algebra under monodromy has been
assumed.

\noindent\textit{2. Recover monodromy from residues.} To compute a translated residue, use the right deck action,
written $z\mapsto z\gamma$, with the convention
\[
 e(z\gamma)=e(z)\rho_E(\gamma),\qquad
 \beta(z\gamma)=b_\gamma(z)\beta(z),
\]
where the fibres over $z$ and $z\gamma$ are identified by the
covering map and $b_\gamma$ is nowhere zero. The kernel descends,
so changing its first variable gives
\[
 \mathsf S(z\gamma,w)
   =b_\gamma(z)^{-1}\rho_E(\gamma)^{-1}\mathsf S(z,w).
\]
Choose a local coordinate $t$ at $\widetilde x$ with
$t(\widetilde x)=0$ and $dt=\beta^2$. The prescribed residue in
the \emph{second} variable gives
$\mathsf S(z,\widetilde x)=-\Id/t(z)+O(1)$.
Put $t_\gamma(z)=t(z\gamma^{-1})$ near $\widetilde x\gamma$.
Consequently
\[
 \res_{z=\widetilde x\gamma}
       \bigl(\mathsf S(z,\widetilde x)\,dt_\gamma(z)\bigr)
   =-b_\gamma(\widetilde x)^{-1}\rho_E(\gamma)^{-1}.
\]
The residue belongs to $\mathfrak A_x(E)$ because the meromorphic
matrix function takes values in that finite-dimensional subspace.
Its scalar factor is nonzero. As $\gamma$ varies, the inverse
monodromy matrices span $R_x(E)$, proving
$R_x(E)\subset\mathfrak A_x(E)$.

\noindent\textit{3. Form the quotient by the monodromy span.} For the reverse inclusion, put $W=V_{E,x}$ and fix $y$ near $x$.
Use parallel transport on a small disc to identify $V_{E,y}$ with
$W$. Left multiplication by every $\rho_E(\gamma)$ preserves
$R_x(E)\subset\Hom(W,W)$, so it defines a flat subbundle
\[
 \mathcal R\subset V_E\otimes W^*,\qquad
 \mathcal Q=(V_E\otimes W^*)/\mathcal R.
\]
Write $d=\operatorname{rank}\mathcal R$.
The holomorphic flat connection gives $\deg\mathcal R=0$,
and Riemann--Roch gives
\[
 \chi(\mathcal R\otimes\vartheta)
      =d(g-1)+d(1-g)=0.
\]
The inclusion
$\mathcal R\otimes\vartheta\subset E\otimes W^*$
implies $H^0(X,\mathcal R\otimes\vartheta)=0$.
Therefore $H^1(X,\mathcal R\otimes\vartheta)=0$ as well.
The exact sequence
\[
 0\longrightarrow\mathcal R\otimes\vartheta
 \longrightarrow E\otimes W^*
 \longrightarrow\mathcal Q\otimes\vartheta
 \longrightarrow0
\]
now implies $H^0(X,\mathcal Q\otimes\vartheta)=0$.

\noindent\textit{4. Use acyclicity to kill the projected kernel.} With a frame of the fixed one-dimensional factor $\vartheta_y$,
the kernel $s_E(\,\cdot\,,y)$ is a meromorphic section of
$(V_E\otimes W^*)\otimes\vartheta$, with at most a simple pole
at $y$. Its polar coefficient is a scalar multiple of $\Id$,
which belongs to $R_x(E)$. Its projection to
$\mathcal Q\otimes\vartheta$ is consequently holomorphic, and
must be zero. Hence the matrix kernel takes values in $R_x(E)$.
Varying $y$ in the same disc shows that all $B_j$ and their
derivatives at $x$ belong to this algebra. This proves the
remaining inclusion.
\end{proof}

The vanishing of $H^0(X,\mathcal Q\otimes\vartheta)$ is where
acyclicity and compactness enter the identification with the
canonical connection.

For lines, the theorem also explains separation by the whole
hierarchy. If $\mathcal P_j(L)=\mathcal P_j(M)$ for every $j$,
then every invariant of $L\oplus M$ is scalar. Its jet algebra
is $\C\Id$, so its two monodromy characters agree and $L\simeq M$.
The pole estimate gave the stronger, effective conclusion of
Theorem~\ref{II:thm:bounded-line-separation}.

For extensions, the algebra has a particularly simple form.
Let $\mathsf T_2(\C)$ denote the upper triangular matrix algebra.
When $L=M$, write $\varepsilon_E=i\pi$, so that
$\varepsilon_E^2=0$. With the endpoint identifications retained,
the algebras are, up to an adapted change of frame,
\begin{equation}\label{II:eq:extension-jet-algebras}
 \begin{array}{c|c|c}
   \text{endpoints}&\xi=0&\xi\ne0\\ \hline
   L\not\simeq M&\C\oplus\C&\mathsf T_2(\C)\\
   L=M&\C\Id&\C\Id+\C\varepsilon_E .
 \end{array}
\end{equation}
Here the last entry is the dual-number algebra.

To prove the first row, the two canonical line characters are
distinct: equality of characters would give a flat, hence
holomorphic, isomorphism $L\simeq M$.
The diagonal projection of $R_x(E)$ is therefore all of $\C^2$.
If that projection were injective, its two idempotents would give
complementary invariant lines and split the flat extension.
For $\xi\ne0$ this is impossible. Its kernel must therefore be
the one-dimensional strictly upper triangular space, giving
$R_x(E)=\mathsf T_2(\C)$.
For $\xi=0$, naturality gives the direct sum of the line
connections and the diagonal algebra.

For the second row, $\varepsilon_E$ is horizontal by naturality,
and the monodromy matrices have the form
\[
 \rho_E(\gamma)=\rho_L(\gamma)
                    \bigl(\Id+u_\gamma\varepsilon_E\bigr).
\]
If every $u_\gamma$ were zero, a constant complement in a
horizontal frame would descend to a splitting. Thus $\xi\ne0$
gives $\C\Id+\C\varepsilon_E$; the split kernel gives $\C\Id$.
Theorem~\ref{II:thm:kernel-monodromy-algebra} proves
\eqref{II:eq:extension-jet-algebras}.
The algebra detects nonsplitting, but its abstract isomorphism
type is the same for all nonzero classes with the same type of
endpoints. To recover the class we retain the actual invariant
coefficients.

\subsection{What is finite, and what is bounded}
\label{II:subsec:hierarchy-bounds}

The line-bundle results use a bounded list of global sections:
$\mathcal P_2,\ldots,\mathcal P_{4g-1}$. Recovery is given by
\eqref{II:eq:higher-recovery-formula} for distinct endpoints and by
\eqref{II:eq:bounded-self-injection} for identical endpoints.

\begin{remark}[Finite dimension and effective bounds]\label{II:rem:finite-jet-scope}

The algebra \eqref{II:eq:hierarchy-jet-algebra} instead uses jets
at one point. Its finite-dimensionality has a different meaning.
For a bundle of rank $r$, it implies that some finite collection
of words in the jets spans the algebra:
\begin{equation}\label{II:eq:finite-jet-words}
 \mathfrak A_x(E)=\operatorname{Span}_{\C}\{C_1,\ldots,C_d\},
 \qquad d\leq r^2,\qquad
 C_b=\prod_s\partial_z^{a_{b,s}}P_{j_{b,s},z}(x),\qquad 1\le b\le d.
\end{equation}
Indeed, all finite words span the generated algebra, and a basis
can be selected from that spanning set.
The number $r^2$ bounds the size of a basis; it does not bound
the indices $j_{b,s}$ or the derivative orders $a_{b,s}$.
No bound for these indices in arbitrary rank is supplied here.
\end{remark}

Three questions remain within this construction. Can the uniform
upper bound $4g-1$ for separating lines be lowered? The theta
formula proves that a general pair has different cubics, but it
does not prove that
\[
 \operatorname{Pic}^{g-1}(X)\setminus\Theta
       \longrightarrow H^0(X,K^3),\qquad L\longmapsto\cub_L,
\]
is generically injective; does that stronger statement hold?
Finally, for rank $r>1$, can one bound both the degrees $j$ and
the derivative orders $a$ needed among
$\partial_z^aP_{j,z}(x)$ in terms of $g$ and $r$?

Thus the monodromy theorem explains the information retained by
the full hierarchy, while the kernel pole count provides a
specific finite procedure for the line-bundle extensions studied
in this paper.

\section{An explicit construction on \texorpdfstring{$X_0(23)$}{X0(23)}}
\label{II:sec:x023}

We finish by carrying the cubic construction through to a
Fourier coefficient. On $X=X_0(23)$ we will choose acyclic lines
$L,M$ with different cubics and
$\dim H^1(X,LM^{-1})=1$. After fixing a horizontal frame, let
$A_9(\xi)$ denote the coefficient of $q^9$ in the upper-right
entry of the normalized weight-fourteen form of an extension
$\xi$. We will compute a generator $\zeta$ of the Serre-dual
space and obtain
\[
 \langle\xi,\zeta\rangle=\frac{A_9(\xi)}{2484}.
\]
Thus this Fourier coefficient is a coordinate on the marked
extension space. The choices fixing its normalization are part
of the construction below.

We first use rational frames to compute the kernel, its
connection, and the intrinsic differential $c_\xi$. We then
change to a horizontal frame, where its coefficient is a modular
form. This conversion will produce both the displayed recovery
rule and the multiplier of the form.

\subsection{The curve and the endpoint bundles}
\label{II:subsec:x023-curve}

Let $X=X_0(23)$, with cusps ${P_\infty}=\infty$ and ${P_0}=0$. The projective
image of $\Gamma_0(23)$ has index $24$, has no elliptic points,
and has two cusps, of widths $1$ and $23$. The genus formula gives
$g=1+24/12-2/2=2$. We use this effective projective action, or a
lift as in Section~\ref{II:sec:modular-calculus}.

Put $q=e^{2\pi i\tau}$ and $D=q\,d/dq$. To avoid confusion with
the scalar connection $\eta$ of Section~\ref{II:sec:modular-calculus},
write $\eta_{\mathrm D}$ for the Dedekind eta function. Define
\begin{align}\label{II:eq:x023-basis}
 h&=\eta_{\mathrm D}(\tau)\eta_{\mathrm D}(23\tau),&
 \Theta&=\sum_{m,n\in\mathbb Z}q^{m^2+mn+6n^2},\nonumber\\
 f_2&=h^2,& f_1&=h\Theta-h^2 .
\end{align}
The eta and binary theta transformation laws give $h,\Theta$
weight one and the same quadratic character
$\left(\frac{-23}{\,\cdot\,}\right)$. Thus $f_1,f_2$ are weight-two
cusp forms with trivial character. Their expansions begin
\begin{align}\label{II:eq:x023-basis-series}
 f_1&=q-q^3-q^4-2q^6+2q^7-q^8+2q^9+O(q^{10}),\nonumber\\
 f_2&=q^2-2q^3-q^4+2q^5+q^6+2q^7-2q^8+O(q^{10}).
\end{align}
They are independent and $\dim S_2(\Gamma_0(23))=g=2$, so they
form a basis. This is the basis recorded in
\cite[Example~3.29]{Stein}. Let
\[
 \omega_i=f_i\,\frac{dq}{q}\in H^0(X,K),\qquad \omega=\omega_2 .
\]
The eta product has no zeros on $\mathbb H$, and its Fricke
transformation gives order two for $f_2$ at each cusp. Consequently
\begin{equation}\label{II:eq:x023-canonical-divisor}
 \operatorname{div}(\omega)={P_\infty}+{P_0},\qquad K\simeq\OO({P_\infty}+{P_0}).
\end{equation}
At ${P_\infty}$, the orders of $\omega_1,\omega_2$ are $0,1$.
Thus ${P_\infty}$ is not a Weierstrass point.

The modular functions
\begin{equation}\label{II:eq:x023-coordinates}
 x=\frac{f_1}{f_2},\qquad y=\frac{Dx}{f_2}
\end{equation}
identify $X$ with the smooth completion of
\begin{align}\label{II:eq:x023-model}
 y^2=f(x)
  &=(x^3-x+1)(x^3-8x^2+3x-7)\nonumber\\
  &=x^6-8x^5+2x^4+2x^3-11x^2+10x-7.
\end{align}
This model is recorded by Schoof \cite[Section~6]{Schoof}.
In these coordinates $\omega=dx/y$ and $\omega_1=x\,dx/y$.
The two points at infinity are
${P_\infty}:y/x^3\to-1$ and ${P_0}:y/x^3\to1$.

The model can also be verified directly from the modular forms.
The following finite coefficient check verifies the identity
\eqref{II:eq:x023-model} for the coordinates
\eqref{II:eq:x023-coordinates}. The expression
$W=(Df_1)f_2-f_1Df_2$ has weight six, and
$W^2-f_2^6f(f_1/f_2)$ is a holomorphic form of weight twelve.
Its coefficients through $q^{24}$ vanish on substituting
\eqref{II:eq:x023-basis}; the weight-index bound is
$12[\operatorname{SL}_2(\mathbb Z):\Gamma_0(23)]/12=24$
\cite[Theorem~9.18]{Stein}. This proves the identity globally.
Since $x=\omega_1/\omega_2$ is the canonical degree-two map of a
genus-two curve, the identity gives its full function field.

We take
\begin{equation}\label{II:eq:x023-lines}
 L=\OO(2{P_\infty}-{P_0}),\qquad M=KL^{-1}=\OO(2{P_0}-{P_\infty}),\qquad
 T=LM^{-1}=\OO(3{P_\infty}-3{P_0}).
\end{equation}
These bundles are acyclic. Indeed, ${P_\infty}$ is ordinary, so
$h^0(\OO(2{P_\infty}))=1$; the only rational functions with these pole
bounds are constants, and a nonzero constant does not vanish at
${P_0}$. Hence $H^0(L)=0$. Since $\deg L=1=g-1$,
Riemann--Roch gives $H^1(L)=0$, and Serre duality gives the same
vanishing for $M$.

The modular unit
\begin{equation}\label{II:eq:x023-unit}
 \mathfrak u=\left(\frac{\eta_{\mathrm D}(\tau)}
                    {\eta_{\mathrm D}(23\tau)}\right)^{12}
 \quad\hbox{has divisor }11({P_0}-{P_\infty}).
\end{equation}
The class ${P_\infty}-{P_0}$ is nonzero, since a function with divisor ${P_\infty}-{P_0}$
would give a degree-one map $X\to\mathbb P^1$. It therefore has
exact order $11$. In particular $T$ is nontrivial of degree zero,
and
\begin{equation}\label{II:eq:x023-extension-dimension}
 \dim H^1(X,T)=1.
\end{equation}
We fix a theta characteristic $\vartheta$ and construct the
connection on $V_E=E\vartheta^{-1}$ as in
Section~\ref{II:sec:kernel}. One possible choice is
$\vartheta=\OO(6{P_\infty}-5{P_0})$: its square differs from $K$ by
$11({P_\infty}-{P_0})$, which is principal by \eqref{II:eq:x023-unit}.
The adjoint connection and all the invariants below are
independent of this choice.

\subsection{A rational kernel and a nonzero cubic}
\label{II:subsec:x023-line-kernel}

We now apply Lemma~\ref{II:lem:hyperelliptic-kernel} to the
model~\eqref{II:eq:x023-model}. Its nonzero cubic will provide
the input for extension recovery.

Write $e_L$ for the rational section represented by $1$ in the divisor
model $L=\OO(2P_\infty-P_0)$. Its divisor as a section of $L$ is
$\operatorname{div}(e_L)=2P_\infty-P_0$. Thus a rational function
multiplies $e_L$ to represent a section of $L$; the same convention
fixes the induced frames on its tensor powers and on $T$.

For the model \eqref{II:eq:x023-model}, the two points at infinity have
the orientation used in Lemma~\ref{II:lem:hyperelliptic-kernel}, and
\[
 S(x)=x^3-4x^2-7x-27,\qquad
 R(x)=f(x)-S(x)^2=-276x^2-368x-736.
\]
Thus $R$ is nonconstant, and the lemma already gives $\cub_L\ne0$.
Let ${e_L}$ be the rational divisor frame of $L=\OO(2{P_\infty}-{P_0})$.
For later reference the kernel in this frame is
\begin{equation}\label{II:eq:x023-line-kernel}
 s_L(z,w)=
 \frac{y(z)+y(w)-S(x(z))+S(x(w))}
      {2(x(w)-x(z))}\,\frac{dx(w)}{y(w)},
\end{equation}
with the factors ${e_L}(z){e_L}(w)^{-1}$ understood. In
\eqref{II:eq:hyperelliptic-factorial-expansion}, use the functions $u,v,a_3$
obtained from this $f$ and $S$. The arithmetic identity
\[
 S''R'-S'R''=-184(9x^2+12x+5)
\]
then gives
\begin{equation}\label{II:eq:x023-cubic}
 \cub_L=-23(9x^2+12x+5)\omega^3.
\end{equation}
By duality $\cub_M=-\cub_L$, so the hypothesis of
Theorem~\ref{II:thm:extension-recovery} holds for this pair.

\subsection{The extension and its canonical connection}
\label{II:subsec:x023-extension}

Choose $R_0=(0,c)$, where $c^2=-7$, and a small disc $U_1$ about
$R_0$. Put $U_0=X\setminus\{R_0\}$. The coordinate $x$ is regular
at $R_0$. On each $U_i$ take a copy of $L\oplus M$. On the
annulus $U_0\cap U_1$, the endpoint frames ${e_L}$ and $dx/{e_L}$
are regular; express the two splittings in these frames and glue by
\begin{equation}\label{II:eq:x023-gluing}
 e_1=e_0G,\qquad
 G(x)=\begin{pmatrix}1&\lambda/x\\0&1\end{pmatrix}.
\end{equation}
Here $e_i$ denotes these local frame rows on the annulus.
The splittings on $U_i$ are direct sums of bundles, not chosen
global holomorphic trivializations of those bundles.
The extension class $\xi_\lambda\in H^1(X,T)$ is represented,
with the convention of Section~\ref{II:sec:extension-recovery},
by $(\lambda/x){e_L}^2/dx$ on $U_0\cap U_1$.
The endpoint bundles are fixed throughout this construction.

The kernel can be written down in the splitting on $U_0$.
Let $k_L(z,w)$ be the coefficient of $dx(w)$ in
\eqref{II:eq:x023-line-kernel}, and set
\begin{equation}\label{II:eq:x023-g}
 g(x)=\frac{y(x)-S(x)+c-27}{2xc}.
\end{equation}
Duality gives $k_M(z,w)=-k_L(w,z)$, while
$k_L(z,R_0)=-g(z)$ and $k_M(R_0,w)=g(w)$. The matrix kernel is
\begin{equation}\label{II:eq:x023-matrix-kernel}
 K_\lambda(z,w)=
 \begin{pmatrix}
 k_L(z,w)&\lambda g(z)g(w)\\
 0&k_M(z,w)
 \end{pmatrix}dx(w).
\end{equation}
The entries refer to the frames ${e_L},dx/{e_L}$ in each variable.
In particular the differential factor in the upper-right entry
cancels the one in the dual frame of $M$; intrinsically that entry
takes values in $L_z\otimes L_w$.

Here $g=x^{-1}+O(1)$ at $R_0$ and is regular at $(0,-c)$.
Changing the frame in the first variable gives the upper entry
$\lambda g(z)g(w)-\lambda k_M(z,w)/x(z)$; its pole at $R_0$
cancels. Changing it in the second variable gives
$\lambda g(z)g(w)+\lambda k_L(z,w)/x(w)$, with the same
cancellation. These are exactly the transformations
$G(z)^{-1}K_\lambda G(w)$, and they contain no term quadratic in
$\lambda$. The diagonal residue remains the identity.
The line-kernel divisor bounds handle the other points. After
these cancellations there are no divisorial poles beyond the
diagonal; holomorphic extension across the remaining isolated
points follows on the smooth surface $X\times X$.
Since $E_{\xi_\lambda}$ is acyclic, uniqueness identifies
\eqref{II:eq:x023-matrix-kernel} with its canonical kernel.

On a disc in $U_0$ where $x$ is a coordinate, use the
determinant-compatible frames
${e_L}/\sqrt{dx}$ and $\sqrt{dx}/{e_L}$ of
$V_{E_{\xi_\lambda}}$. The first regular coefficient of
\eqref{II:eq:x023-matrix-kernel} gives the complete connection matrix
\begin{equation}\label{II:eq:x023-connection}
 \nabla^{V_{E_{\xi_\lambda}}}=d+\mathsf A(x)\,dx,\qquad
 \mathsf A(x)=\begin{pmatrix}u&\lambda g^2\\0&-u\end{pmatrix},
 \qquad u=\frac{S'-y'}{2y}.
\end{equation}
This formula is in meromorphic frames; the connection itself is
holomorphic on all of $X$, by its kernel construction. In
particular it extends across both cusps and across $R_0$.
We use this canonical connection in all subsequent covariant
derivatives of the matrix invariants.

\subsection{The entire extension map}
\label{II:subsec:x023-map}

The Serre pairing turns the extension class into a scalar once
we choose a nonzero section of $KT^{-1}$. We use
\begin{equation}\label{II:eq:x023-dual-generator}
 \zeta={e_L}^{-2}\bigl((3x-16)(y+S)-414\bigr)\omega^2
       \in H^0(X,KT^{-1}).
\end{equation}
To check holomorphicity of $\zeta$, the factor in parentheses has
a zero of order at least two at ${P_\infty}$ and pole order four at ${P_0}$.
The divisor of ${e_L}^{-2}\omega^2$ is $-2{P_\infty}+4{P_0}$.
Thus the product is holomorphic and nonzero.
We normalize the pairing $\langle\, ,\,\rangle$ by residues, as in
the cocycle convention of Section~\ref{II:sec:extension-recovery}.

On the target side, consider the meromorphic differential
\begin{equation}\label{II:eq:x023-target-direction}
 {e_L}^2r(x,y)\omega^6,\qquad
 r(x,y)=\tfrac12 f'(x)y-S'(x)f(x).
\end{equation}
Its bundle is $TK^7=L^2K^6$. The calculation below will show that
it is holomorphic and that every bracket is a multiple of it.
By \eqref{II:eq:x023-extension-dimension} and Serre duality,
$h^0(X,KT^{-1})=h^1(X,T)=1$. With these choices the map from the one-dimensional extension
space is determined by one scalar. The next theorem computes
that scalar.

\begin{theorem}[The explicit recovery map]\label{II:thm:x023-map}
For the bundles \eqref{II:eq:x023-lines}, the canonical map of
Theorem~\ref{II:thm:extension-recovery} is
\begin{equation}\label{II:eq:x023-map}
 c_\xi=-414\langle\xi,\zeta\rangle\,{e_L}^2r(x,y)\omega^6
       \in H^0(X,TK^7).
\end{equation}
Consequently its inverse on its image is
\[
 c_\xi=A_\xi{e_L}^2r\omega^6
       \quad\longmapsto\quad
 \langle\xi,\zeta\rangle=-A_\xi/414.
\]
For the cocycle \eqref{II:eq:x023-gluing},
$\langle\xi_\lambda,\zeta\rangle=-2\lambda(9-8c)/7$.
\end{theorem}

\begin{lemma}[The rational recovery certificate]
\label{II:lem:x023-rational-certificate}
Use $f$ from~\eqref{II:eq:x023-model}, put
$S=x^3-4x^2-7x-27$, and take $g$ from~\eqref{II:eq:x023-g},
with $y^2=f$ and $c^2=-7$.
Put
\[
 u=\frac{S'-y'}{2y},\qquad
 v=\frac{3(f')^2}{8f^2}-\frac{f''}{4f}
          +\frac{S''}{2y}-\frac{S'f'}{2fy},\qquad
 a_x=-\frac{23(9x^2+12x+5)}{y^3},
\]
and $B=gg''-2(g')^2-2ugg'+(2u^2+4u'-3v)g^2$.
Then
\begin{equation}\label{II:eq:x023-rational-reduction}
 12\{a_x(B'+2uB-2a_xg^2)-a_x'B\}
   =\frac{828(9-8c)}7\,\frac{y'-S'}{y^4}.
\end{equation}
\end{lemma}
\begin{proof}
\noindent\textit{1. Separate the quadratic-field components.} Separate the
rational part of $B$ from its coefficient of $y$ before
differentiating. The auxiliary polynomials below are taken in
$\C[x,c]/(c^2+7)$. Put
\begin{align*}
 \alpha&=-23(9x^2+12x+5),& d&=S+27-c,\\
 U_B&=3(27-c)(S+27)+4(21-11c)x-230,\\
 V_B&=46\{9(27x^2+34x+15)-c(9x^2-4x+5)\}-SU_B.
\end{align*}
Thus $a_x=\alpha/y^3$ and $g=(y-d)/(2cx)$.
Substitution in the expression for $B$, followed by $y^2=f$
and $c^2=-7$, gives
\[
 B=\frac{V_B+yU_B}{14xf}.
\]
After differentiating this expression, the left side of
\eqref{II:eq:x023-rational-reduction} becomes
$12(N_0+yN_1)/(14x^2f^3)$, where
\begin{align*}
 N_0={}&x\bigl\{f(\alpha U_B'-\alpha'U_B)
           +\tfrac12\alpha f'U_B+\alpha S'V_B\bigr\}
           -\alpha fU_B+\alpha^2(f+d^2),\\
 N_1={}&x(\alpha V_B'-\alpha'V_B+\alpha S'U_B)
           -\alpha V_B-2\alpha^2d.
\end{align*}
\noindent\textit{2. Verify two polynomial identities.} Only two polynomial identities remain:
\[
 N_0=-138x^2(9-8c)S'f,\qquad
 N_1=69x^2(9-8c)f'.
\]
They follow by inserting the displayed polynomials $f,S,U_B,V_B$
and replacing $c^2$ by $-7$. In particular, the factor $x^2$
cancels the apparent denominator introduced by $g$. \noindent\textit{3. Recombine the components.} Combining
the two terms leaves
\[
 \frac{12(N_0+yN_1)}{14x^2f^3}
 =\frac{828(9-8c)}7\,
       \frac{\tfrac12yf'-S'f}{f^3}
 =\frac{828(9-8c)}7\,\frac{y'-S'}{y^4},
\]
as required.

\end{proof}

\begin{proof}[Proof of Theorem~\ref{II:thm:x023-map}]
It is enough to compute the bracket for the cocycle $\lambda/x$
and compare it with the residue pairing in the same frames.
The nonzero pairing will show that this cocycle spans the
extension space; linearity will then give the map for every class.

\noindent\textit{1. The bracket in the divisor frame.} Use the factorial expansion \eqref{II:eq:hyperelliptic-factorial-expansion}.
Its second coefficient is
\[
 v=\frac{3(f')^2}{8f^2}-\frac{f''}{4f}
          +\frac{S''}{2y}-\frac{S'f'}{2fy}.
\]
For the matrix kernel, the upper entry of the regularized kernel
is $\lambda(w-x)g(x)g(w)$. Since
$a_j=\partial_w^jH|_{w=x}$, differentiation gives the upper
entries of $a_1,a_2,a_3$ as
$\lambda g^2,2\lambda gg',3\lambda gg''$, respectively.
The quadratic coefficient $I_2$ is scalar. Substitution in the
ordered Ore formulas therefore gives, with
$a_x=-23(9x^2+12x+5)/y^3$,
\begin{align}\label{II:eq:x023-bracket-calculation}
 \cub_E&=
 \begin{pmatrix}a_x&\lambda B\\0&-a_x\end{pmatrix}(dx)^3,
 \nonumber\\
 B&=gg''-2(g')^2-2ugg'+(2u^2+4u'-3v)g^2,\nonumber\\
 (\mathcal C_E)_{12}
 &=12\lambda\{a_x(B'+2uB-2a_xg^2)-a_x'B\}(dx)^7.
\end{align}
The connection terms $2uB-2a_xg^2$ come from $[\mathsf A,F]_{12}$.
\noindent\textit{2. The rational identity.} Keeping these terms and using
$y'=f'/(2y)$, $y^2=f$, and $c^2=-7$, the remaining rational
identity follows from
Lemma~\ref{II:lem:x023-rational-certificate}. It gives
\eqref{II:eq:x023-rational-reduction}, including the cancellation of
the denominator introduced by $g$.

\noindent\textit{3. The residue pairing.} The upper entry uses the frame ${e_L}^2/dx$ of $T$. Since
$r=f(y'-S')$ and $\omega=dx/y$, the intrinsic differential is
\[
 c_{\xi_\lambda}
   =\frac{828\lambda(9-8c)}7\,{e_L}^2r\omega^6.
\]
At $R_0$, the coefficient of $\zeta$ in the frame
$dx^2/{e_L}^2$ is $-2(9-8c)/7$. Thus
\[
 \langle\xi_\lambda,\zeta\rangle
 =\res_{R_0}\left(\frac{\lambda}{x}\frac{{e_L}^2}{dx}\,\zeta\right)
 =-\frac{2\lambda(9-8c)}7.
\]
It is nonzero for $\lambda\ne0$, so these classes span $H^1(T)$.
Comparison proves \eqref{II:eq:x023-map}, with the displayed sign.
The final inverse is meaningful because $\zeta$ is a basis of
the dual space.
\end{proof}

We can also follow the polar recovery of
Section~\ref{II:sec:extension-recovery} directly.
Put $d_{LM}=2\cub_L$ and let $Z$ be its zero divisor.
On $U_0$, the quotient of the upper cubic entry by $d_{LM}$ is
\begin{equation}\label{II:eq:x023-polar-section}
 t_0=\frac{\lambda B}{2a_x}\frac{{e_L}^2}{dx}\in T.
\end{equation}
Its polar parts along $Z$ are exactly the polar primitives of
$c_\xi/(6d_{LM}^2)$ by \eqref{II:eq:extension-basic-identity}.
At $R_0$, one has $a_x(R_0)\ne0$. The gluing law gives
$t_1=t_0+(\lambda/x){e_L}^2/dx$, with $t_1$ holomorphic there.
Thus $t_0$ has principal part $-(\lambda/x){e_L}^2/dx$ at $R_0$
and no poles outside $Z\cup\{R_0\}$. The residue theorem for the
global meromorphic one-form $t_0\zeta$ now gives
\begin{equation}\label{II:eq:x023-polar-pairing}
 \sum_{p\in Z}\res_p(t_0\zeta)
   =-\res_{R_0}(t_0\zeta)
   =-\frac{2\lambda(9-8c)}7
   =\langle\xi_\lambda,\zeta\rangle.
\end{equation}
The left side is the Serre pairing of the recovered principal
parts with $\zeta$. Hence the polar algorithm returns the same
coordinate as \eqref{II:eq:x023-map}, including its sign. It
recovers the marked extension class, not just its split or
nonsplit isomorphism type.

\subsection{Horizontal normalization and Fourier coefficients}
\label{II:subsec:x023-fourier}

The rational frame made the kernel calculation explicit. To
read its result as a modular form with a fixed multiplier, we
now use a horizontal frame. Normalize
$\langle\xi,\zeta\rangle=-1/414$, or equivalently
$\lambda=7/(828(9-8c))$, so that $c_\xi={e_L}^2r\omega^6$.
Put $t=\log q=2\pi i\tau$. Then
$D=d/dt$ and $\omega=f_2\,dt$.

In the rational frame $n={e_L}^2/\omega$ of $T$, the induced
connection is
\begin{equation}\label{II:eq:x023-line-connection}
 \nabla^T n=\frac{S'}{y}\,dx\otimes n.
\end{equation}
Indeed, its coefficient in ${e_L}^2/dx$ is $2u$ from
\eqref{II:eq:x023-connection}. The change $n=y{e_L}^2/dx$
adds $y'/y$, leaving $S'/y$.
The horizontal frame at ${P_\infty}$ is $n_h=\psi n$, where
\begin{equation}\label{II:eq:x023-horizontal-series}
 \psi=q^{-3}U(q),\qquad U(0)=1,\qquad
 D\log U=3-3\frac{f_1^2}{f_2}+8f_1+7f_2.
\end{equation}
The factor $q^{-3}$ cancels the zero of $n$ at ${P_\infty}$.
Thus $n_h$ is holomorphic and nonvanishing at the cusp.

The change of frame and parameter now gives one conversion:
\begin{equation}\label{II:eq:x023-frame-conversion}
 \begin{aligned}
 c_\xi&={e_L}^2r\omega^6=n\,r\omega^7\\
      &=n_h\,\frac{r f_2^7}{\psi}(dt)^7.
 \end{aligned}
\end{equation}
Thus $r f_2^7/\psi$ is the coefficient in the horizontal frame
and the parameter $t$. To relate it to the $\tau$ convention of
Section~\ref{II:sec:kernel}, use
$dt=\kappa^{-1}d\tau$, where $\kappa=(2\pi i)^{-1}$.
Writing $F_E,\Phi_E$ for the forms defined there, set
\begin{equation}\label{II:eq:x023-normalized-forms}
 \widehat F=\kappa^3F_E,\qquad
 \widehat\Phi=\kappa^6\Phi_E=6[\widehat F,D\widehat F].
\end{equation}
Thus the pulled-back tensors are
$\cub_E=\widehat F(dt)^3$ and
$\mathcal C_E=\widehat\Phi(dt)^7$.
The two powers of $\kappa$ in
\eqref{II:eq:x023-normalized-forms} record this change of parameter
and the derivative normalization in the bracket. In an adapted
horizontal frame in which the matrix unit
$E_{12}=\left(\begin{smallmatrix}0&1\\0&0\end{smallmatrix}\right)$
represents $n_h$,
\eqref{II:eq:x023-frame-conversion} gives
\begin{equation}\label{II:eq:x023-c14}
 \widehat\Phi=c_{14}E_{12},\qquad
 c_{14}=\frac{r(x,y)f_2^7}{\psi}.
\end{equation}
Expanding \eqref{II:eq:x023-horizontal-series} gives
$U=1+2q+q^2-12q^3-\frac{65}{2}q^4-\frac{151}{5}q^5+\cdots$,
and hence
\begin{align}\label{II:eq:x023-q14}
 c_{14}={}&-6q^9+64q^{10}-230q^{11}+120q^{12}
           +1401q^{13}\nonumber\\
          &-\frac{18736}{5}q^{14}+\frac{3869}{5}q^{15}
           +\frac{260588}{21}q^{16}+O(q^{17}).
\end{align}
Equations \eqref{II:eq:x023-basis}, \eqref{II:eq:x023-horizontal-series},
and \eqref{II:eq:x023-c14} determine every coefficient by rational
arithmetic.

The cubic itself is equally explicit. Its diagonal coefficient is
\begin{equation}\label{II:eq:x023-a-modular}
 a=-23(9f_1^2f_2+12f_1f_2^2+5f_2^3)
   =-207q^4+138q^5+1610q^6+O(q^7).
\end{equation}
A constant upper triangular change of horizontal frame makes
the constant term of $b/a$ zero. In that frame,
\begin{equation}\label{II:eq:x023-full-cubic}
 \widehat F=\begin{pmatrix}a&b\\0&-a\end{pmatrix},\qquad
 b=\frac{q^5}{414}-\frac{8q^6}{621}
       +\frac{q^7}{46}-\frac{q^8}{414}+O(q^9).
\end{equation}
The identity
\begin{equation}\label{II:eq:x023-cubic-primitive}
 D(b/a)=\frac{c_{14}}{12a^2}
\end{equation}
both computes $b$ and verifies
$6[\widehat F,D\widehat F]=c_{14}E_{12}$.
Its right side starts in degree one in $q$, so the prescribed
constant term fixes the primitive.

For an arbitrary $\xi$, keep the same horizontal normalization
of $T$. If $\widehat\Phi_\xi=(A_9q^9+O(q^{10}))E_{12}$, then
linearity and \eqref{II:eq:x023-map} give the explicit recovery rule
\begin{equation}\label{II:eq:x023-fourier-recovery}
 \langle\xi,\zeta\rangle=\frac{A_9}{2484},
 \qquad
 \xi=\xi_\lambda,\quad
 \lambda=-\frac{7A_9}{4968(9-8c)}.
\end{equation}
For the unscaled form $\Phi_E$, multiply its $q^9$ coefficient
by $\kappa^6$ before applying this formula.
In particular the extension splits if and only if $A_9=0$.

\subsection{The multiplier and the noncommutative pair}
\label{II:subsec:x023-multiplier}

We compute the multiplier in two steps. The action on the line
$T$ first determines the character of $c_{14}$; the primitive
$b/a$ then determines the adjoint action on the two matrix forms.
Finally, a scalar factor will specify the full monodromy on
$V_E$.

Let
\[
 A_i(\tau)=\int_{q=0}^{q} f_i(s)\,\frac{ds}{s}.
\]
These are primitives on $\mathbb H$, normalized to have zero
constant Fourier term at ${P_\infty}$. The identity
\begin{equation}\label{II:eq:x023-connection-identity}
 S'(x)f_2=-\frac3{11}D\log\mathfrak u
                    -\frac{58}{11}f_1-\frac{86}{11}f_2
\end{equation}
implies
\begin{equation}\label{II:eq:x023-psi}
 \psi=\mathfrak u^{3/11}
       \exp\left(\frac{58A_1+86A_2}{11}\right),
 \qquad \psi\sim q^{-3}.
\end{equation}
For example, \eqref{II:eq:x023-connection-identity} follows by
multiplying by $f_2$ and checking the resulting holomorphic
weight-four identity through $q^8$. Its bound is $4\cdot24/12=8$.
The exact series also checks this identity.

Since the derivative of $\log\psi$ is an invariant one-form,
$\psi(\gamma\tau)/\psi(\tau)$ is constant. Define
\begin{equation}\label{II:eq:x023-character}
 \chi_\gamma=\frac{\psi(\tau)}{\psi(\gamma\tau)}.
\end{equation}
Then
$c_{14}(\gamma\tau)=j(\gamma,\tau)^{14}\chi_\gamma c_{14}(\tau)$.
This gives the character of the invariant line $T$.
To specify the full adjoint multiplier, put $v=b/a$ using
\eqref{II:eq:x023-cubic-primitive}, and define
\begin{equation}\label{II:eq:x023-full-multiplier}
 d_\gamma=v(\gamma\tau)-\chi_\gamma v(\tau),\qquad
 R_\gamma=
 \begin{pmatrix}\chi_\gamma&-d_\gamma/2\\0&1\end{pmatrix}.
\end{equation}
The numbers $d_\gamma$ are constant. Indeed the derivative of
$v$ transforms with weight two and character $\chi$, by
\eqref{II:eq:x023-cubic-primitive}. At zeros of $a$, the primitive
is interpreted meromorphically: it is the already defined
quotient $b/a$, so it has no logarithmic ambiguity.
Composition gives
$d_{\gamma\delta}=d_\gamma+\chi_\gamma d_\delta$.
Thus $R_{\gamma\delta}=R_\gamma R_\delta$, and direct conjugation
gives both transformation laws
\begin{align}\label{II:eq:x023-adjoint-laws}
 \widehat F(\gamma\tau)
  &=j(\gamma,\tau)^6 R_\gamma\widehat F(\tau)R_\gamma^{-1},
  \nonumber\\
 \widehat\Phi(\gamma\tau)
  &=j(\gamma,\tau)^{14}R_\gamma\widehat\Phi(\tau)R_\gamma^{-1}.
\end{align}
Scalar multiples of $R_\gamma$ have the same adjoint action, so
the transformation laws just obtained do not yet fix the
monodromy on $V_E$. The endpoint line connections determine the
remaining scalar. For the choice $\vartheta=\OO(6{P_\infty}-5{P_0})$ above,
take its rational frame $d$ with
$d^2=\mathfrak u^{-1}\omega$. The frames ${e_L}/d$ of $V_L$ and
$(\omega/{e_L})/d$ of $V_M$ have determinant $\mathfrak u$ and ratio $n$.
Their connection forms therefore have sum $d\log\mathfrak u$ and
difference $-d\log\psi$. The horizontal scalar factors are
\[
 \psi_L=(\psi/\mathfrak u)^{1/2}\sim q^4,
 \qquad \psi_M=(\mathfrak u\psi)^{-1/2}\sim q^7,
\]
with these leading terms fixing the branches on $\mathbb H$.
It follows that the full rank-two multiplier is
\begin{equation}\label{II:eq:x023-rank-two-multiplier}
 \rho_M(\gamma)=\frac{\psi_M(\tau)}{\psi_M(\gamma\tau)},
 \qquad \rho_E(\gamma)=\rho_M(\gamma)R_\gamma.
\end{equation}
Its determinant is one, and its adjoint action is that of
$R_\gamma$. Thus both the coefficient connection and its
monodromy representation have been specified.

The character $\chi$ has infinite image.
The factor $\mathfrak u^{3/11}$ has finite monodromy, whereas
$\alpha=(58\omega_1+86\omega_2)/11$ is a nonzero holomorphic
one-form on $X$. If exponentials of all its periods had finite
image, some positive integer multiple of $\alpha$ would have
all periods in $2\pi i\mathbb Z$. Its real part would then
have zero periods and be the differential of a single-valued
harmonic function on the compact curve. That function is
constant. The real part of $\alpha$, and hence $\alpha$
itself, would vanish, a contradiction.
Thus $c_{14}$ is a cusp form with the specified infinite-image
character $\chi$.

Finally, the rational expression exhibits the cusp orders
without a numerical approximation. The function $r$ has poles
of orders eight at ${P_\infty}$ and four at ${P_0}$, whereas
${e_L}^2\omega^6$ has divisor $10{P_\infty}+4{P_0}$. Therefore $c_\xi$
vanishes to order two at ${P_\infty}$ and is nonzero at ${P_0}$.
The normalized modular orders are nine and seven.
The coefficients through $q^{21}$ detect every nonzero section
of $TK^7$. Indeed, if they all vanish, its modular coefficient is
$O(q^{22})$. Since $(dt)^7=(dq/q)^7$, the intrinsic section
vanishes to order at least $22-7=15$ at ${P_\infty}$, exceeding
$\deg(TK^7)=14$. For the present one-dimensional image,
\eqref{II:eq:x023-fourier-recovery} reduces the test to $q^9$.

In this example the bracket is square zero and has zero trace
and determinant. Nevertheless
\[
 \widehat F\,\widehat\Phi=a\widehat\Phi,\qquad
 \widehat\Phi\,\widehat F=-a\widehat\Phi .
\]
Where $ac_{14}\ne0$, the identity and these two matrices span
the upper triangular algebra. The modular forms thus retain
the extension coordinate in a matrix entry that every scalar
characteristic invariant discards.

\clearpage
\part{Siegel modular varieties and Picard--Fuchs systems}\label{part:III}
\markboth{Part III: Siegel and period equations}{}
The universal expressions have two geometric realizations here.
Sections~\ref{III:sec:siegel-ore}--\ref{III:sec:schwarzians-modular-constructions}
evaluate them in the full symmetric tensor algebra of a Siegel modular
variety. Sections~\ref{III:sec:filtered-connections}--\ref{III:sec:scalarizations}
start instead from a flat symplectic bundle on a curve and a Lagrangian
subbundle. Differentiation of the filtration gives a matrix equation;
its Schwarzian controls the cyclicity divisors in
Theorem~\ref{intro:cyclicity}. Sections~\ref{III:sec:explicit-gm}--\ref{III:sec:igusa-applications}
compute the resulting scalar equations and modular divisors for
genus-two families.

A reader interested in the period equations may begin with
Section~\ref{III:sec:filtered-connections}, using the one-variable
normalization of Part~\ref{part:I}. The acceleration comparison in
Theorem~\ref{III:thm:acceleration-comparison} gives the precise
relation with the ambient construction. No canonical kernel on an
individual fibre is required: the filtered connection varies along
the base, whereas the kernel of Part~\ref{part:II} is attached to a
specified bundle on a fixed curve.

\section{Siegel modular bundles and the Ore calculus}
\label{III:sec:siegel-ore}

Differentiation changes the automorphic type of a modular form.
We include all the resulting symmetric tensors in one algebra, so that
covariant differentiation becomes a derivation. The bundle conventions
follow van der Geer \cite{vdG}, and the symmetric calculus follows
Yang--Yin \cite{YY}; see also the earlier differential operators of
Maass and Resnikoff \cite{Maass,Resnikoff}.

\subsection{Automorphic bundles and tensor types}
\label{III:subsec:automorphic-bundles}

The automorphy representation specifies the bundle in which each
coefficient takes values. Fix $g\geq1$, put $V=\C^g$, and let
\[
 \Hg=\{Z\in\operatorname{Mat}_g(\C):Z^t=Z,\ \operatorname{Im}Z>0\}.
\]
For $\gamma=\left(\begin{smallmatrix}a&b\\c&d\end{smallmatrix}\right)
\in\Sp_{2g}(\mathbb R)$, write
\begin{equation}\label{III:eq:siegel-action}
 \gamma Z=(aZ+b)(cZ+d)^{-1},
 \qquad J_\gamma(Z)=cZ+d.
\end{equation}
The factor of automorphy satisfies
$J_{\gamma_1\gamma_2}(Z)=
J_{\gamma_1}(\gamma_2Z)J_{\gamma_2}(Z)$.
Let $\Gamma\subset\Sp_{2g}(\Z)$ be a torsion-free congruence subgroup and set
$X=\Gamma\backslash\Hg$. For a general congruence subgroup the same
constructions are interpreted equivariantly on the modular orbifold.

A finite-dimensional rational representation
$\rho:\GL(V)\to\GL(V_\rho)$ defines an automorphic bundle $\VV_\rho$
on $X$. Its meromorphic sections are represented upstairs by maps
$f:\Hg\to V_\rho$ satisfying
\begin{equation}\label{III:eq:automorphy-rho}
 f(\gamma Z)=\rho(J_\gamma(Z))f(Z).
\end{equation}
We denote their space by $M_\rho^{\mer}(\Gamma)$, writing
$M_k^{\mer}=M_{\det^k}^{\mer}$ for scalar weight $k$.
Equivalently, they
are the fixed points of the slash action
\[
 (f|_\rho\gamma)(Z)=\rho(J_\gamma(Z))^{-1}f(\gamma Z).
\]
Throughout this section meromorphic sections are taken on the open
quotient $X$. No extension across a compactification boundary is
imposed on this space of meromorphic sections. Holomorphicity and
cusp conditions are asserted separately for the particular classical
modular forms used below.

The standard representation gives the Hodge bundle $\EE$, and we put
$\lambda=\det\EE$. The representation
$\det^k\otimes\Sym^mV$ corresponds to
$\lambda^k\otimes\Sym^m\EE$.
The cotangent bundle has the fundamental identification
\begin{equation}\label{III:eq:cotangent-hodge}
 \Omega_X^1\simeq\Sym^2\EE.
\end{equation}
To see the coordinate convention, write $dZ=(dZ_{ij})$ and define
the symmetric matrix derivative by
\begin{equation*}
 \partial_Z=
 \left(\frac{1+\delta_{ij}}{2}
              \frac{\partial}{\partial Z_{ij}}\right)_{i,j},
 \qquad df=\tr((\partial_Zf)\dZ).
\end{equation*}
Here $Z_{ij}$ with $i\leq j$ are the independent coordinates.
Differentiating \eqref{III:eq:siegel-action} gives
\begin{equation}\label{III:eq:dZ-law}
 d(\gamma Z)=J_\gamma^{-t}\dZ J_\gamma^{-1}.
\end{equation}
Thus $\tr(Q\dZ)$ descends precisely when the symmetric coefficient
matrix satisfies $Q(\gamma Z)=J_\gamma Q(Z)J_\gamma^t$, which is the
$\Sym^2V$ transformation law. We use unnormalized derivatives in the
geometric formulas; the Fourier normalization is $(2\pi i)^{-1}\partial_Z$.

We shall also use the following elementary evaluation principle
\cite[Observation 3.1]{CvG}. If
\[
 P:V_{\rho_1}\oplus\cdots\oplus V_{\rho_s}\longrightarrow V_\sigma
\]
is a $\GL(V)$-equivariant polynomial map and
$f_i\in M_{\rho_i}^{\mer}(\Gamma)$, then
$P(f_1,\ldots,f_s)\in M_\sigma^{\mer}(\Gamma)$.
Indeed, this follows by applying $P$ to
\eqref{III:eq:automorphy-rho}. In particular, the output is a scalar
determinant-weight form when $\sigma=\det^\ell$.

\subsection{A modular connection}
\label{III:subsec:modular-connection}

Differentiating the automorphy factor introduces an inhomogeneous term.
For example, if $f$ has scalar weight $k$, then
\begin{equation}\label{III:eq:scalar-anomaly}
 (\partial_Zf)(\gamma Z)
 =\det(J)^k
   \bigl(J(\partial_Zf)(Z)J^t+kf(Z)Jc^t\bigr),
 \qquad J=J_\gamma(Z).
\end{equation}
We remove this term by fixing a connection. The following normalization
is that of Yang--Yin \cite{YY}.

\begin{definition}\label{III:def:modular-matrix}
A \emph{modular connection matrix} is a meromorphic symmetric matrix
$G:\Hg\to\Sym^2V$ satisfying
\begin{equation}\label{III:eq:G-law}
 G(\gamma Z)=JG(Z)J^t+2Jc^t
 \qquad(\gamma\in\Gamma).
\end{equation}
\end{definition}

\begin{proposition}\label{III:prop:hodge-connection}
Put $\Theta_G=\tfrac12G\dZ$. Then
\begin{equation*}
 \nabla^G=d-\Theta_G
\end{equation*}
defines a meromorphic connection on $\EE$. On an associated bundle
$\VV_\rho$ its expression is
$d-\rho_*(\Theta_G)$, where $\rho_*$ is the induced Lie-algebra
representation. In particular, on $\lambda^k$,
\begin{equation}\label{III:eq:det-connection}
 \nabla^Gf=df-\frac{k}{2}\tr(G\dZ)f
          =\tr\left(\left(\partial_Zf-\frac{k}{2}Gf\right)\dZ\right).
\end{equation}
\end{proposition}

\begin{proof}
The symplectic identities give $Jc^t=cJ^t$, and $dJ=c\dZ$.
Equations \eqref{III:eq:dZ-law} and \eqref{III:eq:G-law} therefore imply
\[
 \gamma^*\Theta_G
 =J\Theta_GJ^{-1}+dJ\,J^{-1}.
\]
For a coefficient column transforming by $v(\gamma Z)=Jv(Z)$, this
is exactly the rule making $\nabla^Gv$ transform by $J$.
The associated representations give the remaining assertions;
on $\det^k$ the infinitesimal action is $k\tr$.
\end{proof}

This is a particular family of connections on $\EE$, sufficient for
the calculus below. It has an immediate meromorphic construction.
If $h$ is a nonzero scalar modular form of weight $w\ne0$, put
\begin{equation*}
 G_h=\frac{2}{w}\,h^{-1}\partial_Zh.
\end{equation*}
Equation \eqref{III:eq:scalar-anomaly} proves \eqref{III:eq:G-law}.
The same argument permits a locally constant character for $h$.
The connection is holomorphic wherever $h$ is holomorphic and
nonvanishing.

For a scalar form $f$ of weight $k$, set
$Q_f^G=\partial_Zf-\tfrac{k}{2}Gf$.
Its transformation law and determinant are
\begin{equation}\label{III:eq:det-derivative}
 Q_f^G(\gamma Z)=\det(J)^kJQ_f^G(Z)J^t,
 \qquad
 \det Q_f^G\in M_{gk+2}^{\mer}(\Gamma).
\end{equation}
This recovers the determinant construction of Yang--Yin and
illustrates the equivariant evaluation principle. In genus one,
choosing the elliptic discriminant $h=\Delta$ gives
$G_\Delta=(\pi i/3)E_2$, where $E_2$ is the normalized
weight-two Eisenstein series. With $D=(2\pi i)^{-1}d/d\tau$,
division of the coefficient in \eqref{III:eq:det-connection} by
$2\pi i$ gives the Serre derivative $Df-\tfrac{k}{12}E_2f$.

\subsection{The symmetric differential algebra}
\label{III:subsec:symmetric-algebra}

Successive derivatives add cotangent factors. Including their symmetric
powers allows the same differentiation rule to be iterated.
The connection on $\EE$ induces connections on $\lambda$, on all
tensor representations, and, through \eqref{III:eq:cotangent-hodge},
on $\Omega_X^1$. To include noncommutative coefficients, let $\AA$
be a locally free sheaf of finite-rank unital associative $\OO_X$-algebras,
equipped with a meromorphic \emph{algebra connection}:
\begin{equation}\label{III:eq:algebra-connection}
 \nabla^{\AA}(ab)=(\nabla^{\AA}a)b+a\nabla^{\AA}b.
\end{equation}
The basic example is $\AA=A\otimes_\C\OO_X$ for a
finite-dimensional associative algebra $A$, with entrywise
differentiation. A vector bundle $\mathcal W$ with connection
$d+B$ gives another example, $\AA=\End(\mathcal W)$, with
$\nabla^\AA a=da+[B,a]$.
We fix $G$ and $\nabla^\AA$ from now on.

Let $\MM_X$ be the sheaf of meromorphic functions and set
\begin{equation*}
 \RR=\bigoplus_{\substack{k\in\Z\\m,r\geq0}}\RR_{k,m,r},
 \qquad
 \RR_{k,m,r}
 =\MM_X\otimes_{\OO_X}
   \bigl(\AA\otimes\lambda^k\otimes\Sym^m\EE
                       \otimes\Sym^r\Omega_X^1\bigr).
\end{equation*}
Multiplication uses the given order in $\AA$ and the ordinary
commutative products in the symmetric and determinant factors.
In particular,
$\RR_{k,m,r}\RR_{\ell,n,s}\subset\RR_{k+\ell,m+n,r+s}$.
The symmetric variables commute with the coefficient algebra.

The tensor product connection induces
\begin{equation*}
 \delta=\operatorname{sym}\circ\nabla:
 \RR_{k,m,r}\longrightarrow\RR_{k,m,r+1},
\end{equation*}
where $\operatorname{sym}$ multiplies the new cotangent factor
into $\Sym^r\Omega_X^1$.
These are symmetric products, not exterior products.

\begin{proposition}\label{III:prop:graded-derivation}
The map $\delta$ is a $\C$-linear derivation of the unital
associative algebra $\RR$:
\begin{equation}\label{III:eq:delta-leibniz}
 \delta(ab)=\delta(a)b+a\delta(b).
\end{equation}
It is defined on $X$, hence commutes upstairs with the automorphy
actions. Its construction requires no flatness assumption.
\end{proposition}

\begin{proof}
The induced connections commute with tensor multiplication,
and \eqref{III:eq:algebra-connection} preserves the order of the
coefficient factors. Applying symmetric multiplication to their
Leibniz rule gives \eqref{III:eq:delta-leibniz}.
All the connections and multiplication maps descend to $X$.
\end{proof}

On scalar functions $\delta f=df$. On the symmetric algebra of
coordinate differentials, the induced connection has the
particularly simple expression
\begin{equation}\label{III:eq:YY-symmetric-calculus}
 \delta(dZ_{rs})=-\sum_{i,j=1}^gG_{ij}\,dZ_{ri}\,dZ_{sj},
 \qquad\text{or}\qquad
 \delta(dZ)=-dZ\,G\,dZ.
\end{equation}
This is Yang--Yin's symmetric calculus
\cite[Definition 3 and Theorem 2]{YY}.
Together with the Leibniz rule it determines differentiation of
all symmetric cotangent tensors. The notation $\delta^r$
always means iteration in this graded algebra, with the
connection acting on every tensor factor.

Retaining the whole symmetric algebra is important. For $g=2$,
\begin{equation}\label{III:eq:genus-two-decomposition}
 \Sym^2(\Omega_X^1)
 \simeq\Sym^2(\Sym^2\EE)
 \simeq\Sym^4\EE\oplus\lambda^2.
\end{equation}
The determinant component is therefore present already in the
second symmetric derivative. Subsequent equivariant tensor maps
may select either component. With $\AA$-valued inputs we apply
such tensor maps while keeping the coefficient multiplication
ordered. When $\AA=\End(\mathcal W)$, taking the coefficient
trace removes $\AA$ but leaves the automorphic tensor type.

\subsection{The highest symmetric-power component}
\label{III:subsec:highest-component}

Selecting the highest symmetric power replaces a cotangent tensor by
a polynomial in one auxiliary vector. This selection commutes with
multiplication and covariant differentiation. Indeed, symmetric
multiplication gives natural maps
\[
 \pi_{m,r}:
 \Sym^m\EE\otimes\Sym^r(\Sym^2\EE)
       \longrightarrow\Sym^{m+2r}\EE.
\]
Together they define an algebra map
\begin{equation*}
 \pi:\RR\longrightarrow
 \SSS=\bigoplus_{\substack{k\in\Z\\m\geq0}}\SSS_{k,m},
 \qquad
 \SSS_{k,m}=\MM_X\otimes_{\OO_X}
           \bigl(\AA\otimes\lambda^k\otimes\Sym^m\EE\bigr),
\end{equation*}
where $\RR_{k,m,r}$ maps to $\SSS_{k,m+2r}$.
On $\SSS_{k,m}$ define
$\delta^{\mathrm{top}}=\pi_{m,1}\circ\nabla$,
using \eqref{III:eq:cotangent-hodge} for the new cotangent factor.
Compatibility of the induced connections with symmetric
multiplication makes this a derivation
$\SSS_{k,m}\to\SSS_{k,m+2}$ and gives
\[
 \pi\delta=\delta^{\mathrm{top}}\pi.
\]

For its coordinate expression, realize $\Sym^mV$ as homogeneous
polynomials $F[u]$ in a column $u=(u_1,\ldots,u_g)^t$, with
$\Sym^m(M)F[u]=F[M^tu]$. A form of determinant weight $k$
then satisfies
\[
 F(\gamma Z)[u]=\det(J)^kF(Z)[J^tu].
\]
For constant associative coefficients with entrywise connection,
the projected derivative is
\begin{equation}\label{III:eq:projected-derivative}
 \boxed{\quad
 \delta^{\mathrm{top}}F[u]
 =u^t(\partial_ZF[u])u-\frac{k+m}{2}\,G[u]F[u],
 \qquad G[u]=u^tG u.
 \quad}
\end{equation}
Indeed, the determinant factor contributes
$-\tfrac{k}{2}G[u]F[u]$. Evaluating the Hodge connection in the
rank-one direction $uu^t$ gives the infinitesimal action of
$Guu^t/2$. In the polynomial model its contribution is
$-\tfrac12G[u]\sum_i u_i\partial F/\partial u_i
=-\tfrac m2G[u]F[u]$.
This proves the coefficient $k+m$ without a separate automorphy
calculation. Iteration uses the new symmetric degree after
each application.

The map $\pi$ discards representation components. In genus two
it annihilates the $\lambda^2$ summand in
\eqref{III:eq:genus-two-decomposition}; in coordinates,
$dZ_{ij}\mapsto u_i u_j$ sends
$dZ_{11}dZ_{22}-dZ_{12}^2$ to zero.
We shall use \eqref{III:eq:projected-derivative} when this component
is intended, and retain $\RR$ for the full tensor calculus.

\subsection{Ore operators and their action}
\label{III:subsec:ore-operators}

A derivation of an associative algebra determines an Ore extension.
For the symmetric calculus this is
\begin{equation}\label{III:eq:ore-extension}
 \mathscr O_\nabla=\RR[\Dt;\delta],
 \qquad \Dt a=a\Dt+\delta(a)\quad(a\in\RR).
\end{equation}
The symbol $\Dt$ is formal; its action on $\RR$ is $\delta$.
The normal-ordering convention and identities of
Section~\ref{I:sec:local-ore} apply to this differential algebra.

The same action is available on a left $\RR$-module $\mathcal N$
equipped with an operator $\delta_{\mathcal N}$ satisfying
\[
 \delta_{\mathcal N}(as)
    =\delta(a)s+a\delta_{\mathcal N}(s):
 \qquad \Dt\cdot s=\delta_{\mathcal N}(s).
\]
Thus the construction accommodates both algebra-valued and
vector-valued unknowns.

\begin{definition}\label{III:def:homogeneous-operator}
A \emph{homogeneous monic operator of order $n$} is an expression
\begin{equation*}
 L=\sum_{j=0}^n\binom nj a_j\Dt^{\,n-j},
 \qquad a_0=1,\qquad a_j\in\RR_{0,0,j}.
\end{equation*}
This is the binomial convention of \eqref{I:eq:local-L}, with the indicated tensor degrees.
\end{definition}

Both multiplication by $a_j$ and iteration of $\delta$ have a
specified symmetric cotangent degree. Consequently,
\begin{equation*}
 L:\RR_{k,m,r}\longrightarrow\RR_{k,m,r+n}.
\end{equation*}
If the coefficients are global sections, this is an operator on
the corresponding meromorphic automorphic sections. Under $\pi$,
$a_j$ has type $\AA\otimes\Sym^{2j}\EE$, and the projected
operator raises the symmetric $\EE$-degree by $2n$.
All Ore-algebra identities apply to the fixed differential
algebra $(\RR,\delta)$; dependence on its chosen connections
is a separate question.

For $g>1$, equation $LF=0$ is a tensor-valued system of partial
differential equations: in a local symmetric basis, each
coefficient of the output tensor must vanish. The Ore symbol represents the iterates of $\delta$.
Away from coefficient poles, the holomorphic solution sheaf is the
kernel of the resulting map of sheaves of complex vector spaces.
For example, with the flat cotangent connection on a $d$-dimensional
affine chart, $\delta^2f=0$ says that every second partial derivative
vanishes. Its local solutions are the affine functions, a space of
dimension $d+1$. Thus the Ore order alone does not determine the
solution rank. Restriction to geodesics is described by
\eqref{III:eq:geodesic-jet-restriction}.

On a parameterized curve, an algebra connection gives the
ordinary one-variable construction directly. If $\mathcal B$
is an associative coefficient bundle with connection, its
meromorphic sections form an algebra $\mathcal K$, and
$\nabla_t^{\mathcal B}$ is a derivation. The relevant Ore algebra is
\[
 \mathcal K[\partial_t;\nabla_t^{\mathcal B}],
 \qquad
 \partial_t a=a\partial_t+\nabla_t^{\mathcal B}a.
\]
This applies, in particular, to connections pulled back to a
curve wherever the pullback is defined. The acceleration of the curve relates symmetric derivatives to
ordinary directional derivatives. If $\nabla^T$ is the dual tangent
connection, then for a scalar function $f$ and a tangent field $v$,
\[
 (\delta^2f)(v,v)=v(v(f))-df(\nabla^T_vv).
\]
A Picard--Fuchs system on a curve will be treated using its
Gauss--Manin connection in the one-variable construction.

\Needspace{30\baselineskip}

\subsection{Notation for ambient tensors and period equations}
\label{III:subsec:notation}
The ambient construction uses tensor-valued derivatives. Along a
parameter curve these become ordinary matrix covariants.
\begin{center}
\small\renewcommand{\arraystretch}{1.16}
\begin{tabularx}{\textwidth}{@{}lXl@{}}
\toprule
Symbol & Meaning & Definition \\
\midrule
$\AA,\RR,\SSS$ & Coefficient algebra bundle, symmetric tensor algebra, and highest-component algebra & \S\ref{III:subsec:symmetric-algebra}, \S\ref{III:subsec:highest-component} \\
$G,\Theta$ & Ambient modular connection matrix; $\Theta=G[u]/2$ & \eqref{III:eq:G-law}, Definition~\ref{III:def:projected-derivative} \\
$\mathfrak d,\mathcal D_{\mathrm{raw}},w$ & The same local raw derivative; effective weight $w=k+m$ & Definition~\ref{III:def:projected-derivative} \\
$\HH,\EE,K_U$ & Flat symplectic bundle, Lagrangian subbundle, and $K_U=\Omega_U^1$ & \S\ref{III:subsec:filtered-input} \\
$\mathsf K$ & Symmetric pairing matrix $(Q(\omega_i,\nD\omega_j))$ & \eqref{III:eq:K-three} \\
$\cov,I,S$ & $\partial_t+[a_1,\cdot]$, matrix Schwarzian, and its trace-free part & \eqref{III:eq:I-S-D-three} \\
$q,C,c$ & $\tr(S^2)$, $[I,\cov I]$, and $\tr(C^2)$ & \eqref{III:eq:q-c-three} \\
$\operatorname{Wr}_s,\mathcal W(v)$ & Cyclic determinant of a section; primitive cyclic determinant of a Hodge wedge & \eqref{III:eq:cyclic-W-four}, \S\ref{III:subsec:primitive-four} \\
$\mathfrak j$ & Joint determinant of three binary quadratics & \eqref{III:eq:joint-J-four} \\
$\mathscr B_m(t_0)$ & Algebra generated by the Schwarzian jets through order $m$ at $t_0$ & Definition~\ref{III:def:jet-algebra} \\
$\rho,\delta_*$ & Scalar coefficients of the auxiliary equation for the trace-free entries & \eqref{III:eq:rho-delta-six} \\
$\alpha_Q,\mathfrak q$ & Connection-change one-form $\tr(Q\,dZ)$; binary quartic & \S\ref{III:subsec:full-connection-change}, \S\ref{III:subsec:binary-covariants-two} \\
\bottomrule
\end{tabularx}
\end{center}
$W_m$ continues to denote a higher Schwarzian. The Euler derivative
$\theta=z\partial_z$ is used only in the Yang--Zudilin dictionary.
The symbols $h,H$ are local inputs or pairings, with definitions in
their displayed formulas; $H_\pm$ and $H_4$ are distinct named
polynomials. The coefficient algebra $\mathcal K$ is not the pairing
matrix $\mathsf K$ or the canonical line $K_U$.

\section{Higher Schwarzians and invariant-theoretic modular constructions}
\label{III:sec:schwarzians-modular-constructions}

The normalized coefficients of an Ore operator transform by
conjugation under a change of coefficient frame. Covariant
differentiation and equivariant tensor maps then produce further
covariants, brackets, and scalar modular forms.
Throughout this section the connections of Section~\ref{III:sec:siegel-ore}
define the differential algebra $(\RR,\delta)$. We write
$M_{k,m}^{\mer}=M_{\det^k\otimes\Sym^mV}^{\mer}$, with determinant
weight first. Coefficients in $\AA$ are understood when specified.

\subsection{Realization in the symmetric tensor algebra}
\label{III:subsec:tensor-realization}

Apply Proposition~\ref{I:prop:local-normalization} and
Theorem~\ref{I:thm:rp-higher} to the differential algebra
$(\RR,\delta)$ of Section~\ref{III:sec:siegel-ore}. For a homogeneous
monic operator as in Definition~\ref{III:def:homogeneous-operator}, write
\begin{equation}\label{III:eq:adapted-derivation}
 \mathsf T=\Dt+a_1,\qquad
 \Delta_L C=[\mathsf T,C]=\delta C+[a_1,C].
\end{equation}
The normalized coefficients $I_j$ and the higher Schwarzians $W_j$
are the evaluations of the universal polynomials of
Part~\ref{part:I} in this algebra. Since $\Delta_L$ raises symmetric
cotangent degree by one and the universal formulas are homogeneous,
\[
 I_j\in\RR_{0,0,j}\quad(2\le j\le N),\qquad
 W_j\in\RR_{0,0,j}\quad(3\le j\le N).
\]
Their automorphy follows from the equivariance of $\delta$.
The gauge laws of Proposition~\ref{I:prop:local-covariance} and
\eqref{I:eq:local-Delta-law} give simultaneous conjugation of all their
covariant jets. Universal freeness in
Theorem~\ref{I:thm:local-generation} is asserted before evaluation;
the specialized tensors may satisfy additional relations.

Ordinary parameter covariance is a separate application of
Theorem~\ref{I:thm:rp-higher}. For a Picard--Fuchs equation it concerns
the one-variable curve operator of Section~\ref{III:subsec:ore-operators},
with its own coefficient connection. The automorphic tensor construction
does not identify differentiation along an arbitrary curve with
unrestricted ambient differentiation.

For $\AA=\End(\mathcal W)$, with its induced algebra connection,
\[
 \delta\tr C=\tr(\Delta_L C).
\]
The coefficient connection preserves trace and the commutator has
trace zero. Trace therefore commutes with covariant differentiation,
but retains the symmetric cotangent factors; scalar automorphic forms
require the further contractions developed below.

\subsection{The Serre comparison and the first bracket}
\label{III:subsec:serre-first-bracket}

Pairing the derivatives of two forms gives the simplest cancellation
of the modular connection. For a determinant-weight input
$f\in M_{k,0}^{\mer}(\Gamma;\AA)$, $\delta f$ has type
$\AA\otimes\lambda^k\otimes\Omega_X^1$.
If the coefficient algebra is constant with entrywise differentiation,
its matrix expression is
$Q_f^G=\partial_Zf-\tfrac{k}{2}Gf$.
Thus the determinant in \eqref{III:eq:det-derivative} is a polynomial
operation applied after a linear covariant derivative.

Two modular connection matrices differ by a tensor. Write
$G'=G+2Q$; then $Q(\gamma Z)=JQ(Z)J^t$. Keeping the coefficient
connection fixed gives
\begin{equation*}
 \delta_{G'}f=\delta_Gf-k\,\tr(Q\dZ)f.
\end{equation*}
Consequently the ordered first bracket
\begin{equation}\label{III:eq:first-bracket-full}
 \mathcal B_1(f,h)=k f\,\delta h-\ell(\delta f)h,
 \qquad f\in M_{k,0}^{\mer},\quad h\in M_{\ell,0}^{\mer},
\end{equation}
is independent of $G$ and lies in
$\AA\otimes\lambda^{k+\ell}\otimes\Omega_X^1$.
The cancellation uses the centrality of $\tr(Q\dZ)$ and preserves
the order of $f$ and $h$. With entrywise coefficients its matrix is
$k f\partial_Zh-\ell(\partial_Zf)h$.

For $g=1$, suppressing the cotangent symbol and using
$D=(2\pi i)^{-1}d/d\tau$, the choice $G=G_\Delta$ gives
\begin{equation*}
 \mathcal S_k f=Df-\frac{k}{12}E_2f,
 \qquad
 [f,h]_1=k f\mathcal S_\ell h-\ell(\mathcal S_kf)h
        =k fDh-\ell(Df)h.
\end{equation*}
The correction is holomorphic in this genus-one realization.
For $g\geq2$, Hofmann--Kohnen's theorem excludes a globally
holomorphic matrix $G$ with \eqref{III:eq:G-law} on a congruence quotient
\cite{HK}. It does not exclude holomorphic expressions such as
\eqref{III:eq:first-bracket-full}, in which the connection cancels.

There is a useful geometric obstruction behind this distinction.
A holomorphic $G$ would induce a holomorphic connection on
$\lambda=\det\EE$. Its restriction to a smooth complete curve
would be flat, since exterior two-forms on a curve vanish, and hence
would have degree zero. Any complete curve of positive Hodge degree
therefore obstructs such a connection. This is the line-bundle
case of the connection obstruction studied by Atiyah \cite{Atiyah};
Hofmann--Kohnen give the stated global nonexistence theorem for the
modular transformation law.

For gauge covariants the adapted derivative also includes
$[a_1,\cdot]$, as in \eqref{III:eq:adapted-derivation}. Below, independence
from $G$ always means that the inputs and their coefficient connection
are held fixed.

\subsection{Higher brackets in the highest symmetric-power component}
\label{III:subsec:higher-brackets}

\begin{definition}[Projected derivative and effective weight]\label{III:def:projected-derivative}

The projection of \S\ref{III:subsec:highest-component} gives a precise
comparison in every order. In the polynomial model put
\begin{equation*}
 \Theta=\frac12G[u],\qquad
 \mathfrak d F=\delta^{\mathrm{top}}F+(k+m)\Theta F
       \quad(F\in\SSS_{k,m}).
\end{equation*}
Locally $\mathfrak d$ is independent of $G$ and is a derivation.
For constant coefficients,
$\mathcal D_{\mathrm{raw}}F=u^t(\partial_ZF)u$, with $u$ held fixed. In general it also includes the fixed
algebra connection. The relevant integer for this projected calculus
is $w=k+m$; after one differentiation it becomes $w+2$.
Thus $\mathfrak d=\mathcal D_{\mathrm{raw}}$ on these local polynomial
representatives, with the coefficient connection included in both.
\end{definition}

The second derivative requires a scalar correction. Its transformation
law will also cancel the connection terms in higher orders.

\begin{lemma}\label{III:lem:R-theta}
The expression
\begin{equation}\label{III:eq:R-theta}
 R_\Theta=\mathcal D_{\mathrm{raw}}\Theta-\Theta^2
\end{equation}
is a scalar-coefficient modular tensor of type $\Sym^4\EE$.
For another modular connection, write $\Theta'=\Theta+q$.
Then $q\in M_{0,2}^{\mer}$ and
\begin{equation*}
 R_{\Theta'}=R_\Theta+\delta^{\mathrm{top}}q-q^2.
\end{equation*}
\end{lemma}

\begin{proof}
One may check the first assertion by differentiating
\eqref{III:eq:G-law}. The underlying one-variable calculation is as
follows. Fix $Z_0,u$, put $J_0=cZ_0+d$, $v=J_0^{-t}u$, and
$\kappa=u^tJ_0^{-1}cu$. The symplectic action sends the rank-one line
$Z(t)=Z_0+tuu^t$ to
\[
 \gamma Z(t)=\gamma Z_0+svv^t,
 \qquad s=\frac{t}{1+\kappa t}.
\]
Writing $\Theta_0(t)=\tfrac12G(Z(t))[u]$ and
$\Theta_1(s)=\tfrac12G(\gamma Z(t))[v]$, the connection law gives
\[
 \Theta_1(s)=(1+\kappa t)^2\Theta_0(t)+\kappa(1+\kappa t).
\]
Hence $d\Theta_1/ds-\Theta_1^2
=(1+\kappa t)^4(d\Theta_0/dt-\Theta_0^2)$.
At $t=0$, varying $u$ proves the asserted tensor law.
Finally, expand \eqref{III:eq:R-theta} after replacing $\Theta$ by
$\Theta+q$, and use
$\delta^{\mathrm{top}}q=\mathcal D_{\mathrm{raw}}q-2\Theta q$.
\end{proof}

The tensor $R_\Theta$ belongs to the symmetric algebra; exterior
curvature uses alternating products. For
$F\in M_{k,m}^{\mer}(\Gamma;\AA)$, define corrected derivatives
$F_{[r]}$ by
\begin{equation}\label{III:eq:corrected-jet-recurrence}
 \begin{aligned}
 F_{[0]}&=F,\qquad F_{[1]}=\delta^{\mathrm{top}}F,\\
 F_{[r+1]}&=\delta^{\mathrm{top}}F_{[r]}
              +r(w+r-1)R_\Theta F_{[r-1]}\quad(r\geq1).
 \end{aligned}
\end{equation}
Then $F_{[r]}\in M_{k,m+2r}^{\mer}(\Gamma;\AA)$.
In particular, $F_{[2]}=(\delta^{\mathrm{top}})^2F+wR_\Theta F$.
These are corrected derivatives, not just iterates of
$\delta^{\mathrm{top}}$.

\begin{lemma}[Closed corrected derivatives]
\label{III:lem:corrected-jet-closed}
For the recurrence~\eqref{III:eq:corrected-jet-recurrence}, with
$w=k+m$, one has
\begin{equation}\label{III:eq:corrected-jet-closed}
 F_{[r]}=\sum_{j=0}^r(-1)^j\binom rj
           (w+r-j)_{j}\,\Theta^j\mathfrak d^{r-j}F,
 \qquad (x)_j=x(x+1)\cdots(x+j-1),\quad(x)_0=1.
\end{equation}
\end{lemma}
\begin{proof}
The cases $r=0,1$ are the definitions. In the induction step use
$\delta^{\mathrm{top}}F_{[r]}=
\mathfrak dF_{[r]}-(w+2r)\Theta F_{[r]}$ and
$\mathfrak d\Theta=R_\Theta+\Theta^2$.
The terms containing $R_\Theta$ cancel because
\[
 j\binom rj(w+r-j)_j
 =r(w+r-1)\binom{r-1}{j-1}(w+r-j)_{j-1}
 \qquad(1\leq j\leq r).
\]
Pascal's identity combines the remaining powers of $\Theta$ into
the formula for $r+1$. These are polynomial identities in $w$,
so the induction includes zero and negative weights.
\end{proof}

\begin{proposition}\label{III:prop:projected-RC}
Let $F\in M_{k,m}^{\mer}$ and $H\in M_{\ell,p}^{\mer}$, and put
$w=k+m$, $v=\ell+p$. For $n\geq0$, the expression
\begin{equation}\label{III:eq:corrected-bracket}
 \mathcal B_n(F,H)
 =\sum_{i=0}^n(-1)^i
       \binom{w+n-1}{n-i}\binom{v+n-1}{i}
       F_{[i]}H_{[n-i]}
\end{equation}
belongs to $M_{k+\ell,m+p+2n}^{\mer}(\Gamma;\AA)$.
It is independent of $G$, with the inputs and their coefficient
connection fixed, and has the local formula
\begin{equation}\label{III:eq:raw-bracket}
 \mathcal B_n(F,H)
 =\sum_{i=0}^n(-1)^i
       \binom{w+n-1}{n-i}\binom{v+n-1}{i}
       (\mathfrak d^i F)(\mathfrak d^{n-i}H).
\end{equation}
Binomial coefficients have their polynomial meaning for arbitrary
integer weights. Products retain the indicated order.
\end{proposition}

\begin{proof}
\noindent\textit{1. Tensor type.} The tensor type follows from the
recurrence and Lemma~\ref{III:lem:R-theta}.

\noindent\textit{2. Cancel the scalar connection.} By
Lemma~\ref{III:lem:corrected-jet-closed},

as a polynomial in $\Theta$ with the
$\mathfrak d$-jets fixed,
$\partial F_{[r]}/\partial\Theta=-r(w+r-1)F_{[r-1]}$.
Let $c_i$ be the coefficient of $F_{[i]}H_{[n-i]}$ in
\eqref{III:eq:corrected-bracket}. The identity
\[
 (i+1)(w+i)c_{i+1}+(n-i)(v+n-i-1)c_i=0
\]
shows that consecutive terms in
$\partial\mathcal B_n/\partial\Theta$ cancel.
\noindent\textit{3. Recover the raw formula.} Thus the bracket is independent of $\Theta$; setting $\Theta=0$
in \eqref{III:eq:corrected-jet-closed} proves \eqref{III:eq:raw-bracket}.
This is a polynomial identity, so the argument covers all weights
without division by a weight. It also applies to associative
coefficients, since only the central factors have been moved.
\end{proof}

For scalar inputs, \eqref{III:eq:raw-bracket} is the symmetric-power
Rankin--Cohen construction of Eholzer--Ibukiyama
\cite[\S6, Proposition 6.1]{EI}, with the displayed normalization.
The recurrence provides its expression in the present connection
calculus. For entrywise coefficients it is holomorphic on $\Hg$
when the inputs are holomorphic; the poles of $G$ have canceled.
It produces the specified highest component. Other components of
higher tensor derivatives require their own equivariant contractions.

The need for the correction is already visible in degree two.
Write $\partial=\delta^{\mathrm{top}}$, with the type updated at
each iteration. Then
\begin{equation*}
 \begin{aligned}
 \mathcal B_2(F,H)={}&
   \binom{w+1}{2}F\partial^2H
   -(w+1)(v+1)(\partial F)(\partial H)
   +\binom{v+1}{2}(\partial^2F)H\\
 &\quad+\frac{wv(w+v+2)}2 R_\Theta FH.
 \end{aligned}
\end{equation*}
For $g=1$ and Fourier-normalized derivatives,
$\Theta=E_2/12$ and
$R_\Theta=D(E_2/12)-(E_2/12)^2=-E_4/144$.
Thus the last term becomes
$-k\ell(k+\ell+2)E_4fh/288$ for scalar inputs.
This is the correction to the expression obtained from unmodified
iterates of the Serre derivative. More generally,
$(2\pi i)^{-n}\mathcal B_n(f,h)$ in the unnormalized convention
of \eqref{III:eq:raw-bracket} is the classical elliptic bracket. For the established elliptic theory
of covariant higher derivatives and modular differential operators,
see \cite{Zagier1994,NSZ}.

The same construction can be applied to gauge covariants of $L$:
replace the algebra connection by
$\nabla^{\AA}+[a_1,\cdot]$, so that the projected derivative is
$\Delta^{\mathrm{top}}=\delta^{\mathrm{top}}
+[\pi(a_1),\cdot]$.
The scalar tensor $R_\Theta$ is unchanged, and the recurrence and
its brackets preserve conjugation covariance. As before, connection independence of the operation is asserted with
its inputs fixed; an input $I_j(L)$ or $W_j(L)$ may itself depend on $G$.

\subsection{Connection changes in the full symmetric calculus}
\label{III:subsec:full-connection-change}

The determinant summand of \eqref{III:eq:genus-two-decomposition} requires
a different cancellation from the highest symmetric-power summand.
We give the complete order-two calculation for determinant-weight
inputs, allowing the fixed associative coefficient connection of
Section~\ref{III:sec:siegel-ore}. All products of differential symbols in
this subsection are symmetric.

Write $H=G/2$. For $G'=G+2Q$, put
\[
 {\alpha_Q}=\tr(Q\dZ),\qquad r_Q=\tr(Q\dZ Q\dZ).
\]
Here $Q$ has scalar entries and ${\alpha_Q}$ is a meromorphic one-form.
The difference of the two derivations is an $\OO_X$-linear derivation:
\begin{equation}\label{III:eq:full-derivation-change}
 \delta_{G'}=\delta_G-\mathcal A_Q.
\end{equation}
On a determinant-weight input $f$, and on the cotangent symbols,
\[
 \mathcal A_Q f=k{\alpha_Q}f,\qquad
 \mathcal A_Q(dZ)=2dZ\,Q\,dZ.
\]
On other associated bundles it is the symmetrized infinitesimal
action of $QdZ$. Thus this rule determines the difference on the
whole algebra $\RR$, with the coefficient connection fixed.
For a one-form $\alpha=\tr(PdZ)$, define
\[
 \mathcal C_Q(\alpha)=2\tr(PdZQdZ).
\]
The trace here contracts the $g\times g$ indices. If $P$ has entries
in $\AA$, it does not contract the coefficient algebra.

\begin{proposition}\label{III:prop:full-second-change}
For $f\in M_{k,0}^{\mer}(\Gamma;\AA)$ and $\alpha=\delta_Gf$,
\begin{align}\label{III:eq:full-second-change}
 \delta_{G'}^2f={}&\delta_G^2f-2k{\alpha_Q}\alpha-\mathcal C_Q(\alpha)
                   -k(\delta_G{\alpha_Q})f\notag\\
                 &\hspace{18mm}+(k^2{\alpha_Q}^2+2kr_Q)f.
\end{align}
More generally, changing the connection changes the first $r$
covariant jets by a triangular transformation with identity leading
term; its coefficients use derivatives of $Q$ through order $r-1$.
\end{proposition}
\begin{proof}
Apply \eqref{III:eq:full-derivation-change} twice. Since
\[
 \mathcal A_Q\alpha=k{\alpha_Q}\alpha+\mathcal C_Q(\alpha),\qquad
 \mathcal A_Q{\alpha_Q}=2r_Q,
\]
the Leibniz rule gives \eqref{III:eq:full-second-change}.
Iteration proves the triangular assertion: the next derivative
raises the highest derivative of the connection difference by at
most one. None of these operations changes the leading ordinary jet.
\end{proof}

There is a full symmetric correction which removes $\delta_G{\alpha_Q}$
from this formula. Define
\begin{equation}\label{III:eq:full-Riccati}
 \mathcal K_G=\tr\bigl((dH-HdZ H)dZ\bigr).
\end{equation}
The entries of $dH$ are ordinary differentials; the final product is
symmetrized. This tensor extends $R_\Theta$:
$\pi\mathcal K_G=R_\Theta$.

\begin{lemma}\label{III:lem:full-Riccati}
The expression $\mathcal K_G$ is a meromorphic section of
$\Sym^2\Omega_X^1$. Under a change $G'=G+2Q$,
\begin{equation*}
 \mathcal K_{G'}=\mathcal K_G+\delta_G{\alpha_Q}-r_Q.
\end{equation*}
\end{lemma}
\begin{proof}
Put $B_H=dH-HdZ H$, a symmetric-matrix-valued one-form.
The affine law for $H$ is $H(\gamma Z)=JHJ^t+Jc^t$.
Differentiating this identity, and using $dJ=c\,dZ$, gives
\[
 \gamma^*B_H=JB_HJ^t:
\]
the three extra terms in $d(JHJ^t+Jc^t)$ cancel the cross terms
in $(JHJ^t+Jc^t)d(\gamma Z)(JHJ^t+Jc^t)$.
Pairing with $d(\gamma Z)=J^{-t}dZJ^{-1}$ and symmetrizing proves
descent. Replacing $H$ by $H+Q$ gives
\[
 B_{H+Q}-B_H=dQ-HdZQ-QdZH-QdZQ.
\]
Its trace pairing with $dZ$ is exactly
$\delta_G{\alpha_Q}-r_Q$, by \eqref{III:eq:YY-symmetric-calculus}.
\end{proof}

In particular, $\mathcal K_G$ is not the exterior curvature
$(\nabla^G)^2=-d(HdZ)+(HdZ)\wedge(HdZ)$.
Their tensor symmetries and their restrictions to curves differ.
For a determinant-weight input define
\begin{equation*}
 \operatorname{Hess}_{k,G}(f)=\delta_G^2f+k\mathcal K_G f.
\end{equation*}
Proposition~\ref{III:prop:full-second-change} and
Lemma~\ref{III:lem:full-Riccati} imply
\begin{equation}\label{III:eq:corrected-Hessian-change}
 \operatorname{Hess}_{k,G'}(f)=\operatorname{Hess}_{k,G}(f)
 -2k{\alpha_Q}\alpha-\mathcal C_Q(\alpha)+(k^2{\alpha_Q}^2+kr_Q)f.
\end{equation}
This transformation contains $Q$ but no derivatives of $Q$.
The cancellation can also be read directly in coordinates. In an
entrywise coefficient frame put
$h=\tr(HdZ)$ and let $d_0$ differentiate the coefficient functions
while keeping the symmetric symbols $dZ_{ij}$ fixed. Direct expansion
gives
\begin{equation}\label{III:eq:corrected-Hessian-raw}
 \operatorname{Hess}_{k,G}(f)=d_0^2f-2\tr((\partial_Zf)dZHdZ)
 -2kh\,d_0f+\bigl(k^2h^2+k\tr(HdZHdZ)\bigr)f.
\end{equation}
This also makes the absence of derivatives of $H$ explicit.

\subsection{Both genus-two components of the second bracket}
\label{III:subsec:both-second-brackets}

In genus two the two projections have different scalar correction
factors. We compute them before forming the brackets. Assume $g=2$,
and let $\Pi_+$ and $\Pi_-$ be the equivariant maps from
$\Sym^2(\Sym^2\EE)$ to $\Sym^4\EE$ and $\lambda^2$, respectively.
The first is evaluation at $dZ=uu^t$. To fix the second normalization,
replace $dZ$ by $X=\left(\begin{smallmatrix}x&y\\y&z\end{smallmatrix}\right)$.
For a homogeneous quadratic polynomial $F(X)$, set
\begin{equation*}
 \Pi_-F=\frac12\left(\partial_x\partial_z
                         -\frac14\partial_y^2\right)F.
\end{equation*}
The value represents the coefficient of a section of $\lambda^2$.
Its embedding back into $\Sym^2\Omega^1$ is
$(4/3)(\Pi_-F)\det(dZ)$: indeed
$\Pi_-(xz-y^2)=\tfrac12(1+\tfrac12)=3/4$, which accounts for
the factor $4/3$. In particular,
\[
 \Pi_-\bigl(\tr(PdZ)\tr(QdZ)\bigr)
 =\langle P,Q\rangle
 :=\tfrac12(P_{11}Q_{22}+P_{22}Q_{11})-P_{12}Q_{12}.
\]
For associative entries this expression retains the displayed order.

The different corrections in the two components follow from the
two-dimensional matrix identity
\[
 \tr(PXQX)=\tr(PX)\tr(QX)-2\langle P,Q\rangle\det X,
\]
used below with $Q$ scalar. It gives
\begin{equation*}
 \begin{array}{c|cc}
 &\Pi_+&\Pi_-\\ \hline
 \mathcal C_Q(\alpha)&2\Pi_+({\alpha_Q}\alpha)&-\Pi_-({\alpha_Q}\alpha)\\
 r_Q&\Pi_+({\alpha_Q}^2)&-\tfrac12\Pi_-({\alpha_Q}^2).
 \end{array}
\end{equation*}
Put $\varepsilon_+=1$ and $\varepsilon_-=-1/2$.
Equation~\eqref{III:eq:corrected-Hessian-change} therefore reads
\begin{equation}\label{III:eq:component-Hessian-change}
 \Pi_\pm\operatorname{Hess}_{k,G'}(f)
 =\Pi_\pm\operatorname{Hess}_{k,G}(f)
 -2(k+\varepsilon_\pm)\Pi_\pm({\alpha_Q}\alpha)
 +k(k+\varepsilon_\pm)\Pi_\pm({\alpha_Q}^2)f.
\end{equation}

\begin{proposition}[The two second-order brackets]
\label{III:prop:full-second-brackets}
Let $f\in M_{k,0}^{\mer}(\Gamma;\AA)$ and
$h\in M_{\ell,0}^{\mer}(\Gamma;\AA)$, with their coefficient
connection fixed. For either sign, the expression
\begin{align}\label{III:eq:full-component-bracket}
 \mathcal B_{2,\pm}(f,h)=\Pi_\pm\bigl[&
 \tfrac12 k(k+\varepsilon_\pm)f\operatorname{Hess}_{\ell,G}(h)
 -(k+\varepsilon_\pm)(\ell+\varepsilon_\pm)
                         (\delta_Gf)(\delta_Gh)\notag\\
 &+\tfrac12\ell(\ell+\varepsilon_\pm)
                              \operatorname{Hess}_{k,G}(f)h\bigr]
\end{align}
is independent of $G$. The targets are, respectively,
$M_{k+\ell,4}^{\mer}(\Gamma;\AA)$ and
$M_{k+\ell+2,0}^{\mer}(\Gamma;\AA)$.
For entrywise coefficients its formula is obtained by replacing
$\operatorname{Hess}_{k,G}(f),\delta_Gf$ by $d_0^2f,d_0f$ in
\eqref{III:eq:full-component-bracket}.
\end{proposition}
\begin{proof}
Automorphy follows from the tensor types. Write
$\alpha=\delta_Gf$, $\beta=\delta_Gh$, and
$\varepsilon=\varepsilon_\pm$. Under a connection change these
become $\alpha-k{\alpha_Q}f$ and $\beta-\ell {\alpha_Q}h$.
Substitute \eqref{III:eq:component-Hessian-change}. The coefficients
of $\Pi_\pm({\alpha_Q}f\beta)$ and $\Pi_\pm({\alpha_Q}\alpha h)$ cancel separately.
The remaining coefficient of $\Pi_\pm({\alpha_Q}^2)fh$ is
\[
 \tfrac12k\ell(k+\varepsilon)(\ell+\varepsilon)
 -k\ell(k+\varepsilon)(\ell+\varepsilon)
 +\tfrac12k\ell(k+\varepsilon)(\ell+\varepsilon)=0.
\]
Only central factors have been moved, so the proof applies to ordered
associative products. Setting $H=0$ in the local polynomial identity
\eqref{III:eq:corrected-Hessian-raw} gives the raw formula.
\end{proof}

For example, writing $Z=\left(\begin{smallmatrix}x&y\\y&z\end{smallmatrix}\right)$
and $\mathscr L=\partial_x\partial_z-\tfrac14\partial_y^2$, the
determinant component for entrywise coefficients is
\begin{equation*}
 \mathcal B_{2,-}(f,h)=
 \frac{k(k-\tfrac12)}2f\mathscr Lh
 -(k-\tfrac12)(\ell-\tfrac12)\langle\partial_Zf,\partial_Zh\rangle
 +\frac{\ell(\ell-\tfrac12)}2(\mathscr Lf)h.
\end{equation*}
The scalar cases are established Siegel Rankin--Cohen operators:
the plus component is \cite[Proposition 6.1]{EI}, and the minus
component is the genus-two, scalar weight-increment-two case of
\cite[\S5.1]{EI}. More precisely, it is $1/8$ of the displayed
$Q_{2,2}(\partial_{Z_1},\partial_{Z_2})(f(Z_1)h(Z_2))$ there,
restricted to the diagonal, with $d_1=2k$, $d_2=2\ell$.
For symmetric matrices $R,S$, this polynomial is
\begin{align*}
 Q_{2,2}(R,S)={}&d_2(d_2-1)\det R+d_1(d_1-1)\det S\\
 &-(d_1-1)(d_2-1)
       (R_{11}S_{22}+R_{22}S_{11}-2R_{12}S_{12}).
\end{align*}
Substitution of $d_1=2k,d_2=2\ell$ and the normalized off-diagonal
matrix derivatives makes the factor $1/8$ explicit.
The calculation here gives their simultaneous expression in the full
modular-connection calculus, including ordered coefficient products.
Within the three-term ansatz in \eqref{III:eq:full-component-bracket},
the two linear cancellation equations determine the coefficients
up to scale when $k\ell(k+\varepsilon)(\ell+\varepsilon)\ne0$.
For higher orders and other input representations, the corresponding
problem is to find corrected jets whose connection-change laws contain
no derivatives of $Q$. Such formulas would reduce the construction of
independent brackets to algebraic cancellation, as above.

\subsection{A quadratic Schwarzian tensor from modular geodesics}
\label{III:subsec:geodesic-Schwarzian}

The classical matrix Schwarzian of a curve involves the inverse of
its velocity matrix. On geodesics of a modular connection that inverse
cancels, giving a tensor in the full symmetric calculus. Fix $G$, put $H=G/2$, and in this subsection take
$\AA=\End(\EE^\vee)$ with the connection induced by $\nabla^G$.
The dual tangent connection is torsion-free: in symmetric matrix
coordinates its Christoffel term is
$\Gamma_G(U,V)=UHV+VHU$. Its affinely parameterized geodesics satisfy
\begin{equation}\label{III:eq:modular-geodesic}
 Z''+2Z'HZ'=0.
\end{equation}
All local statements take place away from the polar locus of $G$.
For every initial point and symmetric initial velocity there is a
unique local holomorphic geodesic.

\Needspace{15\baselineskip}
\begin{theorem}[Geodesic Schwarzian tensor]
\label{III:thm:geodesic-Schwarzian}
For a symmetric tangent matrix $V$, put
\begin{equation}\label{III:eq:geodesic-Schwarzian}
 \boxed{\quad
 \mathbb I_G(V,V)=-V\bigl(dH[V]-HVH\bigr).
 \quad}
\end{equation}
This quadratic polynomial in $V$ defines a meromorphic section of
$\End(\EE^\vee)\otimes\Sym^2\Omega_X^1$. On a geodesic with
$\det Z'\ne0$, its contraction is the matrix Schwarzian
\begin{equation}\label{III:eq:geodesic-graph-comparison}
 \mathbb I_G(Z',Z')
 =\frac12Z'''(Z')^{-1}
       -\frac34\bigl(Z''(Z')^{-1}\bigr)^2.
\end{equation}
Thus it extends that geodesic expression across degenerate initial
velocities. Its scalar trace is $-\mathcal K_G$.
If $\widehat G=G+2Q$, then
\begin{align}\label{III:eq:geodesic-connection-change}
 \mathbb I_{\widehat G}(V,V)-\mathbb I_G(V,V)
 &=-V\left(\nabla^{G,\Sym^2\EE}_VQ-QVQ\right),\\
 \nabla^{G,\Sym^2\EE}_VQ&=dQ[V]-HVQ-QVH.\notag
\end{align}
\end{theorem}
\begin{proof}
Write $V=Z'$. Differentiating \eqref{III:eq:modular-geodesic} gives
\[
 V'=-2VHV,\qquad
 Z'''=8VHVHV-2VdH[V]V.
\]
Substitution in \eqref{III:eq:geodesic-graph-comparison} cancels the
inverse $V$ and gives \eqref{III:eq:geodesic-Schwarzian}. The right-hand side is quadratic in $V$ and therefore extends to every
velocity.

For descent use the tensor $B_H=dH-HdZH$ of
Lemma~\ref{III:lem:full-Riccati}. Under a modular transformation,
$V$ becomes $J^{-t}VJ^{-1}$ and $B_H(V)$ becomes $JB_H(V)J^t$.
Their product therefore transforms by
\[
 \mathbb I_G(V,V)\longmapsto
 J^{-t}\mathbb I_G(V,V)J^t.
\]
This is the endomorphism law for $\EE^\vee$. Polarizing in $V$
gives the asserted symmetric tensor. The trace is the negative of
\eqref{III:eq:full-Riccati}. Finally expand $H+Q$ in
\eqref{III:eq:geodesic-Schwarzian}; the linear and quadratic terms are
exactly \eqref{III:eq:geodesic-connection-change}.
\end{proof}

The transpose $\mathbb I_G^t$ has coefficient bundle $\End(\EE)$.
We retain the dual convention because period unknowns are sections
of $\EE^\vee$; see \eqref{III:eq:period-operator-bundle-three}.
Its dependence on $G$ is given by
\eqref{III:eq:geodesic-connection-change}. Here the induced coefficient
connection varies with $G$ as well; Proposition~\ref{III:prop:full-second-change}
held that connection fixed.

The relation with one-variable operators is exact on geodesics.
If $T$ is a coefficient-valued symmetric $p$-tensor and $\gamma$
is an affine geodesic, then
\begin{equation}\label{III:eq:geodesic-jet-restriction}
 (\delta_G^rT)(\dot\gamma^{\,p+r})
 =\bigl(\nabla_t^{\rm coef}\bigr)^r
                 \bigl(T(\dot\gamma^{\,p})\bigr).
\end{equation}
Indeed the derivative of a tensor contraction has, in addition to
the coefficient derivative, one term for each covariant derivative
of $\dot\gamma$; all those terms vanish. Induction proves the formula.
Consequently a homogeneous tensor equation $LF=0$ restricts to its
contracted Ore equation on every geodesic. Conversely, if these
equations hold along all geodesics through a point and for every
initial velocity, polarization makes every symmetric component of
$LF$ vanish there.

Put $\mathbb S_G=\mathbb I_G-g^{-1}\tr(\mathbb I_G)\Id$. Further
automorphic tensors are
\begin{equation*}
 \mathbb C_G=[\mathbb S_G,\delta_G\mathbb S_G],\qquad
 \mathbf{q}_G=\tr(\mathbb S_G^2),\qquad
 \mathbf{c}_G=\tr(\mathbb C_G^2).
\end{equation*}
Their symmetric cotangent degrees are $5,4,10$, respectively; the
last two have scalar coefficients. The comparison with the invariants of a period curve is made in
Section~\ref{III:subsec:ambient-acceleration}, after the curve operator
has been constructed. The degree-zero rational
expression $\mathbf{c}_G^2/\mathbf{q}_G^5$, where defined, is a
function on $\mathbb P(T_X)$; independence of the tangent direction
would be an additional condition for descent to $X$.

In genus two the coefficient representation is explicit:
\begin{equation*}
 \End_0(\EE^\vee)\otimes\Sym^2(\Sym^2\EE)
 \simeq(\Sym^6\EE\otimes\lambda^{-1})
       \oplus\Sym^4\EE
       \oplus(\Sym^2\EE\otimes\lambda)^{\oplus2}.
\end{equation*}
This follows from $\End_0(\EE^\vee)\simeq\Sym^2\EE\otimes
\lambda^{-1}$ and the Clebsch--Gordan formula. Its four components
have types $(-1,6),(0,4),(1,2),(1,2)$ in our weight convention.
Thus binary invariant theory applies separately to these four
components.

\subsection{A nonzero third Schwarzian on Siegel space}
\label{III:subsec:cubic-Siegel-example}

The quadratic tensor gives an order-three Ore operator. Its third
Schwarzian has a pole along the separating divisor for the connection
defined by $\chi_{10}$, and hence is nonzero. Set
\begin{equation*}
 \mathsf L_G=\Dt^3+3\mathbb I_G\Dt+\delta_G\mathbb I_G.
\end{equation*}
It is homogeneous, with $a_1=0$, $I_2=\mathbb I_G$ and
$I_3=\delta_G\mathbb I_G$. Its third Schwarzian is therefore
\begin{equation}\label{III:eq:cubic-Siegel-W3}
 W_3(\mathsf L_G)=-\tfrac12\delta_G\mathbb I_G
 \in H^0\bigl(X,\End(\EE^\vee)\otimes
                                      \Sym^3\Omega_X^1\bigr)^{\mer}.
\end{equation}

\begin{theorem}\label{III:thm:cubic-Siegel-nonzero}
Let $\chi_{10}$ be the genus-two cusp form of weight ten and choose
$G=G_{\chi_{10}}=\tfrac15\partial_Z\log\chi_{10}$.
At a generic separating point, write
$Z=\left(\begin{smallmatrix}\tau_1&z\\z&\tau_2\end{smallmatrix}\right)$
and $E=\left(\begin{smallmatrix}0&1\\1&0\end{smallmatrix}\right)$.
Then
\begin{equation}\label{III:eq:cubic-Siegel-pole}
 W_3(\mathsf L_G)(E,E,E)
       =\frac{33}{250z^3}\Id+O(z^{-2}).
\end{equation}
In particular, this higher Schwarzian is nonzero.
\end{theorem}
\begin{proof}
The separating expansion is
$\chi_{10}=c z^2\Delta(\tau_1)\Delta(\tau_2)+O(z^4)$, with $c\ne0$
\cite[\S9]{vdG}. The factor $1/2$ in the off-diagonal entry of
$\partial_Z$ therefore gives
\[
 H=\frac{\alpha}{z}E+O(1),\qquad \alpha=\frac1{10}.
\]
Theorem~\ref{III:thm:geodesic-Schwarzian} gives
$\mathbb I_G(E,E)=\alpha(1+\alpha)z^{-2}\Id+O(z^{-1})$.
For the leading term, a geodesic with velocity $vE$ satisfies
$z'=v$ and $v'=-2\alpha v^2/z$. Thus differentiating
$\alpha(1+\alpha)v^2/z^2$ along that geodesic gives
$-2\alpha(1+\alpha)(1+2\alpha)v^3/z^3$.
The leading coefficient commutator vanishes because this term is
scalar. All bounded terms of $H$ contribute lower pole orders.
Apply \eqref{III:eq:geodesic-jet-restriction} and
\eqref{III:eq:cubic-Siegel-W3}, and substitute $v=1$, $\alpha=1/10$.
\end{proof}

This example is an endomorphism-valued meromorphic tensor for the
chosen connection. The pole in \eqref{III:eq:cubic-Siegel-pole} also
specifies its failure to extend holomorphically across the separating
divisor.

\subsection{Equivariant evaluation of differential tensors}
\label{III:subsec:differential-invariant-evaluation}

To obtain a scalar form from a differential tensor, one must contract
its automorphic representation as well as its coefficient algebra.
For a form $f$ of type $\rho$, its $r$th covariant derivative has fiber
representation
\begin{equation}\label{III:eq:jet-representation}
 T_{\rho,r}=V_\rho\otimes\Sym^r(\Sym^2V).
\end{equation}
For a Schwarzian covariant, the corresponding coefficient type is
$\AA\otimes\Sym^{j+r}(\Sym^2\EE)$ for $\Delta^rW_j$;
the same statement holds for $\Delta^rI_j$.
One can decompose these representations before forming further
covariants. For example, in genus two a second derivative has both
the $\Sym^4\EE$ and $\lambda^2$ components of
\eqref{III:eq:genus-two-decomposition}.

\begin{proposition}\label{III:prop:differential-evaluation}
Let $C_1,\ldots,C_s$ be scalar-coefficient modular tensors obtained
from covariant derivatives or scalar contractions of Schwarzian
covariants, with fiber representations $T_1,\ldots,T_s$.
Every $\GL(V)$-equivariant polynomial map
\[
 \Phi:T_1\oplus\cdots\oplus T_s\longrightarrow V_\sigma
\]
gives, by evaluation, a meromorphic modular form of type $\sigma$.
For matrix-valued Schwarzian inputs, take instead a polynomial map
from $\bigoplus_i(\operatorname{Mat}_d\otimes T_i)$ to $V_\sigma$.
If it is invariant under simultaneous coefficient conjugation and
equivariant under $\GL(V)$, its value is a gauge-invariant modular
form of type $\sigma$.
If $\sigma=\det^a$, its value is a scalar form of weight $a$.
\end{proposition}

\begin{proof}
Each input has its stated automorphy law by the equivariance of
covariant differentiation. Applying $\Phi$ to these laws proves
the first assertion, exactly as in
\cite[Observation 3.1]{CvG}. For matrix-valued Schwarzian inputs,
\eqref{I:eq:local-Delta-law} and Proposition~\ref{I:prop:local-covariance} supply simultaneous conjugation; its
invariance under $\Phi$ proves the second assertion.
\end{proof}

For $\AA=\End(\mathcal W)$ of matrix size $d$, the two relevant
groups are $\GL(V)$ and $\GL_d$. A practical construction first
uses the coefficient contraction
\[
 A_1\otimes\cdots\otimes A_q\longmapsto
                      \tr(A_1\cdots A_q),
\]
leaving the automorphic tensor factors separate, and then applies
an equivariant map to those factors. Products of such trace
contractions provide the familiar polynomial invariants of
simultaneous matrix conjugation; see Procesi \cite{Procesi}.
For a general associative coefficient algebra, multilinear tensor
maps can instead be evaluated with a specified order of coefficient
multiplication. The output then remains algebra-valued until a
suitable scalar contraction is applied. In particular, an arbitrary
commutative polynomial formula cannot be evaluated on matrix inputs
without specifying its ordering or polarization.

The determinant operation fits exactly into this construction.
For symmetric $g\times g$ matrices $X_i$ with associative entries,
define
\begin{equation}\label{III:eq:ordered-det-two}
 \operatorname{ODet}(X_1,\ldots,X_g)
 =\frac1{g!}\sum_{\sigma,\tau\in S_g}
       \operatorname{sgn}(\sigma)\operatorname{sgn}(\tau)
       (X_1)_{\sigma(1)\tau(1)}\cdots
       (X_g)_{\sigma(g)\tau(g)}.
\end{equation}
The two alternating tensors show that
\[
 \operatorname{ODet}(MX_1M^t,\ldots,MX_gM^t)
 =\det(M)^2\operatorname{ODet}(X_1,\ldots,X_g).
\]
Hence, for $f_i\in M_{k_i,0}^{\mer}(\Gamma;\AA)$, applying this
map to the matrices representing $\delta f_i$ gives an
$\AA$-valued scalar-type form of weight
$k_1+\cdots+k_g+2$. For commuting entries and identical inputs it
is the ordinary determinant. It generally depends on $G$.
Applying the same map to connection-independent tensors such as
\eqref{III:eq:first-bracket-full} removes that dependence.

For higher derivatives, replace \eqref{III:eq:ordered-det-two} by
equivariant maps between the representations
\eqref{III:eq:jet-representation}. At fixed derivative orders,
polynomial degrees, and target type, this is a finite-dimensional
representation-theoretic problem. The central character gives a necessary weight condition. If the $i$th input has type
$\det^{k_i}V\otimes\Sym^{m_i}V\otimes\Sym^{r_i}(\Sym^2V)$,
and a scalar-valued covariant has degree $d_i$ in that input,
then its output weight $a$ necessarily satisfies
\begin{equation*}
 ga=\sum_i d_i(gk_i+m_i+2r_i).
\end{equation*}
This follows by evaluating the central scalar matrices in
$\GL(V)$. The existence of the equivariant map is a separate
condition.

Connection independence can be tested within the same finite
ansatz. Express the candidate in local jets, replace $G$ by
$G+2Q$, and treat the symmetric matrix $Q$ and its derivatives as
independent formal variables. Universal independence is equivalent
to the vanishing of every coefficient involving those variables.
For a linear combination of a fixed finite list of candidates,
these are linear equations in its coefficients. This provides
a concrete way to seek further anomaly cancellations. The
pluriharmonic-polynomial method for Siegel Rankin--Cohen operators
\cite{EI} is an essential comparison; holomorphic nonlinear
differential constructions also occur in \cite{GIMS}.

\subsection{Binary covariants in genus two}
\label{III:subsec:binary-covariants-two}

For $g=2$, a form in $M_{k,m}^{\mer}$ is a binary form of degree
$m$ with determinant twist $k$. Its determinant weight and binary degree determine the output type
of every polynomial covariant. If a binary multi-covariant has degree $d_i$ in inputs
$F_i\in M_{k_i,m_i}^{\mer}$ and degree $b$ in the binary variables,
its value has type
\begin{equation*}
 M_{a,b}^{\mer},\qquad
 a=\sum_i d_i\left(k_i+\frac{m_i}{2}\right)-\frac b2.
\end{equation*}
For characters $\chi_i$, the output character is
$\prod_i\chi_i^{d_i}$.
This is \cite[Propositions 4.1--4.2]{CvG}, expressed with our
determinant weight first. It follows either from the irreducible
$\GL_2$-representations or from their central characters once
the binary covariant has been fixed.

For example, the normalized $r$th transvectant is
\begin{equation}\label{III:eq:transvectant-two}
 (F,H)_r=\frac{(m-r)!(p-r)!}{m!p!}
       \sum_{i=0}^r(-1)^i\binom ri
       \frac{\partial^rF}{\partial u_1^{r-i}\partial u_2^i}
       \frac{\partial^rH}{\partial u_1^i\partial u_2^{r-i}},
 \qquad 0\leq r\leq\min(m,p).
\end{equation}
It maps $M_{k,m}^{\mer}\times M_{\ell,p}^{\mer}$ to
$M_{k+\ell+r,m+p-2r}^{\mer}$. Its derivatives are with respect
to the auxiliary binary variables, not the entries of $Z$.
With associative coefficients, the displayed order defines a
bilinear equivariant tensor map. If $q_X[u]=u^tXu$ and
$q_Y[u]=u^tYu$ are quadratics, then
\begin{equation*}
 (q_X,q_Y)_2=2\operatorname{ODet}(X,Y),
 \qquad (q_X,q_X)_2=2\det X
       \quad\text{for scalar entries}.
\end{equation*}
Thus the earlier determinant is the first member of a much larger
family of available contractions. In genus two the connection with
binary sextics and vector-valued Siegel forms was developed
systematically in \cite{CFG}; the present evaluation principle
uses that representation-theoretic framework.

To obtain higher differential constructions, one can use inputs
$(\delta^{\mathrm{top}})^rf_i$, the corrected derivatives
\eqref{III:eq:corrected-jet-recurrence}, or the brackets
$\mathcal B_n(f_i,f_j)$, and then apply
\eqref{III:eq:transvectant-two} and more general multi-covariants.
The last choice gives connection-independent operations on fixed
inputs. For instance, for scalar-coefficient forms
$f\in M_{k,0}^{\mer}$ and $h\in M_{\ell,0}^{\mer}$,
\[
 {\mathfrak q}=\mathcal B_2(f,h)\in M_{k+\ell,4}^{\mer}
\]
is a binary quartic made from derivatives of order at most two.

Here is an explicit contraction that also applies directly to
Schwarzian data. Write a scalar-coefficient quartic as
\[
 {\mathfrak q}[u]=c_0u_1^4+4c_1u_1^3u_2+6c_2u_1^2u_2^2
                     +4c_3u_1u_2^3+c_4u_2^4
       \in M_{k,4}^{\mer}.
\]
Its classical invariants are
\begin{equation}\label{III:eq:quartic-invariants-two}
 \begin{aligned}
 A({\mathfrak q})&=c_0c_4-4c_1c_3+3c_2^2,\\
 B({\mathfrak q})&=c_0c_2c_4+2c_1c_2c_3-c_0c_3^2-c_1^2c_4-c_2^3.
 \end{aligned}
\end{equation}
They have degrees two and three in the coefficients. The weight
rule gives
\begin{equation}\label{III:eq:quartic-types-two}
 A({\mathfrak q})\in M_{2k+4}^{\mer},\qquad
 B({\mathfrak q})\in M_{3k+6}^{\mer},\qquad
 \mathcal J_{\mathfrak q}=\frac{B({\mathfrak q})^2}{A({\mathfrak q})^3}\in M_0^{\mer}
        \quad\text{if }A({\mathfrak q})\not\equiv0.
\end{equation}
The relative discriminant is proportional to
$A({\mathfrak q})^3-27B({\mathfrak q})^2$. All these assertions concern scalar
coefficients; matrix inputs may first be contracted by trace as
in Proposition~\ref{III:prop:differential-evaluation}.

For a matrix-coefficient operator $L$, take
\begin{equation*}
 {\mathfrak q}_L=\tr\bigl(\pi I_2(L)\bigr)\in M_{0,4}^{\mer}.
\end{equation*}
Then $A({\mathfrak q}_L)$ and $B({\mathfrak q}_L)$ are gauge-invariant meromorphic scalar
forms of weights four and six, and \eqref{III:eq:quartic-types-two}
gives a modular function wherever its denominator is not
identically zero. Higher data are available from
$\tr(\pi W_3(L))\in M_{0,6}^{\mer}$: binary-sextic invariants of
coefficient degrees $2,4,6,10$ give scalar forms of weights
$6,12,18,30$. Ratios of equal weights give further modular
functions. Nonvanishing and nonconstancy must be checked for each chosen input;
Section~\ref{III:subsec:Igusa-logarithmic-example} gives such a calculation.

For comparison, applying \eqref{III:eq:quartic-invariants-two} to
$\mathcal B_2(f,h)$ gives connection-independent scalar forms of
weights $2(k+\ell)+4$ and $3(k+\ell)+6$. Thus higher brackets
provide modular tensors, and invariant theory supplies systematic
scalar contractions beyond determinants. For differential data
coming from $L$, the input operator and its coefficient connection
remain part of the construction.

\subsection{A nonconstant function from Igusa forms}
\label{III:subsec:Igusa-logarithmic-example}

Logarithmic derivatives give a connection-independent example whose
nonconstancy follows from its divisor. Let $f,h$ be nonzero scalar
forms of nonzero weights $k,\ell$, respectively, and set
\begin{equation*}
 Q_{f,h}=\frac1k\frac{\partial_Zf}{f}
                  -\frac1\ell\frac{\partial_Zh}{h}.
\end{equation*}
The two anomalies cancel, so
$Q_{f,h}(\gamma Z)=JQ_{f,h}(Z)J^t$ and
$\det Q_{f,h}\in M_2^{\mer}$.
This is the normalized first bracket, with the sign determined
by \eqref{III:eq:first-bracket-full}.

The full-group statements below use the equivariant convention of
Section~\ref{III:subsec:automorphic-bundles}; equivalently, their
tensor-valued forms are sections on the modular orbifold.

\begin{proposition}\label{III:prop:Igusa-nonconstant}
On the full genus-two modular quotient, the expression
\begin{equation*}
 \mathcal J_{4,6;10}
   =\frac{\det Q_{E_4,E_6}}{\det Q_{E_4,\chi_{10}}}
\end{equation*}
is a nonconstant meromorphic modular function. It has a zero of
order two upstairs along a generic separating point $z=0$.
\end{proposition}

\begin{proof}
The weights of the two determinants agree. Write
$Z=\left(\begin{smallmatrix}\tau_1&z\\z&\tau_2\end{smallmatrix}\right)$.
Use normalized genus-two Eisenstein series and denote the elliptic
ones by $e_4,e_6$. Their restrictions are
$E_4|_{z=0}=e_4(\tau_1)e_4(\tau_2)$ and
$E_6|_{z=0}=e_6(\tau_1)e_6(\tau_2)$.
Also,
\[
 \chi_{10}(Z)=c\,z^2\Delta(\tau_1)\Delta(\tau_2)+O(z^4),
 \qquad c\ne0;
\]
see \cite[\S9]{vdG}. All three forms are even in $z$.
Put
\[
 a(\tau)=\frac14\frac{e_4'(\tau)}{e_4(\tau)}
                    -\frac16\frac{e_6'(\tau)}{e_6(\tau)}
     =(2\pi i)\frac{144\Delta(\tau)}{e_4(\tau)e_6(\tau)}.
\]
The last identity follows from the elliptic derivative identities
and $e_4^3-e_6^2=1728\Delta$.
At a generic separating point,
\[
 \det Q_{E_4,E_6}=a(\tau_1)a(\tau_2)+O(z^2).
\]
The off-diagonal entry of $\partial_Z\log\chi_{10}$ is
$1/z+O(z)$, because the off-diagonal matrix derivative includes
the factor $1/2$. Its diagonal entries are bounded. Hence
\[
 \det Q_{E_4,\chi_{10}}=-\frac1{100z^2}+O(1),
 \qquad
 \mathcal J_{4,6;10}
     =-100a(\tau_1)a(\tau_2)z^2+O(z^4).
\]
The leading coefficient is generically nonzero. This proves both
the stated order and nonconstancy.
\end{proof}

More explicitly, let $\mathbf B(f,h)=k f\partial_Zh
-\ell h\partial_Zf$ denote the scalar-coefficient matrix of
\eqref{III:eq:first-bracket-full}. Then
\begin{equation*}
 \mathcal J_{4,6;10}
   =\frac{25}{9}\,
      \frac{\chi_{10}^2\det\mathbf B(E_4,E_6)}
           {E_6^2\det\mathbf B(E_4,\chi_{10})}.
\end{equation*}
Both numerator and denominator are holomorphic scalar forms of
weight $42$. This also presents the function as a quotient of
classical modular forms. At full genus-two level such functions
belong to the field determined by Igusa's generators
\cite[Theorem 9.1]{vdG}. The differential expression specifies this member of the classical
modular function field. To obtain geometric information from
Schwarzian inputs, we next construct the operator from a filtered
Gauss--Manin connection.

\begin{remark}[Scope of the ambient construction]
\label{III:rem:ambient-scope}
Sections~\ref{III:sec:siegel-ore} and
\ref{III:sec:schwarzians-modular-constructions} evaluate the universal
polynomials in the symmetric tensor algebra of a Siegel modular
variety. Theorem~\ref{III:thm:cubic-Siegel-nonzero} gives a nonzero
cubic, and the projected brackets provide connection-independent
operations on fixed inputs. The period equations in the following
sections use the one-variable calculus along the base curve.
Theorem~\ref{III:thm:acceleration-comparison} compares the resulting
matrix Schwarzian with the ambient quadratic tensor, retaining the
acceleration term of the period map. These are the two interfaces
used here; no assertion of independence or completeness is made for
the higher ambient tensors on a Siegel variety.
\end{remark}

\section{Filtered symplectic connections and Schwarzian covariants}
\label{III:sec:filtered-connections}

A flat connection becomes trivial in a horizontal frame, but a varying
Lagrangian subbundle still defines a curve in the Lagrangian
Grassmannian. Its first derivative determines whether the filtration
gives a second-order matrix equation or a distinguished kernel line.
The constructions are local on a smooth curve. The specialization
to genus two is marked in Convention~\ref{III:conv:genus-two}.
A variation over a
higher-dimensional base enters through its restriction to a specified
curve.

\subsection{The filtered input and its transverse presentation}
\label{III:subsec:filtered-input}

Let $U$ be a smooth complex curve. Let $\HH$ be a holomorphic bundle
of rank $2g$, with a nonsingular flat connection $\nabla$ and a
horizontal nondegenerate alternating form $Q$. Fix a Lagrangian
subbundle $\EE\subset\HH$ of rank $g$. For an abelian scheme, these data are
\[
 (\HH,\nabla,\EE,Q)
 =(H^1_{\mathrm{dR}},\nabla^{\mathrm{GM}},F^1,Q).
\]
The second fundamental map of $\EE$ is
\begin{equation*}
 \beta:\EE\longrightarrow(\HH/\EE)\otimes K_U,
 \qquad s\longmapsto\nabla s\pmod{\EE},
 \qquad K_U=\Omega_U^1.
\end{equation*}
Choose a parameter $t$ and write $\nD=\nabla_{\partial_t}$.
The contraction $\beta_t$ is called \emph{transverse} when it is
invertible. This condition is independent of the nonvanishing
tangent field used to test it.

\begin{proposition}\label{III:prop:matrix-pf-three}
On the transverse locus, a column of local frame elements
$\omega=(\omega_1,\ldots,\omega_g)^t$ determines unique holomorphic
matrices $a_1,a_2$ with
\begin{equation}\label{III:eq:matrix-pf-three}
 \nD^2\omega+2a_1\nD\omega+a_2\omega=0.
\end{equation}
Their construction requires no choice of a complement to $\EE$.
For a horizontal basis of the dual local system, the period matrix
$Y$ satisfies $Y''+2a_1Y'+a_2Y=0$.
\end{proposition}
\begin{proof}
The images of $\nD\omega_1,\ldots,\nD\omega_g$ in $\HH/\EE$
are independent. Thus the entries of $(\omega,\nD\omega)^t$ form
a frame of $\HH$. Expanding $\nD^2\omega$ in this frame proves
existence and uniqueness. Pairing with horizontal dual sections
commutes with differentiation and proves the period equation.
\end{proof}

This second-order matrix equation has a $2g$-dimensional local
solution space: the periods of the $2g$ horizontal dual sections.
For explicit elimination, suppose a full frame $(\omega,\xi)^t$
has connection matrix
\[
 \nD\binom{\omega}{\xi}=
 \begin{pmatrix}B&C_0\\D&E_0\end{pmatrix}
 \binom{\omega}{\xi}.
\]
Then transversality means $\det C_0\ne0$. Put
\[
 T_0=C_0'C_0^{-1}+C_0E_0C_0^{-1}.
\]
Differentiating the first block and eliminating $\xi$ gives
\begin{equation}\label{III:eq:block-elimination-three}
 2a_1=-(B+T_0),\qquad
 a_2=-B'-C_0D+T_0B.
\end{equation}
The inverse of $C_0$ explains the possible additional poles of this
presentation at a failure of transversality. Where $\beta_t$ has
constant rank one, the kernel-line construction below replaces this
elimination.

\subsection{Normalization and the polarization}
\label{III:subsec:polarized-normalization}

Normalizing the first-derivative term gives the matrix Schwarzian.
The polarization will make it self-adjoint in a suitable frame.
Proposition~\ref{III:prop:normal-form-three} gives the normal form
used in the cyclicity and scalar-operator proofs.
For \eqref{III:eq:matrix-pf-three} define
\begin{equation}\label{III:eq:I-S-D-three}
 I=a_2-a_1'-a_1^2,\qquad
 \cov T=T'+[a_1,T],\qquad
 S=I-\frac1g\tr(I)\Id.
\end{equation}
Here $I=I_2$ of Section~\ref{III:subsec:tensor-realization}.
We use a different font for $\cov$ to distinguish this curve
derivation from the symmetric derivation on the Siegel variety.
If $\widetilde\omega=M\omega$, direct conjugation of
$(\partial_t+a_1)^2+I$ gives
\begin{equation}\label{III:eq:frame-laws-three}
 \widetilde a_1=Ma_1M^{-1}-M'M^{-1},\qquad
 \widetilde I=MIM^{-1},\qquad
 \widetilde{\cov}(MTM^{-1})=M(\cov T)M^{-1}.
\end{equation}
The conjugation direction agrees with Section~\ref{III:sec:schwarzians-modular-constructions} by taking $b=M^{-1}$
there.

The polarization yields more structure than an arbitrary matrix
second-order equation. Define
\begin{equation}\label{III:eq:K-three}
 {\mathsf K}_{ij}=Q(\omega_i,\nD\omega_j).
\end{equation}

\begin{proposition}[Polarized normal form]\label{III:prop:normal-form-three}
The matrix ${\mathsf K}$ is symmetric and invertible. Put
$v=\nD\omega+a_1\omega$. Then
\begin{equation}\label{III:eq:normal-first-order-three}
 \nD\omega=v-a_1\omega,\qquad
 \nD v=-I\omega-a_1v,
\end{equation}
and
\begin{equation}\label{III:eq:polarization-identities-three}
 Q(v,v)=0,\qquad
 {\mathsf K}'=-a_1{\mathsf K}-{\mathsf K}a_1^t,\qquad I{\mathsf K}={\mathsf K}I^t.
\end{equation}
Locally there is a holomorphic frame with $a_1=0$ and ${\mathsf K}$ constant.
In that frame $I$ is self-adjoint for the form ${\mathsf K}^{-1}$.
\end{proposition}
\begin{proof}
Differentiating $Q(\omega_i,\omega_j)=0$ proves symmetry.
The polarization identifies $\HH/\EE$ with $\EE^\vee$, so the
invertibility of $\beta_t$ implies that of ${\mathsf K}$.
Let $C_1=Q(\nD\omega,\nD\omega)$, with pairings taken entrywise.
Equation~\eqref{III:eq:matrix-pf-three} gives
${\mathsf K}'=C_1-2{\mathsf K}a_1^t$. Since ${\mathsf K}'$ is symmetric and $C_1$ alternating,
\[
 C_1={\mathsf K}a_1^t-a_1{\mathsf K},\qquad {\mathsf K}'=-a_1{\mathsf K}-{\mathsf K}a_1^t.
\]
Expansion of $Q(v,v)$ gives zero. Differentiating the definition
of $v$ proves \eqref{III:eq:normal-first-order-three}; differentiating
$Q(v,v)=0$ then gives $I{\mathsf K}={\mathsf K}I^t$.
Finally solve $M'=Ma_1$. Equation~\eqref{III:eq:frame-laws-three}
makes the transformed $a_1$ zero, and then ${\mathsf K}'=0$.
\end{proof}

The normal frame is a device for proofs. Even for an algebraic
Gauss--Manin connection it need not be algebraic; the invariants
can instead be computed directly in an algebraic frame by
\eqref{III:eq:I-S-D-three}.

\subsection{The classical matrix Schwarzian and the underlying bundle}
\label{III:subsec:classical-comparison-three}

The period chart identifies the normalized coefficient with the
classical matrix Schwarzian. Choose a horizontal symplectic
trivialization and a nonsingular period block. Normalizing that block
puts the Hodge period matrix in the form $[\Id\;Z(t)]$, with $Z^t=Z$.
The constant block forces $a_2=0$ in the period equation. On the
transverse locus $Z'$ is invertible, so the other block gives
\begin{equation}\label{III:eq:period-Schwarzian-three}
 a_2=0,\qquad a_1=-\tfrac12Z''(Z')^{-1},\qquad
 I=\tfrac12Z'''(Z')^{-1}
                  -\tfrac34\bigl(Z''(Z')^{-1}\bigr)^2.
\end{equation}
The factor order follows from the column convention for $\omega$.
For a polarized weight-one variation, one may choose a period chart
with $Z\in\mathfrak H_g$.

Formula~\eqref{III:eq:period-Schwarzian-three} is a convention for the
classical Lagrangian matrix Schwarzian. In the fanning-frame
convention of \cite{APD}, the transposed frame satisfies
$A''+A'P+AR=0$ with $P=2a_1^t$ and $R=a_2^t$; their Schwarzian
$2R-\tfrac12P^2-P'$ is therefore $2I^t$.
Normal frames and local equivalence belong to that established
theory; see also \cite{Ovsienko,MBS}.
For example, in a normal frame the system
$\omega'=v$, $v'=-I\omega$, together with its constant ${\mathsf K}$,
reconstructs the local polarized filtered connection up to a
constant symplectic identification.

The period equation acts on the dual Hodge bundle. For a horizontal
functional $\ell$ on $\HH$, the values $\ell(\omega_i)$ are the
coefficients of $\ell|_{\EE}\in\EE^\vee$. Thus the matrix operator
glues as
\begin{equation}\label{III:eq:period-operator-bundle-three}
 L:\EE^\vee\longrightarrow\EE^\vee\otimes K_U^2.
\end{equation}
If $\vartheta^2=K_U$ and $\mathcal V=\EE^\vee\otimes\vartheta$,
this is
\[
 L:\mathcal V\otimes\vartheta^{-1}
       \longrightarrow\mathcal V\otimes\vartheta^3,
\]
the order-two case of \eqref{II:eq:curve-operator-type}.
The adapted connection with local coefficient $a_1dt$ belongs to
$\mathcal V$, while
$\End(\mathcal V)=\End(\EE^\vee)$.
The half-form twist will account for the central term in the
parameter law below. No square root is needed for
\eqref{III:eq:period-operator-bundle-three} itself.

\subsection{Parameter laws and intrinsic covariants}
\label{III:subsec:parameter-laws-three}

A change of parameter adds a scalar Schwarzian term to $I$.
Taking the trace-free part removes it; taking a commutator will then
remove the anomaly in the next derivative. Let $t=t(s)$, put
$\mu=dt/ds\ne0$, and keep the same Hodge frame.
The chain rule gives
\begin{equation*}
 a_{1,s}=\mu a_{1,t}-\frac{\mu'}{2\mu}\Id,
 \qquad a_{2,s}=\mu^2a_{2,t},\qquad {\mathsf K}_s=\mu {\mathsf K}_t,
\end{equation*}
where primes on $\mu$ are $s$-derivatives and all $t$-coefficients
are evaluated at $t(s)$. Substitution gives
\begin{equation}\label{III:eq:parameter-I-three}
 I_s=\mu^2I_t+\tfrac12\{t,s\}\Id,\qquad
 S_s=\mu^2S_t,
 \qquad \{t,s\}=\frac{\mu''}{\mu}
                         -\frac32\left(\frac{\mu'}\mu\right)^2.
\end{equation}
This agrees with the Part~\ref{part:I} projective datum $P=I/3$.
On endomorphism coefficients $\cov_s=\mu\cov_t$; hence
\begin{equation}\label{III:eq:parameter-DS-three}
 (\cov S)_s=\mu^3(\cov S)_t+2\mu\mu'S_t.
\end{equation}
In particular, covariant differentiation of $S$ has a parameter
anomaly even after the Hodge gauge has been accounted for.

\begin{convention}[Genus-two specialization]\label{III:conv:genus-two}
From this point to the end of Part~\ref{part:III}, $g=2$ unless a
statement explicitly allows arbitrary rank or genus. The cyclic-section
lemma and the general endomorphism criterion retain their stated
arbitrary-rank hypotheses.
\end{convention}
Set
\begin{equation}\label{III:eq:q-c-three}
 T=\cov S,\qquad C=[S,T]=[I,\cov I],\qquad
 q=\tr(S^2),\qquad c=\tr(C^2).
\end{equation}
Taking a commutator removes the last term of
\eqref{III:eq:parameter-DS-three}:
\begin{equation}\label{III:eq:parameter-C-three}
 C_s=\mu^5C_t,\qquad q_s=\mu^4q_t,
 \qquad c_s=\mu^{10}c_t.
\end{equation}
Thus $Sdt^2$ and $Cdt^5$ are endomorphism-valued differentials,
and $qdt^4$, $cdt^{10}$ are scalar differentials. On $q\ne0$ put
\begin{equation}\label{III:eq:R-three}
 \mathcal R=\frac{c^2}{q^5},\qquad
 T^\perp=T-\frac{q'}{2q}S.
\end{equation}
Both the scalar $\mathcal R$ and the tensor $T^\perp dt^3$ are
independent of the parameter and the Hodge frame, with conjugation
understood for the latter.
Indeed, $\tr(ST)=q'/2$, and
\eqref{III:eq:parameter-DS-three} shows
$(T^\perp)_s=\mu^3(T^\perp)_t$.
For trace-free $2\times2$ matrices, Cayley--Hamilton gives
\begin{equation}\label{III:eq:trace-relations-three}
 S^2=\tfrac q2\Id,\qquad
 ST+TS=\tr(ST)\Id,\qquad
 c=-2\bigl(q\tr(T^2)-\tr(ST)^2\bigr).
\end{equation}
Consequently $\tr(ST^\perp)=0$ and
$\tr((T^\perp)^2)=-c/(2q)$.

\begin{proposition}[Local filtered splitting]\label{III:prop:splitting-three}
On a sufficiently small open set where $q\ne0$, the identity
$C\equiv0$ is equivalent to an orthogonal decomposition of the
polarized filtered connection into two rank-two flat subbundles,
each meeting $\EE$ in a line.
\end{proposition}
\begin{proof}
Use a normal frame. Since $q\ne0$, $S$ has distinct eigenvalues.
Its spectral projectors $P_+,P_-$ are holomorphic after shrinking.
Differentiating $P_\pm^2=P_\pm$ shows that $P_\pm'$ is off-diagonal
for the eigenspace decomposition. The identity $[S,S']=0$ forces
these off-diagonal terms to vanish. Thus the projectors are constant,
and the second-order equation splits in a constant eigenbasis.
Self-adjointness makes the summands orthogonal. Conversely, an
adapted frame for such a filtered splitting makes $a_1,I,\cov I$
diagonal, so $C=0$.
\end{proof}

The condition is $C\equiv0$ on the neighborhood; an isolated zero
is insufficient. This local complex splitting need not preserve an
integral lattice.

\subsection{Comparison with the ambient connection and acceleration}
\label{III:subsec:ambient-acceleration}

Suppose the period curve lies in a Siegel chart where $G$ is regular.
Theorem~\ref{III:thm:geodesic-Schwarzian} supplies an ambient tensor, but an
arbitrary period curve need not be its geodesic. The discrepancy is
explicit. Put $H=G/2$, $V=Z'$, and, on $\det V\ne0$, define
\[
 A=\nabla^T_tV=Z''+2VHV,\qquad \Upsilon=AV^{-1},\qquad B_0=VH.
\]
Here $\Upsilon$ is an endomorphism of the pulled-back $\EE^\vee$,
and $d+B_0dt$ is its underlying dual Hodge connection.

\begin{theorem}\label{III:thm:acceleration-comparison}
For the graph matrix Schwarzian in \eqref{III:eq:period-Schwarzian-three},
\begin{equation}\label{III:eq:acceleration-comparison}
 \boxed{\quad
 I=\mathbb I_G(V,V)
       +\frac12\bigl(\Upsilon'+[B_0,\Upsilon]\bigr)
       -\frac14\Upsilon^2.
 \quad}
\end{equation}
In particular the ambient and curve Schwarzians agree along affine
$G$-geodesics. In the general case the displayed sum is independent
of the choice of ambient modular connection.
\end{theorem}
\begin{proof}
The graph frame has $a_2=0$ and
$a_1=-Z''V^{-1}/2=B_0-\Upsilon/2$. Since
$V'=\Upsilon V-2VHV$, differentiation gives
\[
 -B_0'-B_0^2=\mathbb I_G(V,V)-\Upsilon B_0.
\]
Substitute $a_1=B_0-\Upsilon/2$ into $I=-a_1'-a_1^2$ and retain
factor order. The mixed terms are
$[B_0,\Upsilon]/2$, proving the formula. Its left side does not
use $G$, which proves the final assertion.
\end{proof}

Along an affine $G$-geodesic the graph frame has $a_1=VH$.
The coefficient connection is therefore the pulled-back dual Hodge
connection. Formula~\eqref{III:eq:geodesic-jet-restriction} identifies
the contractions of $\mathbb S_G,\mathbb C_G,\mathbf q_G,\mathbf c_G$
with $S,C,q,c$. For an arbitrary period curve, the acceleration
terms in \eqref{III:eq:acceleration-comparison} give the difference.

\subsection{Dependence on period-curve jets}
\label{III:subsec:jet-descent-three}

Frame and parameter invariance leave open whether a quantity depends
on the curve or only on its point in moduli.
Formula~\eqref{III:eq:period-Schwarzian-three} uses the third period jet;
$C$ uses the fourth. The transformations already proved show that
$\mathcal R$ is a rational invariant of transverse fourth jets,
under Hodge-frame changes, changes of symplectic marking, and
reparametrization. Its denominator requires the indicated open
conditions. This is a statement on jets of period curves, not
descent to the moduli space of individual abelian varieties.

\begin{proposition}\label{III:prop:jet-obstruction-three}
Fix $Z(0)$ and the two-jet of a curve with $\det Z'(0)\ne0$.
Varying the symmetric matrix $Z'''(0)$ makes $S(0)$ range over
all trace-free matrices satisfying $S {\mathsf K}={\mathsf K}S^t$, with ${\mathsf K}=Z'(0)$
in the standard symplectic graph chart.
In particular, a nonconstant function of this freely varying $S$
is not determined by the base point alone.
\end{proposition}
\begin{proof}
Changing $Z'''(0)$ by a symmetric matrix $B$ changes $I(0)$ by
$\tfrac12B(Z'(0))^{-1}$. Multiplication on the right by $Z'(0)$
identifies this space with the symmetric matrices; trace removal
is surjective onto its trace-free part. All these jets are realized
by polynomial curve germs. If $Z(0)\in\mathfrak H_2$, positivity
of its imaginary part persists after sufficiently shrinking the germ.
\end{proof}

For $\mathcal R$ the obstruction can be seen while fixing even the
three-jet: variation of $Z''''(0)$ varies $T(0)$ arbitrarily in the
trace-free self-adjoint space. If $q(0)\ne0$, one can choose $T$
commuting with $S$ or not commuting with $S$, giving different
values of $\mathcal R$ at the same point.
Thus $\mathcal R$ is an invariant of a period curve, while
$\mathbb I_G$ is an ambient tensor obtained by choosing a modular
connection. Their comparison is \eqref{III:eq:acceleration-comparison}.
For the general Gauss--Manin interpretation of differential equations
satisfied by Siegel modular functions, see \cite{BertrandZudilin}.

\subsection{Rank-one directions and the kernel line}
\label{III:subsec:rank-one-three}

When the Kodaira--Spencer map has rank one, its kernel is a line
whose next derivative measures the remaining variation.
Suppose that $\HH$ has rank four and $\beta_t$ has constant rank one.
Put $N=\ker\beta$. Differentiating the polarization
shows that
\begin{equation*}
 N\subset\EE\subset N^\perp=\EE+\nD\EE\subset\HH,
 \qquad\text{with ranks }1,2,3,4.
\end{equation*}
Indeed, $\nD N\subset\EE$ and
$Q(n,\nD e)=-Q(\nD n,e)=0$ for $n\in N$, $e\in\EE$;
rank counting proves the equality. There is a further intrinsic map
\begin{equation}\label{III:eq:nu-three}
 \nu:N\longrightarrow(\EE/N)\otimes K_U,
 \qquad s\longmapsto\nabla s\pmod N.
\end{equation}
Its nonvanishing is the appropriate cyclicity condition in this regime.

\begin{proposition}\label{III:prop:rank-one-cyclicity-three}
If $\nu(t_0)\ne0$, every nonvanishing local generator $s$ of $N$
is cyclic at $t_0$: the four sections
$s,\nD s,\nD^2s,\nD^3s$ are independent.
If $U$ is connected and $\nu\not\equiv0$, every flat complex
endomorphism of $\HH$ preserving $\EE$ is a scalar constant.
\end{proposition}
\begin{proof}
Where $\nu\ne0$, $s,s'$ span $\EE$. The Kodaira--Spencer pairing
has radical $N$, so $Q(s',s'')\ne0$.
The first three derivative sections span $N^\perp$, while
$Q(s,s''')=-Q(s',s'')\ne0$ puts the fourth outside it.
For the endomorphism assertion, a flat endomorphism $A$ preserving
$\EE$ also preserves $N$. Write $As=f s$. Its characteristic
polynomial in a horizontal frame is constant, so the holomorphic
eigenvalue $f$ is constant. Then $As^{(j)}=f s^{(j)}$ for all $j$.
Cyclicity gives $A=f\Id$ on a nonempty open set, and flatness
extends the equality over connected $U$.
\end{proof}

The scalar differential module attached to $N$ is intrinsic; a
displayed monic fourth-order operator still requires a generator
and a parameter. Its ramification and its relation to Yang--Zudilin
will be treated in Section~\ref{III:sec:scalarizations}.
This is the rank-one geometry of \cite{Zelenko,ZelenkoLi}, expressed
in filtered-connection terms. The constant-rank assumption makes
$N$ and $\EE/N$ bundles; points where the rank changes are excluded.
If $\beta$ vanishes identically, $\EE$ is horizontal.

\section{Canonical scalarizations and cyclicity divisors}
\label{III:sec:scalarizations}

A line in a flat bundle gives a scalar differential equation wherever
its successive derivatives span the bundle. The zeros of their
determinant are apparent singularities of that presentation. Under Convention~\ref{III:conv:genus-two}, we
compute this determinant for the Hodge determinant line, the two
null lines on the transverse locus, and the kernel line on the
rank-one locus.

\subsection{Cyclic sections and apparent singularities}
\label{III:subsec:cyclic-sections-four}

Let $\mathcal F$ be a rank-$r$ flat bundle on $U$, nonsingular at
the points under consideration, and let $\mathcal L\subset\mathcal F$
be a line subbundle. A local section $s$ is \emph{cyclic} at a point
if $s,\nD s,\ldots,\nD^{r-1}s$ form a basis there. For a
nonvanishing local generator of $\mathcal L$, define
\begin{equation}\label{III:eq:cyclic-W-four}
 \operatorname{Wr}_s=\det(s,\nD s,\ldots,\nD^{r-1}s)
\end{equation}
in a local nonzero flat determinant frame. Under $s\mapsto fs$ and
$t=t(u)$, triangularity of the derivative columns gives the factors
$f^r$ and $(dt/du)^{r(r-1)/2}$, respectively. Equivalently, the
cyclic determinant is a bundle morphism
\[
 \mathcal L^{\otimes r}\longrightarrow
                  \det\mathcal F\otimes K_U^{r(r-1)/2}.
\]
If it is not identically zero, its zero divisor is intrinsic.
On its complement, expansion of $\nD^r s$ in the first $r$
derivatives gives the unique monic scalar operator for periods of
$s$. Its coefficients continue meromorphically across the zeros.

\begin{lemma}\label{III:lem:apparent-exponents-four}
Suppose $s$ is generically cyclic and let
$n_0<\cdots<n_{r-1}$ be the vanishing orders of an adapted basis
of its scalar periods at a regular point. They are exactly the indices
$j$ at which
$\operatorname{span}\{s(t_0),\nabla_t s(t_0),\ldots,
\nabla_t^j s(t_0)\}$ increases in dimension. Moreover,
\begin{equation}\label{III:eq:Wronskian-orders-four}
 \operatorname{ord}\operatorname{Wr}_s=\sum_{i=0}^{r-1}n_i-\frac{r(r-1)}2.
\end{equation}
The scalar equation is regular singular with these exponents and
trivial local monodromy. If $\operatorname{Wr}_s$ is a unit it is ordinary; if
$\operatorname{Wr}_s$ has a simple zero, the exponents are
$0,1,\ldots,r-2,r$.
\end{lemma}
\begin{proof}
In a horizontal frame the components of $s$ are holomorphic.
The rank of the derivative columns through order $j$ is the rank of
the Taylor coefficient map through degree $j$ on their scalar periods.
Its successive rank jumps give the stated vanishing sequence.
Successive constant row operations give a basis with distinct
leading powers $t^{n_i}$. The leading coefficient of the Wronskian
is a nonzero multiple of $\prod_{i<j}(n_j-n_i)$, proving
\eqref{III:eq:Wronskian-orders-four}. Cramer's rule for the coefficient
of $\partial_t^j$ replaces derivative order $j$ by $r$ in this
determinant. Its numerator has order at least
$\sum_i n_i-r(r-1)/2-(r-j)$, so the coefficient has a pole of
order at most $r-j$. These are the regular-singular bounds.
The leading powers of the solution basis give the indicial roots;
all the solutions are single-valued. For a simple zero, distinct
nonnegative $n_i$ with the sum in
\eqref{III:eq:Wronskian-orders-four} can only have the stated values.
\end{proof}

Multiplying a generator by a function with a zero adds $r$ times
that zero to the displayed Wronskian divisor. Such a generator
does not compute the intrinsic divisor of the line subbundle there.

\subsection{The primitive Hodge determinant}
\label{III:subsec:primitive-four}

The determinant of the Hodge bundle lies in a smaller flat bundle
than the full exterior square. Its cyclic determinant can therefore
be compared directly with the matrix Schwarzian. Assume $g=2$ and
$\beta_t$ invertible, and let
\[
 \mathcal P=\Lambda^2_0\HH
       =\ker\bigl(\Lambda^2\HH\xrightarrow{Q}\OO_U\bigr).
\]
This is a rank-five flat bundle, and
$v=\omega_1\wedge\omega_2$ is a nonvanishing local generator of
$\det\EE\subset\mathcal P$. Fix its invariant volume by assigning
volume one to
\begin{equation}\label{III:eq:primitive-basis-four}
 (e_1\wedge e_2,\ e_1\wedge f_1-e_2\wedge f_2,
       \ e_1\wedge f_2,\ e_2\wedge f_1,\ f_1\wedge f_2)
\end{equation}
for any symplectic frame with $Q(e_i,f_j)=\delta_{ij}$.
The symplectic group preserves this volume. Write
$\mathcal W(v)=\det_{\mathcal P}(v,v',v'',v''',v^{(4)})$.

\begin{theorem}[Commutator and primitive cyclicity]
\label{III:thm:primitive-commutator-four}
On the transverse locus,
\begin{equation*}
 \boxed{\quad
 \mathcal W(\omega_1\wedge\omega_2)^2
       =-32(\det {\mathsf K})^5\tr\bigl([I,\cov I]^2\bigr).
 \quad}
\end{equation*}
The section $\omega_1\wedge\omega_2$ is cyclic in $\mathcal P$
at a point if and only if $C=[I,\cov I]$ is nonzero there.
\end{theorem}
\begin{proof}
\noindent\textit{1. Normalize the frame.} Use Proposition~\ref{III:prop:normal-form-three} and a constant
congruence to arrange $a_1=0$, ${\mathsf K}=\Id$. Put
\[
 I=\begin{pmatrix}p&u\\u&r\end{pmatrix},\qquad
 e=\omega,\quad f=\omega',\quad \tau=p+r.
\]
\noindent\textit{2. Differentiate in the primitive exterior square.} Then $e'=f$, $f'=-Ie$, and direct exterior differentiation gives
\begin{align*}
 v&=e_1\wedge e_2,\\
 v'&=e_1\wedge f_2-e_2\wedge f_1,\\
 v''&=2f_1\wedge f_2-\tau e_1\wedge e_2,\\
 v'''&=-\tau'e_1\wedge e_2
       +2u(e_1\wedge f_1-e_2\wedge f_2)\\
     &\hspace{12mm}-(3p+r)e_1\wedge f_2+(p+3r)e_2\wedge f_1.
\end{align*}
\noindent\textit{3. Take the determinant.} Differentiate the last row and take the determinant in
\eqref{III:eq:primitive-basis-four}. It is $-8\kappa$, where
\[
 \kappa=(p-r)u'-u(p'-r'),\qquad
 C=[I,I']=\begin{pmatrix}0&\kappa\\-\kappa&0\end{pmatrix}.
\]
Thus $\mathcal W^2=64\kappa^2=-32\tr(C^2)$, and cyclicity
is equivalent to $C\ne0$ in this frame.
\noindent\textit{4. Transfer the identity to every frame.} Under $\omega\mapsto M\omega$, the generator $v$ acquires
$\det M$, so $\mathcal W$ acquires $(\det M)^5$.
Meanwhile ${\mathsf K}\mapsto M{\mathsf K}M^t$ and $C\mapsto MCM^{-1}$.
These transformation laws prove the identity and the equivalence
in every frame.
\end{proof}

Under reparametrization, both $\mathcal W$ and the coefficient
of $C$ have the weights prescribed by
\eqref{III:eq:cyclic-W-four} and \eqref{III:eq:parameter-C-three}:
$\mathcal W$ has weight ten, and $(\det {\mathsf K})^5c$ has weight twenty.
Thus the theorem is compatible with all changes used to define
the intrinsic cyclicity divisor.

\begin{corollary}\label{III:cor:primitive-apparent-four}
Suppose $C\not\equiv0$. The primitive Hodge determinant defines
a fifth-order scalar presentation on a nonempty open set. At a
simple zero of its cyclic determinant, the equation has an
apparent singularity with exponents $0,1,2,3,5$.
Locally, if $C$ vanishes to order $m$ in a normal frame, then
$\mathcal W$ vanishes to order $m$ and $c$ to order $2m$.
\end{corollary}
\begin{proof}
In the preceding normal form $C$ has the single coefficient
$\kappa$, $\mathcal W=-8\kappa$, and $c=-2\kappa^2$.
Apply Lemma~\ref{III:lem:apparent-exponents-four}.
\end{proof}

\subsection{A binary joint invariant for the same divisor}
\label{III:subsec:binary-dictionary-four}

Three binary quadratics give an unsquared version of the preceding
identity. Their coefficient determinant changes by the same relative
weight as the primitive cyclic determinant after one factor of
$\det {\mathsf K}$. With $T=\cov S$, set
\[
 \Phi_0={\mathsf K},\qquad\Phi_1=S{\mathsf K},\qquad\Phi_2=T{\mathsf K}.
\]
These three matrices are symmetric: for $\Phi_2$ differentiate
$S{\mathsf K}={\mathsf K}S^t$ and use \eqref{III:eq:polarization-identities-three}.
Define their joint coefficient determinant by
\begin{equation}\label{III:eq:joint-J-four}
 \mathfrak j=\det\begin{pmatrix}
 (\Phi_0)_{11}&2(\Phi_0)_{12}&(\Phi_0)_{22}\\
 (\Phi_1)_{11}&2(\Phi_1)_{12}&(\Phi_1)_{22}\\
 (\Phi_2)_{11}&2(\Phi_2)_{12}&(\Phi_2)_{22}
 \end{pmatrix}.
\end{equation}
Use the polarized determinant pairing $\langle\ ,\ \rangle$ of
Section~\ref{III:subsec:both-second-brackets}. For its Gram matrix write
$m=\tr(ST)$, $p=\tr(T^2)$, and $d_{\mathsf K}=\det {\mathsf K}$.

\begin{theorem}[Binary-quadratic comparison]
\label{III:thm:binary-comparison-four}
The Gram matrix and the primitive cyclic determinant satisfy
\begin{align}
 \bigl(\langle\Phi_i,\Phi_j\rangle\bigr)_{0\leq i,j\leq2}
   &=d_{\mathsf K}\begin{pmatrix}1&0&0\\0&-q/2&-m/2\\0&-m/2&-p/2\end{pmatrix},
                         \label{III:eq:Gram-four}\\
 \mathfrak j^2&=16\det\bigl(\langle\Phi_i,\Phi_j\rangle\bigr),\qquad
 c=-\frac{\mathfrak j^2}{2d_{\mathsf K}^3},                  \notag\\
 \boxed{\quad\mathcal W(\omega_1\wedge\omega_2)=-4d_{\mathsf K}\mathfrak j.\quad}
                                             \label{III:eq:W-J-four}
\end{align}
\end{theorem}
\begin{proof}
\noindent\textit{1. The Gram matrix.} The identity
$\langle {\mathsf K},A{\mathsf K}\rangle=\tfrac12\tr(A)\det {\mathsf K}$ and
Cayley--Hamilton give \eqref{III:eq:Gram-four}. \noindent\textit{2. The normalized joint determinant.} At a point reduce by
congruence to ${\mathsf K}=\Id$ and write
\[
 S=\begin{pmatrix}a&b\\b&-a\end{pmatrix},\qquad
 T=\begin{pmatrix}d&e\\e&-d\end{pmatrix}.
\]
Then $C=\left(\begin{smallmatrix}0&\kappa\\-\kappa&0\end{smallmatrix}\right)$,
$\kappa=2(ae-bd)$, and $\mathfrak j=4(ae-bd)=2\kappa$.
\noindent\textit{3. Frame weights and the unsquared identity.} Under a frame change $M$, $d_{\mathsf K}$ acquires $(\det M)^2$ and
$\mathfrak j$ acquires $(\det M)^3$. This proves the formula for $c$.
The Gram identity follows either from \eqref{III:eq:Gram-four} and
\eqref{III:eq:trace-relations-three}, or from the determinant $1/16$
of the pairing on coefficient vectors $(A_{11},2A_{12},A_{22})$.
The normalized computation $\mathcal W=-8\kappa=-4\mathfrak j$ and the
common transformation factor $(\det M)^5$ prove
\eqref{III:eq:W-J-four}, including its sign.

\end{proof}

\begin{corollary}[The full matrix algebra]
\label{III:cor:binary-matrix-basis}
On the transverse genus-two locus, $C{\mathsf K}$ is alternating, so a nonzero $C$ is invertible and
cannot be nilpotent. In particular, $C\ne0$, $c\ne0$, $\mathfrak j\ne0$, and
$\mathcal W\ne0$ are equivalent. When they hold,
$\Id,S,T,C$ form a basis of $\operatorname{Mat}_2(\C)$.
\end{corollary}
\begin{proof}
Since $S{\mathsf K}={\mathsf K}S^t$ and $T{\mathsf K}={\mathsf K}T^t$, one has
$C{\mathsf K}=-{\mathsf K}C^t$. A nonzero alternating matrix of size two is invertible;
so is $C$ because ${\mathsf K}$ is invertible.
Finally, $\tr(SC)=\tr(TC)=0$, while the trace Gram determinant
of $S,T$ is $qp-m^2=-c/2$. Thus $S,T,C$ are independent when
$c\ne0$, and adjoining $\Id$ gives the stated basis.
\end{proof}

The Gram identity is the classical syzygy of three binary quadratics.
Equation~\eqref{III:eq:W-J-four} identifies its joint determinant with the
cyclic determinant of the specified Hodge section.

The parameter weights also agree. Equation~\eqref{III:eq:parameter-DS-three}
changes $\Phi_2$ by a multiple of $\Phi_1$ in addition to its
homogeneous factor. This shear disappears from the determinant:
\[
 \mathfrak j_s=\mu^8\mathfrak j_t,\qquad (d_{\mathsf K})_s=\mu^2(d_{\mathsf K})_t,
 \qquad
 \mathcal R=\frac{\mathfrak j^4}{4d_{\mathsf K}^6q^5}.
\]
Thus $\mathfrak j$ is a relative invariant of Hodge weight three and
parameter weight eight.

\begin{example}[Different cyclicity divisors in a local model]
\label{III:ex:two-scalarizations}
Take the normal-frame system $e'=f$, $f'=-I(t)e$, with
\[
 \mathsf K=\Id,\qquad
 I(t)=\begin{pmatrix}t&1\\1&-t\end{pmatrix}.
\]
Because $I$ is symmetric, this is a local flat symplectic system
with Lagrangian subbundle spanned by $e_1,e_2$ and invertible
Kodaira--Spencer form. It is a local model; no algebraic family of
curves is asserted. Here $S=I$, $\cov=\partial_t$, and
\[
 q=2(t^2+1),\qquad
 C=[I,I']=\begin{pmatrix}0&-2\\2&0\end{pmatrix},\qquad c=-8.
\]
The primitive determinant is $\mathcal W(e_1\wedge e_2)=16$
in the volume~\eqref{III:eq:primitive-basis-four}, so this section
is cyclic at every parameter.

To see the null lines, make the constant change
\[
 M=\frac1{\sqrt2}\begin{pmatrix}1&i\\1&-i\end{pmatrix},\qquad
 M\mathsf K M^t=\begin{pmatrix}0&1\\1&0\end{pmatrix},\qquad
 MIM^{-1}=\begin{pmatrix}0&t+i\\t-i&0\end{pmatrix}.
\]
In this frame $b=t+i$ and $c_0=t-i$. The factorization in
Proposition~\ref{III:prop:null-factorization} becomes
\[
 L_b=(t+i)\partial_t^2(t+i)^{-1}\partial_t^2-(t^2+1),
\]
and the other operator interchanges $t+i$ and $t-i$.
Their cyclic determinants are $-(t+i)^2$ and $-(t-i)^2$.
The first equation has apparent exponents $0,1,3,4$ at $t=-i$;
the second has them at $t=i$. The primitive equation is ordinary
at both points. At these points $S$ is nonzero and nilpotent,
whereas the commutator $C$ remains invertible. Thus the two
scalarizations measure different losses of cyclicity in the same
regular filtered connection.
\end{example}

\subsection{The two null-line scalarizations}
\label{III:subsec:null-lines-four}

The polarization makes a second choice of scalar sections: the null
lines of the symmetric Kodaira--Spencer form ${\mathsf K}dt$ on $\EE$.
There are two in each fiber. They form an \'etale double cover
$\pi:\widetilde U\to U$, possibly disconnected. Its tautological
line is intrinsic and does not depend on the parameter.

\begin{theorem}[Null-line cyclicity divisor]
\label{III:thm:null-lines-four}
On the transverse rank-four locus, if $q\not\equiv0$, both null
lines are generically cyclic and their intrinsic cyclicity divisor
$Z_N$ on $\widetilde U$ satisfies
\begin{equation}\label{III:eq:null-divisor-four}
 \boxed{\quad\pi_*Z_N=2\operatorname{div}_0(qdt^4).\quad}
\end{equation}
At a simple zero of $q$, one equation has apparent exponents
$0,1,3,4$ and the other is ordinary.

\end{theorem}

The local expression uses the normal frame of
Proposition~\ref{III:prop:normal-form-three}. A constant congruence
makes its polarization matrix off-diagonal.

\begin{proposition}[Null-line factorization and multiplicities]
\label{III:prop:null-factorization}
On the transverse genus-two locus, the null-line equations have the
following local form. For the local formulas, choose a normal frame with
\begin{equation*}
 {\mathsf K}=\begin{pmatrix}0&1\\1&0\end{pmatrix},\qquad
 I=\begin{pmatrix}a&b\\c_0&a\end{pmatrix}.
\end{equation*}
The null lines are generated by $\omega_1,\omega_2$. On $b\ne0$
the first has scalar operator
\begin{equation}\label{III:eq:null-operator-four}
 L_b=b(\partial_t^2+a)b^{-1}(\partial_t^2+a)-bc_0.
\end{equation}
The other operator is obtained by interchanging $b,c_0$.
Their cyclic determinants in the frame
$(\omega_1,\omega_2,\omega_1',\omega_2')$ are $-b^2,-c_0^2$.
At a zero of $b$ of order $m\geq1$ where $c_0$ is a unit, the
first scalar equation has apparent exponents $0,1,m+2,m+3$.
\end{proposition}
\begin{proof}[Proof of the theorem and proposition]
\noindent\textit{1. Eliminate the other period.} A constant congruence after normalization gives the displayed ${\mathsf K}$.
The condition $I{\mathsf K}={\mathsf K}I^t$ forces the diagonal entries of $I$ to agree.
For periods $z,y$ of the two frame elements,
\[
 z''+az+by=0,\qquad y''+c_0z+ay=0.
\]
Eliminating $y$ proves \eqref{III:eq:null-operator-four}.
\noindent\textit{2. Compute the intrinsic divisor.} The derivative rows for the first line are
\[
 \begin{pmatrix}
 1&0&0&0\\0&0&1&0\\-a&-b&0&0\\-a'&-b'&-a&-b
 \end{pmatrix},
\]
with determinant $-b^2$; interchange the indices for the other.
Since $q=2bc_0$, the product of the two determinants is $q^2/4$.
Changes of nonvanishing generators multiply by units, proving
\eqref{III:eq:null-divisor-four}.

\noindent\textit{3. Read the vanishing orders.} For the exponents, take $t_0=0$. Since $b$ vanishes to order $m$,
the first component and its first derivative give the first two
independent Taylor coefficients of $\omega_1$. Its equation
$\omega_1''+a\omega_1=-b\omega_2$ first introduces the independent
vector $\omega_2(0)$ at derivative order $m+2$.
The rank-jump assertion of Lemma~\ref{III:lem:apparent-exponents-four}
therefore gives the first three vanishing orders
$0,1,m+2$. The cyclic determinant has order $2m$;
\eqref{III:eq:Wronskian-orders-four} forces the fourth order to be
$m+3$. The scalar equation is apparent by
Lemma~\ref{III:lem:apparent-exponents-four}.
\end{proof}

For a fixed parameter, the metric connection determined by $a_1$
preserves these null lines. A parallel generator on either sheet
is unique up to a nonzero constant, so its monic scalar operator
is unambiguous locally. The normal frame used to display
\eqref{III:eq:null-operator-four} need not be algebraic.
The null-line cover is determined by ${\mathsf K}$, whereas the spectral
cover of $S$ is locally $\mu^2=q/2$; they are different covers.

The operator also exhibits the usual scalar self-duality:
\[
 b^{-1}L_b=(\partial_t^2+a)b^{-1}(\partial_t^2+a)-c_0
\]
is formally self-adjoint. More invariantly, a null-line generator
$s$ satisfies $Q(s,s')=0$, so $s\wedge s'$ belongs to
$\Lambda^2_0\HH$. An arbitrary cyclic section need not have this
property. Its scalar exterior-square equation can have order six
even when the flat bundle has symplectic monodromy, as emphasized
in \cite[Remark 4]{YZ}. The primitive rank-five bundle exists in
either case; what changes is the particular section being used.

\subsection{Ramification of the rank-one kernel}
\label{III:subsec:rank-one-ramification-four}

In rank one, a zero of the next fundamental map delays the appearance
of independent derivative directions. The polarization determines
this delay exactly. On the constant-rank-one locus of
Section~\ref{III:subsec:rank-one-three}, let $s$ be a nonvanishing generator
of $N$, extend it to a frame $(s,e)$ of $\EE$, and write
\[
 s'=\alpha s+\nu_t e,\qquad \kappa=Q(e,e').
\]
The coefficient $\nu_t$ represents \eqref{III:eq:nu-three}; constant
rank one makes $\kappa$ a unit. Since $Q(s,s')=Q(s,s'')=0$,
\begin{equation*}
 h:=Q(s,s''')=-Q(s',s'')=-\nu_t^2\kappa.
\end{equation*}
The Gram matrix of $s,s',s'',s'''$ has determinant $h^4$.
Its Pfaffian shows that $\operatorname{Wr}_s$ is a nonzero constant times $h^2$
in a flat determinant frame. Thus
\begin{equation}\label{III:eq:rank-one-W-order-four}
 \operatorname{ord}_{t_0}\operatorname{Wr}_s=4\operatorname{ord}_{t_0}\nu_t.
\end{equation}

\begin{theorem}[Kernel-line apparent exponents]
\label{III:thm:rank-one-exponents-four}
Suppose $\nu\not\equiv0$ and its zero order at a regular
constant-rank-one point is $m\geq0$. The kernel-line scalar
equation has solution vanishing orders
\begin{equation*}
 0,\qquad m+1,\qquad m+2,\qquad 2m+3.
\end{equation*}
For $m>0$ the point is apparent, with these exponents and trivial
local monodromy; for $m=0$ the equation is ordinary.
\end{theorem}
\begin{proof}
Rescale $s$ by a holomorphic unit to remove $\alpha$.
Then $s'=\nu_t e$. In a horizontal frame, integration of this
identity shows the first two independent Taylor directions at
orders $0,m+1$. Since $Q(e,e')$ is a unit, $e'(0)$ lies outside
$\EE(0)$ and supplies a third direction at order $m+2$.
By the same rank-jump assertion, the first three vanishing orders are $0,m+1,m+2$.
Equations~\eqref{III:eq:rank-one-W-order-four} and
\eqref{III:eq:Wronskian-orders-four} force the last to be $2m+3$.
Apply Lemma~\ref{III:lem:apparent-exponents-four}.
\end{proof}

These exponents occur at regular points of the flat connection on
the constant-rank-one locus.

\subsection{An explicit dictionary with Yang--Zudilin}
\label{III:subsec:YZ-dictionary-four}

The kernel line identifies the rank-one input in the holomorphic
period construction of Yang--Zudilin \cite[\S4]{YZ}. Write $u(\tau)$
for their auxiliary function called $t$. Their symmetric matrix is
\begin{equation*}
 T(\tau)=\begin{pmatrix}\tau_1&\tau_2\\\tau_2&\tau_3\end{pmatrix},
 \qquad\tau_1=\tau,\quad \tau_2'=-u,\quad\tau_3'=u^2.
\end{equation*}
Hence $T'=\left(\begin{smallmatrix}1&-u\\-u&u^2\end{smallmatrix}\right)$
has rank one. In a standard horizontal symplectic trivialization,
take the graph frame
\[
 \omega_1=(1,0,\tau_1,\tau_2),\qquad
 \omega_2=(0,1,\tau_2,\tau_3).
\]
The radical of its Kodaira--Spencer form is generated by
$s=u\omega_1+\omega_2$, and differentiation gives
\begin{equation*}
 s'=u'\omega_1,\qquad
 Q(s,s''')=-(u')^2,\qquad
 \det(s,s',s'',s''')=(u')^4.
\end{equation*}
For the determinant use the displayed order of horizontal coordinates.
For example, $\omega_1'=(0,0,1,-u)$ and
$\omega_1''=(0,0,0,-u')$; substituting into the four derivative
rows proves the last two equalities. Thus the function
$v=u'$ of \cite[Theorem 4]{YZ} is the coefficient of the kernel
map $\nu$ in this chart. Their $G_3=v'/v$ is its logarithmic
derivative. The zeros of this coefficient have the geometric
interpretation in Theorem~\ref{III:thm:rank-one-exponents-four}.

There is also a direct dictionary for their fifth-order invariants.
Let $z$ be their invariant parameter and $w$ their weight-one
factor, on an open set where the logarithmic derivatives are
defined and $dz/d\tau\ne0$. Put
\[
 G_1=z'/z,\qquad G_2=w'/w,\qquad G_3=v'/v,
 \qquad \theta=z\partial_z.
\]
Primes here denote $\tau$-derivatives. Yang--Zudilin define
\begin{align*}
 p_1&=\frac{2G_1'-G_1(G_2+G_3)}{2G_1^2},\\
 p_2&=\frac{24G_3'-20(G_2+G_3)'+5(G_2+G_3)^2-8G_3^2}
                 {20G_1^2},\\
 p_3&=\frac{-10G_3'''+40G_3G_3''+21(G_3')^2
                         -54G_3^2G_3'+9G_3^4}{50G_1^4}.
\end{align*}
Their monic fifth-order operator has solutions
$w,\tau_1w,\tau_2w,\tau_3w,(\tau_1\tau_3-\tau_2^2)w$.
These are a coordinate realization of a lift of the primitive
Hodge determinant line.

\begin{theorem}[The normalized coefficient dictionary]
\label{III:thm:YZ-invariants-four}
For the fifth-order operator of \cite[Theorem 4]{YZ}, with
derivation $\theta$ and the binomial normalization of Section~\ref{III:sec:schwarzians-modular-constructions},
\begin{equation*}
 \boxed{\quad
 a_1=2p_1,\qquad
 I_2=\tfrac12(p_2-2\theta p_1-p_1^2),\qquad
 W_3=0,\qquad W_4=\tfrac25p_3.
 \quad}
\end{equation*}
Thus $p_3$ is, up to the displayed constant, the fourth Schwarzian
of that scalar fifth-order equation. The parameter for this
normalization is locally $\log z$, since $\theta=\partial_{\log z}$.
\end{theorem}
\begin{proof}
The coefficients of $\theta^4,\theta^3,\theta^2$ in that operator
are
\[
 10p_1,\quad 10\theta p_1+35p_1^2+5p_2,\quad
 5\theta^2p_1+\tfrac{15}{2}\theta p_2
       +45p_1\theta p_1+50p_1^3+30p_1p_2.
\]
The coefficient of $\theta$ is
\begin{align*}
 &46p_1^2\theta p_1+14p_2\theta p_1+24p_1^4+2p_3+4p_2^2
       +11p_1\theta^2p_1+\tfrac92\theta^2p_2+\theta^3p_1\\
 &\hspace{16mm}+7(\theta p_1)^2+52p_1^2p_2+30p_1\theta p_2.
\end{align*}
Dividing these four coefficients by $5,10,10,5$, respectively,
gives $a_1,a_2,a_3,a_4$.
Set $A=p_2-2\theta p_1-p_1^2$. The normalization recurrence in
Proposition~\ref{I:prop:local-normalization} now gives
\[
 I_2=A/2,\qquad I_3=\tfrac34\theta A,\qquad
 I_4=\tfrac9{10}\theta^2A+\tfrac45 A^2+\tfrac25p_3.
\]
For $N=5$, formula~\eqref{I:eq:quartic-W4} is
$W_4=I_4-2\theta I_3+\tfrac65\theta^2I_2-\tfrac{16}{5}I_2^2$.
Substitution cancels the $A$ terms and proves the assertion;
the formula for $W_3$ is immediate.
\end{proof}

The proposition identifies Yang--Zudilin's $p_3$ with the normalized
fourth Schwarzian of their scalar equation. The geometric input is
the kernel line: $T'$ has rank one, so the transverse matrix formula
\eqref{III:eq:period-Schwarzian-three} does not apply. Their holomorphic
$T$ also need not lie in $\mathfrak H_2$; it is distinct from the
nonholomorphic positive period matrix constructed in \cite{YZ}.

\Needspace{16\baselineskip}
\subsection{A finite differential test for generic endomorphisms}
\label{III:subsec:endomorphisms-four}

An endomorphism of an abelian scheme acts horizontally on de Rham
cohomology and preserves the Hodge bundle. In a normal frame it must
therefore commute with the matrix Schwarzian and all its derivatives.
This gives a finite-dimensional linear-algebra test.
\begin{definition}[Schwarzian jet algebra]\label{III:def:jet-algebra}

For a family of matrices, $\operatorname{Cent}$ denotes its common
centralizer. Write $\mathscr B_m(t_0)$ for the unital complex algebra
generated by $\cov^jI(t_0)$, $0\leq j\leq m$.
\end{definition}

\begin{theorem}[Differential and semisimple centralizer criteria]
\label{III:thm:finite-jet-endomorphisms-four}
Let $\mathcal A\to U$ be an abelian scheme of relative dimension
$g$ over a smooth complex algebraic curve. Suppose $\beta_t$ is
invertible near a point $t_0$. If, for some finite $m$, the only
unital semisimple
complex subalgebra of $\operatorname{Cent}(\mathscr B_m(t_0))$ is
$\C\Id$, then $\End(\mathcal A_{\overline{\C(U)}})=\Z$.
In particular, the simpler condition
\begin{equation}\label{III:eq:centralizer-four}
 \{T\in\operatorname{Mat}_g(\C):
        [T,\cov^jI(t_0)]=0\text{ for }0\leq j\leq m\}=\C\Id
\end{equation}
suffices. In genus two either of the following is sufficient:
$C(t_0)\ne0$; or $\mathscr B_m(t_0)$ contains a nonzero nilpotent
matrix. Thus $q(t_0)=0$, $S(t_0)\ne0$ is also a certificate.
The same endomorphism conclusion holds in genus two on a nonempty
constant-rank-one locus if $\nu\not\equiv0$ there.
\end{theorem}
\begin{proof}
\noindent\textit{1. Extend the endomorphism action.} The rational geometric generic endomorphism algebra is finite
 dimensional and semisimple by Poincar\'e reducibility. Its elements
can be defined simultaneously after a finite extension of $\C(U)$
and spread over an open part of the resulting curve. Their de Rham
actions are horizontal and preserve $\EE$. On a disk about the
regular point $t_0$, a horizontal action is constant in a flat frame
and extends across the omitted points, after choosing a branch of
the finite cover. Preservation of $\EE$ extends by continuity.

\noindent\textit{2. The normal-frame equation.} Use a local normal frame. If an endomorphism acts on the column of
Hodge frame elements by $\phi(\omega)=T\omega$, apply it to
$\omega''=-I\omega$ and differentiate $T\omega$ twice. Independence
of $\omega,\omega'$ gives $T'=0$ and $[T,I]=0$, hence
$[T,I^{(j)}]=0$ for every $j$. These statements are $\cov T=0$
and $[T,\cov^jI]=0$ in any frame. \noindent\textit{3. The semisimple image.} The complexified image of the
endomorphism algebra on $\EE(t_0)$ is therefore a semisimple
subalgebra of the displayed centralizer. Indeed, this image is a quotient of the complexified semisimple
algebra (or its opposite in the column-frame convention), and is
therefore semisimple.
\noindent\textit{4. Integrality.} The hypothesis makes this image scalar. Since $T$ is constant in
the normal frame and $\EE+\nD\EE=\HH$, the full action on $\HH$
is the same constant scalar. For an integral endomorphism that
scalar preserves the Betti lattice, hence is an integer.
Faithfulness on $H^1$ finishes the argument.

\noindent\textit{5. The genus-two criteria.} For $2\times2$ matrices, the centralizer of a nonscalar matrix is
the commutative algebra spanned by it and $\Id$. Two noncommuting
matrices therefore have scalar common centralizer, proving the
$C\ne0$ case. If a nonzero nilpotent $N$ belongs to
$\mathscr B_m$, its centralizer is $\C\Id+\C N$, whose only
unital semisimple subalgebra is $\C\Id$. Finally
$S^2=(q/2)\Id$, and $S\in\mathscr B_0$, which proves the stated
nilpotent-Schwarzian criterion. In the rank-one case apply
Proposition~\ref{III:prop:rank-one-cyclicity-three} and the same lattice
argument.
\end{proof}

For fixed $m$, \eqref{III:eq:centralizer-four} is the rank-$g^2-1$
test for the finite linear map
\[
 T\longmapsto([T,I],[T,\cov I],\ldots,[T,\cov^mI]).
\]
One may equivalently state the semisimple criterion using the algebra
generated by all jets: its increasing chain of finite-jet
subalgebras stabilizes because $\operatorname{Mat}_g(\C)$ is
finite dimensional. The argument gives no uniform bound on that jet order. The conclusion
implies absolute simplicity of the geometric generic fiber. Special
fibers may have additional endomorphisms, as the example in
Section~\ref{III:sec:explicit-gm} will illustrate.

The preceding results distinguish the geometric loci needed in
explicit computations:
\begin{center}
\small
\begin{tabular}{@{}p{.34\textwidth}p{.31\textwidth}p{.26\textwidth}@{}}
\toprule
Presentation & Cyclicity condition & Apparent exponents\\
\midrule
Primitive Hodge determinant, transverse case
 & $C\ne0$ & $0,1,2,3,5$ at a simple cyclic zero\\[3pt]
Null line, transverse case
 & $b\ne0$ on one sheet, $c_0\ne0$ on the other
 & $0,1,3,4$ over a simple zero of $q$\\[3pt]
Kernel line, rank-one case
 & $\nu\ne0$ & $0,m+1,m+2,2m+3$ when $\operatorname{ord}\nu=m$\\
\bottomrule
\end{tabular}
\end{center}
All entries concern regular points of the flat connection on
the stated constant-rank locus; the null-line exponents refer to the
failing sheet. They identify failures of the
chosen scalar presentation separately from singular fibers.

\section{Explicit genus-two Gauss--Manin systems}\label{III:sec:explicit-gm}

Throughout this section, use the genus-two convention~\ref{III:conv:genus-two}.

For an algebraic family of genus-two curves, polynomial reduction
computes the Gauss--Manin connection in a Hodge-compatible frame.
We determine its transverse locus and matrix Schwarzian, then identify
the two polynomials governing the scalarizations of
Section~\ref{III:sec:explicit-scalarizations}. The reduction is classical
\cite{KatzOda}; compare the universal genus-two calculation in
\cite[\S2.2]{CMY}.

\subsection{An algebraic frame and exact reduction}\label{III:subsec:exact-reduction-five}
A basis consisting of holomorphic differentials and differentials of
the second kind makes both the filtration and the connection explicit.
Fix constants $a,b\in\C$ and let $\pi:\mathcal C\to U_{a,b}$ be the
smooth projective family obtained from
\begin{equation}\label{III:eq:pencil-five}
 y^2=f_t(x)=x^5+ax^3+bx+t,\qquad
 U_{a,b}=\{t\in\mathbb A^1_\C:d_{a,b}(t)\ne0\},
\end{equation}
where
\begin{equation*}
\begin{split}
 d_{a,b}(t)={}&3125t^4+(108a^5-900a^3b+2000ab^2)t^2\\
              &+16b^3(a^2-4b)^2.
\end{split}
\end{equation*}
This polynomial is $\operatorname{disc}_x(f_t)$. The sections
\[
 \eta_j=\left[x^j\frac{dx}{y}\right]\quad(0\le j\le3),\qquad
 \eta=(\eta_0,\eta_1,\eta_2,\eta_3)^t
\]
form an algebraic frame of $\HH=R^1\pi_*\Omega^\bullet_{\mathcal C/U_{a,b}}$.
The Hodge frame is $\omega=(\eta_0,\eta_1)^t$, and $\eta_2,\eta_3$ are
differentials of the second kind at infinity. A row $r$ represents the
section $r\eta$; thus, if $\nD\eta=A\eta$, its derivative is represented by
\begin{equation}\label{III:eq:row-recurrence-five}
 \nD(r\eta)=(r'+rA)\eta.
\end{equation}
All matrices below follow this row-coefficient and column-frame convention.

For each $0\le j\le3$, there are unique polynomials $U_j,V_j$ satisfying
\begin{equation}\label{III:eq:bezout-five}
 x^j=U_jf_t+V_j(f_t)_x,\qquad \deg_xV_j<5.
\end{equation}
Indeed, $(f_t)_x$ is invertible modulo $f_t$ on $U_{a,b}$; take
$V_j\equiv x^j(f_t)_x^{-1}\pmod {f_t}$ and then divide by $f_t$ to obtain
$U_j$. Since differentiation at fixed $x$ gives
$\nD\eta_j=-\tfrac12[x^jdx/y^3]$, the identity
\[
 d_x(V_j/y)=\left((\partial_xV_j)/y-V_j(f_t)_x/(2y^3)\right)dx
\]
implies
\begin{equation}\label{III:eq:reduction-five}
 \nD\eta_j=-\frac12\left[(U_j+2\partial_xV_j)\frac{dx}{y}\right].
\end{equation}
Here $\deg_xU_j\le3$ and $\deg_x\partial_xV_j\le3$, so no further
reduction at infinity is needed. The coefficient of $x^k$ on the right
is $A_{jk}$. Equations~\eqref{III:eq:bezout-five}--\eqref{III:eq:reduction-five}
therefore certify every row of the connection by an exact differential.
They also show that its only possible finite poles lie on $d_{a,b}=0$.

We fix the following normalization of the polarization throughout the
examples:
\begin{equation}\label{III:eq:polarization-five}
 Q_a=\bigl(Q(\eta_i,\eta_j)\bigr)_{0\le i,j\le3}
 =\frac43
 \begin{pmatrix}
 0&0&0&1\\0&0&3&0\\0&-3&0&a\\-1&0&-a&0
 \end{pmatrix}.
\end{equation}
To compute this pairing, put $x=u^{-2}$ near infinity. Then
\[
 \eta_j=-2u^{2-2j}\bigl(1-\tfrac12au^4+O(u^8)\bigr)du.
\]
The residue pairing of a local primitive of $\eta_i$ with $\eta_j$ gives
$Q(\eta_0,\eta_3)=4/3$, $Q(\eta_1,\eta_2)=4$, and
$Q(\eta_2,\eta_3)=4a/3$; the other upper-triangular entries vanish.
This yields~\eqref{III:eq:polarization-five}, including its sign convention.
Flatness takes the form
\begin{equation*}
 AQ_a+Q_aA^t=0.
\end{equation*}
In particular $\tr A=0$, and the fixed de Rham frame has constant
symplectic volume.

\subsection{The exact transverse locus}
The matrix presentation exists exactly when the derivatives of the
two Hodge sections are independent modulo the Hodge bundle. Write
the connection in the block notation of \eqref{III:eq:block-elimination-three}:
\begin{equation}\label{III:eq:blocks-five}
 A=\begin{pmatrix}B&C_0\\D&E_0\end{pmatrix}.
\end{equation}
Thus $C_0$ represents the directional Kodaira--Spencer map in the chosen
Hodge and quotient frames. The letter $C$ without a subscript continues to
mean the Schwarzian commutator $[S,\cov S]$ from~\eqref{III:eq:q-c-three}.

\begin{proposition}\label{III:prop:transversality-five}
For~\eqref{III:eq:pencil-five},
\begin{equation}\label{III:eq:det-KS-five}
 \det C_0=\frac{3a(3a^2-10b)}{d_{a,b}(t)}.
\end{equation}
Consequently the Hodge frame is transverse at every smooth finite
parameter if and only if $a(3a^2-10b)\ne0$.
\end{proposition}
\begin{proof}
The first two reductions give
\[
 C_0=\frac1{d_{a,b}}
 \begin{pmatrix}2t\mathfrak u&-3\mathfrak v\\
                 \mathfrak w&6t\mathfrak u\end{pmatrix},
\]
where
\begin{align*}
 \mathfrak u&=(a^2-5b)(9a^2-20b),\\
 \mathfrak v&=12a^4b-88a^2b^2+125at^2+160b^3,\\
 \mathfrak w&=4a^3b^2+75a^2t^2-16ab^3-250bt^2.
\end{align*}
The polynomial identity
\[
 4t^2\mathfrak u^2+\mathfrak v\mathfrak w
     =a(3a^2-10b)d_{a,b}
\]
proves~\eqref{III:eq:det-KS-five}. All these identities hold over $\mathbb Q(a,b,t)$.
\end{proof}

The determinant separates two distinct constructions. On its nonzero
locus the full Hodge frame gives the matrix equation of
Proposition~\ref{III:prop:matrix-pf-three}. On the two factors
$a=0$ and $3a^2=10b$, the kernel-line construction of
Section~\ref{III:subsec:rank-one-three} applies instead; these pencils are
worked out in Section~\ref{III:sec:rank-one-applications}. Their intersection
$a=b=0$ is isotrivial after a finite base change and will be excluded.

\subsection{A transverse pencil and its complete connection}
The choice $a=b=1$ gives a transverse family on its whole smooth
affine base. For the remainder of this section and
Section~\ref{III:sec:explicit-scalarizations}, fix
\begin{equation}\label{III:eq:main-family-five}
 \mathcal C_t:\ y^2=x^5+x^3+x+t,\qquad
 d(t)=3125t^4+1208t^2+144,\qquad U=\{d\ne0\}.
\end{equation}
The four roots of $d$ are simple. In the frame $\eta$, the blocks
in~\eqref{III:eq:blocks-five} are
\begin{align}
 B&=\frac1{2d}\begin{pmatrix}
 -5t(375t^2+76)&4(25t^2-36)\\
 -10(43t^2+12)&-t(625t^2-108)
 \end{pmatrix},\label{III:eq:main-B-five}\\
 C_0&=\frac1d\begin{pmatrix}
 88t&-3(125t^2+84)\\-(175t^2+12)&264t
 \end{pmatrix},\qquad \det C_0=-\frac{21}{d},\notag\\
 D&=\frac1d\begin{pmatrix}
 10t(25t^2+8)&-(305t^2+36)\\
 -(t^2-12)&4t(125t^2+18)
 \end{pmatrix},\notag\\
 E_0&=\frac1{2d}\begin{pmatrix}
 t(625t^2+68)&-6(175t^2+12)\\
 4(65t^2+12)&3t(625t^2+68)
 \end{pmatrix}.\label{III:eq:main-E-five}
\end{align}
The blocks specify the connection over $\mathbb Q[t,d^{-1}]$.
Together with the row recurrence \eqref{III:eq:row-recurrence-five},
they determine all the scalar equations below by rational operations
and differentiation.

\begin{proposition}\label{III:prop:matrix-main-five}
The periods of $\omega=(\eta_0,\eta_1)^t$ satisfy
\begin{equation*}
 Y''+2a_1Y'+a_2Y=0,
\end{equation*}
where
\begin{align}
 a_1&=\frac1{7d}\begin{pmatrix}
 t(21875t^2+9244)&-2(3875t^2-1092)\\
 -2(2275t^2-372)&t(21875t^2-788)
 \end{pmatrix},\label{III:eq:a1-five}\\
 a_2&=\frac1{28d}\begin{pmatrix}
 147(125t^2+84)&-3960t\\-12936t&45(175t^2+12)
 \end{pmatrix}.\label{III:eq:a2-five}
\end{align}
The resulting first-order system in $(Y,Y')$ is equivalent to the full
Gauss--Manin system on $U$. In particular this matrix presentation has no
additional finite poles on the smooth base.
\end{proposition}
\begin{proof}
Let $O=(\Id_2\ 0)$, $O_1=OA$, and $O_2=O_1'+O_1A$. Substitution gives
the identity
\begin{equation*}
 O_2+2a_1O_1+a_2O=0
\end{equation*}
in all four de Rham coordinates. Equivalently, substitute
\eqref{III:eq:main-B-five}--\eqref{III:eq:main-E-five} into
\eqref{III:eq:block-elimination-three}. The change of frame from $\eta$ to
$(\omega,\nD\omega)^t$ has matrix
$\left(\begin{smallmatrix}\Id_2&0\\B&C_0\end{smallmatrix}\right)$,
whose determinant is $-21/d$. It is invertible everywhere on $U$.
Thus the relation has exactly the four-dimensional solution space of
the original connection. The denominators in~\eqref{III:eq:a1-five} and
\eqref{III:eq:a2-five} prove the assertion about finite poles.
\end{proof}

For the polarization~\eqref{III:eq:polarization-five} with $a=1$, the
directional form from~\eqref{III:eq:K-three} is
\begin{equation}\label{III:eq:K-five}
 {\mathsf K}=-\frac4d\begin{pmatrix}
 125t^2+84&-88t\\-88t&175t^2+12
 \end{pmatrix},\qquad \det {\mathsf K}=\frac{112}{d}.
\end{equation}
This is the matrix ${\mathsf K}$ of
Proposition~\ref{III:prop:normal-form-three}; hence it satisfies
${\mathsf K}'=-a_1{\mathsf K}-{\mathsf K}a_1^t$ and $I{\mathsf K}={\mathsf K}I^t$. The displayed algebraic frame
suffices for the calculations.

\subsection{Schwarzian covariants and a finite-jet certificate}
The endomorphism criterion requires only finitely many derivatives
at one regular point. Compute $I,S,\cov I$, and $c$ from
\eqref{III:eq:I-S-D-three} and \eqref{III:eq:q-c-three}. At $t=0$ they are
\begin{align*}
 I(0)&=\frac1{1008}\begin{pmatrix}-7769&0\\0&-689\end{pmatrix},
 &S(0)&=\begin{pmatrix}-295/84&0\\0&295/84\end{pmatrix},
 \\
 \cov I(0)&=\begin{pmatrix}0&72257/1512\\72257/10584&0\end{pmatrix},
 &c(0)&=-\frac{454363969114225}{14114653056}.
\end{align*}
These values follow by differentiating the rational matrices
\eqref{III:eq:a1-five}--\eqref{III:eq:a2-five}.

\begin{corollary}\label{III:cor:generic-end-five}
The family~\eqref{III:eq:main-family-five} admits no local polarized
Hodge-compatible splitting near $t=0$. Its geometric generic Jacobian
has endomorphism ring $\Z$ and is absolutely simple.
\end{corollary}
\begin{proof}
Since $q(0)\ne0$ and $C(0)\ne0$, the splitting criterion of
Proposition~\ref{III:prop:splitting-three} excludes such a splitting.
Theorem~\ref{III:thm:finite-jet-endomorphisms-four} gives the endomorphism
assertion. A nonsimple abelian variety has a nontrivial idempotent in
its rational endomorphism algebra, so it cannot have endomorphism ring $\Z$.
\end{proof}

The generic endomorphism conclusion for this pencil is also accessible
from the established theory of hyperelliptic Jacobians; see~\cite{Zarhin}.
The additional information here is an exact certificate in the filtered
Gauss--Manin connection. It concerns the geometric generic fiber. Indeed,
at the certificate point $t=0$ the curve has the extra automorphism
$(x,y)\mapsto(-x,iy)$, which does not extend as an endomorphism over an
open neighborhood of the parameter curve. This distinguishes an
endomorphism of one fiber from a horizontal endomorphism preserving
the Hodge filtration along the family.

\subsection{Two polynomials controlling the scalarizations}
The zeros of $q$ and $c$ will give the two distinguished apparent
loci. Both functions are even in $t$, so their numerators can be
written as polynomials in $u=t^2$. Put
\begin{align}
 D_2(u)&=3125u^2+1208u+144,\label{III:eq:polynomials-five}\\
 H_4(u)&=22968750000u^4-214863009375u^3-205941087232u^2
 \notag\\&\hspace{16mm}-56411909744u-4911163776,\notag\\
 N_6(u)&=67291259765625u^6+842772656250000u^5
 \notag\\&\hspace{16mm}+1019287360500000u^4+660155943176960u^3
 \notag\\&\hspace{16mm}+196495468427520u^2+24928225026048u
                  +1039421030400.\notag
\end{align}
The subscripts denote polynomial degrees, not modular weights. Set
$P_8(t)=H_4(t^2)$ and $N_{12}(t)=N_6(t^2)$, so $d=D_2(t^2)$.

\begin{proposition}\label{III:prop:qc-five}
For the main pencil,
\begin{equation}\label{III:eq:qc-five}
 q=\frac{N_{12}}{98d^4},\qquad
 c=-\frac{28800P_8^2}{16807d^7},\qquad
 \mathcal R=\frac{c^2}{q^5}
       =\frac{26542080000d^6P_8^4}{N_{12}^5}.
\end{equation}
The polynomials $P_8$ and $N_{12}$ are squarefree and mutually coprime;
both are coprime to $d$.
\end{proposition}
\begin{proof}
Substitute~\eqref{III:eq:a1-five}--\eqref{III:eq:a2-five} into the definitions
of $I,S,\cov S$. Taking the two traces gives the first two identities
after cancellation. Their quotient gives the third, since
$28800^2\cdot98^5/16807^2=26542080000$.
For squarefreeness and coprimality, the Euclidean algorithm over
$\mathbb Q$ gives
\[
 \gcd(P_8,P_8')=\gcd(N_{12},N_{12}')
 =\gcd(P_8,N_{12})=\gcd(P_8,d)=\gcd(N_{12},d)=1.
\]
The exact reductions are reproduced by the accompanying verification
scripts.
\end{proof}

The function $\mathcal R$ is invariant under the frame and parameter
changes of Section~\ref{III:sec:filtered-connections}; it is nonconstant
already by~\eqref{III:eq:qc-five}, since it has poles at the roots of
$N_{12}$ and zeros at the roots of $P_8$. Section~\ref{III:sec:explicit-scalarizations}
identifies these two loci as failures of different geometrically selected
cyclic frames. Section~\ref{III:sec:igusa-applications} then expresses them
by scalar Siegel modular forms on the period image.

\subsection{Binary covariants and normalized Schwarzian variation}
The binary comparison of Theorem~\ref{III:thm:binary-comparison-four}
can be made explicit on this pencil. Set $T=\cov S$ and form the three
symmetric matrices ${\mathsf K},S{\mathsf K},T{\mathsf K}$. The determinant $\mathfrak j$ of their coefficient
rows $(\Phi_{11},2\Phi_{12},\Phi_{22})$ satisfies
\begin{equation}\label{III:eq:J-explicit-five}
 \mathfrak j=-\frac{15360P_8}{7d^5}.
\end{equation}
Substitute ${\mathsf K},S{\mathsf K},(\cov S){\mathsf K}$ into the coefficient determinant
\eqref{III:eq:joint-J-four}; cancellation gives
\eqref{III:eq:J-explicit-five}. The unsquared identity
\eqref{III:eq:W-J-four} will then give the primitive cyclic determinant,
including its sign.

Recall from \eqref{III:eq:R-three} that $T^\perp$ is the component of
$T$ orthogonal to $S$ for the trace pairing. On $q\ne0$,
\eqref{III:eq:trace-relations-three} gives
$\tr((T^\perp)^2)=-c/(2q)$ and $[S,T^\perp]=C$.
Replacing $T{\mathsf K}$ by $T^\perp {\mathsf K}$ does not change $\mathfrak j$. In particular
the invariant can be expressed entirely in binary quadratic data:
\begin{equation*}
 \mathcal R=\frac{\mathfrak j^4}{4(\det {\mathsf K})^6q^5}
          =-\frac{\mathfrak j^4}{128\det {\mathsf K}\,\det(S{\mathsf K})^5}.
\end{equation*}
The second identity uses $\det S=-q/2$; the first uses
$c=-\mathfrak j^2/(2(\det {\mathsf K})^3)$ from Section~\ref{III:sec:scalarizations}.

There is a local geometric interpretation of this quotient. Choose
a fourth root of $q$ on a simply connected open set where $q\ne0$
and define a parameter $\xi$ by $d\xi=q^{1/4}dt$. The parameter law
\eqref{III:eq:parameter-I-three} gives
\begin{equation*}
 \overline S=q^{-1/2}S,\qquad
 \tr(\overline S^2)=1,\qquad
 \cov_\xi\overline S=q^{-3/4}T^\perp,
\end{equation*}
and hence
\begin{equation*}
 \mathcal R=4\left[\tr\bigl((\cov_\xi\overline S)^2\bigr)\right]^2.
\end{equation*}
Thus $\mathcal R$ is four times the square of the trace norm of
$\cov_\xi\overline S$. The fourth-root choice changes the local
normalization, but not this final square. At a root of $P_8$, $q$ is a unit and $T^\perp=0$:
in a frame with ${\mathsf K}=\Id$ the trace-free self-adjoint matrices form a
two-dimensional quadratic space, and the complement of a nonisotropic
$S$ is a nondegenerate line. Thus $\tr((T^\perp)^2)=0$ forces
$T^\perp=0$ there. The primitive apparent divisor consequently detects
stationarity of the normalized trace-free Schwarzian.

\subsection{Nilpotent Schwarzian points}
The semisimple refinement in
Theorem~\ref{III:thm:finite-jet-endomorphisms-four} also applies at
zeros of $q$ where $S$ is nonzero. This is useful because scalar
ratios involving $q$ cease to be regular there, whereas the matrix
criterion remains defined.

In the main pencil, the twelve zeros of $q$ satisfy $C\ne0$ by
\eqref{III:eq:qc-five}. At each, $S^2=0$ and $S\ne0$. Both the commutator criterion and the
nilpotent criterion therefore give the generic endomorphism conclusion.
The matrix equation is regular at these points, even though the scalar
ratio $\mathcal R$ has poles.

\section{Distinguished scalar equations and their apparent divisors}
\label{III:sec:explicit-scalarizations}

Throughout this section, use the genus-two convention~\ref{III:conv:genus-two}.

We continue with the transverse pencil~\eqref{III:eq:main-family-five}.
The matrix equation is regular throughout $U$, but its scalar
presentations need not be. We shall distinguish three choices: the
individual differential $\eta_0$, the Hodge determinant
$v=\eta_0\wedge\eta_1$, and the tautological null line of the directional
polarization. The last two are selected by the filtered symplectic
connection. Their apparent divisors are controlled, respectively, by
$c=\tr(C^2)$ and $q=\tr(S^2)$.

\subsection{A scalar equation from one holomorphic differential}
Cramer's rule turns a cyclic section into a monic scalar equation.
For $s=r_0\eta$, define $r_{j+1}=r_j'+r_jA$ and
\[
 R_s=\begin{pmatrix}r_0\\r_1\\r_2\\r_3\end{pmatrix}.
\]
On $\det R_s\ne0$ there is a unique row
$(\ell_0,\ell_1,\ell_2,\ell_3)=-r_4R_s^{-1}$, and the corresponding
periods satisfy
\begin{equation}\label{III:eq:scalar-operator-six}
 L_s z=0,\qquad
 L_s=\partial_t^4+\ell_3\partial_t^3+
          \ell_2\partial_t^2+\ell_1\partial_t+\ell_0.
\end{equation}
Here $\ell_j$ is the coefficient of $\partial_t^j$. In the binomial
convention of Part~\ref{part:I} it equals $\binom4j a_{4-j}$.
The scalar operator depends on the generator: replacing $s$ by
$\varphi s$ gives $L_{\varphi s}=\varphi L_s\varphi^{-1}$.

For $s=\eta_0$, the calculation is small enough to display in full. Put
$p(t)=96t^2+91$. Then
\begin{equation}\label{III:eq:eta0-W-six}
 \det R_{\eta_0}=\frac{32400p}{d^3},
\end{equation}
and
\begin{align}
 \ell_3&=\frac{12t(250000t^4+323031t^2+52660)}{pd},
 \label{III:eq:eta0-coefficients-six}\\
 \ell_2&=\frac{3(4860000t^4+7151921t^2+491372)}{2pd},\notag\\
 \ell_1&=\frac{42t(93000t^2+192659)}{pd},\notag\\
 \ell_0&=\frac{65835(32t^2+133)}{16pd}.\notag
\end{align}
Substitution verifies
$r_4+\ell_3r_3+\ell_2r_2+\ell_1r_1+\ell_0r_0=0$ in the four-dimensional
de Rham module. This supplies an exact equation for all periods of
$\eta_0$, without choosing any individual solution.

\begin{proposition}\label{III:prop:eta0-apparent-six}
The equation~\eqref{III:eq:scalar-operator-six} with
coefficients~\eqref{III:eq:eta0-coefficients-six} has exactly two additional
finite apparent singularities on $U$, at the roots of $p$. Their local
exponents are $0,1,2,4$, and their local monodromy is trivial.
\end{proposition}
\begin{proof}
The remainder of $d$ modulo $p$ is $16652141/9216$, so both roots
of $p$ belong to $U$. At each, \eqref{III:eq:eta0-W-six} has a simple
zero and the connection is regular. Lemma~\ref{III:lem:apparent-exponents-four}
gives the exponents and trivial monodromy. The displayed coefficients
have no other poles on $U$. Abel's identity also gives
$\ell_3=3d'/d-p'/p$, whose residue is $-1$ at either apparent point.
\end{proof}

Thus the two apparent points disappear when the full Hodge frame is
retained.

\subsection{The primitive fifth-order equation}
Let $\mathcal P=\Lambda^2_0\HH$ and $v=\eta_0\wedge\eta_1$ as in
Section~\ref{III:subsec:primitive-four}. Theorem~\ref{III:thm:primitive-commutator-four}
and~\eqref{III:eq:qc-five} give the following exact comparison.

\begin{theorem}\label{III:thm:primitive-example-six}
With the primitive volume fixed in~\eqref{III:eq:primitive-basis-four},
\begin{equation*}
 \mathcal W(v)=\frac{983040P_8(t)}{d(t)^6}.
\end{equation*}
The Hodge determinant is cyclic in $\mathcal P$ precisely on
$U\setminus\{P_8=0\}$. Its fifth-order scalar equation has eight
additional finite apparent singularities, with local exponents
$0,1,2,3,5$ and trivial local monodromy.
\end{theorem}
\begin{proof}
Use the unsquared identity \eqref{III:eq:W-J-four}. Equations
\eqref{III:eq:K-five} and \eqref{III:eq:J-explicit-five} give
\[
 \mathcal W(v)=-4\det({\mathsf K})\mathfrak j
 =-4\frac{112}{d}\left(-\frac{15360P_8}{7d^5}\right)
 =\frac{983040P_8}{d^6}.
\]
This fixes the sign in the prescribed primitive volume.
By Proposition~\ref{III:prop:qc-five}, $P_8$ is squarefree and coprime
to $d$. Its eight roots are therefore regular parameters with simple
cyclicity zeros. Corollary~\ref{III:cor:primitive-apparent-four} gives
the asserted exponents and monodromy.
\end{proof}

There is also a compact specification of the entire scalar operator.
Use the exterior basis
\[
 (01,02,03,12,13,23)
 =(\eta_0\wedge\eta_1,\eta_0\wedge\eta_2,
   \eta_0\wedge\eta_3,\eta_1\wedge\eta_2,
   \eta_1\wedge\eta_3,\eta_2\wedge\eta_3).
\]
Let $A^{[2]}$ be defined by
$\nD(\eta_i\wedge\eta_j)
 =\sum_k A_{ik}\eta_k\wedge\eta_j+
             \sum_k A_{jk}\eta_i\wedge\eta_k$.
Starting with $v_0=(1,0,0,0,0,0)$, set
\begin{equation}\label{III:eq:primitive-recursion-six}
 v_{j+1}=v_j'+v_jA^{[2]}\qquad(0\le j\le4).
\end{equation}
Primitivity is the single constant relation
$(v_j)_{03}+3(v_j)_{12}+(v_j)_{23}=0$.
Consequently deletion of coordinate $03$ identifies $\mathcal P$ with
a fixed five-dimensional coefficient space. Write $\bar v_j$ for
these five coordinates and let $V$ have rows $\bar v_0,\ldots,\bar v_4$.
If $V[j\leftarrow\bar v_5]$ denotes replacement of its $j$th row,
the operator is
\begin{equation}\label{III:eq:primitive-operator-six}
 L_v=\partial_t^5+\sum_{j=0}^4 b_j\partial_t^j,
 \qquad
 b_j=-\frac{\det V[j\leftarrow\bar v_5]}{\det V}.
\end{equation}
Equations~\eqref{III:eq:main-B-five}--\eqref{III:eq:main-E-five} and
\eqref{III:eq:primitive-recursion-six}--\eqref{III:eq:primitive-operator-six}
are a complete rational definition of $L_v$. The deleted coordinate
introduces no additional locus: its projection is an isomorphism on
the entire constant primitive subspace. In particular
\begin{equation}\label{III:eq:primitive-Abel-six}
 b_4=6d'/d-P_8'/P_8.
\end{equation}
The coefficients $b_j$ have at most simple poles at the eight apparent
points, by Cramer's rule in a regular primitive frame.

Thus the zeros of $C$ are precisely the failures of cyclicity of the
primitive Hodge determinant.

\subsection{Writing the primitive operator in scalar covariants}
\label{III:subsec:primitive-coefficients-six}
For computations one can go further than Cramer's rule. The entire
primitive fifth-order equation has a compact formula in scalar traces
of the matrix Schwarzian and its covariant derivatives. Its denominators involve $c$, so it is suited to studying the
primitive equation near the zeros of $q$.

On $c\ne0$ define, in the fixed parameter $t$,
\begin{equation}\label{III:eq:rho-delta-six}
 \tau=\tr I,\qquad
 \rho=\frac{c'}{2c},\qquad
 \delta_*=\frac{\tr\bigl([\cov S,\cov^2 S]C\bigr)}{c}.
\end{equation}
All three functions are invariant under Hodge-frame changes.

\begin{proposition}\label{III:prop:primitive-coefficients-six}
In a normal Hodge frame, with $a_1=0$ and ${\mathsf K}=\Id$, the operator of
$v=\omega_1\wedge\omega_2$ is
\begin{equation}\label{III:eq:primitive-trace-operator-six}
 L_v=\partial_t^5+B_4\partial_t^4+B_3\partial_t^3
                       +B_2\partial_t^2+B_1\partial_t+B_0,
\end{equation}
where
\begin{align}
 B_4&=-\rho,& B_3&=2\tau+\delta_*,
 \label{III:eq:primitive-trace-coefficients-six}\\
 B_2&=5\tau'-2\rho\tau,\notag\\
 B_1&=4\tau''-3\rho\tau'+2\delta_*\tau+2q,\notag\\
 B_0&=\tau'''-\rho\tau''+\delta_*\tau'+3q'-2\rho q.\notag
\end{align}
In an arbitrary Hodge frame, the operator of its determinant section
is obtained by replacing every $\partial_t$ in
\eqref{III:eq:primitive-trace-operator-six} by $\partial_t+\gamma$,
where $\gamma=\tr a_1$, with all $B_j$ on the left.
\end{proposition}
\begin{proof}
\noindent\textit{1. A scalar equation for the trace-free entries.} In a normal frame the matrices $S,S'$ are trace-free symmetric
matrices. Since $C=[S,S']\ne0$, they are independent and span that
two-dimensional space. Write $S''=\rho_0S'-\delta_0S$. Then
$C'=[S,S'']=\rho_0C$ and $[S',S'']=\delta_0C$. Taking the relevant
traces gives $\rho_0=c'/(2c)$ and
$\delta_0=\tr([S',S'']C)/c$. Thus
\begin{equation}\label{III:eq:traceless-ODE-six}
 S''=\rho S'-\delta_*S.
\end{equation}
\noindent\textit{2. The primitive exterior equation.} Write
\[
 I=\frac\tau2\Id+
       \begin{pmatrix}x&y\\y&-x\end{pmatrix},\qquad
 q=2(x^2+y^2).
\]
Equation~\eqref{III:eq:traceless-ODE-six} says that both $x$ and $y$
satisfy $u''=\rho u'-\delta_*u$. Differentiate
$e'=f$, $f'=-Ie$ in the primitive basis
\eqref{III:eq:primitive-basis-four}. Substitution of these two relations
in the fifth derivative of $e_1\wedge e_2$ gives
\[
 v^{(5)}-\rho v^{(4)}+(2\tau+\delta_*)v'''
        +(5\tau'-2\rho\tau)v''+B_1v'+B_0v=0.
\]
\noindent\textit{3. The scalar coefficients.} The coefficient in each of the five primitive basis directions vanishes;
using $q=2(x^2+y^2)$ and $q'=4(xx'+yy')$ gives exactly
\eqref{III:eq:primitive-trace-coefficients-six}. This direct exterior
calculation is also verified as a polynomial identity in the jets of
the three entries of $I$ in the accompanying code. Cyclicity on
$c\ne0$ gives uniqueness of the resulting operator.

\noindent\textit{4. Return to the original Hodge frame.} For the frame assertion, choose $M$ with $M'=Ma_1$ so that $M\omega$
is normal. If $\lambda=\det M$, then $\lambda'/\lambda=\gamma$
and the determinant section changes by $v\mapsto\lambda v$.
The operator for the original section is therefore
$\lambda^{-1}L_{\lambda v}\lambda$, which replaces
$\partial_t$ by $\partial_t+\gamma$. The scalar functions in
\eqref{III:eq:rho-delta-six}, together with $q$, are unchanged by this
frame transformation, completing the proof.
\end{proof}

For the main pencil, the two logarithmic terms are especially simple:
\begin{equation*}
 \gamma=\frac{d'}{2d},\qquad
 \rho=\frac{P_8'}{P_8}-\frac72\frac{d'}d.
\end{equation*}
The coefficient of $\partial_t^4$ after the scalar conjugation is
$5\gamma-\rho=6d'/d-P_8'/P_8$, recovering
\eqref{III:eq:primitive-Abel-six}. The remaining coefficients are given
by~\eqref{III:eq:primitive-trace-coefficients-six} and the rational matrices
\eqref{III:eq:a1-five}--\eqref{III:eq:a2-five}; no integration of the normalizing
frame is necessary. Their displayed invariant denominators involve
$c$, not $q$, so the formula is valid across the twelve zeros of $q$
in this example. The formula therefore reconstructs the primitive scalar operator from
these specified covariants of the matrix equation.

\subsection{An explicit distinction between exterior orders five and six}
The primitive module is intrinsic, but the exterior-square scalar
equation of an arbitrary fourth-order presentation need not have order
five. For $s=\eta_0$,
\begin{equation*}
 Q(s,\nD s)=-\frac{4(125t^2+84)}d\not\equiv0.
\end{equation*}
Thus $s\wedge\nD s$ has a nonzero projection to the trivial quotient
of $\Lambda^2\HH$. In the exterior basis above, its first six derivative
rows have determinant, at $t=1$,
\begin{equation}\label{III:eq:order-six-certificate}
 \frac{20912756039932870656000000}
 {36600668831158425482549281168458881161}\ne0.
\end{equation}
Its scalar exterior-square equation therefore has order six.

In contrast, $v=\eta_0\wedge\eta_1$ belongs to $\mathcal P$ and has
order five by Theorem~\ref{III:thm:primitive-example-six}. This concretely
illustrates the distinction in the Yang--Zudilin comparison of
Section~\ref{III:subsec:YZ-dictionary-four}: a symplectic rank-four connection,
a primitive rank-five connection, and a scalar exterior-square operator
are related objects, but the choice of cyclic section matters.

\subsection{The algebraic null-line cover}\label{III:subsec:null-cover-six}
Theorem~\ref{III:thm:null-lines-four} selects an unordered pair of scalar
presentations from ${\mathsf K}$. In the algebraic frame, a line generated by
$\eta_0+r\eta_1$ is null exactly when
\begin{equation*}
 (175t^2+12)r^2-176tr+(125t^2+84)=0.
\end{equation*}
Put $\alpha=175t^2+12$, $\beta=125t^2+84$. The two roots are
\begin{equation}\label{III:eq:null-cover-six}
 r_\pm=\frac{88t\pm w}{\alpha},\qquad
 \widetilde U:\ w^2=-7d(t).
\end{equation}
This double cover is connected and \'etale over $U$; its smooth
compactification has genus one because $d$ is a squarefree quartic.
Its points parametrize the null lines of ${\mathsf K}$.
The coordinate $t$ is a valid local parameter everywhere on
$\widetilde U$, and differentiation in its function field is determined by
\begin{equation}\label{III:eq:quadratic-differentiation-six}
 w'=\frac{d'}{2d}w.
\end{equation}
The deck involution $\sigma(w)=-w$ exchanges the two scalarizations.

For computations it is convenient to clear the denominator in
\eqref{III:eq:null-cover-six} and use
\begin{equation}\label{III:eq:null-generators-six}
 s_+=\alpha\eta_0+(88t+w)\eta_1,\qquad s_-=\sigma(s_+).
\end{equation}
These polynomial sections generate the null lines on $\alpha\ne0$.
At a zero of $\alpha$ one of them vanishes; the resulting displayed
Wronskian includes a generator zero. The second chart below removes
this extra factor before the intrinsic divisor is stated.

To compute the cyclic determinant, first compute the pairing of a
generator with its third derivative. In the basis $1,w$ of the
quadratic field, the required numerator uses the polynomials
\begin{align*}
 F_8(t)&=1435546875t^8+5975287500t^6+2313346928t^4
 \\&\hspace{18mm}+175760448t^2-12234240,\\
 G_4(t)&=459375t^4-283624t^2-79248,\\
 H_\pm(t,w)&=F_8(t)\mp44tG_4(t)w.
\end{align*}
The subscripts on $F_8,G_4$ denote degrees. The symbols $H_\pm$ refer
to functions on the double cover and are distinct from $H_4(u)$ in
\eqref{III:eq:polynomials-five}.

\begin{proposition}\label{III:prop:null-determinants-six}
In the fixed de Rham basis and polarization,
\begin{align*}
 h_\pm:=Q(s_\pm,\nD^3s_\pm)&=\frac{4H_\pm}{d^2},
 \\
 \operatorname{Wr}_{s_\pm}:=\det(s_\pm,\nD s_\pm,\nD^2s_\pm,\nD^3s_\pm)
   &=-\frac{3H_\pm^2}{d^4}.
\end{align*}
Their norm identities are
\begin{equation}\label{III:eq:null-norm-six}
 H_+H_-=\alpha^2N_{12},\qquad
 \operatorname{Wr}_{s_+}\operatorname{Wr}_{s_-}=\frac{9\alpha^4N_{12}^2}{d^8}.
\end{equation}
\end{proposition}
\begin{proof}
Apply the row recurrence~\eqref{III:eq:row-recurrence-five} to
$(\alpha,88t\pm w,0,0)$, using~\eqref{III:eq:quadratic-differentiation-six}.
Reduction modulo $w^2+7d$ gives the formula for $h_\pm$.
Nullity gives $Q(s,s')=Q(s,s'')=0$ and $Q(s',s'')=-h$.
The Pfaffian of the Gram matrix of $(s,s',s'',s''')$ is therefore
$-h^2$. Since $\operatorname{Pf}(Q_1)=16/3$, the change-of-basis
identity for Pfaffians gives $\operatorname{Wr}_s=-3h^2/16$. This proves the
second formula. Finally
\[
 H_+H_-=F_8^2+13552t^2G_4^2d=\alpha^2N_{12}
\]
is a polynomial identity; squaring it gives the Wronskian norm.
\end{proof}

\begin{lemma}[Changing the null-line generator]
\label{III:lem:null-chart}
At a root of $\alpha$ on the sheet $w=-88t$, the polynomial
generator $s_+$ in~\eqref{III:eq:null-generators-six} has a simple
zero. The same null line has the nonvanishing generator
$\widetilde s_+=(88t-w)\eta_0+\beta\eta_1$, and
$s_+=\alpha(88t-w)^{-1}\widetilde s_+$.
Consequently this zero adds four to the order of the cyclic
determinant of $s_+$ relative to that of $\widetilde s_+$.
The conjugate statement holds for $s_-$.
\end{lemma}
\begin{proof}
The identity
\begin{equation*}
 \alpha\beta=(88t+w)(88t-w)
\end{equation*}
gives the complementary generator
\[
 \widetilde s_+=(88t-w)\eta_0+\beta\eta_1.
\]
At a root of $\alpha$, the polynomial section $s_+$ vanishes simply
on the sheet $w=-88t$, whereas $\widetilde s_+$ is nonzero.
Locally $s_+=\alpha(88t-w)^{-1}\widetilde s_+$, so this generator
zero contributes order four to its cyclic determinant. 
The scaling law~\eqref{III:eq:cyclic-W-four} gives the order four;
applying the deck involution gives the other chart.
\end{proof}

\begin{theorem}\label{III:thm:null-apparents-six}
The intrinsic null-line scalarization on $\widetilde U$ has twelve
finite apparent points, one over each root of $N_{12}(t)$. At each,
the cyclic determinant has order two, the exponents are $0,1,3,4$,
and local monodromy is trivial. The conjugate scalarization and the
primitive fifth-order equation are ordinary at that point.
\end{theorem}
\begin{proof}
\noindent\textit{1. Locate the intrinsic zeros.} Exact Euclidean division gives
\[
 \gcd(N_{12},N_{12}')=1,\qquad
 \gcd(N_{12},d\alpha F_8tG_4P_8)=1.
\]
At a root of $N_{12}$, the section~\eqref{III:eq:null-generators-six} is
nonzero because $\alpha\ne0$. If $H_+$ and $H_-$ vanished at the same
point, their sum would force $F_8=0$, contrary to the displayed
coprimality. Their product has a simple zero by~\eqref{III:eq:null-norm-six},
so precisely one factor vanishes, and it vanishes simply.
Theorem~\ref{III:thm:null-lines-four} now gives exponents $0,1,3,4$ on
the failing line. The other line is cyclic, and $P_8\ne0$ makes
the primitive Hodge determinant cyclic as well.

\noindent\textit{2. Remove zeros of the chosen generator.}
Lemma~\ref{III:lem:null-chart} removes the factor $\alpha$ in
\eqref{III:eq:null-norm-six}. The null line
itself has no cyclicity zero there, since $N_{12}$ is a unit.
Away from these chart changes, \eqref{III:eq:null-norm-six} excludes
any further zeros. Regularity of the pulled-back connection proves
trivial local monodromy at all twelve intrinsic apparent points.
\end{proof}

Writing $Z_N$ for the intrinsic cyclicity divisor of the tautological
line on $\widetilde U$, we have in this example
\begin{equation}\label{III:eq:null-pushforward-six}
 \pi_*Z_N=2\operatorname{div}_0N_{12}
        =2\operatorname{div}_0(q\,dt^4).
\end{equation}
This is the divisor norm of Theorem~\ref{III:thm:null-lines-four}, including
multiplicity. Its support consists of twelve points on the base;
one of the two points above each lies in the tautological cyclicity
divisor. The conjugate operator describes the same construction after the
deck involution.

\subsection{Complete null-line operators over the quadratic field}
The analytic normal-frame factorization~\eqref{III:eq:null-operator-four}
explains the geometry, but the operators can also be specified entirely
over $\mathbb Q(t,w)$. Let $v_j$ now denote the row of $\nD^j s_+$,
and set
\[
 h=Q(v_0\eta,v_3\eta),\qquad
 k=Q(v_2\eta,v_3\eta),\qquad
 j=Q(v_4\eta,v_3\eta).
\]
These three pairings are rational functions in the quadratic field;
in particular $h=4H_+/d^2$. On $h\ne0$, put
\begin{align}
 \ell_3^+&=-2h'/h,&
 \ell_2^+&=-(k+h''+\ell_3^+h')/h,\label{III:eq:null-operator-six}\\
 \ell_1^+&=-(k'+\ell_3^+k)/h,&
 \ell_0^+&=(\ell_1^+h'-\ell_2^+k-j)/h.\notag
\end{align}
Then $L_{s_+}=\partial_t^4+\sum_{j=0}^3\ell_j^+\partial_t^j$,
and $L_{s_-}=\sigma(L_{s_+})$. Together with the displayed matrix $A$
and the generator~\eqref{III:eq:null-generators-six}, these equations
determine every coefficient by rational operations and differentiation.

To verify them, pair
$v_4+\ell_3^+v_3+\ell_2^+v_2+\ell_1^+v_1+\ell_0^+v_0$
successively with $v_0,v_1,v_2,v_3$ using $Q_1$.
The first pairing gives $2h'+\ell_3^+h=0$; the next two give
$k=-h''-\ell_2^+h-\ell_3^+h'$ and
$k'+\ell_3^+k+\ell_1^+h=0$. The last gives the formula for
$\ell_0^+$. Nondegeneracy of the cyclic Gram matrix forces the whole
row residual to vanish. Thus the relation holds in the differential module and annihilates
all periods of the chosen generator.

Eliminating $h$ and $k$ also gives the familiar fourth-order
self-duality relation
\begin{equation}\label{III:eq:selfduality-six}
 \ell_1=\tfrac12\ell_2\ell_3-\tfrac18\ell_3^3+\ell_2'
             -\tfrac34\ell_3\ell_3'-\tfrac12\ell_3''.
\end{equation}
This is the standard fourth-order self-duality identity
\cite[\S2.4]{vanStraten}. Here it follows from the polarization and
the nullity of the selected line.

For this example the scalar exterior square of either null-line
equation has order exactly five. One direct check uses the nonzero
generator $s=\eta_0+r_+\eta_1$ at $t=0$, with $w(0)=12\zeta$ and
$\zeta^2=-7$. For $u=s\wedge\nD s$, the determinant of its first five
derivative rows in columns $(01,02,03,12,13)$ is
$-2234138434375/53747712$. The same value holds on the other sheet.
Nullity puts these rows in the primitive module, and this nonzero
minor proves the claimed order. Together with \eqref{III:eq:order-six-certificate}, this gives scalar
exterior-square orders five and six for different sections of the
same connection.

\subsection{Geometric comparison of the three apparent loci}
The finite apparent loci of the main pencil can now be recorded without
reference to expanded higher-order coefficients:
\begin{center}
\small
\begin{tabular}{@{}p{.34\textwidth}p{.32\textwidth}p{.24\textwidth}@{}}
\toprule
Presentation & Apparent locus on the smooth base & Exponents\\
\midrule
Second-order Hodge matrix & none & ordinary\\
Fourth order from $\eta_0$ & $p(t)=0$ (two points) & $0,1,2,4$\\
Primitive Hodge determinant & $P_8(t)=0$ (eight points) & $0,1,2,3,5$\\
Tautological null line & one point over each $N_{12}(t)=0$ & $0,1,3,4$\\
\bottomrule
\end{tabular}
\end{center}
The last row counts twelve points on the null-line cover. The primitive
row has differential rank five; the other rows present the rank-four
Gauss--Manin module, after the indicated base change when necessary.

At a root of $N_{12}$, Cayley--Hamilton gives $S^2=(q/2)\Id=0$.
Moreover $P_8\ne0$, hence $C\ne0$ and $S\ne0$. Thus $S$ is a nonzero
nilpotent there; its induced dual action on the Hodge bundle singles
out the failing null line, as in Theorem~\ref{III:thm:null-lines-four}.
At a root of $P_8$, by contrast, $q\ne0$ and the primitive cyclic
frame fails although both null-line presentations are ordinary.
The spectral cover of $S$ has local equation $\mu^2=q/2$; it should
not be confused with the polarization cover~\eqref{III:eq:null-cover-six},
which is \'etale at these points.

Finally, on the smooth affine base the rational invariant of
Section~\ref{III:sec:filtered-connections} has divisor
\begin{equation}\label{III:eq:R-divisor-six}
 \operatorname{div}_U(\mathcal R)
 =4\operatorname{div}_0(P_8)-5\operatorname{div}_0(N_{12}).
\end{equation}
Its zeros and poles thus compare the two selected cyclicity failures.

\section{Rank-one pencils and canonical Picard--Fuchs equations}
\label{III:sec:rank-one-applications}

Throughout this section, use the genus-two convention~\ref{III:conv:genus-two}.

On the two rank-one branches of \eqref{III:eq:det-KS-five}, the
Kodaira--Spencer kernel gives a distinguished fourth-order equation.
An exact differential identity shows that its next fundamental map
never vanishes. The resulting scalar equations therefore have no
finite apparent singularities on the smooth base.

\subsection{The pencil \texorpdfstring{$x^5+bx+t$}{x5+bx+t}}
For the first branch, the weighted Euler identity gives both the
kernel section and its first derivative. Fix $b\in\C^\times$ and put
\begin{equation}\label{III:eq:rankone-family-seven}
 \mathcal C_{b,t}:y^2=x^5+bx+t,\qquad
 d_b=3125t^4+256b^5,\qquad U_b=\{d_b\ne0\}.
\end{equation}
Use the de Rham frame $\eta$ and the polarization
\eqref{III:eq:polarization-five} with $a=0$. Reduction as in
\eqref{III:eq:reduction-five} gives
\begin{equation*}
 {\mathsf K}=-\frac{40b}{d_b}
 \begin{pmatrix}4b\\-5t\end{pmatrix}
 \begin{pmatrix}4b&-5t\end{pmatrix}.
\end{equation*}
This matrix has rank one at every point of $U_b$. Its kernel line
$N\subset\EE$ is generated by
\begin{equation*}
 s_b=5t\eta_0+4b\eta_1.
\end{equation*}
The simple shape of this generator has an explanation independent of
matrix elimination. Since $5t+4bx=5f_t-x(f_t)_x$,
\begin{equation}\label{III:eq:Euler-seven}
 \partial_t\!\left((5t+4bx)\frac{dx}{y}\right)
       =\frac72\frac{dx}{y}-d_x(x/y),\qquad
 \nD s_b=\frac72\eta_0.
\end{equation}
The pair $(s_b,\eta_0)$ is a Hodge frame everywhere on $U_b$ because
$b\ne0$. Thus the second fundamental map
$\nu:N\to(\EE/N)\otimes K_{U_b}$ never vanishes.

\begin{proposition}\label{III:prop:rankone-cyclic-seven}
The section $s_b$ is cyclic at every point of $U_b$. In the fixed
de Rham frame and the polarization~\eqref{III:eq:polarization-five},
\begin{equation}\label{III:eq:rankone-W-seven}
 Q(s_b,\nD^3s_b)=\frac{7840b^3}{d_b},\qquad
 \operatorname{Wr}_{s_b}=-\frac{11524800b^6}{d_b^2}.
\end{equation}
Its periods satisfy the fourth-order equation
\begin{equation}\label{III:eq:rankone-PF-seven}
 z^{(4)}+\frac{25000t^3}{d_b}z'''
 +\frac{80625t^2}{2d_b}z''
 +\frac{5625t}{d_b}z'
 -\frac{5355}{16d_b}z=0.
\end{equation}
There are no finite apparent singularities on the smooth base.
\end{proposition}
\begin{proof}
Cyclicity follows from~\eqref{III:eq:Euler-seven} and
Proposition~\ref{III:prop:rank-one-cyclicity-three}. For an exact coefficient
check, differentiate the row $(5t,4b,0,0)$ four times using
\eqref{III:eq:row-recurrence-five} and the reduction for~\eqref{III:eq:rankone-family-seven}.
The displayed coefficients annihilate the resulting row in all four
de Rham coordinates. Pairing the zeroth and third rows gives the first
identity in~\eqref{III:eq:rankone-W-seven}. Put $h=Q(s_b,\nD^3s_b)$. The Gram-Pfaffian argument of
Proposition~\ref{III:prop:null-determinants-six}, with
$\operatorname{Pf}(Q_0)=16/3$, gives $\operatorname{Wr}_{s_b}=-3h^2/16$ and hence
the second identity. Its numerator is a nonzero constant. Therefore
the cyclic frame is invertible at every smooth finite parameter.
Abel's identity is visible directly:
\[
 \frac{25000t^3}{d_b}=2d_b'/d_b=-\operatorname{Wr}_{s_b}'/\operatorname{Wr}_{s_b}.
\]
\end{proof}

This is a distinguished scalar equation because its line is intrinsic
to the filtration. Its particular coefficients still use the specified
generator $s_b$ and parameter $t$. The self-duality relation
\eqref{III:eq:selfduality-six} holds, either by nullity or by substitution.
The absence of apparent points is stronger than generic cyclicity:
it follows from the fact that $\nu$ is a unit everywhere on $U_b$.
For a general constant-rank-one family, a zero of order $m$ of $\nu$
instead gives exponents $0,m+1,m+2,2m+3$, as proved in
Theorem~\ref{III:thm:rank-one-exponents-four}.

\subsection{Hypergeometric comparison and local exponents}
Equation~\eqref{III:eq:rankone-PF-seven} has an exact classical
hypergeometric form. Set
\begin{equation}\label{III:eq:hypergeometric-coordinate-seven}
 v=-\frac{3125t^4}{256b^5},\qquad \vartheta=v\partial_v.
\end{equation}
It is the pullback of
\begin{equation*}
\begin{split}
 \bigl[&\vartheta(\vartheta-\tfrac14)(\vartheta-\tfrac12)
                       (\vartheta-\tfrac34)\\
       &-v(\vartheta-\tfrac7{40})(\vartheta+\tfrac1{40})
                  (\vartheta+\tfrac9{40})(\vartheta+\tfrac{17}{40})\bigr]z=0.
\end{split}
\end{equation*}
To check the identity in the Ore algebra, put $\theta=t\partial_t$.
Multiplication of~\eqref{III:eq:rankone-PF-seven} by $t^4d_b$ gives
\begin{equation}\label{III:eq:Euler-polynomial-seven}
\begin{split}
 256b^5\theta(\theta-1)(\theta-2)(\theta-3)
 +3125t^4(\theta-\tfrac7{10})(\theta+\tfrac1{10})
              (\theta+\tfrac9{10})(\theta+\tfrac{17}{10}).
\end{split}
\end{equation}
Replace $\theta$ by $4\vartheta$ and use~\eqref{III:eq:hypergeometric-coordinate-seven}.
The factor $t^4$ is on the left, in the Ore convention.

In the standard generalized hypergeometric notation~\cite[16.8.3]{DLMF},
the numerator parameters are $(-7,1,9,17)/40$ and the denominator
parameters are $(1/4,1/2,3/4)$. This classical form determines the local exponents and explains the
ramification of the invariant parameter.

The four finite roots of $d_b$ are simple. At each, the residue of the
coefficient of $z'''$ is $2$, and all lower coefficients have at most
simple poles. Hence the indicial polynomial is
$\lambda(\lambda-1)^2(\lambda-2)$ and the exponents are $0,1,1,2$.
These are singular fibers, unlike the apparent points of
Section~\ref{III:sec:explicit-scalarizations}. At infinity the leading
polynomial in~\eqref{III:eq:Euler-polynomial-seven}, expressed in the local
coordinate $u=1/t$, gives exponents
\[
 -\tfrac7{10},\qquad\tfrac1{10},\qquad\tfrac9{10},\qquad\tfrac{17}{10}.
\]
At $v=0$ the hypergeometric exponents are $0,1/4,1/2,3/4$; the
ramified pullback makes them $0,1,2,3$ at $t=0$, where the original
equation is ordinary. Thus the ramification in the hypergeometric
coordinate introduces no apparent singularity at the origin.

\begin{corollary}\label{III:cor:rankone-end-seven}
For every fixed $b\ne0$, the geometric generic Jacobian of
\eqref{III:eq:rankone-family-seven} has endomorphism ring $\Z$.
\end{corollary}
\begin{proof}
The map $\nu$ never vanishes by \eqref{III:eq:Euler-seven}, so the
rank-one endomorphism criterion in
Theorem~\ref{III:thm:finite-jet-endomorphisms-four} applies.
\end{proof}
After scaling over $\C$, these are the familiar pencils
$y^2=x^5-x+t$. Their generic endomorphism conclusion is already covered
by known results; see the discussion in~\cite{Zarhin}.
Here the kernel-line calculation proves the conclusion directly from
the filtered connection.

\subsection{The second rank-one factor}
The second rank-one branch has the same kernel-line mechanism, with
the two Hodge differentials interchanged. Fix $a\ne0$ and put
\begin{equation*}
 f_t=x^5+ax^3+\frac{3a^2}{10}x+t,\qquad
 D_a=54a^{10}+56250a^5t^2+9765625t^4.
\end{equation*}
Here $d_{a,3a^2/10}=D_a/3125$. With the same normalization
\eqref{III:eq:polarization-five},
\begin{equation*}
 {\mathsf K}=-\frac{100a}{D_a}
       \begin{pmatrix}125t\\3a^3\end{pmatrix}
       \begin{pmatrix}125t&3a^3\end{pmatrix},\qquad
 s_a=-3a^3\eta_0+125t\eta_1.
\end{equation*}
There is again an exact differential identity:
\begin{equation*}
 \partial_t\!\left((125tx-3a^3)\frac{dx}{y}\right)
   =\frac{225}{2}\frac{x\,dx}{y}
          -d_x\!\left(\frac{25x^2+10a}{y}\right).
\end{equation*}
Thus $\nD s_a=(225/2)\eta_1$, and $(s_a,\eta_1)$ is a Hodge
frame everywhere on $D_a\ne0$. The kernel is nonparallel and its
canonical scalarization has no finite apparent points.

The periods of $s_a$ satisfy
\begin{equation*}
\begin{split}
 z^{(4)}&+\frac{25000t(9a^5+3125t^2)}{D_a}z'''
 +\frac{625(357a^5+390625t^2)}{2D_a}z''\\
 &+\frac{9765625t}{D_a}z'
 -\frac{56109375}{16D_a}z=0.
\end{split}
\end{equation*}
It is obtained either by the row recurrence or by the pairing formulas
\eqref{III:eq:null-operator-six}. Those formulas apply here because
$Q(s_a,s_a')=Q(s_a,s_a'')=0$, even though ${\mathsf K}$ is degenerate.
For the polarization \eqref{III:eq:polarization-five},
\begin{equation*}
 Q(s_a,\nD^3s_a)=\frac{11390625a^7}{D_a},\qquad
 \operatorname{Wr}_{s_a}=-\frac{389239013671875a^{14}}{16D_a^2}.
\end{equation*}
The fourth-order row residual vanishes exactly, and
$\operatorname{Wr}_{s_a}=-3Q(s_a,\nD^3s_a)^2/16$ independently checks its determinant.
The generic endomorphism criterion applies just as for the first pencil.
These constructions give the scalar equations on both nonisotrivial
rank-one branches of \eqref{III:eq:det-KS-five}.

\subsection{Ranks and derivatives of periods}
\label{III:subsec:andre-correction-seven}
The first pencil separates differential generation from first-order
transversality. We compute its ranks and period derivatives, then
compare them with the statements of Andr\'e \cite{Andre}.

In the notation of~\cite[p.~1397]{Andre}, let
\[
 r=\operatorname{rank}(\mathscr D_{U_-}\EE/\EE),\quad
 r'=\operatorname{rank}\bigl(T_{U_-}\otimes\EE\longrightarrow\HH/\EE\bigr),
 \quad r''=\max_\partial\operatorname{rank}\beta_\partial.
\]
The symbol $\mathscr D_{U_-}$ here denotes the sheaf of differential
operators on the base, not the covariant derivative $\cov$.
\begin{proposition}[A rank-one pencil with full differential generation]
\label{III:prop:rank-one-full-generation}
For the genus-two family $\mathcal C_t:y^2=x^5-x+t$ over
$U_-=\operatorname{Spec}\C[t,(3125t^4-256)^{-1}]$, the ranks above are
\begin{equation}\label{III:eq:andre-ranks-seven}
 \boxed{\quad r'=r''=1,\qquad r=2.\quad}
\end{equation}
Its complex geometric monodromy is Zariski dense in $\Sp_4$.
There is a local horizontal integral cycle for which both
$P(t)=\int_{\gamma_t}(5t-4x)dx/y$ and $P'(t)$ are nonzero
Abelian integrals of the first kind.
\end{proposition}
\begin{proof}
\noindent\textit{1. The first-order and differential ranks.}
Specialize the first pencil to $b=-1$, and let
$\mathcal J=\operatorname{Jac}(\mathcal C/U_-)$ over
$U_-=\operatorname{Spec}\C[t,(3125t^4-256)^{-1}]$. Then
\begin{equation*}
 \partial_t\!\left((5t-4x)\frac{dx}{y}\right)
   =\frac72\frac{dx}{y}-d_x\!\left(\frac{x}{y}\right),
 \qquad s=5t\eta_0-4\eta_1.
\end{equation*}
Consequently $\nD s=\tfrac72\eta_0$. In the row convention of \eqref{III:eq:row-recurrence-five}, the
Kodaira--Spencer block is
\begin{equation}\label{III:eq:andre-KS-seven}
 C_0=\frac1{3125t^4-256}
 \begin{pmatrix}200t&480\\250t^2&600t\end{pmatrix}.
\end{equation}
It has rank one everywhere on $U_-$. Its left kernel is generated by
$(5t,-4)$, so $\ker\beta_t=\OO_{U_-}s$. The two sections $s$ and
$\nD s$ are independent holomorphic differentials.

The calculation yields~\eqref{III:eq:andre-ranks-seven}.
Indeed every tangent field on the curve is a local multiple of
$\partial_t$, proving the first equality from~\eqref{III:eq:andre-KS-seven}.
The nonzero cyclic determinant~\eqref{III:eq:rankone-W-seven} proves
$\mathscr D_{U_-}\EE=\HH$ and hence the value of $r$.

\noindent\textit{2. Monodromy.} The relevant monodromy hypothesis is also available. The polynomial
$h(x)=x^5-x$ has critical points $r_0$ satisfying $r_0^4=1/5$;
they are simple and their critical values $h(r_0)=-4r_0/5$ are distinct.
It is therefore a Morse polynomial, and its reduction in characteristic
$7$ is weakly supermorse in the sense of \cite[\S5.5.2]{KatzACT}:
the same critical-point calculation applies, since $2$ and $5$ are
invertible. For a quadratic character, \cite[Theorem 5.13(1)]{KatzACT}
gives geometric monodromy $\Sp_4$ for the family $y^2=u-h(x)$
over the complement of the critical values in that characteristic.
The smooth proper hyperelliptic family spreads over the corresponding
open parameter scheme in mixed characteristic. Pink's specialization
theorem \cite[Theorem 8.18.2]{KatzESDE} puts this group inside the
characteristic-zero geometric monodromy group. The polarization gives
the reverse inclusion in $\Sp_4$. Comparison with Betti cohomology
therefore gives Zariski-dense complex monodromy $\Sp_4$ as well.
The changes $u=-t$ and $y\mapsto iy$ identify the displayed sign
convention with our pencil. Passing to a finite \'etale cover to add level structure
preserves both the connected monodromy group and the computed ranks.

\noindent\textit{3. Differentiate an Abelian integral.}
On a simply connected analytic open set take a
horizontal integral cycle $\gamma_t$ and set
\begin{equation*}
 P(t)=\int_{\gamma_t}(5t-4x)\frac{dx}{y},\qquad
 P'(t)=\frac72\int_{\gamma_t}\frac{dx}{y}.
\end{equation*}
Choose $\gamma_t$ so that the right-hand side is not identically zero;
such a cycle exists because $\eta_0$ is a nonzero cohomology class.
Then both $P$ and $P'$ are nonzero Abelian integrals of the first kind.
\end{proof}

\begin{remark}[Comparison with the published rank statements]
\label{III:rem:andre-comparison}
There is already a comparison to make with
\cite[Theorem 3.1]{Andre}. With sufficiently fine level, the smallest
weakly special subvariety containing this full-monodromy period curve
is the full genus-two Siegel variety. Its universal family has
$r''=2$, because its tangent representation consists of all symmetric
$2\times2$ matrices and includes invertible matrices. The asserted
invariance of $r''$ under passage to that weakly special subvariety
would therefore disagree with \eqref{III:eq:andre-ranks-seven}.
The published proof of Theorem~3.1 invokes the tangent-density
assertion of Theorem~2.2. 

The full symplectic case is explicitly included in Andr\'e's restricted
PEM setting~\cite[pp.~1398, 1413]{Andre}. The equality asserted in his
published Theorem~4.4 would then give $r''=r'=r=2$, in conflict with
\eqref{III:eq:andre-ranks-seven}. There is also a direct comparison with
his Theorem~4.6. The coefficient $5t$ is algebraic on $U_-$ and is permitted by the
$\OO(U_-)$-linear interpretation recorded in footnote~10 of that paper.
Under this interpretation of the hypotheses, the conclusions of
Theorems~3.1, 4.4, and 4.6 would therefore conflict with
Proposition~\ref{III:prop:rank-one-full-generation}. The applicability
of the tangent-density assertion to this period curve remains
unresolved, so this comparison is conditional.
\end{remark}

\begin{remark}[First-order transversality]\label{III:rem:first-order-transversality}

Irreducibility of the connection forces $\mathscr D_U\EE=\HH$, but the first-order map
$\beta_t$ can still have a kernel. The elementary statement about
differentials themselves is precisely
\[
 \ker\beta_t=0
 \quad\Longleftrightarrow\quad
 \nD s\notin\EE\ \text{for every nonzero local meromorphic section }s\in\EE,
\]
interpreted over the function field. Accordingly, a transversality
hypothesis must be supplied before making this first-order conclusion;
full symplectic monodromy alone does not supply it in the example.
The kernel line and its nonzero next fundamental map $\nu$ show
how differential generation continues after the first derivative
remains in the Hodge bundle.
\end{remark}

\section{Igusa equations for period curves and apparent divisors}
\label{III:sec:igusa-applications}

Throughout this section, use the genus-two convention~\ref{III:conv:genus-two}.

The apparent loci are defined by polynomials in the parameter of a
period curve. Expressing that parameter in Igusa modular coordinates
turns the polynomials into scalar Siegel modular forms with the same
pullback zero divisors. We use the classical invariant-theoretic
and modular dictionaries of \cite{CvG,MS}.

\subsection{Binary-sextic conventions and Hodge weights}
The normalization of the sextic invariants fixes the constants in
all subsequent modular equations. Write the binary sextic of the
main pencil as
\begin{equation}\label{III:eq:sextic-eight}
 F_t(u_1,u_2)=u_1^5u_2+u_1^3u_2^3+u_1u_2^5+tu_2^6.
\end{equation}
We denote its Igusa--Clebsch invariants by
$\mathsf J_2,\mathsf J_4,\mathsf J_6,\mathsf J_{10}$, reserving
$I_j$ for the higher Schwarzians of Section~\ref{III:sec:schwarzians-modular-constructions}. The normalization of
the transvectant is \eqref{III:eq:transvectant-two}, with binary variables
$(u_1,u_2)$.

For a binary sextic $F$, put
\[
 i=(F,F)_4,\quad h=(i,i)_2,\quad
 y_1=(F,i)_4,\quad y_2=(i,y_1)_2,\quad y_3=(i,y_2)_2
\]
and introduce the Clebsch invariants
\[
 A_{\mathrm{Cl}}=(F,F)_6,\quad B_{\mathrm{Cl}}=(i,i)_4,\quad
 C_{\mathrm{Cl}}=(i,h)_4,\quad D_{\mathrm{Cl}}=(y_3,y_1)_2.
\]
The conversion used here is
\begin{align}
 \mathsf J_2&=-120A_{\mathrm{Cl}},\label{III:eq:clebsch-conversion-eight}\\
 \mathsf J_4&=-720A_{\mathrm{Cl}}^2+6750B_{\mathrm{Cl}},\notag\\
 \mathsf J_6&=8640A_{\mathrm{Cl}}^3-108000A_{\mathrm{Cl}}B_{\mathrm{Cl}}
                                    +202500C_{\mathrm{Cl}},\notag\\
 \mathsf J_{10}&=-62208A_{\mathrm{Cl}}^5+972000A_{\mathrm{Cl}}^3B_{\mathrm{Cl}}
                    +1620000A_{\mathrm{Cl}}^2C_{\mathrm{Cl}}\notag\\
 &\hspace{10mm}-3037500A_{\mathrm{Cl}}B_{\mathrm{Cl}}^2
                    -6075000B_{\mathrm{Cl}}C_{\mathrm{Cl}}-4556250D_{\mathrm{Cl}}.\notag
\end{align}
This fixes the Mestre/Sage convention~\cite{Sage}; in the dictionary
of~\cite[\S2.3]{MS} these are the invariants denoted $J_k$.
The explicit conversion matters: different normalizations alter the
integer coefficients of every modular relation below.

For~\eqref{III:eq:sextic-eight}, successive transvectants give
\begin{align*}
 A_{\mathrm{Cl}}&=-23/60,& B_{\mathrm{Cl}}&=103/11250,\\
 C_{\mathrm{Cl}}&=(6250t^2+99)/562500,&
 D_{\mathrm{Cl}}&=-\frac{39062500t^4-10237500t^2+8281}{56953125000}.
\end{align*}
Substitution into~\eqref{III:eq:clebsch-conversion-eight} yields
\begin{equation}\label{III:eq:Igusa-values-eight}
 \mathsf J_2=46,\qquad\mathsf J_4=-44,\qquad
 \mathsf J_6=2250t^2-72,\qquad\mathsf J_{10}=d(t).
\end{equation}
The last invariant is the polynomial discriminant $d(t)$.

Let $E_4,E_6$ be the normalized genus-two Eisenstein series of Section~\ref{III:sec:schwarzians-modular-constructions},
and choose the scalings of $\chi_{10},\chi_{12}$ as in~\cite{MS}.
The weighted identities are
\begin{align}
 \mathsf J_2&=-24\frac{\chi_{12}}{\chi_{10}},&
 \mathsf J_4&=4E_4,\label{III:eq:Igusa-modular-eight}\\
 \mathsf J_6&=-\frac83E_6-32E_4\frac{\chi_{12}}{\chi_{10}},&
 \mathsf J_{10}&=-2^{14}\chi_{10}.\notag
\end{align}
Each side has Hodge weight given by its subscript. These equations are
identities of weighted coordinates, or of sections after a compatible
trivialization of the Hodge line $\lambda=\det\EE$ from Section~\ref{III:sec:siegel-ore}.
The algebraic frame and a normalized analytic period frame give
different trivializations. If their generators of the Hodge line
satisfy $\sigma_{\rm alg}=\kappa(t)\sigma_{\rm an}$, then the
coefficients of a weight-$k$ section satisfy
\[
 f_{\rm an}(t)=\kappa(t)^k f_{\rm alg}(t).
\]
For example, \eqref{III:eq:Igusa-values-eight} gives
$E_4(Z(t))=-11\kappa(t)^4$. The factor cancels from every weight-zero
ratio.

\subsection{Recovering the parameter of the period image}
Weight-zero ratios remove the choice of Hodge trivialization. They
will recover a coordinate on the coarse period image. On the open
set where the denominators are nonzero, define
\begin{equation}\label{III:eq:modular-ratios-eight}
 R_4=\frac{E_4\chi_{10}^2}{\chi_{12}^2},\qquad
 R_6=\frac{E_6\chi_{10}^3}{\chi_{12}^3},\qquad
 R_{10}=\frac{\chi_{10}^6}{\chi_{12}^5}.
\end{equation}
Their subscripts name the invariant used in their numerator; each
ratio has weight zero.

\begin{proposition}\label{III:prop:Igusa-ratios-eight}
The restrictions of~\eqref{III:eq:modular-ratios-eight} to the period image
of~\eqref{III:eq:main-family-five} are
\begin{equation*}
 R_4=-\frac{1584}{529},\qquad
 R_6=\frac{432(3375t^2+904)}{12167},\qquad
 R_{10}=\frac{243d(t)}{102981488}.
\end{equation*}
Consequently the modular function
\begin{equation}\label{III:eq:Umod-eight}
 U_{\mathrm{mod}}(Z)=\frac{12167R_6(Z)-390528}{1458000}
\end{equation}
restricts to $t^2$ and is a rational coordinate on the coarse period
image.
\end{proposition}
\begin{proof}
The weight-zero versions of~\eqref{III:eq:Igusa-modular-eight} are
\[
 R_4=\frac{144\mathsf J_4}{\mathsf J_2^2},\qquad
 R_6=\frac{1728(3\mathsf J_6-\mathsf J_2\mathsf J_4)}{\mathsf J_2^3},
 \qquad R_{10}=\frac{486\mathsf J_{10}}{\mathsf J_2^5}.
\]
Substitute~\eqref{III:eq:Igusa-values-eight} to obtain the restrictions.
The coefficient of $t^2$ in $R_6$ is nonzero, proving recovery by
\eqref{III:eq:Umod-eight}. Conversely all four weighted invariants of the
family depend on $t^2$, and the isomorphism
$(x,y)\mapsto(-x,iy)$ identifies $\mathcal C_t$ and $\mathcal C_{-t}$.
Thus the invariant function field of the image is exactly $\C(t^2)$.
\end{proof}

Elimination gives two explicit holomorphic scalar forms vanishing
on the period image:
\begin{align*}
 F_{24}&=529E_4\chi_{10}^2+1584\chi_{12}^2,
 \\
 F_{22}&=48323E_4^3\chi_{10}-37752E_4E_6\chi_{12}
                     -6325E_6^2\chi_{10}\\
       &\hspace{18mm}+12317184000\chi_{10}\chi_{12}.
\end{align*}
All monomials in each line have the indicated weight. Substitution
of~\eqref{III:eq:Igusa-modular-eight} and~\eqref{III:eq:Igusa-values-eight}
annihilates both expressions identically. These are vanishing relations in the classical Igusa ring. Their
common zero scheme may have additional components; only their
restriction to the period image is used here.

At $t=0$, every displayed coarse invariant has zero first derivative,
but $\det C_0=-21/d$ is nonzero. The stabilizer of $\mathcal C_0$ acts on the parameter by
$t\mapsto-t$, so the coarse coordinate $t^2$ has zero differential
there. The Kodaira--Spencer map is computed before taking this
coarse quotient and still has rank two.

\subsection{Modular equations for the three apparent loci}
To turn a polynomial in $t^2$ into a modular form, substitute
$U_{\mathrm{mod}}$ and clear its denominator with a matching Hodge
weight. Set
\begin{equation*}
 A_{36}=12167E_6\chi_{10}^3-390528\chi_{12}^3,
 \qquad B_{36}=1458000\chi_{12}^3.
\end{equation*}
Both have weight $36$, and $U_{\mathrm{mod}}=A_{36}/B_{36}$.
For $P(u)$ of degree $r$, let
\begin{equation}\label{III:eq:homogenization-eight}
 P^{\rm h}(A,B)=B^rP(A/B).
\end{equation}
This is a homogeneous polynomial even where $B=0$.
Using~\eqref{III:eq:polynomials-five}, put
\begin{align*}
 \mathscr P_{36}&=96A_{36}+91B_{36},\\
 \mathscr D_{72}&=D_2^{\rm h}(A_{36},B_{36}),\\
 \mathscr H_{144}&=H_4^{\rm h}(A_{36},B_{36}),\\
 \mathscr N_{216}&=N_6^{\rm h}(A_{36},B_{36}).
\end{align*}
They are holomorphic scalar Siegel modular forms for the full genus-two
group, with the weights indicated by their subscripts. Homogenization
keeps their definitions explicit without a large monomial expansion.

\begin{theorem}\label{III:thm:modular-apparents-eight}
Let $\varphi:U\to\mathcal A_2$ be the period map of the main pencil,
understood on the moduli stack when recording line bundles. Its
pullback zero divisors satisfy
\begin{align}
 \operatorname{div}_0\varphi^*\mathscr P_{36}
      &=\operatorname{div}_0p,\label{III:eq:modular-divisor-pullback-eight}\\
 \operatorname{div}_0\varphi^*\mathscr H_{144}
      &=\operatorname{div}_0P_8,\notag\\
 \operatorname{div}_0\varphi^*\mathscr N_{216}
      &=\operatorname{div}_0N_{12}.\notag
\end{align}
The first two are the cyclicity divisors of the fourth-order equation
from $\eta_0$ and the primitive fifth-order equation, respectively.
For the tautological null-line cyclicity divisor $Z_N$,
\begin{equation*}
 \pi_*Z_N=2\operatorname{div}_0\varphi^*\mathscr N_{216}.
\end{equation*}
Thus the last pullback has the support of the null-line apparent
locus; the factor two is its cyclic determinant multiplicity.
\end{theorem}
\begin{proof}
The pullback of $\chi_{10}$ has no zero on $U$, as also follows from
$\mathsf J_{10}=d$ in the algebraic Hodge trivialization. Since
$\mathsf J_2=46$, the pullback of $\chi_{12}$ is nonzero there as well.
Thus $\varphi^*B_{36}$ is a nowhere-vanishing section of
$\varphi^*\lambda^{36}$. The identity $U_{\mathrm{mod}}\circ\varphi=t^2$
and~\eqref{III:eq:homogenization-eight} give
\begin{align*}
 \varphi^*\mathscr P_{36}&=(\varphi^*B_{36})p(t),\\
 \varphi^*\mathscr H_{144}&=(\varphi^*B_{36})^4P_8(t),\\
 \varphi^*\mathscr N_{216}&=(\varphi^*B_{36})^6N_{12}(t).
\end{align*}
Taking zero divisors proves~\eqref{III:eq:modular-divisor-pullback-eight}.
The scalar interpretations follow from
Proposition~\ref{III:prop:eta0-apparent-six} and
Theorems~\ref{III:thm:primitive-example-six}, \ref{III:thm:null-apparents-six};
the last multiplicity is~\eqref{III:eq:null-pushforward-six}.
\end{proof}

The nonzero restrictions prove that these modular forms are nonzero.
The theorem identifies their pullback divisors on the specified family;
their ambient zero divisors may contain other components.

Homogenization is an instance of the equivariant-evaluation method
of \cite{CvG}. Theorem~\ref{III:thm:modular-apparents-eight} gives the
differential interpretation of these particular restrictions as
cyclicity divisors.

\subsection{A modular extension of the differential invariant}
The same substitution extends the rational curve invariant to a
modular function. Using \eqref{III:eq:Umod-eight} in \eqref{III:eq:qc-five},
set
\begin{equation}\label{III:eq:Rhat-eight}
 \widehat{\mathcal R}(Z)=
 \frac{26542080000D_2(U_{\mathrm{mod}}(Z))^6H_4(U_{\mathrm{mod}}(Z))^4}
      {N_6(U_{\mathrm{mod}}(Z))^5}.
\end{equation}
Equivalently, it is the weight-balanced quotient
\begin{equation*}
 \widehat{\mathcal R}
 =\frac{26542080000 B_{36}^{2}\mathscr D_{72}^{6}
                                      \mathscr H_{144}^{4}}
        {\mathscr N_{216}^{5}}.
\end{equation*}
Both numerator and denominator have weight $1080$.
Therefore $\widehat{\mathcal R}$ is a meromorphic Siegel modular
function, and
\[
 \varphi^*\widehat{\mathcal R}=\mathcal R=c^2/q^5.
\]
It is nonconstant because~\eqref{III:eq:R-divisor-six} gives nontrivial
zeros and poles on the family.

The restriction equality is precise, but the extension is not
canonical. For example
\[
 F_{24}/\chi_{12}^2=529R_4+1584
\]
is a meromorphic modular function vanishing on this period image.
Adding a multiple of it to~\eqref{III:eq:Rhat-eight} gives another
extension wherever the expressions are defined. This ambiguity is
different from the frame and parameter invariance of $\mathcal R$:
the latter is intrinsic to the period curve, while an extension off
that curve requires additional data.

Here extension is possible because $U_{\mathrm{mod}}$ restricts to a
coordinate on the coarse period image. The geodesic construction of
Theorem~\ref{III:thm:geodesic-Schwarzian} instead gives tensors on the
tangent bundle from a chosen connection. Their relation to the actual
pencil includes the acceleration terms of
\eqref{III:eq:acceleration-comparison}.

\subsection{The coarse moduli image of the rank-one pencil}
The rank-one pencil gives a coarse coordinate with a different
ramification order.
For $y^2=x^5-x+t$, with $d_-=3125t^4-256$, one obtains
\begin{equation*}
 \mathsf J_2=-40,\qquad\mathsf J_4=-80,\qquad
 \mathsf J_6=320,\qquad\mathsf J_{10}=d_-.
\end{equation*}
Thus the modular functions~\eqref{III:eq:modular-ratios-eight} restrict to
\[
 R_4=-\frac{36}{5},\qquad R_6=\frac{1512}{25},\qquad
 R_{10}=-\frac{243d_-}{51200000}.
\]
Two equations for the image are
\begin{equation*}
 5E_4\chi_{10}^2+36\chi_{12}^2=0,\qquad
 25E_6\chi_{10}^3-1512\chi_{12}^3=0.
\end{equation*}
The coarse parameter is recovered by
\begin{equation*}
 t^4=\frac{256-(51200000/243)R_{10}}{3125}.
\end{equation*}
The isomorphism $(x,y)\mapsto(ix,\sqrt{i}\,y)$ takes $\mathcal C_t$
to $\mathcal C_{it}$, explaining the occurrence of $t^4$. At the
origin the displayed coarse modular coordinates have vanishing
derivatives through order three, even though the directional
Kodaira--Spencer map has rank one. This gives a second explicit
distinction between derivatives of coarse invariant coordinates and
the filtered connection.

Two questions arise from these comparisons. Which equivariant
contractions of the ambient Schwarzian and its derivatives are
independent both of the modular connection and of the tangent
direction? Such contractions would give functions on the modular
variety directly. Conversely, to what extent do the paired null-line
operators, together with their polarization data, determine the
original filtered connection? The formulas above give the forward
construction and its apparent divisors.

\Needspace{8\baselineskip}
\section*{Acknowledgments}
\addcontentsline{toc}{section}{Acknowledgments}
I thank Mehrzad Ajoodanian for sharing his ideas and his unpublished
preprints on non-abelian Schwarzians and theta--Hessian metrics on modular
curves, and for insightful conversations. The non-abelian perspective
developed in \cite{Ajoodanian2026} helped initiate this project.
I also thank Amin Najafi-Amin for his Mathematica computations of
low-degree scalar invariants, which helped reveal patterns in the
general formulas.

\clearpage
\phantomsection

\clearpage
\appendix
\markboth{Higher Schwarzians: Appendices}{}
\section{Explicit formulas in low weights}\label{I:app:formulas}

These formulas permit direct evaluation through weight six and show
the ordering prescribed by the construction. We use the binomial convention
$L=\sum_{k=0}^N\binom Nk a_kD^{N-k}$, with $a_0=1$.
Each formula applies when $m\le N$. Write $a=a_1$ and
$a^{(r)}=\delta^r(a)$; primes also denote $\delta$.
Products are taken in the displayed order.

\subsection{The normalized coefficients}
The coefficients $I_m$ are independent of $N$ once $N\ge m$.
Recall that
\begin{equation}\label{I:eq:app-I2}
 I_2=a_2-a'-a^2.
\end{equation}
The next four are
\begin{equation}\label{I:eq:app-I3}
\begin{aligned}
I_3&=a_3-3a_2a\\
 &\quad+2a^{3}-2a a'+2a' a-a''.
\end{aligned}
\end{equation}

\begin{equation}\label{I:eq:app-I4}
\begin{aligned}
I_4&=a_4-4a_3a+6a_2(a^2-a')\\
 &\quad-3a^{4}+3a^{2} a'+6a a' a-3a a''-3a' a^{2}\\
 &\quad+3(a')^{2}+3a'' a-a'''.
\end{aligned}
\end{equation}

\begin{equation}\label{I:eq:app-I5}
\begin{aligned}
I_5&=a_5-5a_4a+10a_3(a^2-a')\\
 &\quad-10a_2(a^3-aa'-2a' a+a'')\\
 &\quad+4a^{5}-4a^{3} a'-8a^{2} a' a+4a^{2} a''-12a a' a^{2}\\
 &\quad+12a (a')^{2}+12a a'' a-4a a'''+4a' a^{3}-4a' a a'\\
 &\quad-8(a')^{2} a+4a' a''-6a'' a^{2}+6a'' a'+4a''' a\\
 &\quad-a^{(4)}.
\end{aligned}
\end{equation}

\begin{equation}\label{I:eq:app-I6}
\begin{aligned}
I_6&=a_6-6a_5a+15a_4(a^2-a')\\
 &\quad-20a_3(a^3-aa'-2a' a+a'')\\
 &\quad+15a_2\bigl(a^4-a^2a'-2aa' a+aa''
                 -3a' a^2\\
 &\hspace{5em}+3(a')^2+3a'' a-a'''\bigr)\\
 &\quad-5a^{6}+5a^{4} a'+10a^{3} a' a-5a^{3} a''+15a^{2} a' a^{2}\\
 &\quad-15a^{2} (a')^{2}-15a^{2} a'' a+5a^{2} a'''+20a a' a^{3}-20a a' a a'\\
 &\quad-40a (a')^{2} a+20a a' a''-30a a'' a^{2}+30a a'' a'+20a a''' a\\
 &\quad-5a a^{(4)}-5a' a^{4}+5a' a^{2} a'+10a' a a' a-5a' a a''\\
 &\quad+15(a')^{2} a^{2}-15(a')^{3}-15a' a'' a+5a' a'''+10a'' a^{3}\\
 &\quad-10a'' a a'-20a'' a' a+10(a'')^{2}-10a''' a^{2}+10a''' a'\\
 &\quad+5a^{(4)} a-a^{(5)}.
\end{aligned}
\end{equation}

\subsection{The projective expressions}
Put $\Delta C=\delta(C)+[a,C]$. Formula \eqref{I:eq:rp-E} gives
\begin{equation}\label{I:eq:app-E}
 \begin{aligned}
 E_3&=I_3-\frac32\Delta I_2,\\
 E_4&=I_4-2\Delta I_3+\frac65\Delta^2 I_2,\\
 E_5&=I_5-\frac52\Delta I_4+\frac{15}{7}\Delta^2 I_3
                                      -\frac57\Delta^3 I_2,\\
 E_6&=I_6-3\Delta I_5+\frac{10}{3}\Delta^2 I_4
                         -\frac53\Delta^3 I_3+\frac5{14}\Delta^4 I_2.
 \end{aligned}
\end{equation}
These expressions have weighted covariance under M\"obius changes
of parameter.

\subsection{The higher Schwarzians}
Recall $A\odot B=(AB+BA)/2$ from \eqref{I:eq:rp-Sym}. The normalized
representatives of Theorem~\ref{I:thm:rp-higher} are
\begin{equation}\label{I:eq:app-W}
 \begin{aligned}
 W_3&=E_3,\\[2pt]
 W_4&=E_4-\frac{3(5N+7)}{5(N+1)}I_2^2,\\[2pt]
 W_5&=E_5-\frac{10(7N+13)}{7(N+1)}\,I_2\odot E_3,\\[2pt]
 W_6&=E_6-\frac{5(3N+7)}{N+1}\,I_2\odot E_4\\
    &\quad-\frac{7N+8}{14(N+1)}
           \bigl(4I_2\odot\Delta^2 I_2-5(\Delta I_2)^2\bigr)\\
    &\quad+\frac{30(7N^2+28N+25)}{7(N+1)^2}I_2^3.
 \end{aligned}
\end{equation}
Since $E_3$ and $E_4$ are linear in the normalized jets, the binary
products in \eqref{I:eq:app-W} give exactly the full symmetrization
prescribed in \eqref{I:eq:rp-Sym}. For commuting coefficients, replace
$\odot$ by ordinary multiplication and $\Delta$ by $\delta$.

The formulas for $I_m$ follow from the Bell expansion
\eqref{I:eq:local-I-formulas}. Projective normalization
\eqref{I:eq:rp-normalizer}--\eqref{I:eq:rp-q-definition}, followed by
symmetrization, gives \eqref{I:eq:app-W}. Their scalar specializations
agree with $E_m$ when $I_2$ and all its derivatives vanish; this is
the normalization used throughout the paper.

\clearpage
\section{Computational supplement}\label{app:computations}

The accompanying directory \texttt{anc/} contains the exact calculation
programs and their data in one package. It is independent of the
compilation of this article. The universal formulas through weight six
are given in Appendix~\ref{I:app:formulas}; the supplied programs concern
the curve and Siegel applications. They check finite identities and
explicit examples, and do not replace the proofs of the general
theorems.

\subsection{Curve and modular calculations}

The directory \texttt{anc/curves/} contains checks of the coordinate
and Ore conventions, the kernel cubic, ordered brackets, extension
recovery, and the chosen higher hierarchy. The subdirectory
\texttt{x023\_calculations/} contains the rational and Fourier
calculations for Section~\ref{II:sec:x023}. In that subdirectory the
series program is run before the programs that read its coefficients.
The rational certificate checks the displayed polynomial identities
for the extension map. All products use coefficients on the left;
intrinsic derivatives are ordinary, while modular derivatives use
$D=(2\pi i)^{-1}\partial_\tau$.

\subsection{Siegel tensors and scalarizations}

The directory \texttt{anc/siegel/} contains checks of the projected
calculus, full tensor identities, modular geodesics, and scalarizations.
Its \texttt{applications/} subdirectory contains the exact
Gauss--Manin reductions, primitive and null-line operators, rank-one
pencils, and Igusa substitutions of
Sections~\ref{III:sec:explicit-gm}--\ref{III:sec:igusa-applications}.
The application runner orders the calculations so that each generated
data file is available when needed.

The subdirectory \texttt{independent/} contains additional derivations
supplied with the source material and their attribution. Its
\texttt{gm.py} implements a separate B\'ezout reduction. These
calculations use exact rational or quadratic-field arithmetic.
In code, the variable \texttt{J} denotes the binary joint invariant
written $\mathfrak j$ in the article; the keys \texttt{J2},
\texttt{J4}, \texttt{J6}, and \texttt{J10} denote the binary-sextic
invariants. The residue polarization and left-coefficient Ore
conventions agree with the displayed formulas.

\subsection{Small worked examples and notation in code}
The program \texttt{anc/exposition\_checks.py} checks the local
matrix brackets of Examples~\ref{II:ex:ordered-self-brackets}
and~\ref{II:ex:third-self-bracket}, the
kernel expansion of Example~\ref{II:ex:small-kernel}, and the two
scalarizations of Example~\ref{III:ex:two-scalarizations} by exact
arithmetic. It is included in the combined runner.

The original calculation programs retain their established variable
names. In the curve programs, \texttt{A} is the kernel connection
matrix $\mathsf A$, and a cubic difference called \texttt{delta}
is $d_{LM}$ here. In the period programs, \texttt{K} is the pairing
matrix $\mathsf K$; a cyclic determinant written \texttt{W} in code
is $\operatorname{Wr}_s$ or $\mathcal W(v)$, according to its rank.
The projected scalar connection called \texttt{theta} is $\Theta$.
These translations change no normalization or computational input.

\subsection{Reproduction}

With Python~3 and SymPy installed, run the following commands from the
directory containing this article:
\begin{verbatim}
python -m pip install -r anc/requirements.txt
python anc/run_all.py
\end{verbatim}
The runner executes the suites in dependency order and records each
exit status and output in \texttt{anc/results/}. It also detects
explicit false comparisons printed by the independent checks.
The accompanying \texttt{README.md} lists the individual entry points.

\end{document}